\documentclass{amsproc-jjs}
\usepackage{amsmath,amssymb,amsthm,amsfonts,latexsym}
\usepackage{color}
\usepackage{mathrsfs} %\mathscr T
\usepackage{units}   %\nicefrac %kolejarski ulamek
\usepackage{graphicx}
\makeindex
\def\indexname{List of symbols}
\newcommand*{\idx}[2]{\index{{\hspace{-2ex}\color{white}#2}{#1}}{#1}}
\newcommand*{\idxx}[2]{\index{{\hspace{-2ex}\color{white}#2}{#1}}}
\theoremstyle{plain}
\newtheorem{thm}{Theorem}[subsection]
\newtheorem{pro}[thm]{Proposition}
\newtheorem{lem}[thm]{Lemma}
\newtheorem{cor}[thm]{Corollary}
\newtheorem*{thm*}{Theorem}
\newtheorem*{pro*}{Proposition}

\theoremstyle{definition}
\newtheorem{dfn}[thm]{Definition}
\newtheorem*{dfn*}{Definition}
\newtheorem{exa}[thm]{{\it Example}}
\newtheorem*{exa*}{{\it Example}}
\newtheorem{ozn}[thm]{Notation}

\theoremstyle{remark}
\newtheorem{rem}[thm]{{\it Remark}}
\newtheorem{rems}[thm]{{\it Remarks}}
\newtheorem*{rem*}{{\it Remark}}
\newtheorem{opis}[thm]{Procedure}

\DeclareMathOperator{\dzii}{{\mathsf{Chi}}}
\DeclareMathOperator{\dess}{{\mathsf{Des}}}
\DeclareMathOperator{\paa}{{\mathsf{par}}}

\DeclareMathOperator{\koo}{{\mathsf{root}}}
\DeclareMathOperator{\indd}{{\mathrm{ind}}}
\DeclareMathOperator{\Koo}{{\mathsf{Root}}}
\DeclareMathOperator{\sumo}{{\sum\nolimits^{\smalloplus}}}
\DeclareMathOperator{\supp}{\mathrm{supp}}

\newcommand*{\betab}{{\boldsymbol \beta}}
\newcommand*{\borel}[1]{{\mathfrak B}(#1)}
\newcommand*{\cbb}{\mathbb C}

\newcommand*{\card}[1]{\mathrm{card}(#1)}
\newcommand*{\D}{{\mathrm d}}

\newcommand*{\des}[1]{{\dess(#1)}}
\newcommand*{\dest}[2]{{\dess_{#1}(#2)}}
\newcommand*{\dz}[1]{{\mathscr D}(#1)}
\newcommand*{\dzi}[1]{\dzii(#1)}
\newcommand*{\dziplus}[1]{\dzii_\lambdab^+(#1)}
\newcommand*{\dzin}[2]{\dzii^{\langle#1\rangle}(#2)}
\newcommand*{\dzint}[3]{\dzii_{#1}^{\langle#2\rangle}(#3)}
\newcommand*{\dzit}[2]{\dzii_{#1}(#2)}
\newcommand*{\E}{{\mathrm e}}
\newcommand*{\ee}{{\mathcal E}}
\newcommand*{\escr}{{\mathscr{E}_V}}
\newcommand*{\gcal}{{\mathcal G}}
\newcommand*{\Ge}{\geqslant}
\newcommand*{\hh}{{\mathcal H}}
\newcommand*{\I}{{\mathrm i}}
\newcommand*{\ind}[1]{{\indd(#1)}}
\newcommand*{\is}[2]{\langle#1,#2\rangle}
\newcommand*{\isB}[2]{{\Big\langle#1,#2\Big\rangle}}
\newcommand*{\jd}[1]{{\mathscr N}(#1)}
\newcommand*{\kk}{{\mathcal K}}
\newcommand*{\ko}[1]{\koo(#1)}
\newcommand*{\Ko}[1]{\Koo(#1)}
\newcommand*{\lambdab}{{\boldsymbol\lambda}}
\newcommand*{\Le}{\leqslant}
\newcommand*{\nbb}{{\mathbb N}}
\newcommand*{\nul}{\{0\}}
\newcommand*{\ob}[1]{{\mathscr R}(#1)}
\newcommand*{\ogr}[1]{{\boldsymbol B (#1)}}

\newcommand*{\pa}[1]{\paa(#1)}
\newcommand*{\pib}{{\boldsymbol \pi}}
\newcommand*{\rbb}{{\mathbb R}}

\newcommand*{\smalloplus}{\raise0pt\hbox{$\scriptscriptstyle \oplus$}}
\newcommand*{\slam}{S_{\boldsymbol \lambda}}
\newcommand*{\slamh}{S_{\hat {\boldsymbol \lambda}}}
\newcommand*{\slaml}[1]{S_{\lambdab_\leftarrow\!(#1)}}
\newcommand*{\slamr}[1]{S_{\lambdab_\rightarrow\!(#1)}}
\newcommand*{\smlam}{S_{|\boldsymbol \lambda|}}

\newcommand*{\speca}[1]{\sigma_{\mathrm{ap}}(#1)}
\newcommand*{\teb}{{\boldsymbol t}}
\newcommand*{\tbb}{{\mathbb T}}
\newcommand*{\tcal}{{\mathscr T}}
\newcommand*{\thetab}{\boldsymbol\theta}
\newcommand*{\vplus}{V_\lambdab^+}
\newcommand*{\zbb}{{\mathbb Z}}
\begin{document}
%IMUJ%\phantom{1}
%IMUJ%\vskip 3.8cm
%IMUJ%\phantom{1}
%IMUJ%\vskip .8cm
%IMUJ%\noindent{\rm This is an extended version
%IMUJ%of IMUJ PREPRINT \# ****/**.}

   \title[Weighted shifts on directed trees]
{Weighted shifts on directed trees}

   \author[Z.\ J.\ Jab{\l}o\'{n}ski]{Zenon J. Jab{\l}o\'{n}ski}

%IMUJ%\author[Jan Stochel (and ***)]{Jan Stochel (and
%IMUJ%***) \\ \vskip 1.3cm \hspace{7cm} IMUJ PREPRINT
%IMUJ%****/** \phantom{1} \\ \vskip 1cm}

   \address{Instytut Matematyki, Uniwersytet Jagiello\'{n}ski,
ul. \L ojasiewicza 6, PL-30348 Kra\-k\'ow}

   \email{Zenon.Jablonski@im.uj.edu.pl}

  \author[I.\ B.\ Jung]{Il Bong Jung}
   \address{Department of Mathematics, College of Natural
Sciences, Kyungpook National University, Daegu 702-701
Korea}
   \email{ibjung@knu.ac.kr}

   \author[J. Stochel]{Jan Stochel}

   \address{Instytut Matematyki, Uniwersytet Jagiello\'nski,
ul. \L ojasiewicza 6, PL-30348 Kra\-k\'ow}
\email{Jan.Stochel@im.uj.edu.pl}

\thanks{The first and the third authors were supported by
the MNiSzW grant N201 026 32/1350. The second author
was supported by Basic Science Research Program
through the National Research Foundation of Korea
(NRF) grant funded by the Korea government (MEST)
(2009-0093125).}



   \dedicatory{Dedicated to Professor Franciszek H.
Szafraniec on the occasion of his 70th birthday}
   \begin{abstract}
A new class of (not necessarily bounded) operators
related to (mainly infinite) directed trees is
introduced and investigated. Operators in question are
to be considered as a generalization of classical
weighted shifts, on the one hand, and of weighted
adjacency operators, on the other; they are called
weighted shifts on directed trees. The basic
properties of such operators, including closedness,
adjoints, polar decomposition and moduli are studied.
Circularity and the Fredholmness of weighted shifts on
directed trees are discussed. The relationships
between domains of a weighted shift on a directed tree
and its adjoint are described. Hyponormality,
cohyponormality, subnormality and complete
hyperexpansivity of such operators are entirely
characterized in terms of their weights. Related
questions that arose during the study of the topic are
solved as well. Particular trees with one branching
vertex are intensively studied mostly in the context
of subnormality and complete hyperexpansivity of
weighted shifts on them. A strict connection of the
latter with $k$-step backward extendibility of
subnormal as well as completely hyperexpansive
unilateral classical weighted shifts is established.
Models of subnormal and completely hyperexpansive
weighted shifts on these particular trees are
constructed. Various illustrative examples of weighted
shifts on directed trees with the prescribed
properties are furnished. Many of them are simpler
than those previously found on occasion of
investigating analogical properties of other classes
of operators.
   \end{abstract}
   \maketitle
   \newpage
\setcounter{tocdepth}{2}
\tableofcontents

%IMUJ%\newpage

   \numberwithin{equation}{section}
   \newpage
   \section{Introduction}
The main goal of this paper is to implement some methods of
graph theory into operator theory. We do it by introducing
a new class of operators, which we propose to call {\em
weighted shifts on directed trees}. This considerably
generalizes the notion of a weighted shift, the classical
object of operator theory (see e.g., \cite{shi} for a
beautiful survey article on bounded weighted shifts, and
\cite{ml} for basic facts on unbounded ones). As opposed to
the standard graph theory which concerns mostly finite
graphs (see e.g., \cite{ore,c-d-s}), we mainly deal with
infinite graphs, in fact infinite directed trees. Much part
of (non-selfadjoint) operator theory trivializes when one
considers weighted shifts on finite directed trees. This is
the reason why we have decided to include assorted facts on
infinite graphs. The specificity of operator theory forces
peculiarity of problems to be solved in graph theory. This
is yet another reason for studying infinite graphs.

Matrix theory is always behind graph theory:\ finite
undirected graphs induce adjacency matrices which are
always symmetric. However, if undirected graphs are
infinite, then we have to replace adjacency matrices
by symmetric operators (cf.\ \cite{mo,mo-wo}). It
turns out that adjacency operators may not be
selfadjoint (cf.\ \cite{mu}, see also \cite{ao}). If
we want to study non-selfadjoint operators, we have to
turn our interest to directed graphs, and replace the
adjacency matrix by an (in general, unbounded)
operator, called the adjacency operator of the graph.
This was done for the first time in \cite{f-s-w}. It
turns out that the adjacency operators (``which form a
small fantastic world'', cf.\ \cite{f-s-w}) can be
expressed as infinite matrices whose entries are $0$
or $1$. If we look at the definition of the adjacency
operator of a directed tree $\tcal$ (with bounded
valency), we find that it coincides with that of the
weighted shift $\slam$ on $\tcal$ with weights
$\lambda_v \equiv 1$ (see Definition \ref{defshift}
and Proposition \ref{desc}). The questions of when the
adjacency operator is positive, selfadjoint, unitary,
normal and (co-) hyponormal have been answered in
\cite{f-s-w} (characterizations of some algebraic
properties of adjacency operators have been given
there as well). Spectral and numerical radii of
adjacency operators have been studied in \cite{b-m-s}
(the case of undirected graphs) and in
\cite{f-s-w2,t-te} (the case of directed graphs).

The notion of adjacency operator has been generalized
in \cite[Section 6]{f-f-s-w} to the case of infinite
directed fuzzy graphs $G$ (i.e., graphs whose arrows
have stochastic values); such operator, denoted by
$A(G)$ in \cite{f-f-s-w} (and sometimes called a
weighted adjacency operator of $G$), is assumed to be
bounded. In view of Proposition \ref{desc}, it is a
simple matter to verify that if $G$ is a directed
tree, then the weighted adjacency operator $A(G)$
coincides with our weighted shift operator on $G$. It
was proved in \cite[Theorem 6.1]{f-f-s-w} that the
spectral radius of the weighted adjacency operator
$A(G)$ of an infinite directed fuzzy graph $G$ belongs
the approximate point spectrum of $A(G)$. Our approach
to this question is quite different. Namely, we first
prove that a weighted shift on a directed tree is
circular (cf.\ Theorem \ref{modul}), and then deduce
the Perron-Frobenius type theorem (cf.\ Corollary
\ref{sppr}). As an immediate consequence of
circularity, we obtain the symmetricity of the
spectrum of a weighted shift on a directed tree with
respect to the real axis.

We now explain why in the case of directed trees
we prefer to call a weighted adjacency operator a
{\em weighted shift on a directed tree}. In the
present paper, by a {\em classical weighted
shift} we mean either a unilateral weighted shift
$S$ in $\ell^2$ or a bilateral weighted shift $S$
in $\ell^2(\zbb)$ ($\zbb$ stands for the set of
all integers). To be more precise, $S$ is
understood as the product $VD$, where, in the
unilateral case, $V$ is the unilateral isometric
shift on $\ell^2$ of multiplicity $1$ and $D$ is
a diagonal operator in $\ell^2$ with diagonal
elements $\{\lambda_n\}_{n=0}^\infty$; in the
bilateral case, $V$ is the bilateral unitary
shift on $\ell^2(\zbb)$ of multiplicity $1$ and
$D$ is a diagonal operator in $\ell^2(\zbb)$ with
diagonal elements
$\{\lambda_n\}_{n=-\infty}^\infty$ (diagonal
operators are assumed to be closed, cf.\ Lemma
\ref{lems}). In fact, $S$ is a unique closed
linear operator in $\ell^2$ (respectively,
$\ell^2(\zbb)$) such that the linear span of the
standard orthonormal basis $\{e_n\}_{n=0}^\infty$
of $\ell^2$ (respectively,
$\{e_n\}_{n=-\infty}^\infty$ of $\ell^2(\zbb)$)
is a core\footnote{\;See Section \ref{s01} for
appropriate definitions.} of $S$ and
   \begin{align}     \label{notold}
S e_n = \lambda_n e_{n+1} \quad \text{for } n\in \zbb_+ \;
(\textrm{respectively, for } n \in \zbb),
   \end{align}
where $\zbb_+$ is the set of all nonnegative integers.
Roughly speaking, the operators $\slam$ which are
subject of our investigations in the present paper can
be described as follows (cf.\ \eqref{eu}):
   \begin{align}   \label{eu-1}
\slam e_u = \sum_{v\in\dzi u} \lambda_v e_v, \quad u \in V,
   \end{align}
where $\{e_v\}_{v\in V}$ is the standard orthonormal
basis of $\ell^2(V)$ indexed by a set $V$ of vertexes
of a directed tree $\tcal$, $\dzi u$ is the set of all
children of $u$ and $\{\lambda_v\}_{v\in V^\circ}$ is
a system of complex numbers called the weights of
$\slam$. If we apply this definition to the directed
trees $\zbb_+$ and $\zbb$ (see Remark \ref{re1-2} for
a detailed explanation), we will see that in this
particular situation the equality \eqref{eu-1} takes
the form
   \begin{align}  \label{notnew}
\slam e_n = \lambda_{n+1}e_{n+1} \quad \text{for } n\in
\zbb_+ \; (\textrm{respectively, } n \in \zbb).
   \end{align}
Comparing \eqref{notold} with \eqref{notnew}, one can
convince himself that the operator $\slam$ can be viewed as
generalization of a classical weighted shift operator. This
is the main reason why operators $\slam$ are called here
weighted shifts on directed trees.

The reader should be aware of the difference between
notation \eqref{notold} and \eqref{notnew}. {\em In the
present paper, we adhere to the new convention}
\eqref{notnew}.

It is well known that the adjoint of an injective
unilateral classical weighted shift is not a classical
weighted shift. It is somewhat surprising that the adjoint
of a unilateral classical weighted shift is a weighted
shift in our more general sense (cf.\ Remark \ref{surp}).

Hereafter, we study weighted shifts on directed
trees imposing no restrictions on their
cardinality. However, if one wants to investigate
densely defined weighted shifts with nonzero
weights, then one ought to consider them on
directed trees which are at most countable (cf.\
Proposition \ref{przeldz}).

Less than half of our paper, namely chapters
\ref{chap3} and \ref{chap4}, deals with unbounded
weighted shifts on directed trees. In Chapter
\ref{chap3}, we investigate the question of when
assorted properties of classical weighted shifts
remain valid for weighted shifts on directed trees.
The first basic property of classical weighted shifts
stating that each of them is unitarily equivalent to
another one with nonnegative weights has a natural
counterpart in the context of directed trees (cf.\
Theorem \ref{uni}). Circularity is another significant
property of classical weighted shifts which turns out
to be valid for their generalizations on directed
trees (cf.\ Theorem \ref{modul}). The adjoint and the
modulus of a weighted shift on a directed tree are
explicitly exhibited in Propositions \ref{sprz} and
\ref{3}, respectively. As a consequence, a clearly
expressed description of the polar decomposition of a
weighted shift on a directed tree is derived in
Proposition \ref{polar}. It enables us to characterize
Fredholm and semi-Fredholm weighted shifts on a
directed tree (cf.\ Propositions \ref{fredholm} and
\ref{semifred}). It turns out that the property of
being Fredholm, when considered in the class of
weighted shifts on a directed tree with nonzero
weights, can be stated entirely in terms of the
underlying tree. Such a tree is called here Fredholm
(cf.\ Definition \ref{treeind}). In general, if a
directed tree admits a Fredholm weight shift (with not
necessarily nonzero weights), then it is automatically
Fredholm, but not conversely (cf.\ Propositions
\ref{fredholm} and \ref{bfred}). Proposition
\ref{fredholm} provides an explicit formula for the
index of a Fredholm weighted shift on a directed tree
(cf.\ the formula \eqref{finf}). Owing to this
formula, the index depends on both the underlying tree
and the weights of the weighted shift in question (in
fact, it depends on the geometry of the set of
vertexes corresponding to vanishing weights). However,
if all the weights are nonzero, then the index depends
only on the underlying tree, and as such is called the
index of the Fredholm tree (cf.\ Definition
\ref{treeind}). The index of a Fredholm weighted shift
on a directed tree can take all integer values from
$-\infty$ to $1$ (cf.\ Lemma \ref{le1} and Theorem
\ref{indle1}).

The question of when the domain of a classical
weighted shift is included in the domain of its
adjoint has a simple answer. A related question
concerning the reverse inclusion has an equally simple
answer. However, the same problems, when formulated
for weighted shifts on a directed tree, become much
more elaborate. This is especially visible in the case
of the reverse inclusion in which we require that a
family of rank one perturbations of positive diagonal
operators be uniformly bounded; these operators are
tided up to the vertexes possessing children (cf.\
Theorem \ref{zen}). Some examples of unbounded
weighted shifts on a directed tree illustrating
possible relationships between the domain of the
operator in question and that of its adjoint are
stated in Example \ref{17rem}.

Starting from Chapter \ref{ch5}, we concentrate mainly
on the study of bounded operators. We begin by
considering the question of hyponormality. We first
show that a hyponormal weighted shift on a directed
tree with nonzero weights is injective, and
consequently that the underlying tree is leafless
(this no longer true if we admit zero weights, cf.\
Remark \ref{lam0}). A complete characterization of the
hyponormality of weighted shifts on directed trees is
given in Theorem \ref{hyp}. It turns out that the
property of being hyponormal is not too restrictive
with respect to the underlying tree (even in the class
of weighted shifts with nonzero weights). The
situation changes drastically when we pass to
cohyponormal weighted shifts on a directed tree. If
the tree has a root, then there is no nonzero
cohyponormal weighted shift on it. On the other hand,
if the tree is rootless and admits a nonzero
cohyponormal weighted shift, then the set of vertexes
corresponding to nonzero weights is a subtree of the
underlying tree which can be geometrically interpreted
as either a broom with infinite handle or a straight
line. This property is an essential constituent of the
characterization of cohyponormality of nonzero
weighted shifts on a directed tree that is given in
Theorem \ref{cohyp-opis}. As a consequence, any
injective cohyponormal weighted shift on a directed
tree is a bilateral classical weighted shift (cf.\
Corollary \ref{injcoh}).

The last section of Chapter \ref{ch5} is devoted to
showing how to separate hyponormality and
paranormality classes with weighted shifts on directed
trees. It is well known that the class of paranormal
operators is essentially larger than that of
hyponormal ones (see Section \ref{hypcohyp} for the
appropriate definition). This was deduced by Furuta
\cite{Fur} from the fact that there are non-hyponormal
squares of hyponormal operators. The first rather
complicated example of a hyponormal operator whose
square is not hyponormal was given by Halmos in
\cite{hal1} (however it is not injective). Probably
the simplest example of such operator which is
additionally injective is to be found in \cite[page
158]{I-W} (see \cite[Problem 209]{hal} for details).
One more example of this kind (with the injectivity
property), but still complicated, can be found in
\cite[Example]{di-ca}. In the present article, we
offer two examples of injective hyponormal weighted
shifts on directed trees whose squares are not
hyponormal. The first one, parameterized by three
independent real parameters, is built on a relatively
simple directed three that has only one branching
vertex (cf.\ Example \ref{exa3}). The other one,
parameterized by two real parameters, is built on a
directed tree which is a ``small perturbation'' of a
directed binary tree (cf.\ Example \ref{exa4}). Let us
point out that there are no tedious computations
behind our examples. According to our knowledge, the
first direct example (making no appeal to
non-hyponormal squares) of an injective paranormal
operator which is not hyponormal appeared in
\cite[Example 3.1]{b-j} (see also \cite{JLP,JLL} for
non-injective examples of this kind). In Example
\ref{pra-nothyp} we construct an injective paranormal
weighted shift on a directed tree which is not
hyponormal; the underlying directed tree is the
simplest possible directed tree admitting such an
operator (because there is no distinction between
hyponormality and paranormality in the class of
classical weighted shifts).

Chapter \ref{ch6} is devoted to the study of (bounded)
subnormal weighted shifts on directed trees. The main
characterization of such operators given in Theorem
\ref{charsub} asserts that a weighted shift $\slam$ on
a directed tree $\tcal = (V,E)$ is subnormal if and
only if each vertex $u \in V$ induces a Stieltjes
moment sequence, i.e., $\{\|\slam^n
e_u\|^2\}_{n=0}^\infty$ is a Stieltjes moment
sequence. Hence, it is natural to examine the set of
all vertexes which induce Stieltjes moment sequences.
Since the operator in question is bounded, the
Stieltjes moment sequence induced by $u \in V$ turns
out to be determinate; its unique representing measure
is denoted by $\mu_u$ (cf.\ Notation \ref{ozn2}). The
first question we analyze is whether the property of
inducing a Stieltjes moment sequence is inherited by
the children of a fixed vertex. In general, the answer
to the question is in the negative. The situation in
which the answer is in the affirmative happens
extremely rarely, actually, only when the vertex has
exactly one child (cf.\ Lemma \ref{charsub-1} and
Example \ref{2nitki}). This fact, when applied to the
leafless directed trees without branching vertexes,
leads to the well known Berger-Gellar-Wallen criterion
for subnormality of injective classical weighted
shifts (cf.\ Corollaries \ref{b-g-w} and
\ref{b-g-w-2}). Though the answer to the reverse
question is in the negative, we can find a necessary
and sufficient condition for a fixed vertex (read:\ a
parent) to induce a Stieltjes moment sequence whenever
its children do so (cf.\ Lemma \ref{charsub2}); this
condition is called the consistency condition. Lemma
\ref{charsub2} also gives a formula linking measures
induced by the parent and its children. The key
ingredient of its proof consists of Lemma \ref{bext}
which answers a variant of the question of backward
extendibility of Stieltjes moment sequences.

The usefulness of the consistency condition (as well
as the strong consistency condition) is undoubted.
This is particularly illustrated in the case of
directed trees $\tcal_{\eta,\kappa}$ that have only
one branching vertex (cf.\ \eqref{varkappa}). Such
trees are one step more complicated than those
involved in the definitions of classical weighted
shifts (see Remark \ref{re1-2}). Parameter $\eta$
counts the number of children of the branching vertex
of $\tcal_{\eta, \kappa}$, while $\kappa$ counts the
number of possible backward steps along the tree when
starting from its branching vertex. Employing Lemma
\ref{charsub2}, we first characterize the subnormality
of weighted shifts on $\tcal_{\eta,\kappa}$ with
nonzero weights (cf.\ Theorem \ref{omega} and
Corollary \ref{omega2}) and then build models for such
operators (cf.\ Section \ref{mod-sub}). According to
Procedure \ref{twolin}, to construct the model
weighted shift on $\tcal_{\eta,\kappa}$, we first take
a sequence $\{\mu_i\}_{i=1}^\eta$ of Borel probability
measures on a finite interval $[0,M]$, each of which
possessing finite negative moments up to order
$\kappa+1$ (cf.\ \eqref{0<infty}). The next step of
the procedure depends on whether $\kappa=0$ or $\kappa
\Ge 1$. In the first case, we choose any sequence
$\{\lambda_{i,1}\}_{i=1}^\eta$ of positive real
numbers satisfying the consistency condition
\eqref{zgod} and define the weights of the model
weighted shift by the formula \eqref{lamij}. In the
other case, we choose any sequence
$\{\lambda_{i,1}\}_{i=1}^\eta$ of positive real
numbers satisfying the strong consistency condition
\eqref{zgod'} and the estimate \eqref{fine}, and
define the weights of the model weighted shift by the
formulas \eqref{lamij}, \eqref{lambda-k} and
\eqref{lambda-k'}. The question of the existence of a
sequence $\{\lambda_{i,1}\}_{i=1}^\eta$ which meets
our requirements is answered in Lemma \ref{discus}.
Note that if $1 \Le \kappa < \infty$, then the weight
$\lambda_{-\kappa + 1}$ corresponding to the child of
the root of $\tcal_{\eta,\kappa}$ is not uniquely
determined by the sequences $\{\mu_i\}_{i=1}^\eta$ and
$\{\lambda_{i,1}\}_{i=1}^\eta$; it is parameterized by
a positive real number $\vartheta$ ranging over an
interval in which one endpoint is $0$ and the other is
uniquely determined by $\{\mu_i\}_{i=1}^\eta$ and
$\{\lambda_{i,1}\}_{i=1}^\eta$. If the parameter
$\vartheta$ coincides with the nonzero endpoint, the
corresponding subnormal weighted shift on
$\tcal_{\eta,\kappa}$ is called extremal. The
extremality can be expressed entirely in terms of the
weighted shift in question (cf.\ Remark
\ref{forwhile}). Procedure \ref{twolin} enables us to
link the issue of subnormality of weighted shifts
(with nonzero weights) on the directed tree
$\tcal_{\eta,\kappa}$ with the problem of $k$-step
backward extendibility of subnormal unilateral
classical weighted shifts which was originated by
Curto in \cite{cur} and continued in \cite{CL} (see
also \cite{HJL} and referenced cited in the paragraph
surrounding \eqref{Lee-Cu}). Roughly speaking, the
subnormality of a weighted shift on
$\tcal_{\eta,\kappa}$ with nonzero weights is
completely determined by the $(\kappa+1)$-step
backward extendibility of unilateral classical
weighted shifts which are tied up to the children of
the branching vertex via the formula \eqref{alla}
(cf.\ Proposition \ref{2curto}).

The class of completely hyperexpansive operators was
introduced by Aleman in \cite{Alem} on occasion of his
study of multiplication operators on Hilbert spaces of
analytic functions, and independently by Athavale in
\cite{ath} on account of his investigation of
operators which are antithetical to contractive
subnormal operators. We also point out the trilogy by
Stankus and Agler \cite{AgSt1,AgSt2,AgSt3} concerning
$m$-isometric transformations of a Hilbert space which
are always completely hyperexpansive whenever $m\Le
2$. Again, as in the case of subnormality, the
complete hyperexpansivity of a weighted shift $\slam$
on a directed tree $\tcal=(V,E)$ with nonzero weights
can be characterized by requiring that each vertex $u
\in V$ induces a completely alternating sequence,
i.e., $\{\|\slam^n e_u\|^2\}_{n=0}^\infty$ is a
completely alternating sequence (cf.\ Theorem
\ref{charch}). The structure of the set of all
vertexes of $\tcal$ inducing completely alternating
sequences is studied in two consecutive lemmas (cf.\
Lemmata \ref{charch-1} and \ref{charch2}). The first
of them deals with the question of whether the
property of inducing a completely alternating sequence
is inherited by the children of a fixed vertex. The
answer is exactly the same as in the case of
subnormality. In the latter lemma we formulate a
necessary and sufficient condition for a fixed vertex
$u \in V$ to induce a completely alternating sequence
whenever its children do so; this condition is again
called the consistency condition, but now it is
written in terms of representing measures of
completely alternating sequences $\{\|\slam^n
e_v\|^2\}_{n=0}^\infty$, where $v$ ranges over the set
$\dzi{u}$ of all children of $u$. The proof of Lemma
\ref{charch2} rely on Lemma \ref{bext-ca} which solves
the question of backward extendibility of completely
alternating sequences and provides the formula for
representing measures of backward extensions of a
given completely alternating sequence. As a
consequence, we obtain a formula binding representing
measures of completely alternating sequences induced
by the parent and its children (cf.\ Lemma
\ref{charch2}).

As in the case of subnormality, the directed tree
$\tcal_{\eta,\kappa}$ serves as a good test for the
applicability of Lemmata \ref{charch-1} and
\ref{charch2}. What we get are the characterizations
of complete hyperexpansivity of weighted shifts on
$\tcal_{\eta,\kappa}$ with nonzero weights (cf.\
Theorem \ref{omega-ch} and Corollary \ref{omega2-ch}).
In opposition to subnormality, the only completely
hyperexpansive weighted shifts on
$\tcal_{\eta,\infty}$ with nonzero weights are
isometries. This is the reason why in further parts of
the paper we consider only the case when $\kappa$ is
finite. Modelling of complete hyperexpansivity of
weighted shifts on $\tcal_{\eta,\kappa}$ with nonzero
weights, though still possible, is much more
elaborate. The procedure leading to a model weighted
shifts on $\tcal_{\eta,\kappa}$ starts with a sequence
$\boldsymbol \tau =\{\tau_i\}_{i=1}^\eta$ of positive
Borel measures on $[0,1]$ whose total masses are
uniformly bounded (these measures eventually
represents completely alternating sequences induced by
the children of the branching vertex). The next step
of the procedure requires much more delicate
reasoning. It depends on the behaviour of weights of a
completely hyperexpansive weighted shift on
$\tcal_{\eta,\kappa}$ corresponding to the children of
the branching vertex. They must satisfy the conditions
\eqref{ka+1}, \eqref{ka+2} and \eqref{ka+3} which are
rather complicated and somewhat difficult to deal with
(cf.\ Lemma \ref{omega3-ch}). This means that if we
want $\{t_{i}\}_{i=1}^\eta$ to be a sequence of
weights of some completely hyperexpansive weighted
shift on $\tcal_{\eta,\kappa}$ that correspond to the
children of the branching vertex, it must verify the
conditions \eqref{ka+1}, \eqref{ka+2} and
\eqref{ka+3}. Theorem \ref{par-ch} asserts that the
above necessary conditions turn out to be sufficient
as well. However, what remains quite unclear is under
what circumstances a sequence $\{t_{i}\}_{i=1}^\eta$
satisfying these three conditions exits. The solution
of this problem is given in Proposition \ref{nascfch}.
It is unexpectedly simple: each measure $\tau_i$ must
have a finite negative moment of order $\kappa+1$ and
at least one of them must generate a unilateral
classical weighted shift possessing a completely
hyperexpansive $(\kappa+1)$-step backward extension
(see \eqref{jab-ju-st} for an explanation). This is
another significant difference between complete
hyperexpansivity and subnormality because in the
latter case each unilateral classical weighted shift
generated by the child of the branching vertex must
possess subnormal $(\kappa+1)$-step backward extension
(compare Propositions \ref{2curto-ch} and
\ref{2curto}). The problem of $k$-step backward
extendibility of completely hyperexpansive unilateral
classical weighted shifts was investigated in
\cite{j-j-s}. The whole process of modelling complete
hyperexpansivity on $\tcal_{\eta,\kappa}$ is
summarized in Procedure \ref{proc-ch}.

Section \ref{s7.4} deals with the question of when for
a given sequence $\{t_i\}_{i=1}^\eta$ of positive real
numbers there exists a completely hyperexpansive
weighted shift on $\tcal_{\eta,\kappa}$ whose weights
corresponding to the children of the branching vertex
form the sequence $\{t_i\}_{i=1}^\eta$. The necessary
and sufficient conditions for that are given in
Propositions \ref{norm-sim} and \ref{kap-ch},
respectively.

Chapter \ref{ch7} ends with Section \ref{s751} which
concerns the issue of extendibility of a system of
weights of a completely hyperexpansive weighted shift
on a subtree $\tcal$ of a directed tree $\hat \tcal$
to a system of weights of some completely
hyperexpansive weighted shift on $\hat \tcal$ (both
weighted shifts are assumed to have nonzero weights).
In many cases such a possibility does not exist.
Similar effect appears in the case of subnormal
weighted shifts, however the assumptions imposed on
the pair $(\tcal,\hat\tcal)$ in the former case are
much more restrictive than those in the latter
(compare Propositions \ref{maxsub} and
\ref{maxsub-ch}). Example \ref{ch-n0restr} illustrates
the validity of the phrase ``much more restrictive''
as well as shows that none of the assumptions (i) and
(ii) of Proposition \ref{maxsub-ch} can be removed.

In the last chapter of the paper (i.e., Chapter
\ref{ch8}) we discuss the question of when a directed
tree admits a weighted shift with a prescribed
property (dense range, hyponormality, subnormality,
normality, etc.) and characterize $p$-hyponormality of
weighted shifts on directed trees (cf.\ Theorem
\ref{p-hyp}). In Example \ref{p-hyp-sep}, we single
out a family of weighted shifts on $\tcal_{2,1}$ (with
nonzero weights) in which $\infty$-hyponormality and
subnormality are proved to be independent (modulo an
operator). The same family is used to show how to
separate $p$-hyponormality classes. Note also that
$p$-hyponormal unilateral or bilateral classical
weighted shifts are always hyponormal (cf.\ Corollary
\ref{p-hyp-clas}).

We now make two concluding remarks. First, we note
that a weighted shift on a rootless directed tree
$\tcal$ is a weighted composition operator on $L^2$
space with respect to the counting measure on the set
of vertexes of $\tcal$ (cf.\ Definition
\ref{defshift}). The next observation is that a
weighted shift on a directed tree can be viewed as a
weighted shift with operator weights on one of the
following simple directed trees
   \begin{align*}
\text{$\zbb_+$, $\zbb$, $\zbb_-$ and
$\{1,\ldots,\kappa\}$ ($\kappa < \infty$).}
   \end{align*}
This can be inferred from a decomposition of a
directed tree described in the conditions (vi) and
(viii) of Proposition \ref{generation}. In general,
the $n$th operator weight is an unbounded operator
acting between different Hilbert spaces whose
dimensions vary in $n$.

\vspace{5ex}

In this paper we use the following notation. The
fields of real and complex numbers are denoted by
\idx{$\rbb$}{01} and \idx{$\cbb$}{02}, respectively.
The symbols \idx{$\zbb$}{03}, \idx{$\zbb_+$}{04} and
\idx{$\nbb$}{05} stand for the sets of integers,
nonnegative integers and positive integers,
respectively. Given a topological space $X$, we write
\idx{$\borel X$}{06} for the $\sigma$-algebra of all
Borel subsets of $X$. If $\zeta \in X$, then
\idx{$\delta_\zeta$}{07} stands for the Borel
probability measure on $X$ concentrated on
$\{\zeta\}$. We denote by \idx{$\chi_Y$}{08} and
\idx{$\card Y$}{09} the characteristic function and
the cardinal number of a set $Y$, respectively (it is
clear from the context on which set the characteristic
function $\chi_Y$ is defined).

   \numberwithin{equation}{subsection}
   \newpage
   \section{Prerequisites}
   \subsection{\label{s0}Directed trees}
   Since the graph theory is mainly devoted to the
study of finite graphs and our paper deals mostly with
infinite graphs, we have decided to include in this
section some basic notions and facts on the subject
which are essential for the rest of the paper. For the
basic concepts of the theory of graphs, we refer the
reader to \cite{ore}. We say that a pair $\gcal=(V,E)$
is a {\em directed graph} if $V$ is a nonempty set and
$E$ is a subset of $V\times V\setminus\{(v,v)\colon
v\in V\}$. Put \idxx{$\tilde E$}{10}
   \begin{align*}
\widetilde E = \{\{u,v\} \subseteq V \colon (u,v) \in E
\text{ or } (v,u) \in E\}.
   \end{align*}
For simplicity, we suppress the explicit dependence of
$V$, $E$ and $\widetilde E$ on $\gcal$ in the
notation. An element of $V$ is called a {\em vertex}
of $\gcal$, a member of $E$ is called an {\em edge} of
$\gcal$, and finally a member of $\widetilde E$ is
called an {\em undirected edge}. If $W$ is a nonempty
subset $V$, then obviously the pair
\idxx{${\gcal}_W$}{10}
   \begin{align} \label{subtree}
\gcal_W := (W, (W \times W) \cap E)
   \end{align}
is a directed graph which will be called a (directed)
{\em subgraph} of $\gcal$. A directed graph $\gcal$ is
said to be {\em connected} if for any two distinct
vertexes $u$ and $v$ of $\gcal$ there exists a finite
sequence $v_1, \ldots,v_n$ of vertexes of $\gcal$ ($n
\Ge 2$) such that $u=v_1$, $\{v_j,v_{j+1}\} \in
\widetilde E$ for all $j = 1, \ldots, n-1$, and
$v_n=v$; such a sequence will be called an {\em
undirected path} joining $u$ and $v$. Set \idxx{$\dzi
u$}{11}
   \begin{align*}
\dzi u = \{v\in V\colon (u,v)\in E\}, \quad u \in V.
   \end{align*}
A member of $\dzi u$ is called a {\em child} of $u$.
If for a given vertex $u \in V$, there exists a unique
vertex $v\in V$ such that $(v,u)\in E$, then we say
that $u$ has a parent $v$ and write \idx{$\pa u$}{12}
for $v$. Since the correspondence $u \mapsto \pa u$ is
a partial function (read:\ a relation) in $V$, we can
compose it with itself $k$-times ($k \Ge 1$); the
result is denoted by \idx{$\paa^k$}{13}. We adhere to
the convention that $\paa^0$ is the identity mapping
on $V$. We will write $\paa^k(u)$ only when $u$ is in
the domain of $\paa^k$. A finite sequence
$\{u_j\}_{j=1}^n$ of distinct vertexes is said to be a
{\em circuit} of $\gcal$ if $n\ge 2$, $(u_j,u_{j+1})
\in E$ for all $j=1, \ldots, n-1$, and $(u_n,u_1) \in
E$. A vertex $v$ of $\gcal$ is called a {\em root} of
$\gcal$, or briefly \idxx{$\Ko \gcal$}{14} $v \in \Ko
\gcal$, if there is no vertex $u$ of $\gcal$ such that
$(u,v)$ is an edge of $\gcal$. Clearly, the
cardinality of the set $\Ko \gcal$ may be arbitrary.
If $\Ko \gcal$ is a one-element set, then its unique
element is denoted by \idx{$\ko \gcal$}{15}, or simply
by \idx{$\koo$}{16} if this causes no ambiguity. We
write \idxx{$V^\circ$}{17} $V^\circ = V \setminus
\Ko{\gcal}$.

The proof of the following fact is left to the
reader\footnote{\;All facts stated in this section without
proofs can be justified by induction or methods employed in
the proof of Proposition \ref{witr}.}.
   \begin{pro}\label{1}
Let $\gcal$ be a directed graph satisfying the following conditions
   \begin{enumerate}
   \item[(i)] $\gcal$ is connected,
   \item[(ii)] each vertex $v \in V^\circ$ has a parent.
   \end{enumerate}
Then the set $\Ko \gcal$ contains at most one element.
   \end{pro}
We say that a directed graph $\tcal$ is a {\em
directed tree} if it has no circuits and satisfies the
conditions (i) and (ii) of Proposition \ref{1}. Note
that none of these three properties defining the
directed tree follows from the others. A subgraph of a
directed tree $\tcal$ which itself is a directed tree
will be called a {\em subtree} of $\tcal$. A directed
tree may or may not possess a root, however, in the
other case, a root must be unique. Note also that each
finite directed tree always has a root. The reader
should be aware of the fact that we impose no
restriction on the cardinality of the set $V$. A
directed tree $\tcal$ such that $\card {\dzi u} = 2$
for all $u \in V$ will be called a {\em directed
binary tree}.

Given a directed tree $\tcal$, we put
\idxx{$V^\prime$}{18} $V^\prime = \{u \in V \colon
\dzi u \neq \varnothing\}$ and \idxx{$V_{\prec}$}{19}
   \begin{align} \label{prec}
V_{\prec}=\{u \in V\colon \card{\dzi u} \Ge 2\}.
   \end{align}
A member of the set $V\setminus V^\prime$ is called a {\em
leaf} of $\tcal$, while a member of the set $V_{\prec}$ is
called a {\em branching vertex} of $\tcal$. A directed tree
$\tcal$ is said to be {\em leafless} if $V=V^\prime$. Every
leafless directed tree is infinite, and every directed
binary tree is leafless.

The following decomposition of the set $V^\circ$ plays
an important role in our further investigations. Its
proof is left to the reader.
   \begin{pro} \label{46}
If $\tcal$ is a directed tree, then $\dzi u \cap \dzi
v = \varnothing$ for all $u, v\in V$ such that $u \neq
v$, and\,\footnote{\;The notation ``\,$\bigsqcup$\,''
is reserved to denote pairwise disjoint union of
sets.}
   \begin{align} \label{sumchi}
V^\circ= \bigsqcup_{u\in V} \dzi u.
   \end{align}
   \end{pro}
Let $\tcal$ be a directed tree. Given a set $W
\subseteq V$, we put \idxx{$\dzi W$}{20} $\dzi W =
\bigsqcup_{v \in W} \dzi v$ (in view of Proposition
\ref{46}, $\dzi W$ is well-defined). Define
\idxx{$\dzii^{<n>}(W)$}{21} \idxx{$\des W$}{22}
    \begin{gather} \notag
\dzin{0}{W} = W, \quad \dzin{n+1}{W} = \dzi{\dzin{n}{W}}, \quad
n=0,1,2,\ldots,
    \\
\des W = \bigcup_{n=0}^\infty \dzin n W. \label{ytag}
    \end{gather}
The members of $\des W$ are called {\em descendants}
of $W$. Since $\dzi{\cdot}$ is a monotonically
increasing set-function, so is $\dzin{n}{\cdot}$. As a
consequence, we have
   \begin{align} \label{monot}
W_1\subseteq W_2 \subseteq V \implies \des{W_1}
\subseteq \des{W_2}.
   \end{align}
An induction argument shows that
   \begin{align} \label{n+1}
\dzin{n+1}{W} = \bigcup_{v \in \dzi{W}} \dzin{n}{\{v\}}.
   \end{align}
In general, the sets $\dzin n W$, $n=0,1,2,\ldots$,
are not pairwise disjoint. For $u \in V$, we shall
abbreviate $\dzin n {\{u\}}$ and $\des {\{u\}}$ to
$\dzin n u$ \idxx{$\dzii^{<n>}(u)$}{23} and
\idx{$\des{u}$}{24}, respectively. It is clear that
(use an induction argument)
   \begin{align} \label{minimality}
\textrm{$\des u \subseteq W$ whenever $\dzi W
\subseteq W$ and $u \in W$,}
   \end{align}
which means that $\des u$ is the smallest subset of $V$
which is ``invariant'' for $\dzi {\cdot}$ and which
contains $u$ (cf.\ \eqref{dziinv} below). Since, by
\eqref{ytag},
   \begin{align} \label{dziinv}
\dzi{\des u} = \bigcup_{n=1}^\infty \dzin n u
\subseteq \des u, \quad u \in V,
   \end{align}
we get
    \begin{align}\label{inv}
\des{\des u} = \des u, \quad u \in V.
    \end{align}
It follows from the definition of the partial function
$\paa$ and the fact that $\tcal$ has no circuits that the
sets $\dzin n u$, $n=0,1,2,\ldots$, are pairwise disjoint,
and hence
    \begin{align} \label{decom}
\des u= \bigsqcup_{n=0}^\infty \dzin n u, \quad u\in V.
    \end{align}
It may happen that $\dzin n u=\varnothing$ for some $u
\in V$ and $n \Ge 1$. In what follows, we also use the
modified notation \idx{$\dzit\tcal u$}{25},
$\dzint{\tcal}{n}{u}$
\idxx{$\dzii_{\tcal}^{<n>}(u)$}{25} and
\idx{$\dest\tcal u$}{26} in order to make clear the
dependence of $\dzi u$, $\dzin{n}{u}$ and $\des u$ on
the underlying directed tree $\tcal$.

The following simple observation turns out to be useful
(its proof is left to the reader).
   \begin{lem} \label{parchi}
If $\tcal$ is a directed tree and $X\subseteq V$ is such
that $\pa x \in X$ for all $x \in X$, then
   \begin{align*}
u \in V \setminus X \implies \dzi u \cap \big(X \cup
\dzi{X}\big) = \varnothing.
   \end{align*}
   \end{lem}
We now prove an important property of directed trees.
    \begin{pro} \label{witr}
If $\tcal$ is a directed tree, then for every finite subset
$W$ of $V$ there exists $u \in V$ such that $W \subseteq
\des u$.
    \end{pro}
    \begin{proof}
If $W = \{w\}$ with some $w \in V$, then setting $u=w$
does the job. If $W=\{a, b\}$ with some distinct $a,b
\in V$, then we proceed as follows. Since $\tcal$ is
connected, there exists a finite sequence $v_1,
\ldots, v_n$ of vertexes of $\gcal$ ($n \Ge 2$) such
that $a=v_1$, $\{v_j,v_{j+1}\} \in \widetilde E$ for
all $j = 1, \ldots, n-1$, and $v_n=b$. Denote by
$\mathcal W(a,b)$ the set of all such sequences.
Without loss of generality we can assume that the
length $n$ of our sequence $v_1, \ldots, v_n$ is the
smallest among lengths of all sequences from $\mathcal
W(a,b)$. We first show that the vertexes $v_1, \ldots,
v_n$ are distinct. Suppose that, contrary to our
claim, there exist $i,j \in \{1, \ldots, n\}$ such
that $i<j$ and $v_i=v_j$. Since evidently the sequence
$v_1, \ldots, v_i, v_{j+1}, \ldots, v_n$ belongs to
$\mathcal W(a,b)$, we are led to a contradiction.

We now consider two disjunctive cases which cover all
possibilities.

{\sc Case 1.} Suppose first that $(v_1,v_2) \in E$. Then
there exists the largest integer $k \in \{2, \ldots, n\}$
such that $(v_{j-1},v_j) \in E$ for all $j \in \{2, \ldots,
k\}$. We claim that $k=n$. Indeed, otherwise by the
maximality of $k$, $(v_{k+1}, v_k) \in E$, which together
with $(v_{k-1}, v_k) \in E$ and $v_{k-1} \neq v_{k+1}$,
contradicts the definition of $\pa {v_k}$. Hence, $v_{j-1} =
\pa {v_j}$ for all $j \in \{2, \ldots, n\}$, which implies
that $b \in \dzin {n-1} a \subseteq \des a$. Then $u=a$
meets our requirements.

{\sc Case 2.} Assume now that Case 1 does not hold. This
implies that $(v_2,v_1) \in E$. Then there exists the
largest integer $p \in \{2, \ldots, n\}$ such that $(v_j,
v_{j-1}) \in E$ for all $j \in \{2, \ldots, p\}$. If $p=n$,
then $v_j = \pa {v_{j-1}}$ for all $j \in \{2, \ldots,
n\}$, which yields $a \in \dzin {n-1} b \subseteq \des b$.
Therefore $u=b$ meets our requirements. In turn, if $p <
n$, then by the maximality of $p$, $(v_p, v_{p+1}) \in E$.
Arguing as in Case 1, we show that $b \in \dzin {n-p} {v_p}
\subseteq \des {v_p}$. Since $v_j = \pa {v_{j-1}}$ for all
$j \in \{2, \ldots, p\}$, we see that $a \in \dzin {p-1}
{v_p} \subseteq \des {v_p}$. Hence $a,b \in \des{v_p}$,
which completes the proof of the case when $W$ is a
two-point set.

Finally, we have to consider $W$ of cardinality $m$,
which is greater than $2$. We use an induction on $m$.
If $\widehat W =W \cup \{w\}$ for some $w \notin W$,
then by the induction hypothesis, there exists $u \in
V$ such that $W \subseteq \des u$. By the first part
of the proof, there exists $u^\prime \in V$ such that
$u,w \in \des {u^\prime}$. Thus, by \eqref{inv}, we
have
    \begin{align*}
\widehat W = W \cup \{w\} \subseteq \des u \cup
\des{u^\prime} \subseteq \des{\des{u^\prime}} \cup
\des{u^\prime} = \des{u^\prime}.
    \end{align*}
This completes the proof.
    \end{proof}
    \begin{cor} \label{przem}
If $\tcal$ is a directed tree with root, then
   \begin{align*}
V = \des{\koo} = \bigsqcup_{n=0}^\infty \dzin n
{\koo}.
   \end{align*}
   \end{cor}
   \begin{proof}
If $w \in V$, then by Proposition \ref{witr} there
exists $u \in V$ such that $\{\koo, w\} \subseteq \des
u$. By \eqref{decom}, this implies that $u=\koo$, and
consequently $w \in \des \koo$. An application of
\eqref{decom} completes the proof.
    \end{proof}
It turns out that the set $V$ can be described with the
help of the operation $\des{\cdot}$ even in the case of a
rootless directed tree.
   \begin{pro} \label{xdescor}
Let $\tcal$ be a rootless directed tree and $u\in V$.
Then
   \begin{enumerate}
   \item[(i)] $\paa^k(u)$ make sense for all  $k \in \nbb$,
$\paa^k(u) \neq \paa^l(u)$ for all nonnegative
integers $k\neq l$,
   \item[(ii)] $\des{\paa^l(u)} \subseteq \des {\paa^j(u)}$ for all
nonnegative integers $l<j$,
   \item[(iii)] $V=\bigcup_{k\in J} \des{\paa^k(u)}$ for every
infinite subset $J$ of $\nbb$,
   \item[(iv)] if $\card{\dzi{\paa^k(u)}}=1$ for all $k \in
\nbb$, then $V = \{\paa^{k}(u)\}_{k=1}^\infty \sqcup
\des u$.
   \end{enumerate}
   \end{pro}
   \begin{proof}
Condition (i) follows from \eqref{decom} and the fact
that $\tcal$ is rootless.

(iii) We only have to prove the inclusion
``$\subseteq$''. Take $v\in V$. Then, by Proposition
\ref{witr}, there exists $w \in V$ such that $v, u \in
\des w$. Owing to \eqref{decom}, there exists a unique
$l \in \zbb_+$ such that $u \in \dzin{l}{w}$. This
implies that $w=\paa^l(u)$. Since $J$ is infinite,
there exists $j \in J$ such that $l <j$. Hence
$\paa^l(u) \in \dzin{j-l}{\paa^j(u)} \subseteq
\des{\paa^j(u)}$, which implies that
   \begin{align*}
v \in \des{\paa^l(u)} \overset{\eqref{monot}}\subseteq
\des{\des {\paa^j(u)}} \overset{\eqref{inv}}= \des
{\paa^j(u)} \subseteq \bigcup_{k\in J}
\des{\paa^k(u)}.
   \end{align*}
Looking more closely at the last line, we get (ii).

(iv) Put $W=\{\paa^{n}(u)\}_{n=1}^\infty \sqcup \des
u$. It follows from \eqref{dziinv} that $\dzi{W}
\subseteq W$. Hence, by \eqref{minimality},
$\des{\paa^k(u)} \subseteq W$ for all $k \in \nbb$.
This combined with the equality
$V=\bigcup_{k=1}^\infty \des{\paa^k(u)}$ (see (iii))
yields $V = \{\paa^{n}(u)\}_{n=1}^\infty \sqcup \des
u$, which completes the proof.
   \end{proof}
Using Corollary \ref{przem} and arguing as in the proof of
Proposition \ref{xdescor}\,(iv), we obtain a version of the
latter for a directed tree with root.
   \begin{pro}\label{xdescor2}
Let $\tcal$ be a directed tree with root, and let $u \in
V^\circ$. Then there exists a unique $m \in \nbb$ such that
$\paa^m(u)=\koo$; moreover, $\paa^k(u) \neq \paa^l(u)$ for
all $k,l \in \{0, \ldots, m\}$ such that $k\neq l$. If
$\card{\dzi{\paa^j(u)}}=1$ for all $j \in \{1, \ldots,
m\}$, then $V = \{\paa^{j}(u)\colon j=1, \ldots,m\} \sqcup
\des u$.
   \end{pro}
As will be shown below, descenders of a fixed vertex
generate a decomposition of a directed tree. The reader is
referred to \eqref{subtree} for the necessary notation.
    \begin{pro} \label{subdir}
Let $\tcal$ be a directed tree and $u \in V.$ Then
   \begin{enumerate}
   \item[(i)] \idx{$\tcal_{\des u}$}{27} is a directed tree with the root $u$,
   \item[(ii)] \idx{$\tcal_{V \setminus \des u}$}{28} is a
directed tree provided $V \setminus \des u \neq
\varnothing$; moreover, if the directed tree $\tcal$ has a
root, then so does $\tcal_{V \setminus \des u}$ and
   \begin{align*}
\ko\tcal = \ko{\tcal_{V \setminus \des u}}.
   \end{align*}
   \end{enumerate}
In particular, if $\des u =\des v$ for some $v\in V$, then
$u=v$.
    \end{pro}
    \begin{proof}
   Certainly, the directed graphs $\tcal_{\des u}$ and
$\tcal_{V \setminus \des u}$ satisfy the condition (ii) of
Proposition \ref{1}, and they have no circuits. We show
that both of them are connected which will imply that they
are directed trees. Suppose that $u_1$ and $u_2$ are
distinct elements of $\tcal$. By Proposition \ref{witr},
there exists $w \in V$ such that $\{u_1,u_2\} \subseteq
\des w$. It follows from \eqref{decom} that there exist
integers $m_1,m_2 \Ge 0$ such that $u_1 \in \dzin {m_1}{w}$
and $u_2 \in \dzin {m_2}{w}$. Since $\pa{\dzin {k+1}w}
\subseteq \dzin {k}w$ for all integers $k \Ge 0$, we see
that $\{\paa^i(u_1)\colon i=0, \ldots, m_1\} \subseteq \des
w$, $\{\paa^j(u_2)\colon j=0, \ldots, m_2\} \subseteq \des
w$ and $ \paa^{m_1}(u_1) = w = \paa^{m_2}(u_2)$. This means
that the sequence
   \begin{align} \label{sciecha}
\paa^0(u_1), \paa^1(u_1), \ldots, \paa^{m_1-1}(u_1), w,
\paa^{m_2-1}(u_2), \ldots, \paa^1(u_2), \paa^0(u_2)
   \end{align}
is an undirected path joining $u_1$ and $u_2$. If $u_1,u_2
\in \des u$, then applying the above to $w=u$, we see that
$\tcal_{\des u}$ is connected. If $u_1,u_2 \in V \setminus
\des u$, then no vertex of the undirected path
\eqref{sciecha} belongs to $\des u$ which can be deduced
from \eqref{inv}. Hence the graph $\tcal_{V \setminus \des
u}$ is connected.

Suppose that, contrary to our claim, $u$ is not a root
of $\tcal_{\des u}$. Then there exists $v \in \des u$
such that $u \in \dzi v$. Since, by \eqref{decom},
there exists an integer $n \Ge 0$ such that $v\in
\dzin n u$, we see that $u\in \dzin {n+1} u \cap \dzin
0 u$, which contradicts \eqref{decom}. Thus, by
Proposition \ref{1}, $u = \ko{\tcal_{\des u}}$, which
implies the ``in particular'' part of the conclusion.
The ``moreover'' part of (ii) is easily seen to be
true. This completes the proof.
   \end{proof}
   \begin{rem} \label{dzi}
Regarding Proposition \ref{subdir}, note that for every $v
\in \des u$, the set of all children of $v$ counted in the
graph $\tcal_{\des u}$ is equal to $\dzi v$. In turn, if $v
\in V \setminus \des u$, then the set of all children of
$v$ counted in the graph $\tcal_{V \setminus \des u}$ is
equal to either $\dzi v$ if $v \neq \pa u$, or $\dzi
v\setminus \{u\}$ otherwise.
   \end{rem}
A subtree $\tcal$ of a directed tree $\hat\tcal$
containing all $\hat\tcal$-descendants of each vertex
of $\tcal$ can be characterized as follows.
   \begin{pro}\label{preheredit}
Let $\tcal = (V,E)$ be a subtree of a directed tree
$\hat\tcal=(\hat V,\hat E)$. Then the following conditions
are equivalent\/{\em :}
   \begin{enumerate}
   \item[(i)] $\dzit{\tcal}{u} =  \dzit{\hat\tcal}{u}$ for all
$u \in V$,
   \item[(ii)] $\dzit{\hat\tcal}{u} \subseteq V$ for all
$u \in V$,
   \item[(iii)] $\dest{\tcal}{u} =
\dest{\hat\tcal}{u}$ for all $u \in V$,
   \item[(iv)] $\dest{\hat\tcal}{u} \subset V$
for all $u \in V$,
   \item[(v)] $V =
   \begin{cases}
   \dest{\hat\tcal}{\ko{\tcal}} & \text{if $\tcal$ has
a root,}
   \\
   \hat V & \text{if $\tcal$ is rootless.}
   \end{cases}$
   \end{enumerate}
   \end{pro}
   \begin{proof}
(i)$\Rightarrow$(iii) An induction argument shows that
$\dzint{\tcal}{n}{u}=\dzint{\hat\tcal}{n}{u}$ for all $n
\in \zbb_+$ and $u \in V$. This and \eqref{decom}, applied
to $\tcal$ and $\hat\tcal$, lead to (iii).

(iii)$\Rightarrow$(v) If $\tcal$ has a root, then we apply
(iii) to $u=\ko{\tcal}$, and then Corollary \ref{przem} to
$\tcal$. If $\tcal$ is rootless, then employing Proposition
\ref{xdescor}\,(iii) to $\tcal$ and $\hat\tcal$, we get
   \begin{align*}
V = \bigcup_{k=1}^\infty \dest{\tcal}{\paa^k(u)}
\overset{\mathrm{(iii)}}= \bigcup_{k=1}^\infty
\dest{\hat\tcal}{\paa^k(u)} = \hat V, \quad u \in V.
   \end{align*}

(v)$\Rightarrow$(ii) If $\tcal$ has a root, then by
\eqref{dziinv}, applied to $\hat\tcal$, we have
   \begin{align*}
\dzit{\hat\tcal}{u} \subseteq \dzit{\hat\tcal}{V}
\overset{\mathrm{(v)}}=
\dzit{\hat\tcal}{\dest{\hat\tcal}{\ko{\tcal}}}
\subseteq \dest{\hat\tcal}{\ko{\tcal}}
\overset{\mathrm{(v)}}= V, \quad u \in V.
   \end{align*}
The other case is trivially true.

Since the implications (iii)$\Rightarrow$(iv),
(iv)$\Rightarrow$(ii) and (ii)$\Rightarrow$(i) are
obvious, the proof is complete.
   \end{proof}
We now formulate a useful criterion for a directed tree to
have finite number of leaves. Directed trees taken into
consideration in Proposition \ref{fnl} below are called
Fredholm in Section \ref{fredo} (cf.\ Definition
\ref{treeind}).
   \begin{pro}\label{fnl}
If $\tcal$ is a directed tree such that $\card{\dzi u}
< \infty$ for all $u \in V$ and $\card{V_{\prec}} <
\infty$ {\em (cf.\ \eqref{prec})}, then $\card{V
\setminus V^\prime} < \infty$.
   \end{pro}
   \begin{proof}
We first show that there is no loss of generality in
assuming that $\tcal$ has a root. Indeed, otherwise
$\tcal$ is rootless, which together with
$\card{V_{\prec}}<\infty$ and Proposition
\ref{xdescor}\,(i) implies that there exists $u \in V$
such that $\card{\dzi{\paa^{k}(u)}}=1$ for all $k \in
\nbb$. By Proposition \ref{xdescor}\,(iv), $V =
\{\paa^{k}(u)\}_{k=1}^\infty \sqcup \des u$. In view
of \eqref{dziinv}, we have $V^\prime =
\{\paa^{k}(u)\}_{k=1}^\infty \sqcup V_{\des u}^\prime$
and thus $V\setminus V^\prime = V_{\des u} \setminus
V_{\des u}^\prime$. Moreover, $\dzit{\tcal}v =
\dzit{\tcal_{\des u}}v$ for every $v \in \des u$, and
the directed trees $\tcal$ and $\tcal_{\des u}$ have
the same branching vertexes.

Suppose now that $\tcal$ has a root. Certainly, we can
assume that $\tcal$ is infinite and $V\setminus
V^\prime \neq \varnothing$. Take $w \in V\setminus
V^\prime$. Then there exists a positive integer $n$
such that $\paa^{n}(w) \in V_{\prec}$. If not, then by
Corollary \ref{przem} there would exist $n\in \nbb$
such that $\paa^{n}(w) = \koo$ and
$\card{\dzi{\paa^{j}(w)}}=1$ for $j=1,\ldots,n$. Since
$\card{\dzi{w}}=0$, we would deduce that
$V=\{\paa^{j}(w)\}_{j=0}^n$, a contradiction. Let
$k(w)$ be the least positive integer such that
$\paa^{k(w)}(w) \in V_{\prec}$. Set
$\varTheta(w)=\paa^{k(w)}(w)$. Define the equivalence
relation $\mathcal R$ on $V\setminus V^\prime$ by $w_1
\mathcal R \, w_2$ if and only if
$\varTheta(w_1)=\varTheta(w_2)$. Denote by
$[w]_{\mathcal R}$ the equivalence class of $w \in
V\setminus V^\prime$ with respect to $\mathcal R$.
Using the minimality of $k(v)$ and the fact that $v$
is a leaf of $\tcal$, one can show that the mapping
$[w]_{\mathcal R} \ni v \mapsto \paa^{k(v)-1}(v) \in
\dzi{\varTheta(w)}$ is injective. This implies that
$\card{[w]_{\mathcal R}} \Le \card{\dzi{\varTheta(w)}}
<\infty$. Since the mapping $(V\setminus
V^\prime)/\mathcal R \ni [w]_{\mathcal R} \mapsto
\varTheta(w) \in V_{\prec}$ is a well defined
injection, the proof is complete.
   \end{proof}
We conclude this section by introducing an equivalence
relation partitioning the given directed tree into
disjoint classes composed of vertexes of the same
generation.

Suppose that $\tcal$ is a directed tree. We say that
vertexes $u,v \in V$ are of the {\em same generation},
and write \idxx{$\sim_\tcal \; \equiv \; \sim$}{29} $u
\sim_\tcal v$, or shortly $u \sim v$, if there exists
$n\in \zbb_+$ such that $\paa^n(u)=\paa^n(v)$ (and
both sides of the equality make sense). It is easily
seen that $\sim$ is an equivalence relation in $V$.
Denote by \idx{$[u]_{\sim}$}{30} the equivalence class
of $u \in V$ with respect to $\sim$. Note that if $u
\in V^\circ$ and $v \in [u]_\sim$, then $v \in
V^\circ$. Evidently $[\koo]_\sim = \{\koo\}$ if
$\tcal$ has a root. However, if $u \in V^\prime$ and
$v \in [u]_\sim$, then it may happen that $v \notin
V^\prime$.

Given $u\in V$, we define \idxx{$N(u)=N_\tcal(u)$}{31}
$N(u)=N_\tcal(u)=\sup\{n \in \zbb_+ \colon \paa^n(u)
\text{ makes sense}\}$. Clearly, if $n \in \zbb_+$ and
$n \Le N(u)$, then $\paa^n(u)$ makes sense. Let us
collect the basic properties of the relation $\sim$.
   \begin{pro}\label{generation}
If $\tcal$ is a directed tree, then the equivalence
relation $\sim$ has the following
properties\,\footnote{\;$\pa{X}:=\{v \in V\colon
\text{ there exists $x \in X$ such that $\pa{x}$ makes
sense and $v=\pa x$}\}$ for $X \subseteq V$.}\/{\em :}
\idxx{$\pa{X}$}{32}
   \begin{enumerate}
   \item[(i)] $\pa{[u]_{\sim}} \subseteq [\pa{u}]_{\sim}$ for $u \in
V^\circ,$
   \item[(ii)] for all $u,v \in V,$ $u \sim v$ if and only
if $\pa{[u]_\sim} = \pa{[v]_\sim}$; moreover, if $u, v \in
V^\circ,$ then $u \sim v$ if and only if $\pa{[u]_\sim}
\cap \pa{[v]_\sim} \neq \varnothing$,
   \item[(iii)] if $V^\circ \neq \varnothing$, then
$V^\circ/_{\sim} \ni [u]_\sim \mapsto [\pa u]_\sim \in
V/_{\sim}$ is an injection,
   \item[(iv)]  $\dzi{[\pa{u}]_{\sim}} = [u]_{\sim}$
for $u \in V^\circ,$
   \item[(v)] $[u]_{\sim} = \bigcup_{n=0}^{N(u)} \dzin{n}{\paa^n(u)}$
for $u \in V,$
   \item[(vi)] $V = \bigsqcup_{n=0}^{N(u)} [\paa^n(u)]_\sim
\sqcup \bigsqcup_{n=1}^\infty \dzin{n}{[u]_\sim}$ for $u
\in V,$
   \item[(vii)] if $u \in V$, $n\in \nbb$ and $w \in
\dzin{n}{[u]_\sim}$, then $\dzin{n}{[u]_\sim}=[w]_\sim,$
   \item[(viii)] $\dzi{[\paa^{n}(u)]_\sim} =
[\paa^{n-1}(u)]_\sim$ for all integers $n$ such that $1 \Le
n \Le N(u)$, and $\dzi{\dzin{n}{[u]_\sim}} =
\dzin{n+1}{[u]_\sim}$ for all integers $n\Ge 0$.
   \end{enumerate}
   \end{pro}
   \begin{proof}
The proof of (i)--(v), being standard, is omitted.

(vi) The inclusion ``$\subseteq$'' (and consequently the
equality) can be justified as follows. If $v \in V$, then
by Proposition \ref{witr} and \eqref{decom} there exist $w
\in V$ and $k,l \in \zbb_+$ such that $\paa^k(u) = w =
\paa^l(v)$. If $k\Ge l$, then $v \in [\paa^{k-l}(u)]_\sim$.
In the opposite case, $x:= \paa^{l-k}(v) \sim u$ and
consequently $v \in \dzin{l-k}x \subseteq
\dzin{l-k}{[u]_\sim}$.

It remains to prove that the terms in (vi) are pairwise
disjoint. Take $n \in \nbb$ and $m \in \zbb_+$ such that $m
\Le N(u)$. We show that $[\paa^m(u)]_\sim \cap
\dzin{n}{[u]_\sim} = \varnothing$. Indeed, otherwise there
exists $w \in [\paa^m(u)]_\sim \cap \dzin{n}{[u]_\sim}$,
which implies that $\paa^s(w)=\paa^{m+s}(u)$ and
$\paa^{n+t}(w)=\paa^{t}(u)$ for some $s,t \in \zbb_+$.
Hence, if $m+s \Le t$, then
   \begin{align*}
\paa^{t-m}(w) = \paa^{s+(t - (m+s))}(w) = \paa^t(u) =
\paa^{n+t}(w).
   \end{align*}
By \eqref{decom}, this implies that $m+n=0$, which is a
contradiction. By the same kind of reasoning we see that
$m+s > t$ leads to a contradiction as well.

Using \eqref{decom}, we verify that the sets
$[\paa^m(u)]_{\sim}$, $0 \Le m \Le N(u)$, are pairwise
disjoint. Likewise, we show that the sets
$\dzin{m}{[u]_\sim}$, $m \in \nbb$, are pairwise disjoint.

(vii) Apply (iv) and induction on $n$.

(viii) is a direct consequence of (iv) and the definition
of $\dzin{n}{[u]_\sim}$. This completes the proof.
   \end{proof}
It follows from Proposition \ref{generation}\,(vii) that
the partition of $V$ appearing in (vi) coincides with the
one generated by the equivalence relation $\sim$.

As shown in Example \ref{rosnie} below, the inclusion in
Proposition \ref{generation}\,(i) may be proper. The
conditions (vi), (vii) and (viii) of Proposition
\ref{generation} may suggest that there exists a sequence
(finite or infinite) $\{u_n\}_n \subseteq V$ such that
$\pa{u_n}=u_{n-1}$ for all admissible $n$'s, and
$V=\bigsqcup_{n} [u_n]_\sim$. However, this is not always
the case.
   \begin{exa} \label{rosnie}
Consider the tree $\tcal=(V,E)$ with root defined by
   \allowdisplaybreaks
   \begin{align*}
V & = \{\koo\} \sqcup \{(i,j)\colon i,j \in \nbb, \, i \Le
j\},
   \\
E & = \big\{\big(\koo, (1,j)\big)\colon j \in \nbb\big\}
\sqcup \bigsqcup_{j=2}^\infty
\big\{\big((i,j),(i+1,j)\big)\colon i=1, \ldots, j-1\big\}.
   \end{align*}
Then $\pa{[u]_\sim} = [\pa{u}]_\sim \setminus \{(j-1,j-1)\}
\subsetneq [\pa u]_\sim$ for $u=(j,j)$ with $j\Ge 2$. One
can verify that a sequence $\{u_n\}_{n=0}^\infty$ with the
properties mentioned above does not exist.
   \end{exa}
   \subsection{\label{s01}Operator theory}
By an operator {\em in} a complex Hilbert space $\hh$
we understand a linear mapping $A\colon \hh \supseteq
\dz A \rightarrow \hh$ defined on a linear subspace
\idx{$\dz A$}{33} of $\hh$, called the {\em domain} of
$A$. The kernel, the range and the adjoint of $A$ are
denoted by \idx{$\jd A$}{34}, \idx{$\ob A$}{35} and
\idx{$A^*$}{36}, respectively. A densely defined
operator $A$ in $\hh$ is called {\em selfadjoint}
(respectively:\ {\em normal}\/) if $A^*=A$
(respectively:\ $A^*A=AA^*$), cf.\ \cite{b-s,weid}. We
denote by \mbox{$\|\cdot\|_A$}
\idxx{$\lVert\cdot\rVert_A$}{37} and
\idxx{$<\hspace{-.5ex}\cdot,\textrm{-}\hspace{-.5ex}>_A$}{38}
$\is{\cdot}{\mbox{-}}_A$ the graph norm and the graph
inner product of $A$, respectively, i.e., $\|f\|_A^2 =
\|f\|^2 + \|Af\|^2$ and $\is{f}{g}_A = \is f g +
\is{Af}{Ag}$ for $f,g \in \dz A$. If $A$ is closable,
then the closure of $A$ will be denoted by \idx{$\bar
A$}{39}. A linear subspace $\ee$ of $\dz A$ is called
a {\em core} of a closed operator $A$ in $\hh$ if
$\overline{A|_\ee} = A$ or equivalently if $\ee$ is
dense in the graph norm \mbox{$\|\cdot\|_A$} in $\dz
A$. If $A$ is a closed densely defined operator in
$\hh$, then $|A|$ stands for the square root of the
positive selfadjoint operator $A^*A$. For real $\alpha
> 0$, the $\alpha$-root
\idxx{$\mid \hspace{-.5ex} A \hspace{-.5ex} \mid^\alpha$}{40} $|A|^\alpha$ of $|A|$ is
defined by the Stone-von Neumann operator calculus,
i.e.,
   \begin{align*}
|A|^\alpha = \int_0^\infty x^\alpha E(\D x),
   \end{align*}
where $E$ is the spectral measure of $|A|$ (from now
on, we abbreviate $\int_{[0,\infty)}$ to
$\int_0^\infty$). The operator $|A|^\alpha$ is
certainly positive and selfadjoint. Given operators
$A$ and $B$ in $\hh$, we write \idx{$A \subseteq
B$}{41} if $\dz A \subseteq \dz B$ and $Ah = Bh$ for
all $h \in \dz A$.

In what follows, \idx{$\ogr \hh$}{42} stands for the
$C^*$-algebra of all bounded operators in $\hh$ with
domain $\hh$. We write \idx{$I=I_\hh$}{43} for the
identity operator on $\hh$. Given $f,g\in \hh$, we
define the operator $f \otimes g \in \ogr \hh$ by
\idxx{$f \otimes g$}{44}
   \begin{align*}
(f \otimes g) (h) = \is h g f, \quad h \in \hh.
   \end{align*}

We say that a closed linear subspace $\mathcal M$ of
$\hh$ {\em reduces} an operator $A$ in $\hh$ if $PA
\subseteq AP$, where $P\in \ogr \hh$ is the orthogonal
projection of $\hh$ onto $\mathcal M$. If $\mathcal M$
reduces $A$, then \idxx{$A \hspace{-.5ex}
\mid_{\mathcal M}$}{45} $A|_{\mathcal M}$ stands for
the restriction of $A$ to $\mathcal M$. To be more
precise, $A|_{\mathcal M}$ is an operator in
${\mathcal M}$ such that $\dz{A|_{\mathcal M}} = \dz A
\cap {\mathcal M}$ and $A|_{\mathcal M}h=Ah$ for $h
\in \dz{A|_{\mathcal M}}$.

For the reader's convenience, we include the proof of the
following result which is surely folklore.
   \begin{lem}\label{lems}
Let $\{e_\iota\}_{\iota \in \varXi}$ be an orthonormal
basis of $\hh$, $A$ be a positive selfadjoint operator
in $\hh$ and $\{t_\iota\}_{\iota \in \varXi}$ be a
family of nonnegative real numbers such that $e_\iota
\in \dz{A}$ and $Ae_\iota = t_\iota e_\iota$ for all
$\iota \in \varXi$. Then for every real $\alpha >0$,
   \begin{enumerate}
   \item[(i)] the linear span $\mathscr E$ of
$\{e_\iota\}_{\iota \in \varXi}$ is contained in
$\dz{A^\alpha}$,
   \item[(ii)] $\mathscr E$ is a core of $A^\alpha$,
i.e., $A^\alpha = \overline{A^\alpha|_{\mathscr E}}$,
   \item[(iii)] $A^\alpha e_\iota = t_\iota^\alpha e_\iota$
for all $\iota \in \varXi$.
   \end{enumerate}
Moreover, $\ob A$ is closed if and only if there
exists a real number $\delta > 0$ such that $t_\iota
\Ge \delta$ for every $\iota \in \varXi$ for which
$t_\iota > 0$.
   \end{lem}
The operator $A$ appearing in Lemma \ref{lems} will be
called a {\em diagonal} operator (subordinated to the
orthonormal basis $\{e_\iota\}_{\iota \in \varXi}$)
with diagonal elements $\{t_\iota\}_{\iota \in
\varXi}$.
   \begin{proof}[Proof of Lemma  \ref{lems}] (i) $\&$ (iii)
   Define the spectral measure $E$ on $[0,\infty)$ by
   \begin{align} \label{Jab1}
E (\sigma) f = \sum_{\iota \in \varXi}
\chi_\sigma(t_\iota) \is f {e_\iota} e_\iota, \quad f
\in \hh, \, \sigma \in \borel{[0, \infty)},
   \end{align}
where the above series is unconditionally convergent
in norm $($equivalently:\ convergent in norm in a
generalized sense, cf.\ \cite{b-p}$)$. Using a
standard measure theoretic argument, we deduce from
\eqref{Jab1} that for every Borel function $\varphi
\colon [0,\infty) \to [0, \infty)$,
   \begin{align}   \label{intsp}
\int_0^\infty \varphi(x) \is{E(\D x)f}f = \sum_{\iota
\in \varXi} \varphi(t_\iota) |\is f {e_\iota}|^2,
\quad f \in \hh.
   \end{align}
It follows from \eqref{intsp}, the selfadjointness of
$A$ and Parseval's identity that
   \begin{align*}
\int_0^\infty x^2 \is{E(\D x)f}f = \sum_{\iota \in
\varXi} t_\iota^2 |\is f {e_\iota}|^2 = \sum_{\iota
\in \varXi} |\is{Af} {e_\iota}|^2 = \|Af\|^2 < \infty,
\quad f \in \dz{A},
   \end{align*}
which means that $\dz A \subseteq \dz{\int_0^\infty x
\, E(\D x)}$. Arguing as above, we see that
   \begin{align*}
\is{Af}f = \Big\langle \sum_{\iota \in \varXi}
\is{Af}{e_\iota}e_\iota, f \Big \rangle = \sum_{\iota
\in \varXi} t_\iota |\is f {e_\iota}|^2
\overset{\eqref{intsp}}{=} \int_0^\infty x \is{E(\D
x)f} f
   \\
= \Big\langle \int_0^\infty x E(\D x) f, f \Big
\rangle, \quad f \in \dz A.
   \end{align*}
Both these facts imply that $A \subseteq \int_0^\infty
x \, E(\D x)$. Since the considered operators are
selfadjoint, we must have $A = \int_0^\infty x \, E(\D
x)$ (use \cite[Theorem 8.14(b)]{weid}), which means
that $E$ is the spectral measure of $A$. It follows
from the measure transport theorem (cf.\ \cite[Theorem
5.4.10]{b-s}) that the spectral measure $E_\alpha$ of
$A^\alpha$ is given by
   \begin{align}          \label{bab}
E_\alpha (\sigma) = E \circ \pi_\alpha^{-1}(\sigma),
\quad \sigma \in \borel{[0,\infty)},
   \end{align}
where $\pi_\alpha \colon [0,\infty) \to [0,\infty)$ is
defined by $\pi_\alpha(x) = x^\alpha$ for $x \in
[0,\infty)$. Hence, by \cite[Theorem 6.1.3]{b-s}, we
have
   \begin{align*}
\jd{t_\iota^\alpha I_\hh - A^\alpha} = \ob{E_\alpha
(\{t_\iota^\alpha\})} \overset{\eqref{bab}}=
\ob{E(\{t_\iota\})} = \jd{t_\iota I_\hh - A}, \quad
\iota \in \varXi.
   \end{align*}
This and our assumptions imposed on the operator $A$
imply (i) and (iii).

   (ii) In view of the above, it is enough to show
that $\mathscr E$ is a core of $A$. For this, take
a vector $f \in \dz A$ which is orthogonal to
$\mathscr E$ with respect to the graph inner
product $\is{\cdot}{\mbox{-}}_A$. Then
   \begin{align*}
0 = \is f {e_\iota} + \is {Af} {Ae_\iota} = \is f
{e_\iota} + t_\iota \is {Af} {e_\iota} = (1+t_\iota^2
) \is {f} {e_\iota}, \quad \iota \in \varXi.
   \end{align*}
Since $\{e_\iota\}_{\iota \in \varXi}$ is an
orthonormal basis of $\hh$, we conclude that $f=0$.

If $T$ is any normal operator in $\hh$, then $T=T_0
\oplus T_1$, where $T_0$ is the zero operator on
$\jd{T}$ and $T_1$ is an injective normal operator in
$\hh \ominus \jd{T}$ with dense range. Hence, $\ob{T}
= \ob{T_1}=\dz{T_1^{-1}}$, and consequently, by the
inverse mapping theorem, $\ob{T}$ is closed if and
only if $T_1^{-1}$ is bounded. Applying the above
characterization to $T=A$, part (ii) to $A_1^{-1}$ and
the fact that $\jd{A}$ equals the closed linear span
of $\{e_\iota\colon t_\iota=0, \, \iota \in \varXi\}$,
we get the ``moreover'' part of the conclusion.
   \end{proof}
   \newpage
   \section{\label{chap3}Fundamental properties}
   \subsection{\label{s1}An invitation to weighted shifts}
From now on, $\tcal=(V,E)$ is assumed to be a directed
tree. Denote by \idx{$\ell^2(V)$}{46} the Hilbert
space of all square summable complex functions on $V$
with the standard inner product
   \begin{align*}
\is fg = \sum_{u \in V} f(u) \overline{g(u)}, \quad f, g
\in \ell^2(V).
   \end{align*}
For $u \in V$, we define \idxx{$e_u$}{47} $e_u \in
\ell^2(V)$ by
   \begin{align*}
e_u(v) =
   \begin{cases}
   1 & \text{if } u=v, \\
   0 & \text{otherwise.}
   \end{cases}
   \end{align*}
The set $\{e_u\}_{u\in V}$ is an orthonormal basis of
$\ell^2(V)$. Denote by \idx{$\escr$}{48} the linear
span of the set $\{e_u\colon u \in V\}$. Let us point
out that $\ell^2(V)$ is a reproducing kernel Hilbert
space which is guaranteed by the reproducing property
   \begin{align} \label{rkhs}
f(u) = \is f {e_u}, \quad f \in \ell^2(V), \, u \in V.
   \end{align}
If $W$ is a nonempty subset of $V,$ then we regard the
Hilbert space $\ell^2(W)$ as a closed linear subspace
of $\ell^2(V)$ by identifying each $f\in \ell^2(W)$
with the function $\widetilde f \in \ell^2(V)$ which
extends $f$ and vanishes on the set $V \setminus W$.
   \begin{dfn}   \label{defshift}
Given $\lambdab = \{\lambda_v\}_{v \in V^\circ}$, a
family of complex numbers, we define the operator
\idx{$\slam$}{49} in $\ell^2(V)$ by
\idxx{$\varLambda_\tcal$}{50}
   \begin{align}   \label{lamtf+}
   \begin{aligned}
\dz {\slam} & = \{f \in \ell^2(V) \colon \varLambda_\tcal f
\in \ell^2(V)\},
   \\
\slam f & = \varLambda_\tcal f, \quad f \in \dz {\slam},
   \end{aligned}
   \end{align}
where $\varLambda_\tcal$ is the mapping defined on
functions $f\colon V \to \cbb$ by
   \begin{align} \label{lamtf}
(\varLambda_\tcal f) (v) =
   \begin{cases}
\lambda_v \cdot f\big(\pa v\big) & \text{if } v\in V^\circ,
   \\
   0 & \text{if } v=\koo.
   \end{cases}
   \end{align}
The operator $\slam$ will be called a {\em weighted
shift} on the directed tree $\tcal$ with weights
$\{\lambda_v\}_{v \in V^\circ}$.
   \end{dfn}
It is worth noting that the extremal situation
$V^\circ = \varnothing$ is not excluded; then, by
\eqref{lamtf}, $\slam$ is the zero operator on a
one-dimensional Hilbert space.

The proof of the following fact is based only on the
reproducing property of $\ell^2(V)$. Proposition
\ref{clos} can also be deduced from parts (i) and (ii)
of Proposition~\ref{desc}.
   \begin{pro} \label{clos}
Any weighted shift $\slam$ on a directed tree $\tcal$ is a closed
operator.
   \end{pro}
   \begin{proof}
Suppose that a sequence $\{f_n\}_{n=1}^\infty \subseteq \dz
\slam$ is convergent to a vector $f \in \ell^2(V)$ and the
sequence $\{\slam f_n\}_{n=1}^\infty$ is convergent to a
vector $g \in \ell^2(V)$. Take $u \in V$. By \eqref{rkhs},
the sequence $\{(\slam f_n)(u)\}_{n=1}^\infty$ is
convergent to $g(u)$. If $u \in V^\circ$, then, again by
\eqref{rkhs}, applied to $(\slam f_n)(u) = \lambda_u
f_n(\pa u)$, we see that the sequence $\{(\slam
f_n)(u)\}_{n=1}^\infty$ is convergent to $\lambda_u f(\pa
u)$. Thus $(\varLambda_\tcal f)(u)=g(u)$. If $u=\koo$, then
evidently $(\varLambda_\tcal f)(\koo)=0=g(\koo)$.
Summarizing, we have shown that $\varLambda_\tcal f=g \in
\ell^2(V)$, which means that $f \in \dz \slam$ and $g=\slam
f$.
   \end{proof}
   Next we describe the domain and the graph norm of the
operator $\slam$. In what follows, we adopt the conventions
that $0 \cdot \infty = 0$ and $\sum_{v\in\varnothing}
x_v=0$.
   \begin{pro} \label{desc}
Let $\slam$ be a weighted shift on a directed tree $\tcal$
with weights $\lambdab = \{\lambda_v\}_{v \in V^\circ}$.
Then the following assertions hold\/{\em :}
   \begin{enumerate}
   \item[(i)] $\dz{\slam} = \big\{f \in \ell^2(V)
\colon \sum_{u \in V} \big(\sum_{v \in\dzi u} |\lambda_v|^2\big)
|f(u)|^2 < \infty\Big\}$,
   \item[(ii)] $\|f\|_{\slam}^2 = \sum_{u \in V} \big
(1+\sum_{v \in\dzi u} |\lambda_v|^2\big) |f(u)|^2$ for all
$f \in \dz{\slam}$,
   \item[(iii)] $e_u$ is in $\dz{\slam}$ if and only if
$\sum_{v\in\dzi u} |\lambda_v|^2 < \infty$; if $e_u
\in \dz{\slam}$, then
   \begin{align} \label{eu}
\slam e_u = \sum_{v\in\dzi u} \lambda_v e_v, \quad
\|\slam e_u\|^2 = \sum_{v\in\dzi u} |\lambda_v|^2,
   \end{align}
   \item[(iv)] if $f \in \dz{\slam}$ and $W$ is a
subset of $V,$ then $f \chi_W \in \dz{\slam}$,
   \item[(v)] $\slam$ is densely defined if and only if
$\{e_u\colon u \in V\} \subseteq \dz{\slam}$,
   \item[(vi)] $\slam = \overline{\slam|_{\escr}}$ provided $\slam$
is densely defined.
  \end{enumerate}
   \end{pro}
   \begin{proof}
If $f \colon V \to \cbb$ is any function, then
   \begin{align}  \label{wwz}
\sum_{u \in V} |(\varLambda_\tcal f) (u)|^2 &
\overset{\eqref{lamtf}}= \sum_{u \in V^\circ} |\lambda_u|^2
|f(\pa u)|^2
   \\
& \overset{ \eqref{sumchi} } = \sum_{u \in V} \Big(\sum_{v
\in \dzi u} |\lambda_v|^2 \Big) |f(u)|^2, \notag
   \end{align}
which implies (i), (ii) and the first part of (iii). The
proof of \eqref{eu} is left to the reader (use
\eqref{rkhs}).

   (iv) is a direct consequence of (i).

   (v) The ``if'' part is clear because $\{e_u\}_{u\in V}$ is
an orthonormal basis of $\ell^2(V)$. Now we justify the
``only if'' part. Suppose that, contrary to our claim, $e_u
\notin\dz{\slam}$ for some $u \in V$. It follows from (i)
and (iii) that the vector $e_u$ is orthogonal to
$\dz{\slam}$, which is a contradiction.

   (vi) By (i), (ii), (iii) and (v), the Hilbert space $(\dz{\slam},
\|\cdot\|_{\slam})$ is the weighted $\ell^2$ space on $V$ with
weights $\big\{1+\sum_{v \in\dzi u} |\lambda_v|^2\big\}_{u \in V}$
in which the set of all complex functions on $V$ vanishing off
finite sets is dense. This means that $\escr$ is a core of $\slam$,
which completes the proof.
   \end{proof}
It is worth noting that, in general, the linear space
$\escr$ is not invariant for a densely defined
weighted shift $\slam$ on a directed tree $\tcal.$
This happens when the set $\dzi u$ is infinite for at
least one $u \in V$, and all the weights
$\{\lambda_v\}_{v \in \dzi u}$ are nonzero (use
\eqref{eu}). However, if the set $\dzi u$ is finite
for every $u \in V$, then $\escr$ is invariant for
$\slam$.
   \begin{rem} \label{re1-2}
The unilateral and bilateral classical weighted shifts
fit into our definition. Indeed, it is enough to
consider directed trees $(\zbb_+, \{(n,n+1)\colon n
\in \zbb_+\})$ and $(\zbb, \{(n,n+1)\colon n \in
\zbb\})$, respectively (they will be shortly denoted
by $\zbb_+$ and $\zbb$). Then the first equality in
\eqref{eu} reads as follows:
   \begin{align}  \label{notnew+}
\slam e_n = \lambda_{n+1}e_{n+1}.
   \end{align}
The reader should be aware that this is something
different from the conventional notation $\slam
e_n = \lambda_n e_{n+1}$ which abounds in the
literature. In the present paper, we use only the
new convention. Let us mention that according to
Proposition \ref{desc} any weighted shift $\slam$
on the directed tree $\zbb_+$ is densely defined
and the linear span of $\{e_n\colon n \in
\zbb_+\}$ is a core of $\slam$. This fact and
\eqref{notnew+} guarantee that $\slam$ is a
unilateral classical weighted shift (cf.\
\cite[equality (1.7)]{ml}). The same reasoning
applies to the case of a bilateral classical
weighted shift.
   \end{rem}
   \begin{ozn} \label{poddrz}
Given a weighted shift $\slam$ on a directed tree
$\tcal$ with weights $\lambdab = \{\lambda_v\}_{v \in
V^\circ}$ and $u \in V^\circ$, we denote by
\idx{$\slamr u$}{51} and \idx{$\slaml u$}{52} the
weighted shifts on directed trees $\tcal_{\des u}$ and
$\tcal_{V \setminus \des u}$ with weights
\idxx{$\lambdab_\rightarrow\!(u)$}{53}
$\lambdab_\rightarrow\!(u) :=\{\lambda_v\}_{v \in \des
u \setminus \{u\}}$ and \idxx{$\lambdab_\leftarrow\!
(u)$}{54} $\lambdab_\leftarrow\! (u):=\{\lambda_v\}_{v
\in V \setminus (\des u \cup \Ko \tcal)}$,
respectively (cf.\ Proposition \ref{subdir}). If
$\tcal$ has a root and $u=\koo$, then we write $\slamr
u := \slam$.
   \end{ozn}
We show that if at least one weight of the weighted
shift $\slam$ on a directed tree vanishes, then
$\slam$ is an orthogonal sum of two weighted shifts on
directed trees.
   \begin{pro}\label{dirsum}
Let $\slam$ be a weighted shift on a directed tree
$\tcal$ with weights $\lambdab = \{\lambda_v\}_{v \in
V^\circ}$. Assume that $\lambda_u=0$ for some $u \in
V^\circ$. Then
   \begin{align*}
\slam = \slamr u \oplus \slaml u.
   \end{align*}
   \end{pro}
   \begin{proof}
Since $V^\circ \neq \varnothing$, we infer from Proposition
\ref{subdir} that the graphs $\tcal_{\des u}$ and $\tcal_{V
\setminus \des u}$ are directed trees. Denote by $P_u$ the
orthogonal projection of $\ell^2(V)$ onto $\ell^2(\des u)$,
i.e., $P_u f = \chi_{\des u} f$ for $f \in \ell^2(V)$. We
show that $P_u \slam \subseteq \slam P_u$. For this, let $f
\in \dz {\slam}$. By Proposition \ref{desc}\,(iv) $P_u f
\in \dz \slam$. If $v \in V^\circ$, then either $v \in
V^\circ \setminus \{u\}$ and, consequently, by Proposition
\ref{subdir} $\chi_{\des u}(v)= \chi_{\des u}(\pa v)$, or
$v=u$ and hence $\lambda_v=0$. This implies that for all $v
\in V^\circ$,
   \begin{align*}
(P_u \slam f)(v) & = \chi_{\des u}(v) (\slam f)(v) =
\lambda_v \chi_{\des u}(v) f(\pa v)
   \\
& = \lambda_v \chi_{\des u}(\pa v) f(\pa v) =
\lambda_v (P_u f)(\pa v) = (\slam P_u f)(v).
   \end{align*}
In turn, if $v = \ko{\tcal}$, then by \eqref{lamtf} we have
$(P_u \slam f)(v)= 0 = (\slam P_u f)(v)$. This means that
$P_u \slam \subseteq \slam P_u$. Hence $\slam =
\slam|_{\ell^2(\des u)} \oplus \slam|_{\ell^2(V \setminus
\des u)}$. Using Proposition \ref{subdir} as well as Remark
\ref{dzi}, one can show that $\slam|_{\ell^2(\des u)} =
\slamr u$ and $\slam|_{\ell^2(V \setminus \des u)} = \slaml
u$. Looking at the equality $\dz{\slam|_{\ell^2(V \setminus
\des u)}} = \dz{\slaml u}$, the reader should be aware of
the fact that
   \begin{align*}
\sum_{v \in \dzi {\pa u}} |\lambda_v|^2 = \sum_{v \in
\dzii_\leftarrow\!(\pa u)} |\lambda_v|^2 +
|\lambda_u|^2,
   \end{align*}
where $\dzii_\leftarrow\!(w)$ is the set of all
children of $w$ counted in the graph
$\tcal_{V\setminus \des u}$. This completes the proof.
  \end{proof}
The injectivity of a weighted shift on a directed tree is
characterized by a condition which essentially refers to
the graph structure of the tree. In particular, there may
happen that an injective weighted shift on a directed tree
has many zero weights (which never happens for classical
weighted shifts).
   \begin{pro}\label{dzisdesz}
Let $\slam$ be a weighted shift on a directed tree $\tcal$
with weights $\lambdab = \{\lambda_v\}_{v \in V^\circ}$.
Then the following conditions are equivalent{\em :}
   \begin{enumerate}
   \item[(i)] $\slam$ is injective,
   \item[(ii)] $\tcal$ is leafless and
$\sum_{v\in \dzi u} |\lambda_v|^2 > 0$ for all $u \in
V$.
   \end{enumerate}
   \end{pro}
   It follows from Proposition \ref{dzisdesz} that a
directed tree which admits an injective weighted shift
must be leafless.
   \begin{proof}[Proof of Proposition \ref{dzisdesz}]
(i)$\Rightarrow$(ii) Suppose that contrary to our
claim \linebreak $\sum_{v\in \dzi u} |\lambda_v|^2
= 0$ for some $u \in V$ (of course, this includes
the case of $\dzi u = \varnothing$). Then, by
Proposition \ref{desc}\,(iii), $e_u \in \dz{\slam}$
and $\slam e_u = 0$, a contradiction.

(ii)$\Rightarrow$(i) Take $f \in \dz \slam$ such that
$\slam f = 0$. Then, by \eqref{lamtf+} and \eqref{wwz}, we
have
   \begin{align*}
0 = \|\slam f\|^2 = \sum_{u \in V} \Big(\sum_{v \in \dzi u}
|\lambda_v|^2 \Big) |f(u)|^2,
   \end{align*}
which, together with (ii), implies that $f(u)=0$ for all $u
\in V$.
   \end{proof}
In view of Propositions \ref{dirsum} and
\ref{dzisdesz}, one can construct a reducible
injective and bounded weighted shift on a
directed tree with root (see \eqref{varkappa} for
examples of directed trees admitting such
weighted shifts). This is again something which
cannot happen for (bounded or unbounded)
injective unilateral classical weighted shifts
(see \cite[Theorem (3.0)]{ml}).

The question of when a weighted shift $\slam$ on a
directed tree is bounded has a simple answer. Let us
point out that implication (i)$\Rightarrow$(ii) of
Proposition \ref{ogrs} below is also an immediate
consequence of the closed graph theorem and
Proposition \ref{clos}.
   \begin{pro}\label{ogrs}
Let $\slam$ be a weighted shift on a directed tree $\tcal$ with
weights $\lambdab = \{\lambda_v\}_{v \in V^\circ}$. Then the
following conditions are equivalent{\em :}
   \begin{enumerate}
   \item[(i)] $\dz \slam = \ell^2(V)$,
   \item[(ii)] $\slam \in \ogr {\ell^2(V)}$,
   \item[(iii)] $\sup_{u\in
V}\sum\nolimits_{v\in\dzi u} |\lambda_v|^2 < \infty$.
   \end{enumerate}
If $\slam \in \ogr {\ell^2(V)}$, then
   \begin{align} \label{pnor}
\|\slam\| = \sup_{u\in V} \|\slam e_u\| = \sup_{u\in
V} \, \sqrt{\sum\limits_{v\in\dzi u} |\lambda_v|^2}.
   \end{align}
   \end{pro}
   \begin{proof}
   (i)$\Leftrightarrow$(iii) It follows from Proposition
\ref{desc}\,(i) that $\dz \slam = \ell^2(V)$ if and only if
for every complex function $f$ on $V$,
   \begin{align*}
\sum_{u \in V} |f(u)|^2 < \infty \; \iff \; \sum_{u
\in V} \Big(1 + \sum_{v \in\dzi u} |\lambda_v|^2\Big)
|f(u)|^2 < \infty.
   \end{align*}
This in turn is easily seen to be equivalent to (iii).

   (ii)$\Rightarrow$(iii) If $\slam \in \ogr {\ell^2(V)}$, then by
\eqref{eu} we have
   \begin{align}  \label{szac}
\sum_{v\in\dzi u} |\lambda_v|^2 = \|\slam e_u\|^2 \Le
\|\slam\|^2, \quad u \in V.
   \end{align}

   (iii)$\Rightarrow$(ii) Setting $c=\sup_{u\in
V} \sum\nolimits_{v\in\dzi u} |\lambda_v|^2$, we
get
   \begin{align*}
\sum_{u \in V} \sum_{v \in\dzi u} |\lambda_v|^2
|f(u)|^2 \Le c \sum_{u \in V} |f(u)|^2, \quad f \in
\ell^2(V),
   \end{align*}
which, by Proposition \ref{desc}\,(i) and \eqref{wwz},
implies that $\dz \slam = \ell^2(V)$ and $\|\slam\|^2
\Le c$. This and \eqref{szac} give \eqref{pnor}.
   \end{proof}
According to Propositions \ref{46} and \ref{ogrs},
$\sup_{v \in V^\circ} |\lambda_v| < \infty$ whenever
$\slam \in \ogr{\ell^2(V)}$. However, in general,
$\sup_{v \in V^\circ} |\lambda_v| < \infty$ does not
imply $\slam \in \ogr{\ell^2(V)}$. What is worse, the
above inequality may not imply that the operator
$\slam$ is densely defined (cf.\ Proposition
\ref{desc}).
   \begin{cor} \label{parc}
Let $\slam$ be a weighted shift on a directed tree $\tcal$ with
weights $\lambdab = \{\lambda_v\}_{v \in V^\circ}$, and let $\sup_{u
\in V} \card {\dzi u} < \infty$. Then $\slam \in \ogr {\ell^2(V)}$
if and only if $\sup_{v \in V^\circ} |\lambda_v| < \infty$.
   \end{cor}
If we want to investigate densely defined
weighted shifts on a directed tree with nonzero
weights, then we have to assume that the tree
under consideration is at most countable, and if
the latter holds, we can always find a bounded
weighted shift on it with nonzero weights.
   \begin{pro} \label{przeldz}
If there exists a densely defined weighted shift
$\slam$ on a directed tree $\tcal$ with nonzero
weights $\lambdab = \{\lambda_v\}_{v \in V^\circ}$,
then $\card{V} \Le \aleph_0$. Conversely, if $\tcal$
is a directed tree such that $\card{V} \Le \aleph_0$,
then there exists a weighted shift $\slam \in
\ogr{\ell^2(V)}$ with nonzero weights.
   \end{pro}
   \begin{proof}
Suppose first that there exists a densely defined
weighted shift $\slam$ on $\tcal$ with nonzero
weights. It follows from parts (iii) and (v) of
Proposition \ref{desc} and \cite[Corollary
19.5]{b-p} that $\card{\dzi u} \Le \aleph_0$ for
all $u \in V$. An induction argument combined
with \eqref{n+1} shows that $\card{\dzin n {u}}
\Le \aleph_0$ for all $u \in V$ and $n\in
\zbb_+$. Hence, by \eqref{ytag}, $\card{\des u}
\Le \aleph_0$ for all $u \in V$. If $\tcal$ has a
root, then Corollary \ref{przem} implies that
$\card{V} \Le \aleph_0$. If $\tcal$ is rootless,
then the same inequality holds due to Proposition
\ref{xdescor}\,(iii).

Assume now that $\card{V} \Le \aleph_0$. It is
then clear that for every $u \in V^\prime$, there
exists a system $\{\lambda_{u,v}\}_{v\in
\dzi{u}}$ of positive real numbers such that
$\sum_{v\in \dzi{u}} \lambda_{u,v}^2 \Le 1$.
Hence, by \eqref{sumchi}, the system $\lambdab =
\{\lambda_u\}_{u \in V^\circ}$ given by
$\lambda_v = \lambda_{u,v}$ for $v \in \dzi{u}$
and $u \in V^\prime$ is well defined, and by
Proposition \ref{ogrs} $\slam\in
\ogr{\ell^2(V)}$.
   \end{proof}
We now discuss the question of when the space
$\ell^2(V)$ built on a subtree $\tcal$ of a directed
tree $\hat\tcal$ is invariant for a weighted shift on
$\hat\tcal$ with nonzero weights. Note that if $\tcal
= (V,E)$ is a subtree of a directed tree
$\hat\tcal=(\hat V,\hat E)$, then $V^\circ \subseteq
\hat V^\circ$; moreover, if $\hat \tcal$ has a root,
so does $\tcal$. For equivalent forms of the condition
(ii) of Proposition \ref{heredit} below, we refer the
reader to Proposition \ref{preheredit}.
   \begin{pro}\label{heredit}
Let $\tcal = (V,E)$ be a subtree of a directed tree
$\hat\tcal=(\hat V,\hat E)$, and let $\slamh\in
\ogr{\ell^2(\hat V)}$ be a weighted shift on $\hat\tcal$
with nonzero weights $\hat\lambdab =
\{\hat\lambda_u\}_{u\in \hat V^\circ}$. Then the following
two conditions are equivalent\/{\em :}
   \begin{enumerate}
   \item[(i)] $\ell^2(V)$ is invariant for $\slamh$,
   \item[(ii)] $V =
   \begin{cases}
\dest{\hat\tcal}{\ko{\tcal}} & \text{if $\tcal$ has a
root,}
   \\
\hat V & \text{if $\tcal$ is rootless.}
   \end{cases}$
   \end{enumerate}
Moreover, if {\em (i)} holds, then $\slamh|_{\ell^2(V)} =
\slam$, where $\slam \in \ogr{\ell^2(V)}$ is a weighted
shift on $\tcal$ with weights $\lambdab =
\{\lambda_u\}_{u\in V^\circ}$ given by $\lambda_u =
\hat\lambda_u$ for $u\in V^\circ$.
   \end{pro}
Observe that the implication (ii)$\Rightarrow$(i) remains
valid without assuming that $\slamh$ has nonzero weights.
   \begin{proof}[Proof of Proposition \ref{heredit}]
(i)$\Rightarrow$(ii) It follows from \eqref{eu}, applied to
$\slamh$, that $\dzit{\hat\tcal}u \subseteq V$ for every
$u\in V$. Applying Proposition \ref{preheredit}, we get
(ii).

(ii)$\Rightarrow$(i) By Proposition \ref{preheredit},
$\dzit{\tcal}{u} = \dzit{\hat\tcal}{u}$ for all $u \in V$.
This together with \eqref{eu} yields (i).

If the space $\ell^2(V)$ is invariant for $\slamh$, then
the equality $\slamh|_{\ell^2(V)} = \slam$ can be inferred
from Proposition \ref{preheredit}\,(i) and \eqref{eu}.
   \end{proof}
In view of Proposition \ref{przeldz}, the situation
discussed in Proposition \ref{heredit} may happen only if
$\card{\hat V} \Le \aleph_0$.
   \subsection{\label{s1.5}Unitary equivalence}
We begin by showing that, from the Hilbert space point
of view, the study of weighted shifts on directed
trees can be reduced to the case of weighted shifts
with nonnegative weights. Comparing with the
analogical result for classical weighted shifts, the
reader will find that in the present situation the
proof is much more complicated (mostly because it
essentially depends on the complexity of graphs under
consideration). To make the proof as clear and short
as possible, we have decided to use a topological
argument which seems to be of independent interest. We
are aware of the fact that a more elementary proof of
Theorem \ref{uni} is available. However, it is
essentially longer and more technical (compare with
the proof of Theorem \ref{modul}).
    \begin{thm} \label{uni}
A weighted shift $\slam$ on a directed tree
$\tcal$ with weights $\lambdab = \{\lambda_v\}_{v
\in V^\circ}$ is unitarily equivalent to the
weighted shift $\smlam$ on $\tcal$ with weights
$|\lambdab| = \{|\lambda_v|\}_{v \in V^\circ}$.
    \end{thm}
    \begin{proof}
Set \idxx{$\tbb$}{55} $\tbb=\{z \in \cbb\colon
|z|=1\}$. For $\betab = \{\beta_u\}_{u \in V}
\subseteq \tbb$, we define the unitary operator
$U_\betab \in \ogr{\ell^2(V)}$ by $(U_{\betab} f)(u) =
\beta_u f(u)$ for $u \in V$ and $f \in \ell^2(V)$.
Since $(U_\betab^* f)(u) = \overline{\beta_u} f(u)$
for $u \in V$ and $f \in \ell^2(V)$, we infer from
Proposition \ref{desc}\,(i) that $\dz{\smlam} =
\dz{\slam} = \dz{U_\betab\slam U_\betab^*}$. Hence,
for every $f \in \dz{\smlam}$,
    \begin{align*}
(U_\betab\slam U_\betab^* f) (v)
\overset{\eqref{lamtf}}=
   \begin{cases}
   \lambda_v \beta_v (U_\betab^* f)(\pa v) =
\lambda_v \beta_v \bar \beta_{\pa v} f(\pa v) &
\text{if } v \in V^\circ,
   \\
   0 & \text{if } v=\koo.
   \end{cases}
   \end{align*}
To complete the proof it is therefore enough to
show that there exists a system
$\betab=\{\beta_v\}_{v \in V} \subseteq \tbb$
such that
    \begin{align} \label{ukr2}
\lambda_v \beta_v \bar \beta_{\pa v} =
|\lambda_v|, \quad v \in V^\circ.
    \end{align}
We do this in two steps.

{\sc Step 1.} For each $(u,\gamma) \in V \times \tbb$,
there exists $\{\beta_{v}\}_{v \in \des u} \subseteq
\tbb$ such that
    \begin{align} \label{indn}
\beta_{u} & = \gamma,
    \\
\lambda_v \beta_{v} & = |\lambda_v| \beta_{\pa v}, \quad v
\in \des u \setminus \{u\}. \label{indn+}
    \end{align}

Indeed, since $\des u \setminus \{u\} =
\bigsqcup_{n=1}^\infty \dzin n u$ (use the decomposition
\eqref{decom}) and $\pa{\dzin{n+1}u} \subseteq \dzin n u$,
we can define the wanted system $\{\beta_{v}\}_{v \in \des
u}$ recursively. We begin with \eqref{indn}, and then
having defined $\beta_w$ for all $w \in \dzin n u$, we
define $\beta_v$ for every $v \in \dzin {n+1} u$ by
$\beta_v = \lambda_v^{-1}|\lambda_v| \beta_{\pa v}$
whenever $\lambda_v\neq 0$ and by $\beta_v=1$ otherwise.
Hence, an induction argument completes the proof of Step 1.

Step 1 and Corollary \ref{przem} enable us to solve
\eqref{ukr2} in the case when $\tcal$ has a root. We
now consider the other case when $\tcal$ has no root.

{\sc Step 2.} There exists $\{\beta_{v}\}_{v \in V}
\subseteq \tbb$ such that
    \begin{align} \label{indn1}
\lambda_v \beta_{v} = |\lambda_v| \beta_{\pa v}, \quad v
\in V.
    \end{align}

To prove this, denote by $\tbb^V$ the set of all functions
from $V$ to $\tbb$, and equip it with the topology of
pointwise convergence on $V$. By Tihonov's theorem, $\tbb^V$
is a compact Hausdorff space. Given $u \in V$, we set
    \begin{align*}
\varOmega_u = \Big\{\{\beta_v\}_{v \in V} \in \tbb^V \colon
\lambda_v \beta_{v} = |\lambda_v| \beta_{\pa v} \text{ for
all } v \in \des u \setminus \{u\}\Big\}.
    \end{align*}
Plainly, each set $\varOmega_u$ is closed in $\tbb^V$. We
claim that the family $\{\varOmega_u\}_{u \in V}$ has the
finite intersection property. Indeed, if $W$ is a finite
nonempty subset of $V$, then by Proposition \ref{witr}
there exists $u \in V$ such that $W \subseteq \des u$.
Hence
    \begin{align} \label{dse2}
\des w \subseteq \des {\des u} \overset{\eqref{inv}} = \des
u, \quad w \in W.
    \end{align}
This implies that $\des w \setminus \{w\}
\subseteq \des u \setminus \{u\}$ for all $w \in
W$ (because the only dubious case $u \in \des w
\setminus \{w\}$, when combined with \eqref{inv}
and \eqref{dse2}, yields $\des u = \des w $,
which contradicts Proposition \ref{subdir}). As a
consequence, $\varOmega_u \subseteq \bigcap_{w
\in W} \varOmega_w$. Since, by Step 1, the set
$\varOmega_u$ is nonempty, we conclude that the
family $\{\varOmega_u\}_{u \in V}$ has the finite
intersection property. Thus, by the compactness
of $\tbb^V$, $\bigcap_{u\in V} \varOmega_u \neq
\varnothing$. If $\betab \in \bigcap_{u\in V}
\varOmega_u$ and $v \in V$, then $\betab \in
\varOmega_{\pa v}$, which implies \eqref{indn1}.
    \end{proof}
    \begin{rem}
We now discuss the question of uniqueness of solutions in
Steps 1 and 2 of the proof of Theorem \ref{uni}. Certainly,
we lose uniqueness if some of the weights $\lambda_v$
vanish. The situation is quite different if the weights of
$\slam$ are nonzero.

In Steps $1^\prime$ and $2^\prime$ below we assuming that
$\lambda_v \neq 0$ for all $v \in V^\circ$.

{\sc Step $1^\prime$.} If $u \in V$ is fixed and
$\{\beta_{v}\}_{v \in \des u}, \{\beta_{v}^\prime\}_{v \in
\des u} \subseteq \tbb$ satisfy \eqref{indn+}, then
$\beta^\prime_v = \gamma \beta_v$ for all $v \in \des u$
with $\gamma = \beta^\prime_u \overline{\beta_u}$.

The proof of Step $1^\prime$ is similar to that
of Step 1.

{\sc Step $2^\prime$.} Suppose that $\tcal$ has
no root. If $u\in V$ is fixed and $\betab,
\betab^\prime \in \tbb^V$ satisfy \eqref{indn1},
then $\beta^\prime_v = \gamma \beta_v$ for all $v
\in V$ with $\gamma = \beta^\prime_u
\overline{\beta_u}$.

Indeed, if $v \in V$, then by Proposition
\ref{witr} there exists $w \in V$ such that
$\{u,v\} \subseteq \des w$. According to Step
$1^\prime$, there exists $\gamma \in \tbb$ such
that $\beta^\prime_x = \gamma \beta_x$ for all $x
\in \des w$. Since $u,v \in \des w$, we get
$\beta^\prime_v = \gamma \beta_v$ and
$\beta^\prime_u = \gamma \beta_u$, which yields
$\gamma = \beta^\prime_u \overline{\beta_u}$.
This means that $\gamma$ does not depend on $w$.
    \end{rem}
   \subsection{Circularity}
We now prove that a densely defined weighted shift on
a directed tree is a circular operator. The definition
of a circular operator was introduced in
\cite{a-h-h-k}. As shown in \cite[Proposition
1.3]{a-h-h-k}, an irreducible bounded operator on a
complex Hilbert space is circular if and only if it
possesses a circulating $C_0$-group of unitary
operators. The latter property was then undertaken by
Mlak and used as the definition of circularity in the
more general context of unbounded operators (cf.\
\cite{ml0,ml1,mslo}).
    \begin{thm} \label{modul}
Let $\slam$ be a weighted shift on a directed
tree $\tcal$. Then for every $c \in \rbb$ there
exists $\thetab=\{\theta_u\}_{u \in V} \subseteq
\rbb$ such that
    \begin{align} \label{cogrr}
\E^{-\I t N} \slam \E^{\I t N} = \E^{\I t c}
\slam, \quad t \in \rbb,
    \end{align}
where $N=N_{\thetab}$ is a unique selfadjoint
operator in $\ell^2(V)$ such that $\{e_u\}_{u \in
V} \subseteq \dz{N}$ and $N e_u = \theta_u e_u$
for all $u \in V$.
    \end{thm}
    \begin{proof}
Fix $c \in \rbb$ and take
$\thetab=\{\theta_u\}_{u \in V} \subseteq \rbb$.
Define the operator $N=N_{\thetab}$ in
$\ell^2(V)$ by $\dz{N} = \{f \in \ell^2(V)\colon
\thetab f \in \ell^2(V)\}$ and $Nf = \thetab f$
for $f \in \dz{N}$, where $(\thetab f)(u)=
\theta_u f(u)$ for $u \in V$. Clearly, $N$ is
selfadjoint, $\{e_v\}_{v \in V} \subseteq \dz N$
and $N e_u = \theta_u e_u$ for all $u \in V$. By
Lemma \ref{lems}, such $N$ is unique. Moreover,
$\{\E^{\I t N}\}_{t\in \rbb}$ is a $C_0$-group of
unitary operators. Using an explicit description
of the spectral measure of $N$ (as in
\eqref{Jab1}), we verify that $\E^{\I t N}e_u =
\E^{\I t \theta_u} e_u$ for all $u \in V$ and $t
\in \rbb$. Hence, for all $u \in V$ and $f \in
\ell^2(V)$, we have
   \begin{align*}
(\E^{\I tN}f)(u) \overset{\eqref{rkhs}} =
\is{\E^{\I tN}f}{e_u} = \is{f}{\E^{-\I tN}e_u} =
\E^{\I t \theta_u} \is{f}{e_u} = \E^{\I t
\theta_u} f(u).
   \end{align*}
In view of Proposition \ref{desc}, this implies
that $\dz{\E^{-\I t N} \slam \E^{\I t N}} =
\dz{\slam}$ for all $t \in \rbb$. Moreover, if $f
\in \dz{\slam}$, then
   \begin{align*}
(\E^{-\I t N} \slam \E^{\I t N}f)(v) = \E^{-\I t \theta_v}
\lambda_v (\E^{\I t N}f)(\pa{v}) = \E^{\I t
(\theta_{\pa{v}}-\theta_v)} (\slam f)(v), \quad v \in
V^\circ,
   \end{align*}
and $(\E^{-\I t N} \slam \E^{\I t N}f)(v)= (\slam
f)(v) = 0$ for $v=\koo$. Consequently, it remains
to prove that there exists a solution
$\{\theta_u\}_{u \in V} \subseteq \rbb$ of the
equation
    \begin{align} \label{equiv}
\E^{\I t (\theta_{\pa v} - \theta_v)} = \E^{\I t c}, \quad
v \in V^\circ, \, t \in \rbb.
    \end{align}
Differentiating both sides of the above equality with
respect to $t$ at $t=0$, we see that \eqref{equiv} is
equivalent to
    \begin{align} \label{equiv+}
\theta_{\pa v} - \theta_v = c, \quad v \in V^\circ.
    \end{align}

Take $u \in V$. As in the proof of Step 1 of Theorem
\ref{uni}, we show that for each $\zeta \in \rbb$, there
exists a unique system $\{\theta_{v}\}_{v \in \des
u}\subseteq \rbb$ such that $\theta_u = \zeta$ and
    \begin{align} \label{cosr}
\theta_{\pa v} - \theta_v = c, \quad v \in \des u \setminus
\{u\}.
    \end{align}
Therefore, if $\{\theta_{v}\}_{v \in \des u},
\{\theta_{v}^\prime\}_{v \in \des u} \subseteq \rbb$ are
solutions of \eqref{cosr}, then so is the system
$\{\theta_{v} + (\theta_u^\prime - \theta_u)\}_{v \in \des
u}$ with the same value at $u$ as $\{\theta_{v}^\prime\}_{v
\in \des u}$. Thus, by uniqueness, we have $\theta_v^\prime
= \theta_{v} + (\theta_u^\prime - \theta_u)$ for all $v \in
\des u$.

Suppose now that $u_0 \in \des u$ and $\zeta \in
\rbb$. Take any solution $\{\theta_{v}\}_{v \in \des
u} \subseteq \rbb$ of \eqref{cosr}. Then
$\{\theta_{v}+(\zeta - \theta_{u_0})\}_{v \in \des u}$
is a solution of \eqref{cosr} with value $\zeta$ at
$u_0$. Note that such solution is unique. Indeed, if
$\{\theta_{v}\}_{v \in \des u},
\{\theta_{v}^\prime\}_{v \in \des u} \subseteq \rbb$
are solutions of \eqref{cosr} with the same value
$\zeta$ at $u_0$, then by the previous paragraph there
exists $a \in \rbb$ such that $\theta_v^\prime =
\theta_{v} + a$ for all $v \in \des u$. Substituting
$v=u_0$, we obtain $a=0$, which gives the required
uniqueness.

In view of the above discussion and Corollary
\ref{przem}, the equation \eqref{equiv+} has a
solution in the case when $\tcal$ has a root.

Let us pass to the other case when $\tcal$ has no root. Fix
any $u_0\in V$. If $u \in V$ is such that $u_0 \in \des u$,
then by the penultimate paragraph there exists a unique
system $\{\theta_{u,v}\}_{v \in \des u} \subseteq \rbb$
solving \eqref{cosr} and such that $\theta_{u,u_0} = 0$. We
now define the required solution $\{\theta_v\}_{v \in V}
\subseteq \rbb$ of \eqref{equiv+} as follows. If $v \in V$,
then by Proposition \ref{witr} there exists $u \in V$ such
that $v, u_0 \in \des u$. Define $\theta_v = \theta_{u,v}$
(note that $\{\theta_v\}_{v \in V}$ depends of $u_0$).
First we prove that this definition is correct. So, let
$u^\prime \in V$ be such that $v, u_0 \in \des {u^\prime}$.
We claim that $\theta_{u,v} = \theta_{u^\prime,v}$. Indeed,
by Proposition \ref{witr}, there exists $w \in V$ such that
$u, u^\prime \in \des w$. Then, by \eqref{inv}, we have
$\des u \cup \des{u^\prime} \subseteq \des w$. As in the
proof of Theorem \ref{uni}, we show that the last inclusion
implies $\des u \setminus \{u\} \subseteq \des w \setminus
\{w\}$ and $\des{u^\prime}\setminus \{u^\prime\} \subseteq
\des w \setminus \{w\}$. Hence the system
$\{\theta_{w,x}\}_{x \in \des u}$ is a solution of
\eqref{cosr}, and $\theta_{w,u_0}=0$. By uniqueness
property, we must have $\theta_{w,x} = \theta_{u,x}$ for
all $x \in \des u$. Substituting $x=v$, we get
$\theta_{w,v} = \theta_{u,v}$. Applying similar argument to
the system $\{\theta_{w,x}\}_{x \in \des {u^\prime}}$, we
obtain $\theta_{w,v} = \theta_{u^\prime,v}$, which shows
that our definition of $\{\theta_v\}_{v \in V}$ is correct.
Using Proposition \ref{witr} again, we find $\tilde u \in
V$ such that $\{v, \pa v, u_0\} \subseteq \des{\tilde u}$.
Since $v \in \des {\pa v}\setminus \{\pa v\} \subseteq
\des{\tilde u} \setminus \{\tilde u\}$, we have $\theta_v =
\theta_{\tilde u,v}$, $\theta_{\pa v} = \theta_{\tilde
u,\pa v}$ and consequently $\theta_{\pa v} - \theta_v = c$.
As $v \in V$ is arbitrary, the proof is complete.
    \end{proof}
A careful inspection of the proof of Theorem \ref{modul}
shows that if a tree $\tcal$ has no root, then for every
$c\in \rbb$ and for every $(u_0,\zeta) \in V \times \rbb$
there exists a unique system $\{\theta_v\}_{v \in V}
\subseteq \rbb$ such that $\theta_{u_0} = \zeta$ and
$\theta_{\pa v} - \theta_v = c$ for all $v \in V$.

We conclude this section by mentioning some spectral
properties of weighted shifts on a directed tree. The fact
that the spectral radius of the weighted adjacency operator
$A(G)$ of an infinite directed fuzzy graph $G$ belongs to
the approximate point spectrum of $A(G)$ was proved in
\cite[Theorem 6.1]{f-f-s-w} (see also \cite{mo} for the
case of infinite undirected graphs). The weighted adjacency
operator $A(G)$ is defined in \cite{f-f-s-w} for a directed
fuzzy graph a vertex of which may have more than one server
(read:\ parent). The reader should also convince himself
that in the case of a directed tree our weighted shift
operator coincides with the weighted adjacency operator
$A(G)$ (note that only the bounded weighted adjacency
operators are taken into consideration in \cite{f-f-s-w}).

Given a densely defined closed operator $A$ in a
complex Hilbert space $\hh$, we denote by
\idx{$\sigma(A)$}{56} and \idx{$\speca A$}{57} the
spectrum and the approximate point spectrum of $A$,
respectively. If $A \in \ogr \hh$, then
\idx{$r(A)$}{58} stands for the spectral radius of
$A$. A subset $\sigma$ of $\cbb$ is said to be {\em
circular} if
    \begin{align*}
\text{$\E^{\I t}z \in \sigma$ for all $t \in \rbb$ and
$z \in \sigma$.}
    \end{align*}
    \begin{cor} \label{sppr}
If $\slam$ is a densely defined weighted shift on a directed tree
$\tcal$, then the sets $\sigma(\slam)$, $\sigma(\slam^*)$, $\speca
\slam$ and $\speca {\slam^*}$ are circular. Moreover, if $\slam \in
\ogr{\ell^2(V)}$, then $\{z \in \cbb \colon |z| = r(\slam)\}
\subseteq \speca{\slam} \cap \speca{\slam^*}$.
    \end{cor}
    \begin{proof}
Let $N$ be as in Theorem \ref{modul} with $c=1$. Since the
operators $\E^{\I t N}$, $t \in \rbb$, are unitary and
$(\E^{\I t N})^* = \E^{-\I t N}$ for all $t \in \rbb$, we
deduce that
    \begin{align*}
\sigma(\slam) = \sigma\big((\E^{\I t N})^* \slam \E^{\I t N}\big)
\overset{\eqref{cogrr}} = \E^{\I t} \sigma(\slam), \quad t \in \rbb,
    \end{align*}
which means that $\sigma(\slam)$ and consequently
$\sigma(\slam^*)$ are circular. The same
reasoning shows that the approximate point
spectra of $\slam$ and $\slam^*$ are circular.

Suppose now that the operator $\slam$ is bounded.
Since $\sigma(\slam)$ is a nonempty compact
subset of $\cbb$, there exists $z_0 \in
\sigma(\slam)$ such that $|z_0| = r(\slam)$. By
the circularity of $\sigma (\slam)$, we see that
the circle $\varGamma := \{z \in \cbb \colon |z|
= r(\slam)\}$ is contained in $\sigma(\slam)$.
This means that $\varGamma$ is a subset of the
boundary of $\sigma(\slam)$. Hence, by
\cite[Corollary XI.1.2]{con}, $\varGamma
\subseteq \speca \slam \cap \speca {\slam^*}$.
This completes the proof.
    \end{proof}
The properties of spectra mentioned in Corollary \ref{sppr} are true
for general circular operators. For the reader's convenience we have
included their proofs. Certainly, other spectra of a circular
operator, like the point spectrum, the continuous spectrum and the
residual spectrum are circular.
    \subsection{\label{s2}Adjoints and moduli}
We begin by giving an explicit description of the
adjoint $\slam^*$ of $\slam$. Recall that $\escr$ is
the linear span of the set $\{e_u\colon u \in V\}$.
   \begin{pro} \label{sprz}
If $\slam$ is a densely defined weighted shift on a directed tree
$\tcal$ with weights $\lambdab = \{\lambda_v\}_{v \in V^\circ}$,
then the following assertions hold\/{\em :}
   \begin{enumerate}
   \item[(i)]
$\sum_{v \in \dzi u} |\lambda_v f(v)| < \infty$ for
all $u \in V$ and $f \in \ell^2(V)$,
   \item[(ii)] $\escr \subseteq \dz{\slam^*}$
and
   \begin{align} \label{sl*}
\slam^*e_u=
   \begin{cases}
   \overline{\lambda_u} e_{\pa u} & \text{if } u \in V^\circ, \\
0 & \text{if } u = \koo,
   \end{cases}
   \end{align}
   \item[(iii)] $(\slam^*f) (u) = \sum_{v \in \dzi u}
\overline{\lambda_v} f(v)$ for all $u \in V$ and $f
\in \dz{\slam^*}$,
   \item[(iv)] $\dz{\slam^*} = \big\{ f \in \ell^2(V)
\colon \sum_{u \in V} \big|\sum_{v \in \dzi u}
\overline{\lambda_v} f(v) \big|^2 < \infty \big\}$,
   \item[(v)] $\|f\|_{\slam^*}^2 = \sum_{u \in V}
\big( |f(u)|^2 + \big|\sum_{v \in \dzi u}
\overline{\lambda_v} f(v) \big|^2\big)$ for all $f \in
\dz{\slam^*}$,
   \item[(vi)] $\ell^2(\dzi u) \subseteq \dz{\slam^*}$
for every $u \in V,$
   \item[(vii)] $\slam^* = \overline{\slam^*|_{\escr}}$.
   \end{enumerate}
   \end{pro}
   \begin{proof}
   (i) By the Cauchy-Schwarz inequality and Proposition
\ref{desc}\,(iii) and (v), we have
   \begin{align*}
\Big(\sum_{v \in \dzi u} |\lambda_v f(v)|\Big)^2 \Le
\sum_{v \in \dzi u} |\lambda_v|^2 \sum_{v \in \dzi
u}|f(v)|^2 <\infty, \quad u\in V, \, f \in \ell^2(V).
   \end{align*}

   (ii) Since
   \begin{align*}
\is{\slam f}{e_{\koo}} \overset{\eqref{rkhs}} =
(\slam f)(\koo) = 0, \quad f \in \dz \slam,
   \end{align*}
we get $e_{\koo} \in \dz{\slam^*}$ and
$\slam^*e_{\koo}=0$. Assume now that $u \in
V^\circ$. Then
   \begin{align*}
\is{\slam f}{e_u} = (\slam f)(u) = \lambda_u \cdot
f\big(\pa u\big) = \is{f}{\overline{\lambda_u} e_{\pa
u}}, \quad f\in \dz\slam,
   \end{align*}
which implies that $e_u \in \dz{\slam^*}$ and $\slam^*
e_u = \overline{\lambda_u} e_{\pa u}$.

   (iii) Applying \eqref{rkhs} and Proposition
\ref{desc}\,(v), we deduce that
   \begin{multline*}
(\slam^*f)(u) = \is{\slam^* f}{e_u} \hspace{1ex}=
\is{f}{\slam e_u} \overset{\eqref{eu}}=
\isB{f}{\sum_{v\in\dzi u} \lambda_v e_v}
   \\
= \sum_{v\in\dzi u} \overline{\lambda_v} \is f {e_v} =
\sum_{v\in\dzi u} \overline{\lambda_v} f(v), \quad u
\in V, f \in \dz{\slam^*}.
   \end{multline*}

   (iv) If $f \in \dz{\slam^*}$, then
   \begin{align} \label{grapl}
\sum_{u \in V} \Big|\sum_{v \in \dzi u}
\overline{\lambda_v} f(v)\Big|^2
\overset{\mathrm{(iii)}}= \sum_{u \in V}
|(\slam^*f)(u)|^2 = \|\slam^*f\|^2 < \infty.
   \end{align}
Conversely, if $f$ belongs to the right-hand side of
(iv), then (i) enables us to define the function $g
\colon V \to \cbb$ by
   \begin{align} \label{defg}
g(u) := \sum_{v \in \dzi u} \overline{\lambda_v} f(v),
\quad u \in V.
   \end{align}
By our assumption, $g \in \ell^2(V)$. Moreover,
   \begin{multline*}
\is{\slam h}f = \sum_{u\in V} (\slam h)(u) \cdot
\overline{f(u)} \overset{\eqref{lamtf}}= \sum_{u\in
V^\circ} h(\pa u) \lambda_u \overline{f(u)}
   \\
\overset{\eqref{sumchi}}= \sum_{u\in V} \sum_{v \in
\dzi u} h(\pa v) \lambda_v \overline{f(v)} =
\sum_{u\in V} h(u) \sum_{v \in \dzi u} \lambda_v
\overline{f(v)}
   \\
\overset{ \eqref{defg} }= \sum_{u\in V} h(u)
\overline{g(u)} = \is h g, \quad h \in \dz{\slam},
   \end{multline*}
which implies that $f \in \dz{\slam^*}$ and
$g=\slam^*f$.

Assertion (v) is a direct consequence of
\eqref{grapl}.

Assertion (vi) follows from (i), (iv) and Proposition
\ref{46}.

   (vii) Take $f \in \dz{\slam^*}$ which is
orthogonal to $\{e_w\colon w \in V\}$ with
respect to the graph inner product
$\is{\cdot}{\mbox{-}}_{\slam^*}$. If $w=\koo$,
then
   \begin{align} \label{rootko}
0 = \is f {e_{\koo}}_{\slam^*} \overset{\eqref{sl*}}=
f(\koo).
   \end{align}
We show that $f$ vanishes on $\dzi u$ for every $u \in V$,
which in view of \eqref{rootko} and \eqref{sumchi} will
complete the proof. Fixing $u \in V$, we get
   \begin{align*}
0 = \is f {e_w}_{\slam^*} \overset{\eqref{sl*}} = f(w) +
\is{\slam^* f} {\overline{\lambda_w} e_{\pa w}} = f(w) +
\lambda_w (\slam^* f) (u), \quad w \in \dzi u.
   \end{align*}
Multiplying the left and the right side of the
above chain of equalities by $\overline{f(w)}$
and then summing over all $w \in \dzi u$, we
deduce from (iii) that
   \begin{align*}
0 = \Big(\sum_{w \in \dzi u} |f(w)|^2\Big) + |(\slam^* f)
(u)|^2,
   \end{align*}
which implies that $f$ vanishes on $\dzi u$. This completes
the proof.
   \end{proof}
It follows from Proposition \ref{sprz}\,(ii) that the
linear space $\escr$ is always invariant for the
adjoint $\slam^*$ of a densely defined weighted shift
$\slam$ on a directed tree $\tcal$. This is opposed to
the fact that $\escr$ may not be invariant for $\slam$
(see the comments after Proposition \ref{desc}).
   \begin{rem} \label{surp}
We now show that the adjoint of a unilateral classical
weighted shift is a weighted shift on a very
particular directed tree. Indeed, let us regard
\idxx{$\zbb_-$}{59} $\zbb_- := \{\ldots, -2, -1,0\}$
as a subtree of the directed tree $\zbb$ (cf.\ Remark
\ref{re1-2}). Certainly, $\zbb_-$ is a rootless
directed tree with only one leaf $0$. Let $\slam$ be a
weighted shift on $\zbb_-$ with weights $\lambdab =
\{\lambda_{-n}\}_{n=0}^\infty$. Then by Proposition
\ref{desc} and the equality \eqref{eu} the operator
$\slam$ is densely defined, $\slam e_{-n} =
\lambda_{-(n-1)} e_{-(n-1)}$ for all $n \in \nbb$, and
$\slam e_{0}=0$. This fact combined with Proposition
\ref{desc}\,(vi) guarantees that $\slam$ can be
thought of as the adjoint of the unilateral classical
weighted shift with weights
$\{\bar\lambda_{-(n-1)}\}_{n=1}^\infty$ (cf.\
\cite[equality (1.11)]{ml}).
   \end{rem}
We now describe powers of the modulus of $\slam$.
   \begin{pro}\label{3}
If $\slam$ is a densely defined weighted shift on
a directed tree $\tcal$ with weights $\lambdab =
\{\lambda_v\}_{v \in V^\circ}$, then for every
real $\alpha > 0$,
   \begin{enumerate}
   \item[(i)] $\escr \subseteq \dz{|\slam|^\alpha}$,
   \item[(ii)] $\escr$ is a core of $|\slam|^\alpha$,
i.e., $|\slam|^\alpha =
\overline{|\slam|^\alpha|_{\escr}}$,
   \item[(iii)] $|\slam|^\alpha e_u = \|\slam e_u\|^\alpha e_u$ for
$u \in V,$
   \item[(iv)] $(|\slam|^\alpha f) (u) = \|\slam e_u\|^\alpha
f(u)$ for $u \in V$ and $f \in \dz{|\slam|^\alpha}$.
   \end{enumerate}
   \end{pro}
   \begin{proof}
   (i)--(iii) Applying Proposition \ref{sprz}\,(iii),
we get
   \begin{multline} \label{kwad}
(\slam^* \slam f) (u) = \sum_{v \in \dzi u}
\overline{\lambda_v} (\slam f)(v)
\overset{\eqref{lamtf}}= \sum_{v \in \dzi u}
|\lambda_v|^2 f(\pa v)
   \\
= \sum_{v\in \dzi{u}} |\lambda_v|^2 f(u)
\overset{\eqref{eu}}= \|\slam e_u\|^2 f(u), \quad u
\in V, \, f \in \dz{\slam^*\slam}.
   \end{multline}
Now, we show that $\escr \subseteq \dz
{\slam^*\slam}$. Indeed, if $u \in V$, then
   \begin{align*}
\sum_{w \in V} \Big| \sum_{v \in \dzi w} \overline
{\lambda_v} (\slam e_u)(v)\Big|^2
\overset{\eqref{lamtf}}= \sum_{w \in V} \Big| \sum_{v
\in \dzi w} |\lambda_v|^2 e_u(w)\Big|^2
\overset{\eqref{eu}}= \|\slam e_u\|^4 < \infty,
   \end{align*}
   which, by Proposition  \ref{sprz}\,(iv), implies that
$\slam e_u \in \dz{\slam^*}$. Hence, \eqref{kwad}
leads to
   \begin{align} \label{slgsl}
\slam^*\slam e_u = \|\slam e_u\|^2 e_u, \quad u \in V.
   \end{align}
In view of Proposition \ref{clos} and
\cite[Theorem 5.39]{weid}, the operator
$\slam^*\slam$ is positive and selfadjoint. Using
\eqref{slgsl} and applying Lemma \ref{lems} with
$\alpha/2$ in place of $\alpha$ to the operator
$\slam^*\slam$, we obtain (i), (ii) and (iii).

   (iv) It follows from \eqref{rkhs} that for all $f
\in \dz{|\slam|^\alpha}$ and $u \in V$,
   \begin{align*}
(|\slam|^\alpha f) (u) = \is{|\slam|^\alpha f}{e_u}
\overset{\mathrm{(i)}}= \is{f}{|\slam|^\alpha e_u}
\overset{\mathrm{(iii)}}= \is{f}{\|\slam e_u\|^\alpha
e_u} = \|\slam e_u\|^\alpha f(u).
   \end{align*}
This completes the proof.
   \end{proof}
   \begin{cor} \label{chariso}
A weighted shift $\slam$ on a directed tree $\tcal$ is
an isometry on $\ell^2(V)$ if and only if
$\sum_{v\in\dzi u} |\lambda_v|^2=1$ for all $u\in V$.
   \end{cor}
   \begin{proof}
Apply Propositions \ref{desc}, \ref{ogrs} and \ref{3}
(with $\alpha=2$).
   \end{proof}
Since the compactness and the membership in the
Schatten-von Neumann $p$-class depend on
analogous properties of the modulus of the
operator in question (cf.\ \cite{Sch,Rin}), the
following corollary is an immediate consequence
of Proposition~ \ref{3}. Below, the limit
$\lim_{u \in V}$ and the sum $\sum_{u \in V}$ are
understood in a generalized sense.
   \begin{cor}
Let $\slam \in \ogr{\ell^2(V)}$ be a weighted shift on
a directed tree $\tcal$ and let $p\in [1,\infty)$.
Then the following two assertions hold.
   \begin{enumerate}
   \item[(i)] $\slam$ is compact if and only if $\lim_{u \in V}
\|\slam e_u\| = 0$,
   \item[(ii)] $\slam$ is in the Schatten $p$-class
if and only if $\sum_{u \in V} \|\slam e_u\|^p
<\infty$.
   \end{enumerate}
   \end{cor}
    \subsection{\label{s2-polar}The polar decomposition}
We now describe the polar decomposition of a densely
defined weighted shift on a directed tree. Recall,
that if $T$ is a closed densely defined operator in a
complex Hilbert space $\hh$, then there exists a
unique partial isometry $U \in \ogr{\hh}$ with initial
space $\overline{\ob{|T|}}$ such that $T=U|T|$. Such
decomposition is called the {\em polar decomposition}
of $T$ (see e.g., \cite[Theorem 8.1.2]{b-s}). If
$T=U|T|$ is the {\em polar decomposition} of $T$, then
the final space of $U$ equals $\overline{\ob T}$.
   \begin{pro}\label{polar}
Let $\slam$ be a densely defined weighted shift on a
directed tree $\tcal$ with weights $\lambdab =
\{\lambda_v\}_{v \in V^\circ}$, and let $\slam =
U|\slam|$ be the polar decomposition of $\slam$. Then
$|\slam|$ is the diagonal operator subordinated to the
orthonormal basis $\{e_u\}_{u \in V}$ with diagonal
elements $\{\|\slam e_u\|\}_{u \in V}$, and $U$ is the
weighted shift $S_\pib$ on $\tcal$ with weights
$\pib=\{\pi_{u}\}_{u \in V^\circ}$ given
by\,\footnote{\;For simplicity, we suppress the
explicit dependence of $\pib$ on $\lambdab$ in the
notation.}
   \begin{align} \label{pidef}
\pi_u =
   \begin{cases}
\cfrac{\lambda_u}{\|\slam e_{\pa u}\|} & \text{if }
\pa u \in \vplus,
   \\[2.5ex]
   0 & \text{if } \pa u \notin \vplus,
   \end{cases}
\quad u \in V^\circ,
   \end{align}
where \idxx{$\vplus$}{60} $\vplus:= \{u \in V\colon
\|\slam e_u\| > 0\}$. Moreover, the following
assertions hold\/{\em :}
   \begin{enumerate}
   \item[(i)] $\jd{U}=\jd{\slam} = \jd{|\slam|}=\ell^2(V \setminus
\vplus)$ and $\overline{\ob{\slam^*}}=\ell^2(\vplus)$,
   \\[-0.5ex]
   \item[(ii)]   $\jd{\slam^*} =
   \begin{cases}
\langle e_{\koo} \rangle \oplus \bigoplus_{u \in
V^\prime} \big(\ell^2(\dzi u) \ominus \langle
\lambdab^u \rangle\big) & \text{if $\tcal$ has a
root,}
   \\[.5ex]
\bigoplus_{u \in V^\prime} \big(\ell^2(\dzi u)
\ominus \langle \lambdab^u \rangle\big) &
\text{otherwise,}
   \end{cases}$
\\[2ex] where \idxx{$\lambdab^u$}{60} $\lambdab^u
\in \ell^2(\dzi u)$ is given by $\lambdab^u\colon \dzi
u \ni v \to \lambda_v \in \cbb$, and
\idxx{$<\hspace{-.5ex}\lambdab^u \hspace{-.6ex}>$}{61}
$\langle \lambdab^u\rangle$ is the linear span of
$\{\lambdab^u\}$,
   \item[(iii)]  the initial space of $U$ equals
$\ell^2(\vplus)$,
   \item[(iv)] $\ob{U} = \overline{\ob{\slam}} =\bigoplus_{u
\in V^\prime} \langle \lambdab^u\rangle$.
   \end{enumerate}
   \end{pro}
   \begin{proof}
It follows from Proposition \ref{clos} that $\slam$ is
closed. By Lemma \ref{lems} and Proposition \ref{3},
$|\slam|$ is the diagonal operator subordinated to the
orthonormal basis $\{e_u\}_{u \in V}$ with diagonal
elements $\{\|\slam e_u\|\}_{u \in V}$. Hence, the
initial space $\overline{\ob{|\slam|}}$ of $U$ equals
$\ell^2(\vplus)$. This means that (iii) holds. As a
consequence of (iii), we obtain (i). If $u \in
\vplus$, then
   \begin{multline} \label{10.IV}
Ue_u = \frac 1 {\|\slam e_u\|} U |\slam| e_u = \frac 1
{\|\slam e_u\|} \slam e_u
   \\
\overset{\eqref{eu}}= \sum_{v \in \dzi u} \frac
{\lambda_v}{\|\slam e_u\|} e_v = \sum_{v \in \dzi u}
\frac {\lambda_v}{\|\slam e_{\pa v}\|} e_v.
   \end{multline}
In turn, if $u \in V\setminus \vplus$, then by (i) $e_u \in
\ell^2(V \setminus \vplus) = \jd{U}$, and so $Ue_u = 0$.
Using \eqref{pidef} and Proposition \ref{ogrs}, we see that
$S_\pib \in \ogr{\ell^2(V)}$. Since, by \eqref{eu},
\eqref{pidef} and \eqref{10.IV}, both operators $U$ and
$S_\pib$ coincide on basic vectors $e_u$, $u \in V$, we
deduce that $U=S_\pib$.

We now describe the final space of $U$. Consider first
the case when $\tcal$ has a root. It follows from
Proposition \ref{sprz}\,(iii) and (iv), and
Proposition \ref{46} that
   \begin{multline}     \label{10.IVb}
\jd{\slam^*} = \{f \in \ell^2(V) \colon \is{f|_{\dzi
u}}{\lambdab^u}=0 \text{ for all } u \in V^\prime\}
   \\
= \langle e_{\koo} \rangle \oplus \bigoplus_{u \in
V^\prime} \Big(\ell^2(\dzi u) \ominus \langle
\lambdab^u \rangle\Big).
   \end{multline}
Since, by Proposition~ \ref{46}, $\ell^2(V) = \langle
e_{\koo} \rangle \oplus \bigoplus_{u \in V^\prime}
\ell^2(\dzi{u})$, we deduce from \eqref{10.IVb} that
$\overline{\ob{\slam}} = \bigoplus_{u \in V^\prime}
\langle \lambdab^u\rangle$. If $\tcal$ is rootless,
then \eqref{10.IVb} holds with $\langle
e_{\koo}\rangle$ removed. As a consequence, we get the
same formula for $\overline{\ob{\slam}}$. This proves
assertions (ii) and (iv), and hence completes the
proof.
   \end{proof}
Using the description of the polar decomposition of
$\slam$ given in Proposition \ref{polar} and the last
paragraph of the proof of Lemma \ref{lems} (with
$A=|\slam|$), we see that $\ob{\slam} = S_\pib
(\ob{|\slam|}) = S_\pib (\dz{A_1^{-1}})$, where
$A_1^{-1}$ is a diagonal operator with diagonal
elements $\big\{\frac 1 {\|\slam e_u\|}\big\}_{u \in
\vplus}$. The details are left to the reader.
   \subsection{\label{fredo}Fredholm directed trees}
We begin by describing weighted shifts on directed
trees with closed ranges. Recall that $\vplus= \{u \in
V\colon \|\slam e_u\|
> 0\} \subseteq V^\prime$.
   \begin{pro} \label{domknietosc}
Let $\slam$ be a densely defined weighted shift on a
directed tree $\tcal$ with weights $\lambdab =
\{\lambda_v\}_{v \in V^\circ}$. Then the following
conditions are equivalent\/{\em :}
   \begin{enumerate}
   \item[(i)] $\ob{\slam}$ is closed,
   \item[(ii)] $\ob{\slam^*}$ is closed,
   \item[(iii)] there exists a real number $\delta >0$
such that $\|\slam e_u\| \Ge \delta$ for every $u \in
\vplus$.
   \end{enumerate}
   \end{pro}
   \begin{proof}
Equivalence (i)$\Leftrightarrow$(ii) holds for general
closed densely defined Banach space operators (cf.\
\cite[Theorem IV.1.2]{Gol}).

(i)$\Leftrightarrow$(iii) It follows from the polar
decomposition of $\slam$ that $\ob{\slam}$ is closed
if and only if $\ob{|\slam|}$ is closed. This fact
combined with Proposition \ref{polar} and Lemma
\ref{lems} completes the proof.
   \end{proof}
   Recall that a closed densely defined operator $T$
in a complex Hilbert space $\hh$ is said to be {\em
Fredholm} if its range is closed and the spaces
$\jd{T}$ and $\jd{T^*}$ are finite dimensional. It is
well known that a closed densely defined operator $T$
in $\hh$ is Fredholm if and only if the spaces
$\jd{T}$ and $\hh / \ob{T}$ are finite dimensional
(cf.\ \cite[Corollary IV.1.13]{Gol}). The {\em index}
\idx{$\ind T$}{62} of a Fredholm operator $T$ is given
by $\ind T=\dim \jd{T}-\dim\jd{T^*}$ (see \cite[\S
4.2]{Gol} and \cite{Goh-Gol-Kaa} for the case of
unbounded operators and \cite{con} for bounded ones).

Fredholm weighted shifts on directed trees and their
indexes can be characterized as follows (below we adopt the
convention that $\inf \varnothing = \infty$).
   \begin{pro}\label{fredholm}
Let $\slam$ be a densely defined weighted shift on a
directed tree $\tcal$ with weights $\lambdab =
\{\lambda_v\}_{v \in V^\circ}$. Then the following
conditions are equivalent\/{\em :}
   \begin{enumerate}
   \item[(i)] $\slam$ is a Fredholm operator,
   \item[(ii)] $\mathfrak c(\slam) > 0$ and
$\mathfrak b(\slam) < \infty$, where \idxx{$\mathfrak
b(\slam)$}{63}
   \begin{align*}
\mathfrak b(\slam) & :=\sum_{u \in \vplus}
(\card{\dzi u}-1) + \sum_{u \in V^\prime
\setminus \vplus} \card{\dzi u},
      \\
\mathfrak c(\slam) & := \inf\{|\lambda_u|\colon u
\in V^\circ, \lambda_u \neq 0,
\card{\dzi{\pa{u}}} = 1\},
   \end{align*}
   \item[(iii)] $\mathfrak c(\slam) > 0$, $\card{\dzi u} < \infty$ for all $u \in V$,
$\card{V_{\prec}} < \infty$ $($see \eqref{prec}
for the definition of $V_\prec$$)$ and
$\card{V^\prime \setminus \vplus} < \infty$.
   \end{enumerate}
If $\slam$ is Fredholm, then \idxx{$\mathfrak
a(\slam)$}{64} $\mathfrak a(\slam) := \card{V
\setminus \vplus} < \infty$ and
   \begin{align} \label{finf}
\ind \slam =
   \begin{cases}
   \mathfrak a(\slam)-\mathfrak b(\slam) - 1 &
\text{if $\tcal$ has a root,}
   \\
   \mathfrak a(\slam)-\mathfrak b(\slam) &
\text{otherwise.}
   \end{cases}
   \end{align}
   \end{pro}
   \begin{proof}
First we prove that if $\mathfrak b(\slam) < \infty$,
then $\mathfrak a(\slam) < \infty$, $\card{\dzi u} <
\infty$ for all $u \in V$ and $\card{V_{\prec}} <
\infty$. We begin by recalling that $\vplus \subseteq
V^\prime \subseteq V$. Then an easy computation shows
that
   \begin{align*}
   \begin{aligned} &\card{V_{\prec}} \Le \mathfrak b(\slam) <
\infty,
   \\
&\card{\dzi{u}} \Le \mathfrak b(\slam) + 1 < \infty,
\quad u \in V.
   \end{aligned}
   \end{align*}
Hence, by Proposition \ref{fnl}, $\card{V \setminus
V^\prime} < \infty$. Since $\card{V^\prime \setminus
\vplus} \Le \mathfrak b(\slam) < \infty$, we get
$\mathfrak a(\slam)=\card{V \setminus V^\prime} +
\card{V^\prime \setminus \vplus} < \infty$. This
proves our claim.

Reversing the above reasoning we deduce that if $\card{\dzi
u} < \infty$ for all $u \in V$, $\card{V_{\prec}} < \infty$
and $\card{V^\prime \setminus \vplus} < \infty$, then
$\mathfrak b(\slam) < \infty$ and consequently $\mathfrak
a(\slam) < \infty$.

(i)$\Rightarrow$(ii) Employing Proposition
\ref{polar}, we see that $\mathfrak a(\slam) <
\infty$, $\mathfrak b(\slam)<\infty$ and \eqref{finf}
holds. By Proposition \ref{domknietosc}, we conclude
that $\mathfrak c(\slam) > 0$.

(ii)$\Rightarrow$(i) In view of the first
paragraph of this proof, the implication
(ii)$\Rightarrow$(i) can be deduced from
Propositions \ref{polar} and \ref{domknietosc}.

(ii)$\Leftrightarrow$(iii) Apply two first paragraphs
of this proof.
   \end{proof}
Owing to Proposition \ref{fredholm}, a densely defined
weighted shift $\slam$ on a directed tree $\tcal$ with
nonzero weights $\lambdab = \{\lambda_v\}_{v \in V^\circ}$
and with closed range is Fredholm if and only if
$\card{\dzi u} < \infty$ for all $u \in V$ and
$\card{V_{\prec}} < \infty$. Moreover,
   \begin{align} \label{indsl}
\ind\slam =
   \begin{cases}
\card{V \setminus V^\prime} - 1 - \sum\limits_{u \in
V^\prime} \big(\card{\dzi u}-1\big) & \text{if $\tcal$
has a root},
   \\[2.5ex]
\card{V \setminus V^\prime} - \sum\limits_{u \in
V^\prime} \big(\card{\dzi u}-1\big) &
\text{otherwise.}
   \end{cases}
   \end{align}
A trivial verification shows that
   \begin{align} \label{indsl1}
\ind\slam =
   \begin{cases}
\card{(V \setminus V^\prime)\sqcup V_{\prec}} - 1 -
\sum\limits_{u \in V_{\prec}} \card{\dzi u} & \text{if
$\tcal$ has a root},
   \\[2.5ex]
\card{(V \setminus V^\prime) \sqcup V_{\prec}} -
\sum\limits_{u \in V_{\prec}} \card{\dzi u} &
\text{otherwise.}
   \end{cases}
   \end{align}
Noting that the right-hand side of \eqref{indsl} does
not depend on $\slam$, we propose the following
definition.
   \begin{dfn} \label{treeind}
A directed tree $\tcal$ such that $\card{\dzi u} <
\infty$ for all $u \in V$ and $\card{V_{\prec}} <
\infty$ is called {\em Fredholm}. The right-hand side
of \eqref{indsl} (which is equal to the right-hand
side of \eqref{indsl1}) is denoted by $\ind\tcal$ and
called the {\em index} of a Fredholm directed tree
$\tcal$.
   \end{dfn}
   The definition of $\ind{\slam}$ is correct due
to Proposition \ref{fnl}. We can rephrase
Definition \ref{treeind} as follows:\ a directed
tree is Fredholm if and only if it has finitely
many branching vertexes and each branching vertex
has finitely many children. Moreover, each
Fredholm directed tree has finitely many leaves.
As a consequence, we see that there exists
countably many non-isomorphic Fredholm directed
trees.

The following is a beneficial excerpt from the
proof of Proposition \ref{fredholm}.
   \begin{pro}\label{bfred}
If $\slam$ is a densely defined weighted shift on a
directed tree $\tcal$ with weights $\lambdab =
\{\lambda_v\}_{v \in V^\circ}$ and $\mathfrak
b(\slam)<\infty$, then $\tcal$ is Fredholm and
$\mathfrak a(\slam)<\infty$.
   \end{pro}
It is worth mentioning that for every Fredholm directed
tree $\tcal$ we may construct a bounded Fredholm weighted
shift $\slam$ on $\tcal$ with nonzero weights. Indeed, in
view of Propositions \ref{ogrs} and \ref{fredholm}, the
weighted shift $\slam$ with weights $\lambda_u \equiv 1$
meets our requirements.

The assertion (i) of Lemma \ref{przycinanie} below
shows that after cutting off a leaf of a Fredholm
directed tree $\tcal$ the index of the trimmed tree
remains the same as that of $\tcal$. The assertion
(ii) states that after cutting off a straight infinite
branch $\des u$ from a simply branched leafless
subtree $\des w$ of $\tcal$, where $u$ is a child of
$w$, the index of the trimmed tree enlarges by $1$.
Finally, the assertion (iii) says that after cutting
off a trunk of a rootless $\tcal$ the index of the
trimmed tree decreases by $1$. For the definition of
the subgraph $\gcal_W$, we refer the reader to
\eqref{subtree}.
   \begin{lem} \label{przycinanie}
If $\tcal$ is a Fredholm directed tree, then the
following assertions hold.
   \begin{enumerate}
   \item[(i)] If $w \in V \setminus V^\prime$
and $V_w := V \setminus \{w\} \neq \varnothing$, then
$\ind\tcal=\ind{\tcal_{V_w}}$.
   \item[(ii)] If $w \in V$ is such that $\card {\dzi
{w}} \Ge 2$ and $\card{\dzi{v}} =1$ for all $v \in
\des {w} \setminus \{w\}$, then $\ind{\tcal_{V_{[u]}}}
=\ind\tcal+1$ for every $u \in \dzi{w}$, where
$V_{[u]} := V \setminus \des{u}$.
   \item[(iii)] If $\tcal$ is rootless and $w \in V$
is such that $V = \{\paa^{n}(w)\}_{n=1}^\infty \sqcup
\des w$ $($such $w$ always exists$)$, then
$\ind{\tcal_{\des w}} = \ind\tcal-1$.
   \end{enumerate}
   \end{lem}
Note that by Proposition \ref{subdir}, $\tcal_{V_w}$,
$\tcal_{V_{[u]}}$ and $\tcal_{\des w}$ are directed
trees.
   \begin{proof}[Proof of Lemma  \ref{przycinanie}]
   (i) It is clear that
   \begin{align*}
   E_{\tcal} = E_{\tcal_{V_w}} \sqcup \{(\pa w, w)\}
   \end{align*}
($\pa w$ makes sense because $V_w \neq\varnothing$). This
implies that $\dzit{\tcal}u = \dzit{\tcal_{V_w}}u$ for $u
\in V_w \setminus \{\pa w\}$ and $\dzit{\tcal}{\pa w} =
\dzit{\tcal_{V_w}}{\pa w} \sqcup \{w\}$. Consider first the
case when $\card{\dzit{\tcal}{\pa w}} \Ge 2$. Then
$V^\prime = V_w^\prime$ and consequently $\card{V \setminus
V^\prime} = \card{V_w\setminus V_w^\prime} + 1$. This
altogether implies that $\ind\tcal=\ind{\tcal_{V_w}}$. In
turn, if $\card{\dzit{\tcal}{\pa w}} = 1$, then arguing as
above, we see that $V^\prime = V_w^\prime \sqcup \{\pa w\}$
and $\pa w \in V_w \setminus V_w^\prime$, hence that $V
\setminus V^\prime = \{w\} \sqcup \big((V_w \setminus
V_w^\prime) \setminus \{\pa w\}\big)$ and $\card{V
\setminus V^\prime} = \card{V_w\setminus V_w^\prime}$, and
finally that $\ind\tcal=\ind{\tcal_{V_w}}$. In particular,
we have the following equalities
   \begin{align} \label{vvprim}
   \card{V_w\setminus V_w^\prime} =
   \begin{cases}
\card{V \setminus V^\prime} - 1 & \text{if }
\card{\dzit{\tcal}{\pa w}} \Ge 2,
   \\
\card{V \setminus V^\prime} & \text{if }
\card{\dzit{\tcal}{\pa w}} = 1.
   \end{cases}
   \end{align}

   (ii) It is plain that $V^\prime = V_{[u]}^\prime
\sqcup \des{u}$ and consequently that $V \setminus
V^\prime = V_{[u]} \setminus V_{[u]}^\prime$. If $v
\in V_{[u]}^\prime \setminus \{w\}$, then
$\dzit{\tcal}v = \dzit{\tcal_{V_{[u]}}}v$. In turn, if
$v=w$, then $\dzit{\tcal}v = \dzit{\tcal_{V_{[u]}}}v
\sqcup \{u\}$. Finally, if $v \in \des u$, then
$\card{\dzit{\tcal}v}=1$. All this implies that
$\ind{\tcal_{V_{[u]}}} = \ind\tcal+1$.

   (iii) Arguing as in the first paragraph of the
proof of Proposition \ref{fnl} and using
\eqref{indsl1}, we show that $\ind{\tcal_{\des w}} =
\ind\tcal-1$ (remember that the directed tree
$\tcal_{\des w}$ has a root, while $\tcal$ not). This
completes the proof.
   \end{proof}
We now show that the index of a Fredholm directed tree
does not exceed $1$.
   \begin{lem}\label{le1}
If $\tcal$ is a Fredholm directed tree, then
   \begin{enumerate}
   \item[(i)] $\ind \tcal \Le 1$ provided $\tcal$ is rootless,
   \item[(ii)] $\ind \tcal =  0$ provided $\tcal$  is
finite,
   \item[(iii)] $\ind \tcal \Le -1$ provided $\tcal$ has a root
and is infinite.
   \end{enumerate}
   \end{lem}
   \begin{proof}
Without loss of generality, we can assume that $\tcal$
is infinite (otherwise $\ind \tcal = 0$). If $\tcal$
is rootless, then by using assertion (iii) of Lemma
\ref{przycinanie} we are reduced to showing that
$\ind\tcal \Le -1$ whenever $\tcal$ has a root and is
infinite (note that the trimmed tree may be finite,
however this case has been covered).

If $\tcal$ is leafless, then by \eqref{indsl} we have
   \begin{align} \label{le-1}
\ind\tcal = -1 - \sum_{u \in V_{\prec}}
\big(\card{\dzi u}-1\big) \Le -1.
   \end{align}

Suppose now that $\tcal$ is not leafless (however
still infinite and with root). Then
$\varkappa(\tcal) := \card{V\setminus V^\prime}
\Ge 1$. We claim that there exists an infinite
subtree $\check\tcal$ of $\tcal$ such that
$\varkappa(\check \tcal) = \varkappa(\tcal) - 1$
and $\ind{\check \tcal} = \ind\tcal$. To prove
our claim, take any $w \in V \setminus V^\prime$.
Let $k$ be the least positive integer such that
$\paa^{k}(w) \in V_{\prec}$ (the existence of
such $k$ is justified in the second paragraph of
the proof of Proposition \ref{fnl}). Applying
\eqref{vvprim} and assertion (i) of Lemma
\ref{przycinanie} $k$ times, we get the required
subtree $\check \tcal$, which proves our claim.
Finally, employing the reduction procedure $\tcal
\rightsquigarrow \check \tcal$ $\varkappa(\tcal)$
times, we find an infinite subtree $\check \tcal$
of $\tcal$ such that $\ind\tcal = \ind{\check
\tcal}$ and $\varkappa(\check \tcal)=0$. Since
the latter means that $\check \tcal$ is leafless,
we deduce from \eqref{le-1} that $\ind\tcal =
\ind{\check \tcal} \Le -1$. This completes the
proof.
   \end{proof}
Applying \eqref{indsl1} and Lemma \ref{le1} we get the
following estimates for Fredholm directed trees.
   \begin{cor}
If $\tcal$ is an infinite Fredholm directed tree,
then
   \begin{align*}
\card{V \setminus V^\prime} + \card{V_{\prec}} & \Le
\sum\limits_{u \in V_{\prec}} \card{\dzi u} \quad
\text{if $\tcal$ has a root,}
   \\
\card{V \setminus V^\prime} + \card{V_{\prec}} & \Le 1
+ \sum\limits_{u \in V_{\prec}} \card{\dzi u} \quad
\text{otherwise.}
   \end{align*}
   \end{cor}
Note that for every integer $k\Le 1$ there exists
a Fredholm directed tree $\tcal$ such that
$\ind{\tcal}=k$. Indeed, it is clear that
$\zbb_-$, $\zbb$ and $\zbb_+$ are Fredholm
directed trees (see Remarks \ref{re1-2} and
\ref{surp} for appropriate definitions), and
$\ind{\zbb_-}=1$, $\ind{\zbb}=0$ and
$\ind{\zbb_+}=-1$. The directed trees $\zbb$ and
$\zbb_+$ are leafless and they have no branching
vertexes. If we consider any rootless and
leafless directed tree $\tcal$ with $|k|$
branching vertexes ($k \Le 0$) each of which
having exactly two children, then $\ind \tcal=k$.

It turns out that the index of a Fredholm weighted
shift on a directed tree does not exceed $1$ even if
some weights of $\slam$ are zero.
   \begin{thm} \label{indle1}
If $\slam$ is a Fredholm weighted shift on a directed
tree $\tcal$ with weights $\lambdab = \{\lambda_v\}_{v
\in V^\circ}$, then $\tcal$ is Fredholm and
$\ind{\slam}=\ind{\tcal} \Le 1$.
   \end{thm}
   \begin{proof}
It follows from Proposition \ref{fredholm} that
$\tcal$ is Fredholm~ and
   \begin{align} \label{temp1}
\card{V^\prime \setminus \vplus} \Le \mathfrak
a(\slam)<\infty.
   \end{align}
Hence, the following equalities hold
   \allowdisplaybreaks
   \begin{align*}
   \mathfrak a(\slam) - \mathfrak b(\slam) &
\hspace{1ex} \hspace{.8ex}=\card{V \setminus \vplus}
-\sum_{u \in \vplus} (\card{\dzi u}-1) - \sum_{u \in
V^\prime \setminus \vplus} \card{\dzi u}
   \\
& \overset{ \eqref{temp1}} = \card{V \setminus
V^\prime} + \card{V^\prime \setminus \vplus}
   \\
& \hspace{10ex}- \sum_{u \in V^\prime} (\card{\dzi
u}-1) + \sum_{u \in V^\prime \setminus \vplus}
(\card{\dzi u}-1)
   \\
& \hspace{40ex} - \sum_{u \in V^\prime \setminus
\vplus} \card{\dzi u}
   \\
&\hspace{1.8ex}= \card{V \setminus V^\prime} - \sum_{u
\in V^\prime} (\card{\dzi u}-1).
   \end{align*}
By \eqref{finf} and \eqref{indsl}, this implies
that $\ind{\slam}=\ind{\tcal}$. Employing Lemma
\ref{le1} completes the proof.
   \end{proof}
It may happen that a directed tree $\tcal$ is
Fredholm but not every densely defined weighted
shift $\slam$ on $\tcal$ with closed range is
Fredholm (e.g., the directed tree $\tcal =
\tcal_{2,0}$ defined in \eqref{varkappa} has the
required properties; it is enough to consider a
weighted shift $\slam$ on $\tcal$ with infinite
number of zero weights, whose nonzero weights are
uniformly separated from zero).

The theory of semi-Fredholm operators can be
implemented into the context of weighted shift
operators on directed trees as well. Recall that
a closed densely defined operator $T$ in a
complex Hilbert space $\hh$ is said to be {\em
left semi-Fredholm} if its range is closed and
$\dim \jd T < \infty$. If $T^*$ is left
semi-Fredholm, then $T$ is called {\em right
semi-Fredholm}. It is well known that a closed
densely defined operator $T$ in $\hh$ is right
semi-Fredholm if and only if $T$ has closed range
and $\dim \jd{T^*} < \infty$. Hence, in view of
\cite[Corollary IV.1.13]{Gol}, a closed densely
defined operator $T$ in $\hh$ is right
semi-Fredholm if and only if the quotient space
$\hh / \ob{T}$ is finite dimensional. If $T$ is
either left semi-Fredholm or right semi-Fredholm,
then its {\em index} $\ind{T}$ is defined by
$\ind{T}= \dim \jd{T}-\dim\jd{T^*}$. In general,
$\ind{T} \in \zbb \sqcup \{\pm \infty\}$ (cf.\
\cite{Gol,con}).
   \begin{pro} \label{semifred}
Let $\slam$ be a densely defined weighted shift on a
directed tree $\tcal$ with weights $\lambdab =
\{\lambda_v\}_{v \in V^\circ}$. Then the following
assertions hold.
   \begin{enumerate}
   \item[(i)] $\slam$ is left
semi-Fredholm if and only if $\slam$ has closed range
$($cf.\ {\em Proposition \ref{domknietosc}}$)$ and
$\card{V \setminus \vplus}<\infty$. Moreover, if
$\slam$ is left semi-Fredholm, then $\ind{\slam} \in
\{n \in \zbb\colon n \Le 1\} \sqcup \{-\infty\}$.
   \item[(ii)] $\slam$ is right semi-Fredholm if and only if it
is Fredholm.
   \end{enumerate}
   \end{pro}
   \begin{proof}
Assertion (i) is a direct consequence Proposition
\ref{polar}\,(i) and Theorem \ref{indle1}.

(ii) If $\slam$ right semi-Fredholm, then by
Proposition \ref{polar}\,(ii) $\mathfrak
b(\slam)<\infty$. This, in view of Proposition
\ref{bfred}, implies that $\mathfrak
a(\slam)<\infty$. Thus $\slam$ is Fredholm.
   \end{proof}
Suppose that $\slam$ is any densely defined weighted shift
on a directed tree {\em with nonzero weights}. It follows
from Proposition \ref{semifred} that $\slam$ is left
semi-Fredholm if and only if it has closed range and the
directed tree $\tcal$ has finitely many leaves. As a
consequence, we see that if $\slam$ has closed range and
$\tcal$ is leafless, then $\slam$ is always left
semi-Fredholm. Certainly, it many happen that $\slam$ is
left semi-Fredholm but not Fredholm (in fact, this happens
more frequently). For example, the isometric weighted shift
$\slam$ on the leafless directed tree $\tcal_{\infty,0}$
defined in Example \ref{sub2} (with one branching vertex
$\omega$ such that $\card{\dzi{\omega}} = \aleph_0$) is
left semi-Fredholm and $\ind{\slam}=-\infty$.
   \newpage
   \section{\label{chap4}Inclusions of domains}
   \subsection{\label{s3}When $\text{$\dz{\slam} \subseteq
\dz{\slam^*}$}$?}
   Our next aim is to characterize the circumstances
under which the inclusion $\dz{\slam} \subseteq
\dz{\slam^*}$ holds.
   \begin{thm}\label{slssl*}
If $\slam$ is a densely defined weighted shift on a
directed tree $\tcal$ with weights $\lambdab =
\{\lambda_v\}_{v \in V^\circ}$, then the following
conditions are equivalent{\em :}
   \begin{enumerate}
   \item[(i)] $\dz{\slam} \subseteq \dz{\slam^*}$,
   \item[(ii)] there exists $c>0$ such that
   \begin{align} \label{incofds}
\sum_{v \in \dzi u} \frac {|\lambda_v|^2} {1+ \|\slam
e_v\|^2}\Le c,\quad u\in V.
   \end{align}
   \end{enumerate}
   \end{thm}
   \begin{proof}
   (i)$\Rightarrow$(ii) Recall that by
Proposition \ref{desc}\,(v), $\escr \subseteq \dz
{\slam}$. Suppose that $\dz{\slam} \subseteq
\dz{\slam^*}$. By Proposition \ref{clos}, the
normed space $(\dz{\slam},
\text{\mbox{$\|\cdot\|_{\slam}$}})$ is complete.
The normed space $(\dz{\slam^*},
\text{\mbox{$\|\cdot\|_{\slam^*}$}})$ is complete
as well (cf.\ \cite[Theorems 3.2.1 and
3.3.2]{b-s}). Applying the closed graph theorem
to the identity embedding mapping $(\dz{\slam},
\text{\mbox{$\|\cdot\|_{\slam}$}})\to
(\dz{\slam^*},
\text{\mbox{$\|\cdot\|_{\slam^*}$}})$ we see that
there exists a positive real number $c$ such that
   \begin{align*}
\|f\|_{\slam^*}^2 \Le c \, \|f\|_{\slam}^2, \quad f
\in \dz{\slam}.
   \end{align*}
By Propositions \ref{sprz}\,(v) and \ref{desc}\,(ii),
the above inequality implies that
   \begin{multline}  \label{malysz}
\sum_{u \in V} \Big| \sum_{v \in \dzi u} \overline
{\lambda_v} f(v)\Big|^2 \Le \|f\|_{\slam^*}^2 \Le c
\|f\|_{\slam}^2
   \\
\overset{\eqref{eu}}= c \sum_{u \in V} \big (1+\|\slam
e_u \|^2\big) |f(u)|^2, \quad f \in \dz{\slam}.
   \end{multline}
First, we consider the case when the tree $\tcal$ has a root.
Employing \eqref{sumchi}, we get
   \begin{multline*}
\sum_{u \in V} \Big| \sum_{v \in \dzi u} \overline
{\lambda_v} f(v)\Big|^2 \Le c (1+\|\slam
e_{\koo}\|^2\big) |f(\koo)|^2
   \\
+ c \sum_{u \in V} \sum_{v \in \dzi u} \big(1+\|\slam
e_v\|^2\big) |f(v)|^2, \quad f \in \dz{\slam}.
   \end{multline*}
Since, by Proposition \ref{desc}\,(iv), the function
$f \cdot \chi_{V^\circ}$ is in $\dz{\slam}$ for every
$f \in \dz{\slam}$, we see that the above inequality
is equivalent to the following one
   \begin{multline}     \label{in1}
\sum_{u \in V} \Big| \sum_{v \in \dzi u} \overline
{\lambda_v} f(v)\Big|^2 \Le c \sum_{u \in V} \sum_{v
\in \dzi u} \big(1+\|\slam e_v\|^2\big) |f(v)|^2, \;\;
f \in \dz{\slam}.
   \end{multline}
If $\tcal$ is rootless, then similar reasoning leads to the
inequality \eqref{in1}. Since, by Proposition
\ref{desc}\,(iv), $f \cdot \chi_{\dzi u} \in \dz {\slam}$
for all $f \in \dz{\slam}$ and $u \in V$, we deduce from
\eqref{in1} and Proposition \ref{46} that
   \begin{align} \label{in2}
\Big| \sum_{v \in \dzi u} \overline {\lambda_v}
f(v)\Big|^2 \Le c \sum_{v \in \dzi u} \big(1+\|\slam
e_v\|^2\big) |f(v)|^2, \quad f \in \dz{\slam}, u \in
V.
   \end{align}
Fix $u \in V^\prime$ and take a finite nonempty
subset $W$ of $\dzi u$. It follows from
Proposition \ref{desc}\,(v) that $\ell^2(W)
\subseteq \dz \slam$. This allows us to
substitute the function
    \begin{align*}
f(v) =
    \begin{cases}
\frac {\lambda_v}{1+\|\slam e_v\|^2} & \text{if } v
\in W,
    \\
0 & \text{otherwise,}
    \end{cases}
    \end{align*}
into \eqref{in2}. What we obtain is
    \begin{align*}
\Big(\sum_{v \in W} \frac {|\lambda_v|^2}{1+\|\slam e_v\|^2}\Big)^2
\Le c \sum_{v \in W} \frac {|\lambda_v|^2}{1+\|\slam e_v\|^2}.
    \end{align*}
This implies that
    \begin{align*}
\sum_{v \in W} \frac {|\lambda_v|^2}{1+\|\slam e_v\|^2}
\Le c.
    \end{align*}
Passing with $W$ to ``infinity'' if necessary, we get (ii).

   (ii)$\Rightarrow$(i) A careful inspection of
the proof of (i)$\Rightarrow$(ii), supported by
the Cauchy-Schwarz inequality, shows that the
reverse implication (ii)$\Rightarrow$(i) is true
as well.
   \end{proof}
   \subsection{\label{s4}When $\text{$\dz{\slam^*}\subseteq
\dz{\slam}$}$?}
   The circumstances under which the inclusion
$\dz{\slam^*} \subseteq \dz{\slam}$ holds are more
elaborate and require much more effort to be
accomplished. For this reason, we attach to a densely
defined weighted shift $\slam$ on a directed tree
$\tcal$ the diagonal operators $M_u$ in $\ell^2(\dzi
u)$, $u \in V^\prime$, given by \idxx{$M_u$}{65}
   \begin{align}     \label{tuformm}
   \begin{aligned}
\dz{M_u} & = \{g \in \ell^2(\dzi u) \colon \sum_{v \in
\dzi u} \|\slam e_v\|^2 |g(v)|^2 < \infty\},
   \\
   (M_ug)(v) & = \|\slam e_v\| g(v), \quad v \in \dzi
u, \, g\in \dz{M_u}.
   \end{aligned}
   \end{align}
If $u \in V^\prime$ is such that the function
$\lambdab^u\colon \dzi u \ni v \to \lambda_v \in \cbb$
belongs to $\dz{M_u}$, then we define the operator
$T_u$ in $\ell^2(\dzi u)$ by \idxx{$T_u$}{66}
   \begin{align}    \label{tuform}
T_u &=M_u^2 - \frac 1 {1+ \|\slam e_u\|^2} \, M_u
(\lambdab^u) \otimes M_u (\lambdab^u), \quad u \in
V^\prime.
   \end{align}
For simplicity, we suppress the explicit dependence of
$M_u$ and $T_u$ on $\lambdab$ in the notation. We
gather below indispensable properties of the operators
$M_u$ and $T_u$.
   \begin{pro}\label{ogrtu}
Let $\slam$ be a densely defined weighted shift
on a directed tree $\tcal$ with weights $\lambdab
= \{\lambda_v\}_{v \in V^\circ}$. If $u \in
V^\prime$ and $\lambdab^u \in \dz{M_u}$, then
   \begin{enumerate}
   \item[(i)]  $\{e_v \colon v \in \dzi u\} \subseteq
\dz{T_u} = \dz{M_u^2} $,
   \item[(ii)] $M_u$ and $T_u$ are positive
selfadjoint operators in $\ell^2(\dzi u)$,
   \item[(iii)] $M_u$ is bounded if and only if $T_u$ is bounded,
   \item[(iv)] $\|\slam e_v\| \Le \sqrt{2\|T_u\|}$ for all but
at most one vertex $v \in \dzi u$ whenever the operator
$T_u$ is bounded.
   \end{enumerate}
   \end{pro}
   \begin{proof}
   (i) Apply parts (iii) and (v) of Proposition \ref{desc}.

   (ii) It is easily seen that $M_u$ and $T_u$ are
selfadjoint operators, and $M_u$ is positive (see
e.g., \cite[Theorem 6.20]{weid}). Moreover, we have
   \begin{multline*}
\big\langle \big(M_u (\lambdab^u) \otimes M_u
(\lambdab^u)\big) g, g \big\rangle = |\is g{M_u
(\lambdab^u)}|^2 = |\is {M_u (g)}{\lambdab^u}|^2
   \\
\Le \|\lambdab^u\|^2 \|M_u(g)\|^2 \overset{\eqref{eu}}\Le
(1+\|\slam e_u\|^2) \is{M_u^2(g)} g, \quad g \in \dz {T_u},
   \end{multline*}
which implies that $T_u$ is positive.

   (iii) Evident.

   (iv) Take vertexes $v_1,v_2 \in \dzi u$ such that
$v_1 \neq v_2$. Without loss of generality, we may
assume that $\|\slam e_{v_1}\| \Le \|\slam e_{v_2}\|$.
It follows that
   \begin{align}    \label{dzsum}
\frac {\|\slam e_v\|^2\big(1 + \sum_{w \in \dzi u \setminus
\{v\}} |\lambda_w|^2\big)}{1+\|\slam e_u\|^2}
\overset{\eqref{tuform}}= \is{T_u e_v}{e_v} \Le \|T_u\|,
\quad v \in \dzi u.
   \end{align}
Hence, we have
   \allowdisplaybreaks
   \begin{align*}
\|\slam e_{v_1}\|^2 & \Le \frac {\|\slam
e_{v_1}\|^2\big(2 + |\lambda_{v_1}|^2 +
|\lambda_{v_2}|^2 + 2 \sum_{w \in \dzi u \setminus
\{v_1,v_2\}} |\lambda_w|^2\big)}{1+\|\slam e_u\|^2}
   \\
& \Le \frac {\|\slam e_{v_1}\|^2\big(1 + \sum_{w \in
\dzi u \setminus \{v_1\}}
|\lambda_w|^2\big)}{1+\|\slam e_u\|^2}
   \\
&\phantom{xxxxx} + \frac{\|\slam e_{v_2}\|^2 \big(1 +
\sum_{w \in \dzi u \setminus \{v_2\}}
|\lambda_w|^2\big)}{1+\|\slam e_u\|^2}
   \\
&\hspace{-1.45ex}\overset{\eqref{dzsum}} \Le 2
\|T_u\|.
   \end{align*}
This enables us to deduce that for every two distinct
vertexes $v_1,v_2 \in \dzi u$ either $\|\slam
e_{v_1}\| \Le \sqrt{2\|T_u\|}$ or $\|\slam e_{v_2}\|
\Le \sqrt{2\|T_u\|}$. As is easily seen, this implies
(iv).
   \end{proof}
Now we characterize all weighted shifts $\slam$ on directed trees
which have the property that $\dz{\slam^*}\subseteq \dz{\slam}$. We
do this with the help of the operators $T_u$ defined in
\eqref{tuform}. The proof of this characterization relies heavily on
the fact that for every $u \in V^\prime$, the operator $T_u$
factorizes as $T_u=R_uR_u^*$ whenever $T_u$ is bounded (see
\eqref{factort}).
   \begin{thm} \label{zen}
If $\slam$ is a densely defined weighted shift on a directed tree
$\tcal$ with weights $\lambdab = \{\lambda_v\}_{v \in V^\circ}$,
then the following two conditions are equivalent{\em :}
   \begin{enumerate}
   \item[(i)] $\dz{\slam^*}\subseteq \dz{\slam}$,
   \item[(ii)]
$T_u \in \ogr{\ell^2(\dzi u)}$ for all $u \in
V^\prime$, and
   \begin{align}  \label{supnor}
\sup_{u\in V^\prime} \|T_u\|<\infty.
   \end{align}
   \end{enumerate}
   \end{thm}
   \begin{proof}
   (i)$\Rightarrow$(ii)
Suppose that $\dz{\slam^*}\subseteq \dz{\slam}$. Since
both operators are closed, we can argue as in the
proof of Theorem \ref{slssl*}. Thus, there exists a
constant $c > 0$ such that
   \begin{align*}
\|\slam f\|^2 \Le c \, (\|f\|^2 + \|\slam^*f\|^2), \quad f
\in \dz{\slam^*}.
   \end{align*}
Similarly to \eqref{malysz}, we see that
   \begin{align}    \label{both}
\sum_{u \in V} \|\slam e_u\|^2 |f(u)|^2 \Le c \sum_{u
\in V} |f(u)|^2 + c \sum_{u \in V} \Big| \sum_{v \in
\dzi u} \overline {\lambda_v} f(v)\Big|^2
\hspace{-1ex},\quad f \in \dz{\slam^*}.
   \end{align}
   If the tree $\tcal$ has a root, then applying \eqref{sumchi} to
the left-hand side of \eqref{both} and to the first term of the
right-hand side of \eqref{both}, we obtain
   \begin{multline*}
\sum_{u \in V} \sum_{v \in\dzi u} \|\slam e_v\|^2
|f(v)|^2
   \\
\Le c |f(\koo)|^2 + c \sum_{u \in V} \Big( \sum_{v \in
\dzi u} |f(v)|^2 + \Big| \sum_{v \in \dzi u} \overline
{\lambda_v} f(v)\Big|^2 \Big), \quad f \in
\dz{\slam^*}.
   \end{multline*}
Since, by Proposition \ref{sprz}\,(iv), there is no
restriction on the value of $f$ at $\koo$, we see that
the above inequality is equivalent to the following
one
   \begin{multline}    \label{malyszi}
\sum_{u \in V} \sum_{v \in\dzi u} \|\slam e_v\|^2
|f(v)|^2
   \\
\Le c \sum_{u \in V} \Big( \sum_{v \in \dzi u}
|f(v)|^2 + \Big| \sum_{v \in \dzi u} \overline
{\lambda_v} f(v)\Big|^2 \Big), \quad f \in
\dz{\slam^*}.
   \end{multline}
Clearly, inequalities \eqref{both} and
\eqref{malyszi} coincide if $\tcal$ is rootless.

Fix $u \in V^\prime$. Recall that by Proposition
\ref{desc}, $\lambdab^u \in \ell^2(\dzi u)$. In
view of Proposition \ref{46} and Proposition
\ref{sprz}\,(vi), the inequality \eqref{malyszi}
yields
   \begin{align}   \label{bin}
\sum_{v \in\dzi u} \|\slam e_v\|^2 |f(v)|^2 \Le c (
\|f\|^2 + |\is{f}{\lambdab^u}|^2), \quad f \in
\ell^2(\dzi u).
   \end{align}
If $\lambdab^u = 0$, then by \eqref{bin}, $M_u \in \ogr
{\ell^2(\dzi u)}$ and $\|M_u\| \le \sqrt c$, and
consequently by \eqref{tuform}, $\|T_u\| \le c$. Assume now
that $\lambdab^u \neq 0$. Define the new inner product
$\is{\cdot}{\textrm{-}}_u$ on $\ell^2(\dzi u)$ by
   \begin{align} \label{nsp}
\is{f}{g}_u = \is fg + \is f {\lambdab^u} \cdot \is
{\lambdab^u} g, \quad f,g \in \ell^2(\dzi u),
   \end{align}
and denote by $\kk_u$ the Hilbert space $(\ell^2(\dzi u),
\is{\cdot}{\textrm{-}}_u)$. It follows from \eqref{bin} that the
operator $R_u \colon \kk_u \to \ell^2(\dzi u)$ defined by
   \begin{align}      \label{gejsza}
(R_u f)(v) = \|\slam e_v\| f(v), \quad v \in \dzi u,
\, f \in \kk_u,
   \end{align}
is bounded and
   \begin{align}      \label{gejsza2}
\|R_u\| \Le \sqrt c.
   \end{align}
Using the Cauchy-Schwarz inequality, we deduce from
\eqref{nsp}, \eqref{gejsza} and \eqref{gejsza2} that
$M_u \in \ogr {\ell^2(\dzi u)}$ and $\|M_u\|^2 \Le c
(1+ \|\lambdab^u\|^2)$. Hence, $T_u \in \ogr
{\ell^2(\dzi u)}$.

   We now compute the adjoint $R_u^* \colon \ell^2(\dzi
u) \to \kk_u$ of $R_u$. Take $g \in \ell^2(\dzi u)$.
Then, we have
   \allowdisplaybreaks
   \begin{align}   \label{szte}
\sum_{v \in \dzi u} \|&\slam e_v\| g(v) \overline {h(v)} =
\is{g}{R_u h} = \is{R_u^*g}{h}_u
   \\
& = \is{R_u^*g}{h} + \is {R_u^*g}{\lambdab^u} \cdot \is
{\lambdab^u} h \notag
   \\
& = \sum_{v \in \dzi u} (R_u^* g)(v) \overline {h(v)} + \is
{R_u^*g}{\lambdab^u} \cdot \is {\lambdab^u} h, \quad h \in
\ell^2(\dzi u). \notag
   \end{align}
Set $\zeta_{u,g}(v) = \|\slam e_v\| g(v) - (R_u^* g)(v)$
for $v \in \dzi u$. By \eqref{bin}, $\zeta_{u,g} \in
\ell^2(\dzi u)$. It follows from \eqref{szte} that
   \begin{align}   \label{zetah}
\is {\zeta_{u,g}} h = \sum_{v \in \dzi u} \zeta_{u,g}(v)
\overline {h(v)} = \is {R_u^*g}{\lambdab^u} \cdot \is
{\lambdab^u}h, \quad h \in \ell^2(\dzi u).
   \end{align}
This implies that $\{\lambdab^u\}^\perp \subseteq
\{\zeta_{u,g}\}^\perp$, where the orthogonality refers
to the original inner product of $\ell^2(\dzi u)$. As
a consequence, $\{\zeta_{u,g}\}^{\perp\perp} \subseteq
\{\lambdab^u\}^{\perp\perp}$, and so there exists
$\alpha_{u,g}\in \cbb$ such that $\zeta_{u,g} =
-\alpha_{u,g} \lambdab^u$. Hence, by the definition of
$\zeta_{u,g}$, we have
   \begin{align}    \label{zetah2}
(R_u^* g)(v) = \|\slam e_v\| g(v) + \alpha_{u,g}
\lambda_v, \quad v \in \dzi u.
   \end{align}
Since $M_u \in \ogr {\ell^2(\dzi u)}$, we see that
   \allowdisplaybreaks
   \begin{align*}
-\alpha_{u,g} \is {\lambdab^u} h & \overset{\eqref{zetah}}
= \is{R_u^* g} {\lambdab^u} \is {\lambdab^u}h
   \\
& \overset{\eqref{zetah2}} = \is{M_u g}{\lambdab^u}
\is {\lambdab^u}h + \alpha_{u,g} \|\lambdab^u\|^2 \,
\is {\lambdab^u}h
   \\
&\hspace{2.1ex} = \is{g}{M_u (\lambdab^u)} \is
{\lambdab^u}h + \alpha_{u,g} \|\lambdab^u\|^2 \, \is
{\lambdab^u}h, \quad h \in \ell^2(\dzi u).
   \end{align*}
This and \eqref{eu} imply that
   \begin{align}    \label{dalf}
-\alpha_{u,g} (1 + \|\slam e_u\|^2) \is {\lambdab^u}h
= \is{g}{M_u (\lambdab^u)} \is {\lambdab^u}h, \quad h
\in \ell^2(\dzi u).
   \end{align}
Substituting $h = \lambdab^u$ into \eqref{dalf} and
dividing both sides of \eqref{dalf} by
$\|\lambdab^u\|^2$ (which, according to our
assumption, is nonzero), we get
   \begin{align*}
\alpha_{u,g} = -\frac{\is{g}{M_u (\lambdab^u)}}{1 + \|\slam
e_u\|^2}, \quad h \in \ell^2(\dzi u).
   \end{align*}
Hence
   \begin{align} \label{rus}
R_u^* g \overset{\eqref{zetah2}}= M_u g - \frac{\is{g}{M_u
(\lambdab^u)} }{1 + \|\slam e_u\|^2} \, \lambdab^u, \quad g
\in \ell^2(\dzi u).
   \end{align}
As a consequence, we have
   \begin{align}   \label{factort}
R_uR_u^* g \overset{\eqref{gejsza}}= M_u^2 g -
\frac{\is{g}{M_u (\lambdab^u)}}{1 + \|\slam e_u\|^2} \, M_u
(\lambdab^u) \overset{\eqref{tuform}}= T_u g, \quad g \in
\ell^2(\dzi u).
   \end{align}
This, together with \eqref{gejsza2}, implies that
$\|T_u\| = \|R_u\|^2 \Le c$ for all $u \in V^\prime$.

   (ii)$\Rightarrow$(i) Since the operator $M_u$ is
bounded, we easily check that the operator $R_u \colon
\kk_u \to \ell^2(\dzi u)$ defined by equalities
\eqref{nsp} and \eqref{gejsza} is bounded as well. As
in the proof of implication (i)$\Rightarrow$(ii), we
verify that the formulas \eqref{rus} and
\eqref{factort} are valid. Hence, by \eqref{factort},
$\|R_u\| \Le \sqrt c$ for all $u \in V^\prime$, where
$c := \sup_{v\in V^\prime} \|T_v\|$. Now, by reviewing
in reverse order this part of the proof of
(i)$\Rightarrow$(ii) which begins with \eqref{gejsza2}
and ends with \eqref{malyszi}, we see that (i) holds.
   \end{proof}
   \begin{rems}
1) It follows from Proposition \ref{ogrtu}\,(iv) and
Theorem \ref{zen} that if $\dz{\slam^*} \subseteq
\dz{\slam}$, then for every $u \in V^\prime$ and for all
but at most one vertex $v \in \dzi u$, $\|\slam e_v\|^2 \Le
C$ with $C = 2 \sup_{v\in V^\prime} \|T_v\| < \infty$. This
fact, when combined with Proposition \ref{ogrs}, may
suggest that in the case of directed binary (or more
complicated) trees the inclusion $\dz{\slam^*}\subseteq
\dz{\slam}$ forces the operator $\slam$ to be bounded.
However, this is not the case, which is illustrated in
Example \ref{17rem} below. It is well known that in the
case of classical weighted shifts none of the inclusions
$\dz{\slam}\subseteq \dz{\slam^*}$ and
$\dz{\slam^*}\subseteq \dz{\slam}$ implies the boundedness
of the operator $\slam$.

2) Looking at the formula \eqref{tuform} and
Theorem \ref{zen}, it is tempting to expect that
the uniform boundedness of the operators
$\{T_u\}_{u \in V^\prime}$, which completely
characterizes the inclusion
$\dz{\slam^*}\subseteq \dz{\slam}$, is equivalent
to the uniform boundedness of the operators
$\{M_u\}_{u \in V^\prime}$. However, this is not
the case. Indeed, assuming that the operator
$\slam$ is densely defined and all the operators
$\{M_u\}_{u \in V^\prime}$ are bounded, one can
infer from \eqref{tuform} and Proposition
\ref{ogrtu}\,(ii) that $0 \Le T_u \Le M_u^2$ and
consequently $\|T_u\| \Le \|M_u\|^2$. This means
that $\sup_{u\in V^\prime} \|M_u\|<\infty$
implies $\sup_{u\in V^\prime} \|T_u\|<\infty$. In
view of \eqref{tuformm}, the inequality
$\sup_{u\in V^\prime} \|M_u\|<\infty$ is
equivalent to $\sup_{u \in V^\circ} \|\slam e_u\|
< \infty$. Since $e_{\koo} \in \dz \slam$
whenever $\tcal$ has a root, the last inequality
is equivalent to $\sup_{u \in V} \|\slam e_u\| <
\infty$, which, by Proposition \ref{ogrs}, is
equivalent to $\slam \in \ogr {\ell^2(V)}$.
Summarizing, we have shown that if the operator
$\slam$ is densely defined, then the inequality
$\sup_{u\in V^\prime} \|M_u\|<\infty$ is
equivalent to the boundedness of $\slam$.

    3) According to \eqref{tuform} and Theorem \ref{zen},
the question of when $\dz{\slam^*}\subseteq \dz{\slam}$
reduces to estimating the norms of rank one perturbations
of positive diagonal operators. The class of such operators
is still not very well understood, despite the structural
simplicity of diagonal operators (cf.\
\cite{Io,f-j-k-p,f-j-k-p2} and references therein; see also
\cite{si} for a more general situation\footnote{\;In
\cite{si} Barry Simon wrote: {\em Finally to rank one
perturbations---maybe something so easy that I can say
something useful! Alas, we'll see even this is hard and
exceedingly rich}.}). As will be seen in Example
\ref{17rem}, the norms of unperturbed operators $M_u^2$ may
tend to $\infty$, while the norms of perturbed operators
$T_u$ may be uniformly bounded.
   \end{rems}
To make the condition \eqref{supnor} of Theorem
\ref{zen} more explicit and calculable, we
consider, instead of the usual operator norm, the
Hilbert-Schmidt norm and the trace norm,
respectively.
    \begin{pro}   \label{ilB}
Let $\slam$ be a densely defined weighted shift on a directed tree
$\tcal$ with weights $\lambdab = \{\lambda_v\}_{v \in V^\circ}$. If
there exists a constant $C>0$ such that
   \begin{align}  \label{wkwdom-}
   \begin{aligned}
\sum_{v \in \dzi u} \left[\|\slam e_v\|^2 \Big(1 -
\frac {|\lambda_v|^2}{1+\| \slam
e_u\|^2}\Big)\right]^2 &\Le C, \quad u\in V^\prime,
    \\
\sum_{\substack{v,w \in \dzi u\\v \neq w}}
\left[\frac{\|\slam e_v\| |\lambda_v| \|\slam e_w\|
|\lambda_w|}{1+ \|\slam e_u\|^2}\right]^2 & \Le C,
\quad u\in V^\prime,
    \end{aligned}
   \end{align}
then $\dz{\slam^*}\subseteq \dz{\slam}$. In
particular, this is the case if for some constant
$C>0$,
    \begin{align}    \label{wkwdom-+}
\sum_{v \in \dzi u} \|\slam e_v\|^2 \Big(1 - \frac
{|\lambda_v|^2}{1+\| \slam e_u\|^2}\Big) &\Le C, \quad
u\in V^\prime.
    \end{align}
    \end{pro}
    \begin{proof}
First we consider the case when \eqref{wkwdom-}
holds. It follows from Proposition \ref{desc}
that $\{e_u\colon u \in V\} \subseteq
\dz{\slam}$. We show that $\lambdab^u \in
\dz{M_u}$ for every $u \in V^\prime$. Indeed,
there are two possibilities, either $\|\slam
e_v\| |\lambda_v|=0$ for all $v \in \dzi u$ and
hence $\lambdab^u \in \dz{M_u}$, or $\|\slam
e_w\| |\lambda_w|>0$ for some $w \in \dzi u$
which together with the second inequality in
\eqref{wkwdom-} implies that $\lambdab^u \in
\dz{M_u}$. This means that we can consider the
operators $T_u$, $u \in V^\prime$. Owing to
Proposition \ref{ogrtu}, for every $u \in
V^\prime$, the operator $T_u$ is closed and its
domain contains the basis vectors $e_v$, $v \in
\dzi u$. It follows from \eqref{tuform} that
    \begin{align} \label{7.2}
\is{T_u e_w} {e_v} = \|\slam e_v\|^2 \delta_{v,w} - \frac{\|\slam
e_v\| \lambda_v \|\slam e_w\| \overline{\lambda_w}}{1+ \|\slam
e_u\|^2}, \quad v,w \in \dzi u, \, u \in V^\prime,
    \end{align}
where $\delta_{v,w}$ is the usual Kronecker delta. Thus, by
\eqref{7.2} and \eqref{wkwdom-}, we have
    \begin{multline} \label{sum}
\sum_{v,w \in \dzi u} |\is{T_u e_v} {e_w}|^2 = \sum_{v \in \dzi u}
\left[\|\slam e_v\|^2 \Big(1 - \frac {|\lambda_v|^2}{1+\| \slam
e_u\|^2}\Big)\right]^2
    \\
+ \sum_{\substack{v,w \in \dzi u\\v \neq w}} \left[\frac{\|\slam
e_v\| |\lambda_v| \|\slam e_w\| |\lambda_w|}{1+ \|\slam
e_u\|^2}\right]^2 \Le 2C, \quad u \in V^\prime.
    \end{multline}
We deduce from \eqref{sum} that the operator $T_u|_\escr$ is bounded
and $\|T_u|_\escr\| \Le \sqrt{2C}$. Since $T_u$ is closed, we see
that $T_u \in \ogr {\ell^2(\dzi u)}$ and $\|T_u\| \Le \sqrt{2C}$ for
all $u \in V^\prime$. By Theorem \ref{zen}, we get $\dz{\slam^*}
\subseteq \dz{\slam}$. Note that $T_u$ is a Hilbert-Schmidt operator
with Hilbert-Schmidt norm $\|T_u\|_2$ (cf.\ \cite[page 66]{stone})
not exceeding $\sqrt{2C}$.

    We now claim that \eqref{wkwdom-+} implies
\eqref{wkwdom-} with the new constant $C^2$. Fix $u
\in V^\prime$. For $v,w \in \dzi u$, we denote by
$\varDelta_{v,w}$ the right-hand side of the equality
\eqref{7.2}. First, we show that
   \begin{align} \label{deltawv}
|\varDelta_{v,w}|^2 \Le \varDelta_{v,v} \varDelta_{w,w},
\quad v, w \in \dzi u.
   \end{align}
Since, $\varDelta_{v,v} \Ge 0$ for all $v\in \dzi u$, it is
enough to consider the case when $v \neq w$. Under this
assumption, we have
   \begin{align} \label{tedwaimpl1}
|\lambda_v|^2 \overset{\eqref{eu}} \Le \|\slam e_u\|^2 -
|\lambda_w|^2 < 1 + \|\slam e_u\|^2 - |\lambda_w|^2,
   \end{align}
and, by symmetry,
   \begin{align} \label{tedwaimpl2}
|\lambda_w|^2 < 1 + \|\slam e_u\|^2 - |\lambda_v|^2.
   \end{align}
It is now easily seen that \eqref{tedwaimpl1} and
\eqref{tedwaimpl2} imply \eqref{deltawv}. Combining
\eqref{deltawv} with \eqref{wkwdom-+}, we get
   \begin{align}     \label{delta2le}
\sum_{v,w \in \dzi u} |\varDelta_{v,w}|^2 \Le \Big(\sum_{v
\in \dzi u} \varDelta_{v,v}\Big)^2 \Le C^2.
   \end{align}
Since the left-hand side of \eqref{delta2le} is equal
to the right-hand side of the equality in \eqref{sum},
our claim is established. In view of \eqref{delta2le}
and the discussion in the previous paragraph, we see
that $T_u \in \ogr {\ell^2(\dzi u)}$ and $\|T_u\| \Le
C$ for all $u \in V^\prime$, which completes the
proof. Looking at \eqref{wkwdom-+} and \eqref{7.2}, we
deduce that $\sum_{v\in \dzi u} \is{T_u e_v} {e_v} \Le
C$, which means that the operator $T_u$ is of trace
class and its trace norm $\|T_u\|_1$ is less than or
equal to $C$.
    \end{proof}
     An inspection of the proof of Proposition \ref{ilB}
shows that \eqref{wkwdom-} is equivalent to $\sup_{u \in
V^\prime}\|T_u\|_2 < \infty$, while \eqref{wkwdom-+} is
equivalent to $\sup_{u \in V^\prime}\|T_u\|_1 < \infty$.
Since $\|T\| \Le \|T\|_2$ and $\|T\|_2 \Le \|T\|_1$ for
Hilbert-Schmidt and trace class operators, respectively, we
see why implications
\eqref{wkwdom-+}$\Rightarrow$\eqref{wkwdom-} (possibly with
different constants $C$) and
\eqref{wkwdom-}$\Rightarrow$\eqref{supnor} are true.
   \begin{cor} \label{zenw}
If $\slam$ is a weighted shift on a directed tree
$\tcal$ with weights $\lambdab = \{\lambda_v\}_{v \in
V^\circ}$, and $\sup_{u \in V} \card {\dzi u} <
\infty$, then the following two assertions hold\/{\em
:}
   \begin{enumerate}
   \item[(i)] $\slam$ is densely defined.
   \item[(ii)] $\dz{\slam^*}\subseteq \dz{\slam}$ if and only if
there exists a positive constant $C$ such that
   \begin{align}    \label{wkwdom}
\|\slam e_v\|^2 \Big(1 - \frac {|\lambda_v|^2}{1+\|
\slam e_u\|^2}\Big) & \Le C, \quad v \in \dzi u, \,
u\in V^\prime.
   \end{align}
   \end{enumerate}
   \end{cor}
   \begin{proof}
That $\slam$ is densely defined follows immediately from
Proposition \ref{desc}.

The ``if'' part of the assertion (ii) is a direct
consequence of Proposition \ref{ilB}. In turn, the ``only
if'' part follows from Theorem \ref{zen} because, by
\eqref{7.2}, the left-hand side of inequality
\eqref{wkwdom} is equal to $\is{T_u e_v}{e_v}$.
   \end{proof}
   The classical weighted shifts fall within the
scope of Theorem \ref{slssl*} and Corollary
\ref{zenw}. We leave it to the reader to write
the explicit inequalities characterizing the
inclusions $\dz\slam \subseteq \dz{\slam^*}$ and
$\dz{\slam^*} \subseteq \dz\slam $ for classical
weighted shifts $\slam$.
   \subsection{An example}
In this section we give an example highlighting the
possible relationships between domains of unbounded
weighted shifts on directed trees and their adjoints.
Since classical weighted shifts are well developed, we
concentrate on weighted shifts on directed binary
trees, the graphs which are essentially more
complicated than ``line trees'' involved in the
definition of classical weighted shifts. The other
reason for this is that Propositions \ref{ogrs} and
\ref{ogrtu}, and Theorem \ref{zen} may suggest at
first glance that weighted shifts $\slam$ on directed
binary trees (or more complicated directed trees)
satisfying the inclusion $\dz{\slam^*}\subseteq
\dz{\slam}$ are almost always bounded. Fortunately, as
shown below, this is not the case.
   \begin{exa} \label{17rem}
Let $\tcal$ be the directed binary tree as in Figure 1 with
$V^\circ$ given by
   \begin{align*}
   V^\circ = \{(i,j) \colon i = 1,2, \ldots, \, j=1, \ldots,2^i\}.
   \end{align*}
Define the sets $V^\circ_{\mathrm g} = \{(i,j) \in
V^\circ \colon j=1\}$ (gray filled ellipses in Figure
1 without root) and $V^\circ_{\mathrm w} = V^\circ
\setminus V^\circ_{\mathrm g}$ (white filled ellipses
in Figure 1).
   \begin{center}
   \includegraphics[width=8cm]
   %{Graf1.png}
   {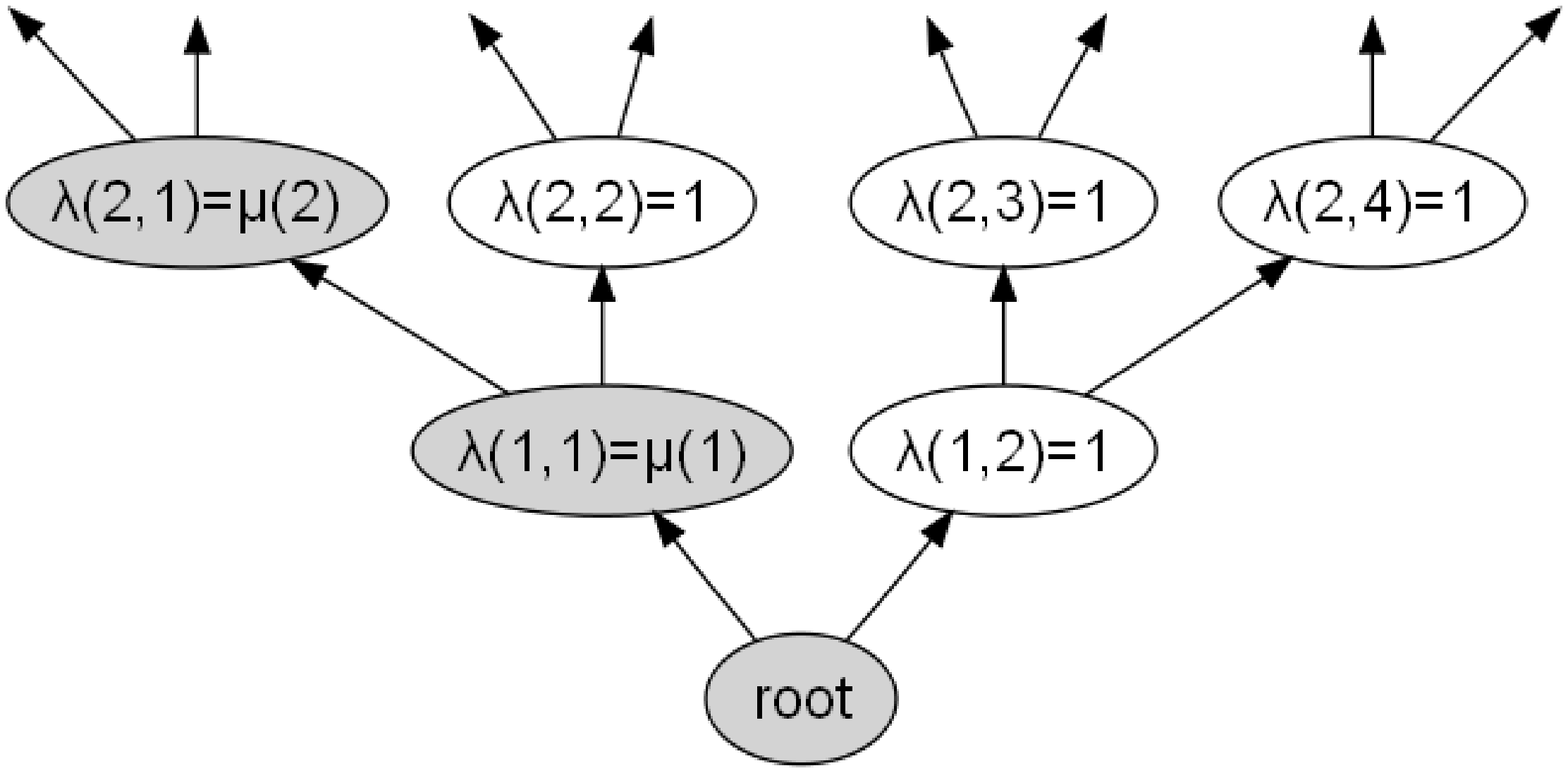}
   \\[1.5ex]
   {\small {\sf Figure 1}}
   \end{center}
   \vspace{1ex}
   Let $\{\mu(i)\}_{i=1}^\infty$ be a sequence of
complex numbers and let $\slam$ be the weighted shift
on $\tcal$ with weights $\boldsymbol
\lambda=\{\lambda(i,j)\}_{(i,j) \in V^\circ}$ given by
$\lambda(i,j)=1$ for $(i,j) \in V^\circ_{\mathrm w}$,
and $\lambda(i,j)=\mu(i)$ for $(i,j) \in
V^\circ_{\mathrm g}$. In view of Corollary \ref{zenw},
$\slam$ is densely defined, and $\dz{\slam^*}\subseteq
\dz{\slam}$ if and only if there exists a positive
constant $C$ such that for all $u \in V = V^\prime$,
   \begin{align} \label{bint}
\frac {\|\slam
e_{v_k}\|^2(1+|\lambda_{v_l}|^2)}{1+|\lambda_{v_1}|^2
+ |\lambda_{v_2}|^2} \Le C, \quad k,l=1,2, \, k \neq
l, \dzi u = \{v_1,v_2\}.
   \end{align}
If $u \in V^\circ_{\mathrm w}$, then \eqref{bint} holds with
$C=4/3$. It is also easily seen that \eqref{bint} is valid for all
$u \in V^\circ_{\mathrm g}$ with some positive constant $C$ if and
only if
   \begin{align}\label{wkwg}
\sup_{i \Ge 1} \frac{1+|\mu(i+1)|^2} {1 + |\mu(i)|^2}
< \infty \quad \Big(\hspace{-1ex}\iff \sup_{i \Ge 1}
\frac{|\mu(i+1)|^2} {1 + |\mu(i)|^2} < \infty \Big).
   \end{align}
Hence, $\dz{\slam^*}\subseteq \dz{\slam}$ if and only if
\eqref{wkwg} holds. Similar analysis based upon Theorem
\ref{slssl*} shows that $\dz{\slam}\subseteq \dz{\slam^*}$
if and only if
   \begin{align}  \label{wkwg+}
\sup_{i \Ge 1} \frac{1 + |\mu(i)|^2} {1 +
|\mu(i+1)|^2} < \infty \quad \Big(\hspace{-1ex}\iff
\sup_{i \Ge 1} \frac{|\mu(i)|^2} {1 + |\mu(i+1)|^2} <
\infty \Big).
   \end{align}
According to Corollary \ref{parc}, $\slam$ is
bounded if and only if $\sup_{i\Ge 1} |\mu(i)| <
\infty$. Clearly, if $\inf_{i\Ge 1} |\mu(i)| > 0$,
then \eqref{wkwg} is equivalent to $\sup_{i\Ge 1}
|\mu(i+1)/ \mu(i)| < \infty$, while \eqref{wkwg+}
is equivalent to $\sup_{i\Ge 1} |\mu(i)/ \mu(i+1)|
< \infty$. This simple observation enables us to
find unbounded sequences $\{\mu(i)\}_{i=1}^\infty
\subseteq \cbb$ (read:\ unbounded $\slam$'s) for
which any of the following four mutually exclusive
conditions may hold
   \begin{enumerate}
   \item[$1^\circ$] $\dz{\slam} = \dz{\slam^*}$
(e.g., $\mu(i)=i^s$ for some $s \in (0,\infty)$, or
$\mu(i)=q^i$ for some $q \in (1,\infty)$),
   \item[$2^\circ$] $\dz{\slam} \varsubsetneq \dz{\slam^*}$
(e.g., $\mu(i)=i!$),
   \item[$3^\circ$] $\dz{\slam^*} \varsubsetneq \dz{\slam}$
(e.g., $\mu(i)=i+1 - k_n$ if $k_n \Le i < k_{n+1}$,
where $\{k_n\}_{n=1}^\infty$ is a strictly increasing
sequence of integers such that $k_1 = 1$ and
$\lim_{n\to \infty} (k_{n+1} - k_n) = \infty$),
   \item[$4^\circ$] $\dz{\slam} \nsubseteq \dz{\slam^*}$ and
$\dz{\slam^*}\nsubseteq \dz{\slam}$ (e.g.,
$\mu(i)=(i+1 - k_n)!$ if $k_n \Le i < k_{n+1}$,
where $\{k_n\}_{n=1}^\infty$ is as in $3^\circ$).
   \end{enumerate}
The examples mentioned in $3^\circ$ and $4^\circ$
fit into a more general scheme. Given a strictly
increasing function $\phi\colon \{1,2, \ldots\}
\to (0,\infty)$ such that $\lim_{i\to \infty}
\phi(i) = \infty$, we set $\mu(i) = \phi(i+1-k_n)
+ l_n$ if $k_n \Le i < k_{n+1}$, where
$\{k_n\}_{n=1}^\infty$ is as in $3^\circ$ and
$\{l_n\}_{n=1}^\infty \subseteq [0,\infty)$.
Choosing appropriate $\phi$,
$\{k_n\}_{n=1}^\infty$ and
$\{l_n\}_{n=1}^\infty$, we can find examples of
unbounded $\slam$'s satisfying $3^\circ$ (e.g.,
$\phi(i)=i$, $k_n = n^3$ and $l_n=n$) and
$4^\circ$ ($\phi(i)=i!$, $k_n = n^2$ and $l_n=n$)
with the additional property that $\lim_{i\to
\infty} |\mu(i)| = \infty$.

One can construct unbounded weighted shifts on
directed binary trees with the required properties
mentioned in conditions $1^\circ$-$4^\circ$ whose
weights are more complicated than those explicated in
Figure 1. The simplest way of doing this is to draw a
directed binary tree with gray and white vertexes
(read:\ filled ellipses) following the rule that only
one child of the gray vertex is gray, the other being
white, and that both children of the white vertex are
white.
   \end{exa}
   \newpage
   \section{\label{ch5}Hyponormality and cohyponormality}
   \subsection{\label{hypcohyp}Hyponormality}
Starting from this section, we shall concentrate
mostly on investigating bounded weight\-ed shifts on
directed trees. We begin with recalling definitions of
some important classes of operators (see also
Corollary \ref{chariso}, Proposition \ref{adcoiso},
Lemma \ref{normal} and Proposition \ref{qn-ch} for
characterizations of isometric, coisometric, normal
and quasinormal weighted shifts, respectively). Let
$\hh$ be a complex Hilbert space. An operator $A \in
\ogr \hh$ is said to be {\em subnormal} if there
exists a complex Hilbert space $\mathcal K$ and a
normal operator $N\in \ogr{\kk}$ such that $\hh
\subseteq \kk$ (isometric embedding) and $Ah=Nh$ for
all $h \in \hh$ (cf.\ \cite{hal1,con2}). An operator
$A \in \ogr \hh$ is called {\em hyponormal} if
$\|A^*f\| \Le \|Af\|$ for all $f \in \hh$. We say that
$A\in \ogr \hh$ is {\em paranormal} if $\|Af\|^2 \Le
\|A^2 f\| \|f\|$ for all $f \in \hh$. It is a well
known fact that subnormal operators are always
hyponormal, but not conversely. The latter can be
easily seen by considering classical weighted shifts.
Also, it is well known that hyponormal operators are
paranormal (cf.\ \cite{Istr,Fur}), but not conversely
(cf.\ \cite[Theorem 2]{Fur}). As noticed by Furuta,
the latter reduces to finding an example of a
hyponormal operator whose square is not hyponormal;
the square is just the wanted paranormal operator.
Since paranormal unilateral classical weighted shifts
are automatically hyponormal, they are not proper
candidates for operators with the aforesaid property.
Example \ref{exa3} shows that weighted shifts on
directed trees are prospective candidates for this
purpose.

As pointed out below, hyponormal weighted shifts on a
directed tree with nonzero weights must be injective.
   \begin{pro} \label{hypcor}
Let $\tcal$ be a directed tree with $V^\circ \neq
\varnothing$. If $\slam \in \ogr{\ell^2(V)}$ is a
hyponormal weighted shift on $\tcal$ whose all weights
are nonzero, then $\tcal$ is leafless. In particular,
$\slam$ is injective and $\card{V} = \aleph_0$.
   \end{pro}
   \begin{proof}
Suppose that, contrary to our claim, $\dzi u =
\varnothing$ for some $u \in V$. It follows from
$V^\circ \neq \varnothing$ and Corollary \ref{przem}
that $u \in V^\circ$. Then we have (see also
\eqref{wkwhyp0})
   \begin{align*}
|\lambda_u|^2 \overset{ \eqref{sl*}}= \|\slam^*e_u\|^2
\Le \|\slam e_u\|^2 \overset{\eqref{eu}}= \sum_{v \in
\dzi u} |\lambda_v|^2 = 0,
   \end{align*}
which is a contradiction. Since each leafless directed
tree is infinite, Propositions \ref{dzisdesz} and
\ref{przeldz} complete the proof.
   \end{proof}
   We now characterize the hyponormality of weighted
shifts on directed trees in terms of their weights
(see Theorem \ref{p-hyp} for the case of
$p$-hyponormality). Given a directed tree $\tcal$ and
$\lambdab = \{\lambda_v\}_{v \in V^\circ} \subseteq
\cbb$, we define \idx{$\dziplus u$}{67}
    \begin{align*}
\dziplus u = \{v \in \dzi u\colon \sum_{w\in\dzi v}
|\lambda_w|^2 > 0\}, \quad u \in V.
    \end{align*}
If $\{e_v\}_{v \in V} \subseteq \dz \slam$ and $u
\in V$, then by \eqref{eu},
   \begin{align*}
\dziplus u = \{v \in \dzi u\colon \|\slam e_v\|
> 0\}=\dzi u \cap \vplus,
   \end{align*}
and for every $v \in \dzi u$, $\|\slam e_v\| = 0$ if and
only if $\lambda_w = 0$ for all $w \in \dzi v$ (of course,
the case of $\dzi v= \varnothing$ is not excluded).
    \begin{thm} \label{hyp}
Let $\slam \in \ogr{\ell^2(V)}$ be a weighted shift on
a directed tree $\tcal$ with weights $\lambdab =
\{\lambda_v\}_{v \in V^\circ}$. Then the following
assertions are equivalent{\em :}
    \begin{enumerate}
    \item[(i)] $\slam$ is hyponormal,
    \item[(ii)] for every $u \in V$, the following two conditions
hold{\em :}
    \begin{gather}
\text{if } v \in \dzi u \text{ and } \|\slam e_v\|=0,
\text{ then } \lambda_v = 0, \label{wkwhyp0}
    \\
\sum_{v \in \dziplus u} \frac {|\lambda_v|^2}{\|\slam
e_v\|^2} \Le 1. \label{wkwhyp}
    \end{gather}
    \end{enumerate}
    \end{thm}
    \begin{proof}
Suppose that $\slam$ is hyponormal. Arguing as in the proof
of Theorem \ref{slssl*} up to the inequality \eqref{in2}
(with $c=1$ and $\|\slam e_v\|^2$ in place of $1+\|\slam
e_v\|^2$), we deduce that
   \begin{align} \label{in2+}
\Big| \sum_{v \in \dzi u} \overline {\lambda_v}
f(v)\Big|^2 \Le \sum_{v \in \dzi u} \|\slam e_v\|^2
|f(v)|^2, \quad f \in \ell^2(\dzi u),\, u \in V.
   \end{align}
If $v \in \dzi u$ is such that $\|\slam e_v\|=0$, then
by substituting $f = e_v$ into \eqref{in2+}, we obtain
$\lambda_v=0$, which proves \eqref{wkwhyp0}. In view
of \eqref{wkwhyp0} and \eqref{in2+}, we have
   \begin{align} \label{in2++}
\Big| \sum_{v \in \dziplus u} \overline {\lambda_v}
f(v)\Big|^2 \Le \sum_{v \in \dziplus u} \|\slam
e_v\|^2 |f(v)|^2, \quad f \in \ell^2(\dziplus u),\, u
\in V.
   \end{align}
This in turn implies \eqref{wkwhyp} (consult the part of
the proof of Theorem \ref{slssl*} which comes after the
inequality \eqref{in2}). It is a matter of routine to
verify that the above reasoning can be reversed.
    \end{proof}
   \begin{rem} \label{lam1}
In view of Proposition \ref{hypcor}, it is
tempting to expect that weights of an injective
hyponormal weighted shift on a directed tree are
all nonzero. However, this is not the case.
Indeed, consider two injective hyponormal
weighted shifts $S_{\boldsymbol \lambda_1}$ and
$S_{\boldsymbol \lambda_2}$ on directed trees
$\tcal_1 = (V_1,E_1)$ and $\tcal_2=(V_2,E_2)$
with nonzero weights $\lambdab_1 =
\{\lambda_{1,v}\}_{v\in V_1^\circ}$ and
$\lambdab_2=\{\lambda_{2,v}\}_{v\in V_2^\circ}$,
respectively, and assume that the directed tree
$\tcal_2$ has a root denoted by $\koo_2$ (cf.\
Examples \ref{exa3} and \ref{exa4}). Set $V = V_1
\sqcup V_2$ and $E = E_1 \sqcup E_2 \sqcup
\{(\omega_1,\koo_2)\}$, where $\omega_1$ is an
arbitrarily fixed vertex of $\tcal_1$. Then
clearly $\tcal=(V,E)$ is a directed tree and
$V^\circ = V_1^\circ \sqcup V_2$. Define a family
$\lambdab = \{\lambda_v\}_{v\in V^\circ}
\subseteq \cbb$ by setting $\lambda_v =
\lambda_{1,v}$ for $v\in V_1^\circ$, $\lambda_v =
\lambda_{2,v}$ for $v\in V_2^\circ$ and
$\lambda_{\koo_2} = 0$. Applying Proposition
\ref{dirsum} to $u=\koo_2$, we deduce that $\slam
= S_{\boldsymbol \lambda_1} \oplus S_{\boldsymbol
\lambda_2}$, which implies that $\slam$ is an
injective hyponormal weighted shifts with exactly
one vanishing weight.
   \end{rem}
   \begin{rem} \label{lam0}
It is worth mentioning that the conclusion of
Proposition \ref{hypcor} is no longer true if some
weights of $\slam$ vanish. Indeed, it is enough to
modify the example in Remark \ref{lam1} simply by
considering an injective hyponormal weighted shift
$S_{\boldsymbol \lambda_1}$ with nonzero weights and
$V_2=\{\koo_2\}$ (see the comment following Definition
\ref{defshift}). Then the weighted shift $\slam =
S_{\boldsymbol \lambda_1} \oplus S_{\boldsymbol
\lambda_2} = S_{\boldsymbol \lambda_1} \oplus 0$ is
hyponormal, it has exactly one vanishing weight and
$\dzi{\koo_2} = \varnothing$.
   \end{rem}
   \begin{rem} \label{unhyp}
The notion of hyponormality can be extended to
the case of unbounded operators. A densely
defined operator $A$ in a complex Hilbert space
$\hh$ is said to be {\em hyponormal} if $\dz{A}
\subseteq \dz{A^*}$ and $\|A^*f\| \Le \|Af\|$ for
all $f \in \dz A$. It is known that hyponormal
operators are closable and their closures are
hyponormal as well (see
\cite{ot-sch,jj1,jj2,jj3,sto} for elements of the
theory of unbounded hyponormal operators). A
close inspection of the proof reveals that {\em
Theorem {\em \ref{hyp}} remains true for densely
defined weighted shifts on directed trees} (in
the present more general context the inequalities
\eqref{in2+} and \eqref{in2++}, which play a
pivotal role in the proof, have to be considered
for $f \in \dz{\slam}$). Note also that if
$\slam$ is a densely defined weighted shift on a
directed tree $\tcal$ with weights $\lambdab =
\{\lambda_v\}_{v \in V^\circ}$, then the
conditions \eqref{wkwhyp0} and \eqref{wkwhyp}
imply \eqref{incofds} with $c=1$. Indeed, for all
$u \in V$,
   \begin{align*}
\sum_{v \in \dzi u} \frac {|\lambda_v|^2} {1+
\|\slam e_v\|^2} \overset{\eqref{wkwhyp0}}=
\sum_{v \in \dziplus u} \frac {|\lambda_v|^2} {1+
\|\slam e_v\|^2} \Le \sum_{v \in \dziplus u}
\frac {|\lambda_v|^2} {\|\slam e_v\|^2}
\overset{\eqref{wkwhyp}} \Le 1.
   \end{align*}
As a consequence of Theorem \ref{slssl*}, we get
$\dz{\slam} \subseteq \dz{\slam^*}$.
   \end{rem}
   \subsection{\label{cohypcohyp}Cohyponormality}
Recall that an operator $A \in \ogr \hh$ is said to be
{\em cohyponormal} if its adjoint $A^*$ is hyponormal.
The question of cohyponormality of weighted shifts on
directed trees is more delicate than hyponormality. It
requires considering two distinct cases.
   \begin{lem} \label{cohyp}
Let $\slam \in \ogr{\ell^2(V)}$ be a weighted shift on
a directed tree $\tcal$ with weights $\lambdab =
\{\lambda_v\}_{v \in V^\circ}$. Then the following
assertions hold{\em :}
    \begin{enumerate}
    \item[(i)] if $\tcal$ has a root, then $\slam$ is cohyponormal
if and only if $\slam = 0$,
    \item[(ii)] if $\tcal$ is rootless, then $\slam$ is cohyponormal
if and only if for every $u \in V$ the following
two conditions are satisfied{\em :}
   \begin{enumerate}
   \item[(a)] $\card {\dziplus u} \Le  1$,
   \item[(b)] if $\card {\dziplus u} = 1$, then
$0 < \|\slam e_v\| \Le |\lambda_v|$ for $v \in
\dziplus u$ and $\lambda_w = 0$ for $w \in \dzi u
\setminus \dziplus u$.
   \end{enumerate}
    \end{enumerate}
    \end{lem}
    \begin{proof}
(i) Assume that $\tcal$ has a root and $\slam$ is
cohyponormal. Note that if $\slam^* e_u = 0$ for some
$u \in V$, then by the cohyponormality of $\slam$ we
have $\slam e_u=0$, which together with \eqref{eu}
implies that $\lambda_v = 0$ for all $v \in \dzi{u}$,
or equivalently that $\slam^* e_v=0$ for all $v \in
\dzi{u}$ (use \eqref{sl*}). Since $\slam^*
e_{\koo}=0$, an induction argument shows that
$\lambda_v = 0$ for all $v \in \dzin n u$ and for all
$n=1,2,\ldots$ Applying Corollary \ref{przem}, wee see
that $\lambda_v = 0$ for all $v \in V^\circ$, which
means that $\slam =0$.

(ii) Assume that $\tcal$ is rootless. Suppose first that
$\slam$ is cohyponormal. Arguing as in the proof of Theorem
\ref{zen} up to the inequality \eqref{both}, we deduce that
    \begin{align}   \label{niercoh1}
\sum_{u \in V} \|\slam e_u\|^2 |f(u)|^2 \Le \sum_{u
\in V} \Big| \sum_{v \in \dzi u} \overline {\lambda_v}
f(v)\Big|^2 \hspace{-1ex},\quad f \in \ell^2(V),
    \end{align}
which by Proposition \ref{46} takes the form
    \begin{align}   \label{niercoh2}
\sum_{u \in V} \sum_{v \in \dzi u} \|\slam
e_v\|^2 |f(v)|^2 \Le \sum_{u \in V} \Big| \sum_{v
\in \dzi u} \overline {\lambda_v} f(v) \Big|^2
\hspace{-1ex},\quad f \in \ell^2(V).
    \end{align}
This in turn is equivalent to
    \begin{align} \label{ineq2}
\sum_{v \in \dzi u} \|\slam e_v\|^2 |f(v)|^2 \Le
|\is{f}{\lambdab^u}|^2,\quad f \in \ell^2(\dzi u), \,
u \in V^\prime,
    \end{align}
where $\lambdab^u = \{\lambda_v\}_{v \in \dzi u}$. It
follows from \eqref{ineq2} that
    \begin{align*}
\ell^2(\dzi u) \ominus \langle \lambdab^u\rangle
\subseteq \ell^2(\dzi u) \ominus \ell^2(\dziplus
u), \quad u \in V^\prime,
    \end{align*}
where $\langle \lambdab^u\rangle$ stands for the
linear span of $\{\lambdab^u\}$. Taking orthogonal
complements in $\ell^2(\dzi u)$, we obtain
$\ell^2(\dziplus u) \subseteq \langle
\lambdab^u\rangle$. This implies that $\card {\dziplus
u} \Le 1$. If $\card {\dziplus u} = 1$, say $\dziplus
u = \{v\}$, then $\ell^2(\dziplus u) = \langle
\lambdab^u\rangle$, which guarantees that $\lambda_v
\neq 0$ and $\lambda_w = 0$ for all $w \in \dzi u
\setminus \{v\}$. In turn, by the cohyponormality of
$\slam$ we have
   \begin{align*}
\|\slam e_v\| \Le \|\slam^* e_v\|
\overset{\eqref{sl*}}= |\lambda_v|.
   \end{align*}

Reversing the above reasoning completes the
proof.
    \end{proof}
Lemma \ref{cohyp} enables us to give a complete description
of cohyponormal weight\-ed shifts on a directed tree. In
particular, the structure of a subtree corresponding to
nonzero weights of a cohyponormal weighted shift is
established.
   \begin{thm} \label{cohyp-opis}
Let $\slam \in \ogr{\ell^2(V)}$ be a nonzero weighted
shift on a directed tree $\tcal$ with weights
$\lambdab = \{\lambda_v\}_{v \in V^\circ}$. Then
$\slam$ is cohyponormal if and only if the tree
$\tcal$ is rootless and one of the following two
disjunctive conditions holds\/{\em :}
   \begin{enumerate}
   \item[(i)] there exists a sequence
$\{u_{n}\}_{n=-\infty}^{\infty} \subseteq V$ such that
   \begin{align} \label{defun}
\text{$0 < |\lambda_{u_{n}}| \Le
|\lambda_{u_{n-1}}|$ and $u_{n-1} = \pa {u_n}$}
   \end{align}
for all $n \in \zbb$, and $\lambda_v = 0$ for all
$v \in V \setminus \{u_{n}\colon n\in\zbb\}$,
   \item[(ii)] there exist a sequence
$\{u_{n}\}_{n=-\infty}^{0} \subseteq V$ such that
   \begin{align} \label{defun2}
\text{$0 < \sum_{v \in \dzi{u_0}}|\lambda_v|^2
\Le |\lambda_{u_{0}}|^2$, $0 < |\lambda_{u_{n}}|
\Le |\lambda_{u_{n-1}}|$ and $u_{n-1} = \pa
{u_n}$}
   \end{align}
for all integers $n \Le 0$, and $\lambda_v = 0$
for all $v \in V \setminus (\{u_{n}\colon n\Le
0\} \cup \dzi{u_0})$.
   \end{enumerate}
   \end{thm}
   \begin{proof}
Suppose that $\slam$ is cohyponormal. Since $\slam
\neq 0$, it follows from Lemma \ref{cohyp}\,(i) that
$\tcal$ is rootless. Hence, we have
   \begin{align} \label{indnz}
|\lambda_u| \overset{ \eqref{eu}} \Le \|\slam e_{\pa
u}\| \Le \|\slam^* e_{\pa u}\| \overset{ \eqref{sl*}}
= |\lambda_{\pa u}|, \quad u \in V.
   \end{align}
Observe that
   \begin{align} \label{step1}
   \begin{minipage}{30em}
if $u \in V$ is such that $\card {\dziplus u} = 1$,
then $\card {\dziplus{\pa{u}}} = 1$ and $u \in
\dziplus{\pa{u}}$.
   \end{minipage}
   \end{align}
   Indeed, it follows from Lemma \ref{cohyp}\,(b) that
$0 < \|\slam e_v\| \Le |\lambda_v|$, where $v \in
\dziplus{u}$. As a consequence, we have $\lambda_v
\neq 0$ and $\|\slam e_u\| > 0$. Hence, $u \in
\dziplus {\pa u}$ and so, by Lemma
\ref{cohyp}\,(a), $\card {\dziplus{\pa{u}}} = 1$.

Since $\tcal$ is rootless and $\slam \neq 0$,
there exists $u \in V$ such that $\card {\dziplus
u} = 1$. Applying an induction procedure, we
infer from \eqref{step1} that there exists a
backward sequence $\{u_n\}_{n=-\infty}^{-1}$ such
that $u_{n-1} = \pa {u_n}$ and $\card
{\dziplus{u_n}} = 1$ for all integers $n \Le -1$
with $u_{-1} = u$ (by Proposition \ref{xdescor},
$u_m \neq u_n$ whenever $m \neq n$). Take $v \in
\dziplus{u}$. If $\card {\dziplus v} = 1$, then
the new sequence $\{\ldots, u_{-2}, u_{-1},
u_{0}\}$ with $u_{0}=v$ shares the same
properties as $\{u_n\}_{n=-\infty}^{-1}$. There
are now two possibilities. The first is that this
procedure never terminates. Let us denote by
$\{u_n\}_{n=-\infty}^\infty$ the resulting
sequence. We show that
$\{u_n\}_{n=-\infty}^\infty$ fulfills (i). For
this, take $v \in V \setminus \{u_{n} \colon n\in
\zbb\}$. Suppose that, contrary to our claim,
$\lambda_v \neq 0$. By Proposition \ref{witr},
there exists an integer $k \Ge 1$ such that
$\paa^k(v) \in \{u_{n}\colon n\in\zbb\}$, say
$\paa^k(v) = u_l$ with $l\in \zbb$, and
$\paa^j(v) \notin \{u_{n}\colon n\in\zbb\}$ for
$j=0, \ldots, k-1$. It follows from \eqref{indnz}
that $\lambda_{\paa^{k-1}(v)} \neq 0$. Since
$\dziplus{u_{l}} = \{u_{l+1}\}$, $\paa^{k-1}(v)
\in \dzi{u_l}$ and $\paa^{k-1}(v) \neq u_{l+1}$,
we get a contradiction with Lemma
\ref{cohyp}\,(b). The inequalities in
\eqref{defun} follow from Lemma \ref{cohyp}\,(b).

The other possibility is that our procedure
terminates, which means that we obtain a sequence
$\{u_n\}_{n=-\infty}^{-1}$ such that $u_{n-1} = \pa
{u_n}$ and $\card {\dziplus{u_n}} = 1$ for all
integers $n \Le -1$, and
   \begin{align} \label{zen0}
\card {\dziplus{u_0}} = 0 \text{ for a unique } u_0
\in \dziplus{u_{-1}}.
   \end{align}
We show that $\{u_n\}_{n=-\infty}^{0}$ fulfills
(ii). Take $v \in V \setminus (\{u_{n}\colon n\Le
0\} \cup \dzi{u_0})$. Suppose that, contrary to
our claim, $\lambda_v \neq 0$. As in the previous
case, we find $k \Ge 1$ such that $\paa^k(v) \in
\{u_{n}\colon n\Le 0\} \cup \dzi{u_0}$ and
$\paa^j(v) \notin \{u_{n} \colon n \Le 0\} \cup
\dzi{u_0}$ for $j=0, \ldots, k-1$. By
\eqref{indnz}, we get $\lambda_{\paa^{k-1}(v)}
\neq 0$. Now we have two possibilities, either
$\paa^{k}(v) \in \dzi{u_0}$ which implies that
$\paa^{k}(v) \in \dziplus{u_{0}}$, in
contradiction with \eqref{zen0}, or $\paa^{k}(v)
\in \{u_{n}\colon n\Le -1\}$, say $\paa^k(v) =
u_l$ with $l\Le -1$ (note that the case of $l=0$
is impossible), which implies that
$\dziplus{u_{l}} = \{u_{l+1}\}$ (use \eqref{zen0}
if $l=-1$, and \eqref{step1} if $l \Le -2$),
$\paa^{k-1}(v) \in \dzi{u_l}$ and $\paa^{k-1}(v)
\neq u_{l+1}$, in contradiction with Lemma
\ref{cohyp}\,(b). Since the inequalities in
\eqref{defun2} hold due to Lemma
\ref{cohyp}\,(b), the ``only if'' part of the
conclusion is justified.

Finally, applying Lemma \ref{parchi} (with $X =
\{u_n\colon n \in \zbb\}$ or $X = \dzi{u_0} \cup
\{u_n\colon n \Le 0\}$) and Lemma
\ref{cohyp}\,(ii) completes the proof of the
``if'' part.
   \end{proof}
The following corollary is a direct consequence of
Proposition \ref{dzisdesz} and Theorem
\ref{cohyp-opis}. It asserts that injective
cohyponormal weighted shifts on directed trees are
bilateral classical weighted shifts.
   \begin{cor}  \label{injcoh}
A nonzero weighted shift $\slam \in
\ogr{\ell^2(V)}$ on a directed tree $\tcal$ with
weights $\lambdab = \{\lambda_v\}_{v \in
V^\circ}$ is injective and cohyponormal if and
only if there exists a sequence
$\{u_{k}\}_{k=-\infty}^{\infty}$ such that $V =
\{u_k\colon k \in \zbb\}$, $u_{n-1} = \pa {u_n}$
and $0 < |\lambda_{u_{n}}| \Le
|\lambda_{u_{n-1}}|$ for all $n \in \zbb$.
   \end{cor}
A typical directed tree admitting a nonzero
cohyponormal weighted shift $\slam$ that satisfies the
condition (ii) of Theorem \ref{cohyp-opis} is
illustrated in Figure 2 below (such directed tree is
infinite). Only edges joining vertexes corresponding
to nonzero weights of $\slam$ are drawn. Of course,
the tree may have more edges however their arrowheads
must correspond to zero weights of $\slam$; they are
not drawn.
   \vspace{2ex}
   \begin{center}
   \includegraphics[width=6cm]
    %{Graf5.png}
    {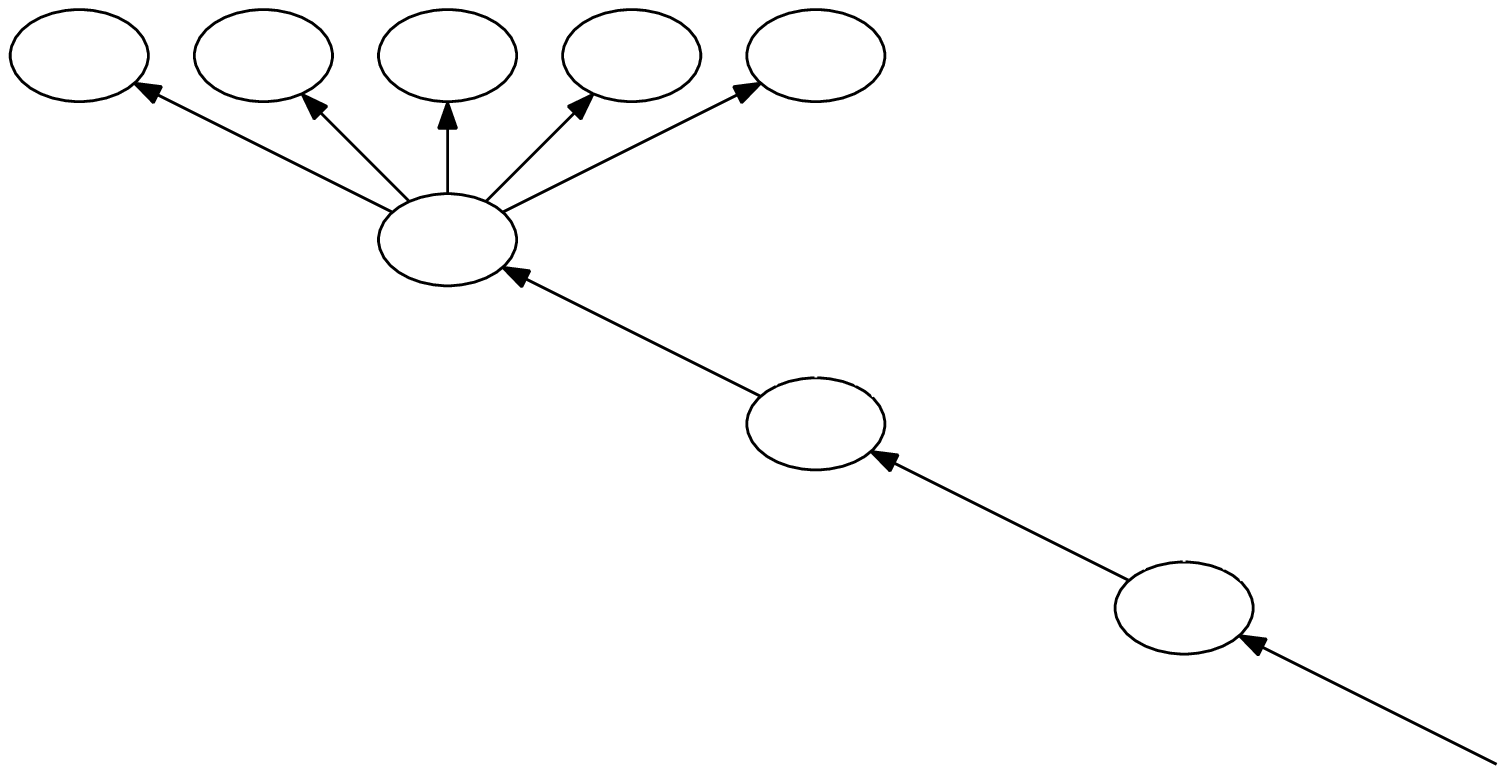}
   \\[1.5ex]
   {\small {\sf Figure 2}}
   \end{center}
   \begin{rem}
A closed densely defined operator $A$ in a
complex Hilbert space $\hh$ is said to be {\em
cohyponormal} if $A^*$ is hyponormal (cf.\ Remark
\ref{unhyp}), i.e., $\dz{A^*} \subseteq \dz{A}$
and $\|Af\| \Le \|A^*f\|$ for all $f \in
\dz{A^*}$. A thorough inspection of proofs shows
that {\em Lemma {\em \ref{cohyp}}, Theorem {\em
\ref{cohyp-opis}} and Corollary {\em
\ref{injcoh}} remain true for densely defined
weighted shifts on directed trees} (now we have
to employ Propositions \ref{desc}\,(v) and
\ref{sprz}\,(ii); moreover, the inequalities
\eqref{niercoh1} and \eqref{niercoh2} have to be
considered for $f \in \dz{\slam^*}$).
   \end{rem}
   \subsection{Examples}
We begin by giving an example of a bounded injective
weight\-ed shift on a directed tree which is
paranormal but not hyponormal (see Section
\ref{hypcohyp} for definitions). In view of Remark
\ref{re1-2}, the directed tree considered in Example
\ref{pra-nothyp} below is one step more complicated
than that appearing in the case of classical weighted
shifts.
   \begin{exa} \label{pra-nothyp}
Let $\tcal$ be the directed tree as in Figure 3 with
$V^\circ$ given by $V^\circ = \{(i,j) \colon i = 1,2,
\, j=1, 2, \ldots\}$ (note that $\tcal=\tcal_{2,0}$,
cf.\ \eqref{varkappa}).
   \vspace{1.5ex}
   \begin{center}
   \includegraphics[width=10cm]
   %{Graf2.png}
   {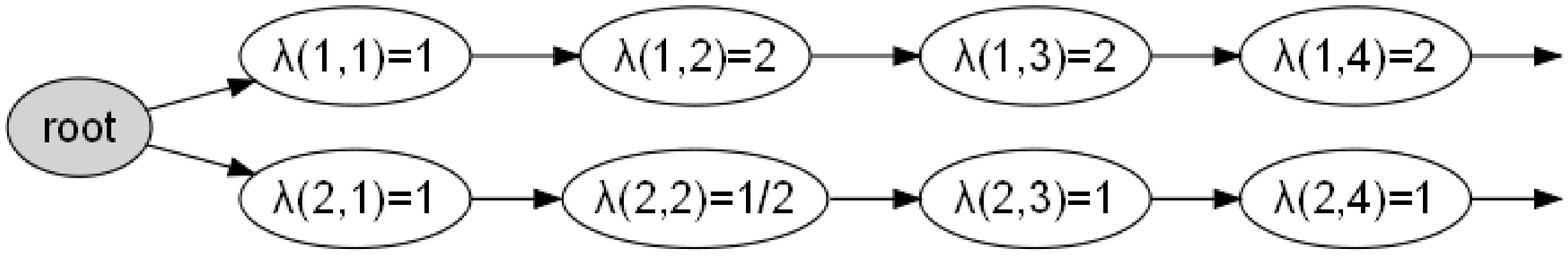}
   \\[1.5ex]
   {\small {\sf Figure 3}}
   \end{center}
   \vspace{1ex}
   Let $\slam$ be the weighted shift on
$\tcal$ with weights $\boldsymbol \lambda =
\{\lambda(i,j)\}_{(i,j) \in V^\circ}$ given by
$\lambda(1,1)=1$, $\lambda(1,j)=2$ for all $j \Ge 2$,
$\lambda(2,2)=1/2$ and $\lambda(2,j)=1$ for all $j
\neq 2$. By Corollary \ref{parc} and Proposition
\ref{dzisdesz}, $\slam \in \ogr {\ell^2(V)}$ and
$\jd{\slam}=\{0\}$. To make the notation more
readable, we write $e_{i,j}$ instead of $e_u$ for $u =
(i,j) \in V^\circ$. It follows from \eqref{eu} that
    \begin{gather} \label{dwa19X}
    \begin{gathered}
    \begin{cases}
\slam (e_{\koo}) = e_{1,1} + e_{2,1},\quad \slam
(e_{1,j}) = 2e_{1,j+1} \text{ for } j\Ge1,
   \\
\slam (e_{2,1}) = \frac12 e_{2,2},\quad \slam
(e_{2,j}) = e_{2,j+1} \text{ for }j\Ge2.
    \end{cases}
    \end{gathered}
    \end{gather}
Take $f= \alpha e_{\koo} + \sum_{j=1}^\infty \alpha_{1,j}
e_{1,j} + \sum_{j=1}^\infty \alpha_{2,j} e_{2,j} \in
\ell^2(V)$. Then, by  \eqref{dwa19X}, we have
    \begin{align*}
\slam f&= \alpha e_{1,1} + \alpha e_{2,1} + 2
\sum_{j=1}^\infty \alpha_{1,j} e_{1,j+1} +
\frac12\alpha_{2,1} e_{2,2} + \sum_{j=2}^\infty
\alpha_{2,j} e_{2,j+1},
    \\
\slam^2 f&= 2\alpha e_{1,2} + \frac12\alpha e_{2,2} +
4\sum_{j=1}^\infty \alpha_{1,j} e_{1,j+2} +
\frac12\alpha_{2,1} e_{2,3} + \sum_{j=2}^\infty
\alpha_{2,j} e_{2,j+2}.
    \end{align*}
Putting all this together implies that
    \begin{align*}
\|f\|^2&= |\alpha|^2 + \sum_{j=1}^\infty |\alpha_{1,j}|^2 +
|\alpha_{2,1}|^2 + \sum_{j=2}^\infty |\alpha_{2,j}|^2,
    \\
\|\slam f\|^2&= 2|\alpha|^2 +4\sum_{j=1}^\infty
|\alpha_{1,j}|^2 + \frac14|\alpha_{2,1}|^2+
\sum_{j=2}^\infty |\alpha_{2,j}|^2, \\
\|\slam^2 f\|^2&= \frac{17}4|\alpha|^2 +
16\sum_{j=1}^\infty |\alpha_{1,j}|^2 +
\frac14|\alpha_{2,1}|^2+ \sum_{j=2}^\infty
|\alpha_{2,j}|^2.
    \end{align*}
Substituting $x_1=|\alpha|^2$, $x_2=\sum_{j=1}^\infty
|\alpha_{1,j}|^2$, $x_3=|\alpha_{2,1}|^2$ and $x_4=
\sum_{j=2}^\infty |\alpha_{2,j}|^2$ into the following
inequality\footnote{\;which can be obtained from the
Cauchy-Schwarz inequality, or by direct computation}
   \begin{align*}
\Big(2x_1+4x_2+\frac14x_3+x_4\Big)^2 \Le \Big(x_1+x_2+x_3+x_4\Big)
\cdot\Big(\frac{17}4x_1+16x_2+\frac14x_3+x_4\Big),
    \end{align*}
which is true for all nonnegative reals
$x_1,x_2,x_3,x_4$, we get $\|\slam f\|^4 \Le \|f\|^2
\|\slam^2 f\|^2$. This means that $\slam$ is
paranormal. Since, by \eqref{dwa19X} and \eqref{sl*},
$\|\slam e_{2,1}\| = \frac 12$ and $\|\slam^*
e_{2,1}\| = 1$, we conclude that $\slam$ is not
hyponormal.
    \end{exa}
Recall that there are bounded hyponormal operators
whose squares are not hyponormal (see
\cite{hal1,I-W,hal,di-ca}). This cannot happen for
classical weighted shifts. However, this really
happens for weighted shifts on directed trees. Our
next example is build on a very simple directed tree
(though a little bit more complicated than that in
Example \ref{pra-nothyp}). It is parameterized by four
real parameters $q$, $r$, $s$ and $t$ (in fact, by
three independent real parameters, cf.\
\eqref{sfera}). In Example \ref{exa4} below, we
demonstrate yet another sample of a hyponormal
weighted shift with non-hyponormal square which is
indexed only by two real parameters, however it is
build on a more complicated directed tree. It seems to
be impossible to reduce the number of parameters.
    \begin{exa}     \label{exa3}
Let $\tcal$ be the directed tree as in Figure 4 with
$V^\circ$ given by $V^\circ = \{(0,0)\} \cup \{(i,j)
\colon i=1,2; \, j=1,2, \ldots\}$ (note that
$\tcal=\tcal_{2,1}$, cf.\ \eqref{varkappa}).
   \vspace{1.5ex}
   \begin{center}
   \includegraphics[width=10cm]
   %{Graf3.png}
   {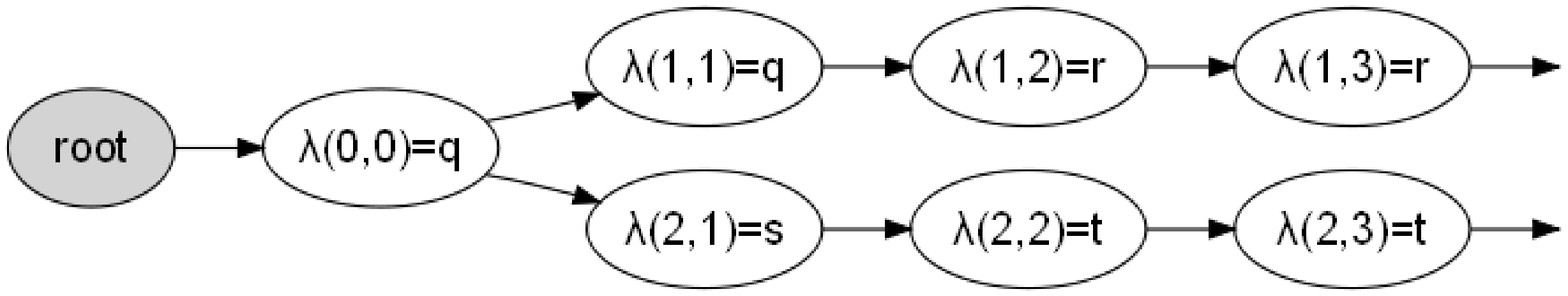}
   \\[1.5ex]
   {\small {\sf Figure 4}}
   \end{center}
   \vspace{1ex}
   Fix positive real numbers $q,r,s,t$ such that
    \begin{align} \label{sfera}
\Big(\frac s t\Big)^2 \Big(\frac rt\Big)^2 +
\Big(\frac qt\Big)^2 \Le \Big(\frac rt\Big)^2 \quad
\text{and } \quad \Big(\frac s t\Big) \Big(\frac
qt\Big) > 1.
    \end{align}
(e.g., $q=4$, $r=8$, $s=1$ and $t \in [2/\sqrt 3,
2)$ are sample numbers satisfying \eqref{sfera}).
Let $\slam$ be the weighted shift on $\tcal$ with
weights $\boldsymbol \lambda =
\{\lambda(i,j)\}_{(i,j) \in V^\circ}$ given by
$\lambda(0,0)=\lambda(1,1)=q$, $\lambda(2,1)=s$,
$\lambda(1,j) = r$ and $\lambda(2,j)=t$ for
$j=2,3, \ldots$ By Corollary \ref{parc} and
Proposition \ref{dzisdesz}, $\slam \in \ogr
{\ell^2(V)}$ and $\jd{\slam}=\{0\}$. It is a
routine matter to verify that $\slam$ satisfies
the inequality \eqref{wkwhyp}. Hence, by Theorem
\ref{hyp}, $\slam$ is hyponormal. It follows from
\eqref{eu} and \eqref{sl*} that $\slam^2 e_{2,1}
= t^2 e_{2,3}$ and $\slam^{*2} e_{2,1} = sq
e_{\koo}$, which, by the right-hand inequality in
\eqref{sfera}, implies that $\|\slam^2 e_{2,1}\|
< \|\slam^{*2} e_{2,1}\|$. This means that
$\slam^2$ is not hyponormal.
    \end{exa}
   \begin{exa} \label{exa4} Let $\tcal$ be the
directed binary tree as in Figure 5 with $V^\circ$ given by
   \begin{align*}
   V^\circ = \{(i,j) \colon i = 1,2, \ldots, \, j=1,
\ldots,2^{i-1}\}.
   \end{align*}
Fix positive real numbers $\alpha, \beta$. Let $\slam$
be the weighted shift on $\tcal$ with weights
$\boldsymbol \lambda = \{\lambda(i,j)\}_{(i,j) \in
V^\circ}$ given by
   \begin{align*}
\lambda(i,j)=
  \begin{cases}
  \alpha & \text{if } i=j=1,
  \\
   \alpha & \text{if } i \Ge 2 \text{ and } j = 1,
\ldots, 2^{i-2},
   \\
   \beta & \text{if } i \Ge 2 \text{ and } j =
2^{i-2}+1, \ldots, 2^{i-1}.
   \end{cases}
   \end{align*}
   By Corollary \ref{parc} and Proposition
\ref{dzisdesz}, $\slam \in \ogr {\ell^2(V)}$ and
$\jd{\slam}=\{0\}$. In virtue of Theorem \ref{hyp},
$\slam$ is hyponormal. Assume now that $\alpha > 2
\beta$. Then $\slam^2$ is not hyponormal because, by
\eqref{eu} and \eqref{sl*}, we have (see the
notational convention used in Example
\ref{pra-nothyp})
   \begin{align*}
\|\slam^{*2} e_{2,2}\| = \alpha \beta > 2 \beta^2 =
\|\slam^2 e_{2,2}\|.
   \end{align*}
   \begin{center}
   \includegraphics[width=8cm]
   %{Graf4.png}
   {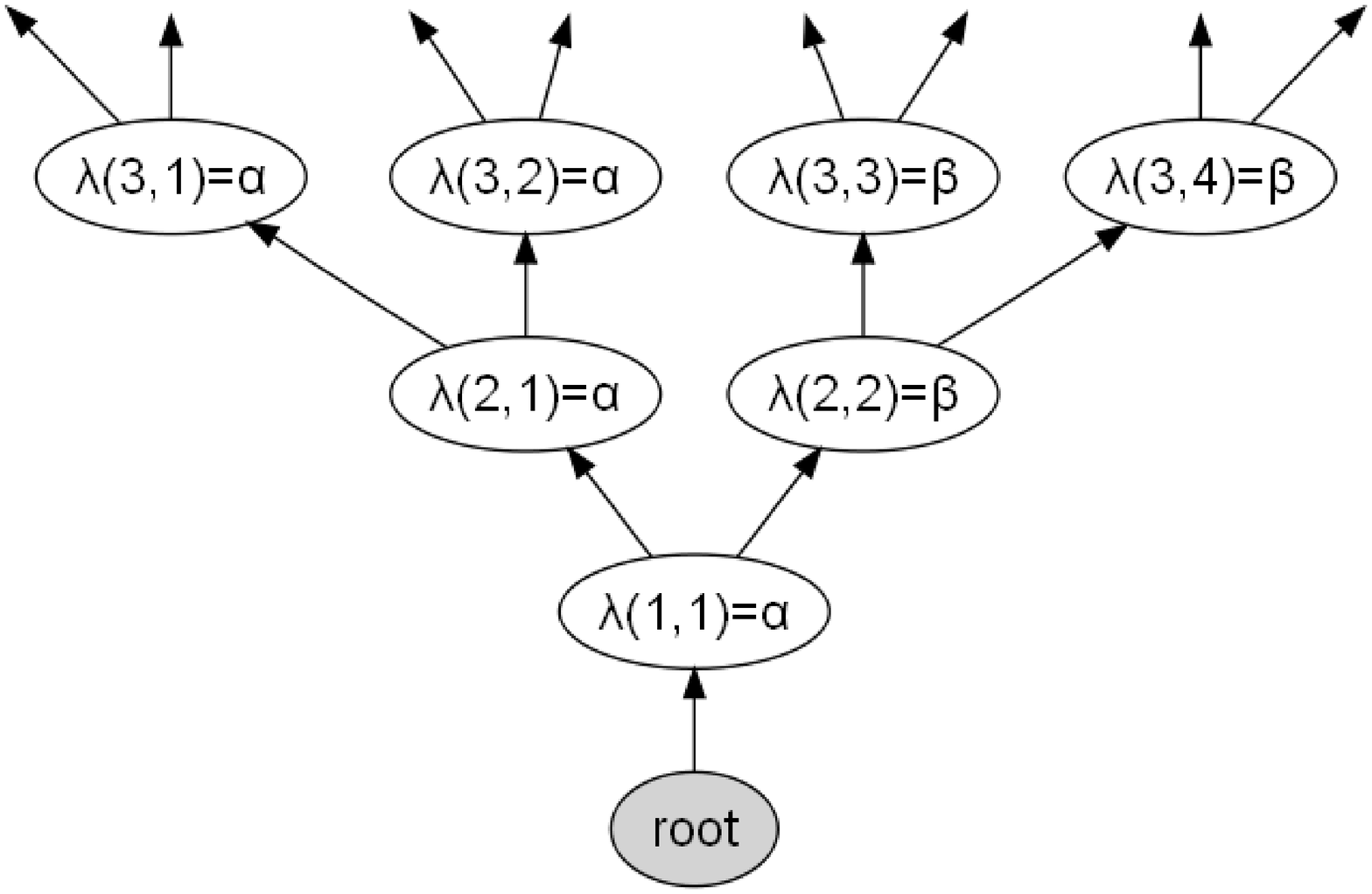} \\[1.5ex]{\small {\sf Figure 5}}
   \end{center}
    \end{exa}
   \newpage
   \section{\label{ch6}Subnormality}
   \subsection{A general approach}
Our goal in this section is to find a characterization
of subnormality of weighted shifts on directed trees
(see Section \ref{hypcohyp} for the definition of a
subnormal operator). We begin by attaching to a family
$\lambdab = \{\lambda_v\}_{v \in V^\circ}$ of weights
of a weighted shift $\slam$ on a directed tree $\tcal$
a new family $\{\lambda_{u\mid v}\}_{u \in V, v \in
\des{u}}$ defined~ by \idxx{$\lambda_{u \mid v}$}{68}
   \begin{align*}
\lambda_{u\mid v} =
   \begin{cases}
   1 & \text{ if } v=u,
   \\
   \prod_{j=0}^{n-1} \lambda_{\paa^{j}(v)} & \text{ if
} v \in \dzin{n}{u}, \, n \Ge 1.
   \end{cases}
   \end{align*}
Owing to \eqref{decom}, the above definition is
correct. It is easily seen that the following
recurrence formula holds:
   \begin{align} \label{recfor}
\lambda_{u\mid v} & = \lambda_{u\mid \pa v} \lambda_v,
\quad u \in V, \, v \in \des u, \, v \neq u,
   \\
   \lambda_{\pa v\mid w} & = \lambda_v \lambda_{v\mid
w}, \quad v \in V^\circ, \, w\in \des v.
\label{recfor2}
   \end{align}
   \begin{lem} \label{pot}
If $\slam \in \ogr {\ell^2(V)}$ is a weighted shift on
a directed tree $\tcal$ with weights $\lambdab =
\{\lambda_v\}_{v \in V^\circ}$, then the following two
conditions hold\/{\em :}
   \begin{enumerate}
   \item[(i)] $\slam^n e_u = \sum_{v \in \dzin{n}{u}}
\lambda_{u\mid v} \, e_v$ for all $u \in V$ and $n\in
\zbb_+$,
   \item[(ii)] $\|\slam^n e_u\|^2 =
\sum_{v \in \dzin{n}{u}} |\lambda_{u\mid v}|^2$ for
all $u \in V$ and $n\in \zbb_+$.
   \end{enumerate}
   \end{lem}
   \begin{proof}
All we need to prove is the equality in the condition
(i). We proceed by induction on $n$. The case $n=0$ is
obvious. Suppose that (i) holds for a given integer
$n\Ge 0$. It follows from the definition of
$\dzin{n+1}{u}$ and Proposition \ref{46} that
   \begin{align} \label{dzinn1}
\dzin{n+1}{u} = \bigsqcup_{v \in \dzin{n}{u}} \dzi{v}.
   \end{align}
Then by the induction hypothesis and the continuity of
$\slam$ we have
   \begin{align*}
\slam^{n+1} e_u = \slam (\slam^n e_u) & \hspace{2.7ex}
= \sum_{v \in \dzin{n}{u}} \lambda_{u\mid v} \, \slam
e_v
   \\
& \hspace{0.9ex}\overset{\eqref{eu}} = \underset{v \in
\dzin{n}{u}\hspace{1ex}}{\sumo} \sum_{w\in\dzi v}
\lambda_{u\mid v} \lambda_w e_w
\\
& \hspace{.9ex}\overset{\eqref{dzinn1}} = \sum_{w \in
\dzin{n+1}{u}} \lambda_{u\mid \pa{w}} \lambda_w e_w
\overset{\eqref{recfor}} = \sum_{w \in \dzin{n+1}{u}}
\lambda_{u\mid w} e_w,
   \end{align*}
where the symbol \idx{$\sumo$}{69} is reserved for
denoting the orthogonal series. This completes the
proof.
   \end{proof}
We say that a sequence $\{t_n\}_{n=0}^\infty$ of real
numbers is a {\em Hamburger moment sequence} if there
exists a positive Borel measure $\mu$ on $\rbb$ such
that
   \begin{align*} %\label{stjelt}
   t_{n}=\int_{\rbb} s^n \D\mu(s),\quad n\in \zbb_+;
   \end{align*}
$\mu$ is called a {\em representing measure} of
$\{t_n\}_{n=0}^\infty$. In view of Hamburger's
theorem (cf.\ \cite[Theorem 1.2]{sh-tam}), a
sequence $\{t_n\}_{n=0}^\infty$ of real numbers
is a Hamburger moment sequence if and only if it
is {\em positive definite}, i.e.,
   \begin{align*}
\sum_{k,l=0}^n t_{k+l} \alpha_k \overline{\alpha_l}
\Ge 0, \quad \alpha_0,\ldots, \alpha_n \in \cbb,\,
\quad n \in \zbb_+.
   \end{align*}
A Hamburger moment sequence is said to be {\em
determinate} if it has only one representing measure.
Let us recall a useful criterion for determinacy.
   \begin{align} \label{determ}
   \begin{minipage}{28em}
If $\{t_n\}_{n=0}^\infty$ is a Hamburger moment
sequence such that $a:=\limsup_{n\to \infty}
t_{2n}^{\nicefrac 1 {2 n}} < \infty$, then it is determinate and
its unique representing measure is concentrated on
$[-a,a]$.
   \end{minipage}
   \end{align}
Indeed, if $\mu$ is a representing measure of
$\{t_n\}_{n=0}^\infty$, then the closed support of
$\mu$ is contained in $[-a,a]$ (cf.\ \cite[page
71]{Rud}), and so $\mu$ is a unique representing
measure of $\{t_n\}_{n=0}^\infty$ (see \cite{fug}).
Yet another approach to \eqref{determ} is in
\cite[Theorem~ 2]{sza}.

A Hamburger moment sequence having a representing
measure concentrated on $[0,\infty)$ is called a {\em
Stieltjes moment sequence}. Clearly
   \begin{align}  \label{st+1}
\text{if $\{t_n\}_{n=0}^\infty$ is a Stieltjes moment
sequence, then so is $\{t_{n+1}\}_{n=0}^\infty$.}
   \end{align}
The above property is no longer valid for Hamburger
moment sequences. By Stieltjes theorem (cf.\
\cite[Theorem~ 1.3]{sh-tam}), a sequence
$\{t_n\}_{n=0}^\infty \subseteq \rbb$ is a Stieltjes
moment sequence if and only if the sequences
$\{t_n\}_{n=0}^\infty$ and $\{t_{n+1}\}_{n=0}^\infty$
are positive definite.

The question of backward extendibility of Hamburger or
Stieltjes moment sequences has known solutions (see
e.g., \cite[Proposition 8]{cur} and \cite{sz}). What
we further need is a variant of this question.
   \begin{lem} \label{bext}
Let $\{t_n\}_{n=0}^\infty$ be a Stieltjes moment
sequence. Set $t_{-1}=1$. Then the following three
conditions are equivalent{\em :}
   \begin{enumerate}
   \item[(i)] $\{t_{n-1}\}_{n=0}^\infty$ is a
Stieltjes moment sequence,
   \item[(ii)]  $\{t_{n-1}\}_{n=0}^\infty$ is positive
definite,
   \item[(iii)] there exists a representing measure $\mu$
of $\{t_n\}_{n=0}^\infty$ concentrated on $[0,\infty)$ such
that\/\footnote{\;\label{foot}We adhere to the convention
that $\frac 1 0 := \infty$. Hence, $\int_0^\infty \frac 1 s
\D \mu(s) < \infty$ implies $\mu(\{0\})=0$.} $\int_0^\infty
\frac 1 s \D \mu(s) \Le 1$.
   \end{enumerate}
If $\mu$ is as in {\em (iii)}, then the positive Borel
measure $\nu$ on $\rbb$ defined by
   \begin{align}   \label{nu}
\nu(\sigma) = \int_\sigma \frac 1 s \D \mu(s) +
\Big(1-\int_0^\infty \frac 1 s \D \mu(s)\Big)
\delta_0(\sigma), \quad \sigma \in \borel{\rbb},
   \end{align}
is a representing measure of $\{t_{n-1}\}_{n=0}^\infty$
concentrated on $[0,\infty)$; moreover, $\nu(\{0\})=0$ if
and only if $\int_0^\infty \frac 1 s \D \mu(s)=1$.
   \end{lem}
   \begin{proof}
(i)$\Leftrightarrow$(ii) Apply the Stieltjes theorem.

(i)$\Rightarrow$(iii) Let $\rho$ be a representing
measure of $\{t_{n-1}\}_{n=0}^\infty$ concentrated on
$[0,\infty)$. Define the positive Borel measure $\mu$
on $[0,\infty)$ by $\D \mu (s) = s \D \rho (s)$. Then
   \begin{align*}
t_n = t_{(n+1)-1} = \int_0^\infty s^n s\D \rho(s) =
\int_0^\infty s^n \D \mu(s), \quad n \in \zbb_+.
   \end{align*}
Clearly, $\mu(\{0\})=0$ and consequently
   \begin{align*}
\int_0^\infty \frac 1 s \D \mu(s) = \int_{(0,\infty)} \D
\rho(s) = \rho((0,\infty)) = \int_{[0,\infty)} s^0 \D
\rho(s) - \rho(\{0\}) = 1 - \rho(\{0\}).
   \end{align*}
This implies that $\int_0^\infty \frac 1 s \D \mu(s) \Le
1$.

(iii)$\Rightarrow$(i) It is easily verifiable that the
measure $\nu$ defined by \eqref{nu} is concentrated on
$[0,\infty)$ and $t_{n-1}= \int_0^\infty s^n \D \nu(s)$ for
all $n \in \zbb_+$. Hence $\{t_{n-1}\}_{n=0}^\infty$ is a
Stieltjes moment sequence. The ``moreover'' part of the
conclusion is obvious.
   \end{proof}
   Let us recall Lambert's characterization of
subnormality (cf.\ \cite{Lam}; see also
\cite[Theorem 7]{StSz} for the general, not
necessarily injective, case):\ an operator $T \in
\ogr \hh$ is subnormal if and only if
$\{\|T^{n}f\|^2\}_{n=0}^\infty$ is a Stieltjes
moment sequence for all $f\in \hh$. Since
$T^{n+1}f=T^{n}(Tf)$, we infer from the Stieltjes
theorem that
   \begin{align}  \label{lamsub}
   \begin{minipage}{28em}
an operator $T\in \ogr \hh$ is subnormal if and only
if $\{\|T^{n}f\|^2\}_{n=0}^\infty$ is a Hamburger
moment sequence for all $f\in \hh$.
   \end{minipage}
   \end{align}
Lambert's theorem enables us to write a complete
characterization of subnormality of weighted shifts on
directed trees.
   \begin{thm} \label{charsub}
If $\slam \in \ogr {\ell^2(V)}$ is a weighted shift on
a directed tree $\tcal$ with weights $\lambdab =
\{\lambda_v\}_{v \in V^\circ}$, then the following
conditions are equivalent{\em :}
   \begin{enumerate}
   \item[(i)] $\slam$ is subnormal,
   \item[(ii)] $\big\{\sum\limits_{v \in \dzin{n}{u}}
|\lambda_{u\mid v}|^2\big\}_{n=0}^\infty$ is a
Stieltjes moment sequence for every $u \in V,$
   \item[(iii)] $\{\|\slam^n e_u\|^2\}_{n=0}^\infty$ is a
Stieltjes moment sequence for every $u \in V,$
   \item[(iv)] $\{\|\slam^n e_u\|^2\}_{n=0}^\infty$ is a Hamburger
moment sequence for every $u \in V,$
   \item[(v)]
$\sum_{k,l=0}^n \|\slam^{k+l} e_u\|^2 \alpha_k
\overline{\alpha_l} \Ge 0$ for all $\alpha_0,\ldots,
\alpha_n \in \cbb$, $n \in \zbb_+$ and $u \in V.$
   \end{enumerate}
   \end{thm}
   \begin{proof}
(i)$\Rightarrow$(ii) Employ the Lambert theorem
(or rather its easier part) and Proposition
\ref{pot}\,(ii).

(ii)$\Rightarrow$(iii) Apply Lemma \ref{pot}\,(ii).

(iii)$\Rightarrow$(iv) Evident.

(iv)$\Rightarrow$(i) Take $f \in \ell^2(V)$. An
induction argument shows that for a fixed integer $n
\Ge 0$, the sets $\dzin{n} u$, $u \in V$, are pairwise
disjoint. By Lemma \ref{pot}\,(i), this implies that
   \begin{align}  \label{sumo}
\|\slam^n f\|^2 = \Big\|\underset{u\in
V\hspace{1ex}}{\sumo} f(u) \slam^n e_u\Big\|^2 =
\sum_{u \in V} |f(u) |^2\|\slam^n e_u \|^2, \quad n
\in \zbb_+.
   \end{align}
Since the sequence $\{\|\slam^n
e_u\|^2\}_{n=0}^\infty$ is positive definite, we can
easily infer from \eqref{sumo} that the sequence
$\{\|\slam^n f\|^2\}_{n=0}^\infty$ is positive
definite as well. Applying the Hamburger theorem and
\eqref{lamsub}, we get the subnormality of $\slam$.

The equivalence (iv)$\Leftrightarrow$(v) is a direct
consequence of the Hamburger theorem. This completes
the proof.
   \end{proof}
One of the consequences of Theorem \ref{charsub} is
that the study of subnormality of weighted shifts on
directed trees reduces to the case of trees with root.
   \begin{cor} \label{subcyc}
Let $\slam \in \ogr {\ell^2(V)}$ be a weighted shift
on a directed tree $\tcal$ with weights $\lambdab =
\{\lambda_v\}_{v \in V^\circ}$. Suppose that $X$ is a
subset of $V$ such that $V=\bigcup_{x\in X} \des x$.
Then $\slam$ is subnormal if and only if the operator
$\slamr x$ is subnormal for every $x \in X$ $($cf.\
Notation {\em \ref{poddrz}}$)$.
   \end{cor}
   \begin{proof}
Note that by \eqref{dziinv} the space $\ell^2(\des u)$ is
invariant for $\slam$ and
   \begin{align*}
   \slamr u = \slam|_{\ell^2(\des u)}.
   \end{align*}
Hence, by Theorem \ref{charsub}, $\slam$ is subnormal if
and only if $\slamr u$ is subnormal for every $u \in V$.
Since $\slamr u \subseteq \slamr x$ whenever $u \in \des
x$, the conclusion follows from the above characterization
of subnormality and the equality $V=\bigcup_{x\in X} \des
x$.
   \end{proof}
It turns out that in some instances the condition (iii) of
Theorem \ref{charsub} can be essentially weakened without
spoiling the equivalence (i)$\Leftrightarrow$(iii). This
effect is similar to that appearing in the case of
classical weighted shifts. The result which follows will be
referred to as the {\em small} lemma (see also Lemma
\ref{charsub2}).
   \begin{lem} \label{charsub-1}
Let $\slam \in \ogr {\ell^2(V)}$ be a weighted shift
on a directed tree $\tcal$ with weights $\lambdab =
\{\lambda_v\}_{v \in V^\circ}$ and let $u_0, u_1 \in
V$ be such that $\dzi{u_0} = \{u_1\}$. If $\{\|\slam^n
e_{u_0}\|^2\}_{n=0}^\infty$ is a Stieltjes moment
sequence and $\lambda_{u_1}\neq 0$, then $\{\|\slam^n
e_{u_1}\|^2\}_{n=0}^\infty$ is a Stieltjes moment
sequence.
   \end{lem}
   \begin{proof} Observing that
   \begin{align*}
\|\slam^n e_{u_1}\|^2 \overset{\eqref{eu}}= \frac 1
{|\lambda_{u_1}|^2}\|\slam^{n+1} e_{u_0}\|^2, \quad n
\in \zbb_+,
   \end{align*}
and applying \eqref{st+1}, we complete the proof.
   \end{proof}
Note that Lemma \ref{charsub-1} is no longer true if
$\card{\dzi{u_0}} \Ge 2$.
   \begin{exa} \label{2nitki}
Let $\slam$ be a weighted shift on the directed tree
$\tcal_{2,0}$ (cf.\ \eqref{varkappa}) with weights
$\{\lambda_{v}\}_{v\in V_{2,0}^\circ}$ given by
$\{\lambda_{1,j}\}_{j=1}^\infty=\{a,\frac ba, \frac
ab, 1,1, \ldots\} $ and
$\{\lambda_{2,j}\}_{j=1}^\infty=\{b,\frac ab, \frac
ba, 1,1, \ldots\}$, where $a, b \in (0,1)$ are such
that $a < b$ and $a^2 + b^2=1$. Then $\slam$ is
bounded, $\{\|\slam^ne_0\|^2\}_{n=0}^\infty = \{1,1,
\ldots\}$ is a Stieltjes moment sequence and neither
of the sequences
$\{\|\slam^ne_{1,1}\|^2\}_{n=0}^\infty = \{1, (\frac
ba)^2, 1,1, \ldots\}$ and
$\{\|\slam^ne_{2,1}\|^2\}_{n=0}^\infty = \{1, (\frac
ab)^2, 1,1, \ldots\}$ is a Stieltjes moment sequence
(consult also Proposition \ref{izometria}).
   \end{exa}
   As an immediate consequence of Theorem
\ref{charsub} and Lemma \ref{charsub-1} (see also
Remark \ref{re1-2}), we obtain the
Berger-Gellar-Wallen criterion for subnormality of
injective unilateral classical weight\-ed shifts (cf.\
\cite{g-w,h}).
   \begin{cor} \label{b-g-w}
A bounded injective unilateral classical weighted
shift with weights $\{\lambda_n\}_{n=1}^\infty$
$($with notation as in \eqref{notnew}$)$ is subnormal
if and only if the sequence \linebreak $\{1,
|\lambda_1|^2, |\lambda_1 \lambda_2|^2, |\lambda_1
\lambda_2 \lambda_3|^2, \ldots\}$ is a Stieltjes
moment sequence.
   \end{cor}
Before formulating the next corollary, we recall that
a two-sided sequence $\{t_n\}_{n=-\infty}^\infty$ of
real numbers is said to be a {\em two-sided Stieltjes
moment sequence} if there exists a positive Borel
measure $\mu$ on $(0,\infty)$ such that
   \begin{align*}
   t_{n}=\int_{(0,\infty)} s^n \D\mu(s),\quad n \in
\zbb;
   \end{align*}
$\mu$ is called a representing measure of
$\{t_n\}_{n=-\infty}^\infty$. It is easily seen that
   \begin{align}   \label{char2sid2}
   \begin{minipage}{29em}
$\{t_n\}_{n=-\infty}^\infty \subseteq \rbb$ is a
two-sided Stieltjes moment sequence if and only if
$\{t_{n+k}\}_{n=-\infty}^\infty$ is a two-sided
Stieltjes moment sequence for some (equivalently:\ for
all\/) $k\in \zbb$.
   \end{minipage}
   \end{align}
It is known that (cf.\ \cite[Theorem 6.3]{j-t-w} and
\cite[page 202]{ber})
   \begin{align} \label{char2sid}
   \begin{minipage}{29em}
$\{t_n\}_{n=-\infty}^\infty \subseteq \rbb$ is a
two-sided Stieltjes moment sequence if and only if the
sequences $\{t_{n-k}\}_{n=0}^\infty$, $k =
0,1,2,\ldots$, are positive definite.
   \end{minipage}
   \end{align}

We are now in a position to deduce an analogue of the
Berger-Gellar-Wallen criterion for subnormality of
injective bilateral classical weighted shifts (cf.\
\cite[Theorem II.6.12]{con2} and \cite[Theorem
5]{StSz}). Another proof of this fact will be given
just after Remark \ref{unbsu}.
   \begin{cor} \label{b-g-w-2}
A bounded injective bilateral classical weighted
shift $S$ with weights $\{\lambda_n\}_{n \in
\zbb}$ $($with notation as in \eqref{notnew}$)$
is subnormal if and only if the two-sided
sequence $\{t_n\}_{n=-\infty}^\infty$ defined by
   \begin{align} \label{twowe}
t_n =
   \begin{cases}
   |\lambda_1 \cdots \lambda_{n}|^2 & \text{ for } n
\Ge 1,
   \\
   1 & \text{ for } n=0,
   \\
   |\lambda_{n+1} \cdots \lambda_{0}|^{-2} & \text{
for } n \Le -1,
   \end{cases}
   \end{align}
is a two-sided Stieltjes moment sequence.
   \end{cor}
   \begin{proof}
By Lemma \ref{charsub-1} and Theorem
\ref{charsub}\,(iii), $S$ is subnormal if and only if
$\{\|S^n (S^{-k}e_0)\|^2\}_{n=0}^\infty$ is a
Stieltjes moment sequence for every $k\in \zbb_+$.
This, when combined with \eqref{char2sid}, completes
the proof.
   \end{proof}
Taking into account \eqref{char2sid}, we can rephrase
Corollary \ref{b-g-w-2} as follows:\ $S$ is subnormal
if and only if $\{\ldots, |\lambda_{-1}
\lambda_{0}|^{-2}, |\lambda_{0}|^{-2},1,
|\lambda_1|^2, |\lambda_1 \lambda_2|^2, \ldots\}$ is a
two-sided Stieltjes moment sequence no matter which
position is chosen as the zero one.

Another question worth exploring is to find relationships
between representing measures of Stieltjes moment sequences
$\{\|\slam^n e_u\|^2\}_{n=0}^\infty$, $u \in V$. We begin
by fixing notation.
   \begin{ozn} \label{ozn2}
Let $\slam \in \ogr {\ell^2(V)}$ be a weighted shift
on a directed tree $\tcal$. If for some $u \in V$,
$\{\|\slam^n e_u\|^2\}_{n=0}^\infty$ is a Stieltjes
moment sequence, then, in view of \eqref{determ}, it
is determinate and its unique representing measure is
concentrated on $[0,\|\slam\|^2]$. Denote this measure
by \idxx{$\mu_u$, $\mu_u^{\tcal}$}{70} $\mu_u$ (or by
$\mu_u^{\tcal}$ if we wish to make clear the
dependence of $\mu_u$ on $\tcal$).
   \end{ozn}
The result which follows will be referred to as the {\em
big} lemma (as opposed to Lemma \ref{charsub-1} which is
called the small lemma).
   \begin{lem} \label{charsub2}
Let $\slam \in \ogr {\ell^2(V)}$ be a weighted shift on a
directed tree $\tcal$ with weights $\lambdab =
\{\lambda_v\}_{v \in V^\circ}$, and let $u \in V^\prime$ be
such that $\{\|\slam^n e_v\|^2\}_{n=0}^\infty$ is a
Stieltjes moment sequence for every $v \in \dzi u$. Then
the following conditions are equivalent\/\footnote{\;In
\eqref{consist}, we adhere to the standard convention that
$0 \cdot \infty = 0$; see also footnote \ref{foot}.}{\em :}
   \begin{enumerate}
   \item[(i)] $\{\|\slam^n e_u\|^2\}_{n=0}^\infty$
is a Stieltjes moment sequence,
   \item[(ii)]  $\slam$ satisfies
the consistency condition at $u$, i.e.,
   \begin{align} \label{consist}
\sum_{v \in \dzi{u}} |\lambda_v|^2 \int_0^\infty \frac 1
s\, \D \mu_v(s) \Le 1.
   \end{align}
   \end{enumerate}
If {\em (i)} holds, then $\mu_v(\{0\})=0$ for every $v
\in \dzi u$ such that $\lambda_v \neq 0$, and the
representing measure $\mu_u$ of $\big\{\|\slam^n
e_u\|^2\big\}_{n=0}^\infty$ is given by
   \begin{align}    \label{muu}
\mu_u(\sigma) = \sum_{v \in \dzi u} |\lambda_v|^2
\int_\sigma \frac 1 s \D \mu_v(s) + \Big(1 - \sum_{v \in
\dzi u} |\lambda_v|^2 \int_0^\infty \frac 1 s \D
\mu_v(s)\Big) \delta_0(\sigma)
   \end{align}
for $\sigma \in \borel{\rbb}$; moreover, $\mu_u(\{0\})=0$
if and only if $\slam$ satisfies the strong consistency
condition at $u$, i.e.,
   \begin{align} \label{consist'}
\sum_{v \in \dzi{u}} |\lambda_v|^2 \int_0^\infty \frac 1
s\, \D \mu_v(s) = 1.
   \end{align}
   \end{lem}
   \begin{proof}
Define the set function $\mu$ on Borel subsets of
$\rbb$ by
   \begin{align*}
\mu(\sigma) = \sum_{v \in \dzi u} |\lambda_v|^2
\mu_v(\sigma), \quad \sigma \in \borel{\rbb}.
   \end{align*}
Then $\mu$ is a positive Borel measure concentrated on
$[0,\infty)$, and\/\footnote{\;Apply the Lebesgue
monotone convergence theorem to measures $\mu$,
$\mu_v$ and to the counting measure on $\dzi u$; note
also that the cardinality of $\dzi u$ may be larger
than $\aleph_0$.}
   \begin{align}  \label{leb2}
\int_0^\infty f \D \mu = \sum_{v \in \dzi u}
|\lambda_v|^2 \int_0^\infty f \D \mu_v
   \end{align}
for every Borel function $f\colon {[0,\infty)} \to
[0,\infty]$. In particular, we have
   \begin{align} \label{1/t}
\int_0^\infty \frac 1 s\, \D \mu (s) = \sum_{v \in
\dzi u} |\lambda_v|^2 \int_0^\infty \frac 1 s\, \D
\mu_v(s).
   \end{align}
   Combining \eqref{n+1} with the fact that the sets
$\dzin{n} u$, $u \in V$, are pairwise disjoint for
every fixed integer $n \Ge 0$, we deduce that
   \begin{align} \label{dzinn2}
\dzin{n+1}{u} = \bigsqcup_{v \in \dzi{u}} \dzin{n}{v}.
   \end{align}
   Employing twice Lemma \ref{pot}\,(ii), we get
   \allowdisplaybreaks
   \begin{align}     \label{sln+1}
\|\slam^{n+1} e_u\|^2 & \hspace{2.2ex}= \sum_{w \in
\dzin{n+1}{u}} |\lambda_{u\mid w}|^2
   \\
& \overset{\eqref{dzinn2}}= \sum_{v \in \dzi u}
\sum_{w \in \dzin{n} v} |\lambda_{u\mid w}|^2 \notag
   \\
& \hspace{.4ex} \overset{\eqref{recfor2}}= \sum_{v \in
\dzi u} |\lambda_v|^2 \sum_{w \in \dzin{n} v}
|\lambda_{v\mid w}|^2 \notag
   \\
& \hspace{2.2ex} =\sum_{v \in \dzi u} |\lambda_v|^2
\|\slam^n e_v\|^2, \quad n \in \zbb_+. \notag
   \end{align}
This implies that
   \begin{align*}
\|\slam^{n+1} e_u\|^2 = \sum_{v \in \dzi u}
|\lambda_v|^2 \int_0^\infty s^n \, \D \mu_v(s)
\overset{\eqref{leb2}}= \int_0^\infty s^n \D \mu(s),
\quad n \in \zbb_+,
   \end{align*}
   which means that $\{\|\slam^{n+1}
e_u\|^2\}_{n=0}^\infty$ is a Stieltjes moment sequence with
a representing measure $\mu$. Since $\limsup_{n\to \infty}
(\|\slam^{n+1} e_u\|^2)^{1/n} \Le \|\slam\|^2$, we deduce
from \eqref{determ} that $\{\|\slam^{n+1}
e_u\|^2\}_{n=0}^\infty$ is a determinate Hamburger moment
sequence, and consequently, $\mu$ is its unique
representing measure. Employing now the equality
\eqref{1/t} and Lemma \ref{bext} with $t_n = \|\slam^{n+1}
e_u\|^2$, we see that the conditions (i) and (ii) are
equivalent. The formula \eqref{muu} can easily be inferred
from \eqref{nu} by applying \eqref{leb2} (consult Notation
\ref{ozn2}). The remaining part of conclusion is now
obvious.
   \end{proof}
   \begin{rem} \label{unbsu}
A thorough inspection of the proof of Lemma
\ref{charsub2} reveals that the implication
(ii)$\Rightarrow$(i) can be justified without recourse
to the determinacy of Stieltjes moment sequences. In
particular, the formula \eqref{muu} gives a
representing measure of $\big\{\|\slam^n
e_u\|^2\big\}_{n=0}^\infty$ provided $\mu_v$ is a
representing measure of $\big\{\|\slam^n
e_v\|^2\big\}_{n=0}^\infty$ concentrated on
$[0,\infty)$ for every $v \in \dzi u$. This
observation seems to be of potential relevance because
it can be used to produce examples of unbounded
weighted shifts on directed trees. However, the proof
of the implication (i)$\Rightarrow$(ii) requires using
the determinacy of the sequence $\{\|\slam^{n+1}
e_u\|^2\}_{n=0}^\infty$.
   \end{rem}
Lemma \ref{charsub2} turns out to be a useful
tool for verifying subnormality of weighted
shifts on directed trees. First, we apply it to
give another proof (without recourse to
\eqref{char2sid}) of the Berger-Gellar-Wallen
criterion for subnormality of injective bilateral
classical weight\-ed shifts.
   \begin{proof}[Second proof of Corollary \ref{b-g-w-2}]
Suppose first that $S$ is subnormal. Applying Lemma
\ref{charsub2} to $u=e_{k-1}$, $k \in \zbb$, we deduce
that $\mu_k({\{0\}})=0$ for all $k \in \zbb$. As a
consequence, we see that the inequality
\eqref{consist} turns into the equality
\eqref{consist'}, and the second term of the
right-hand side of the equality \eqref{muu} vanishes.
This, when applied to $u=e_{-1}$, leads to $\frac 1
{|\lambda_{0}|^2} = \int_0^\infty \frac 1 s \, \D
\mu_0(s)$ and $\D \mu_{-1}(s) = \frac{|\lambda_0|^2}s
\D \mu_0(s)$ (be aware of \eqref{notnew}). Employing
an induction argument, we show that for every $k \in
\zbb_+$,
   \begin{align} \label{mu-k}
\frac 1 {|\lambda_{-k} \cdots \lambda_{0}|^2}
= \int_0^\infty \frac 1 {s^{k+1}} \, \D \mu_0
(s) \text{ and } \D \mu_{-k-1}(s) = \frac
{|\lambda_{-k} \cdots \lambda_{0}|^2}
{s^{k+1}} \D \mu_0(s),
   \end{align}
which completes the proof of the ``only if''
part of the conclusion.

Reversely, if $\{t_n\}_{n=-\infty}^\infty$ in \eqref{twowe}
is a two-sided Stieltjes moment sequence with a
representing measure $\mu_0$, then the equality $\|S^n
e_{-k}\|^2=|\lambda_{-k+1} \cdots \lambda_0|^2 t_{n-k}$
which holds for all $k \in \nbb$ and $n \in \zbb_+$ implies
that for every $k \in \nbb$, the sequence $\{\|S^n
e_{-k}\|^2\}_{n=0}^\infty$ is a Stieltjes moment sequence
with a representing measure $\mu_{-k}$ defined in
\eqref{mu-k}. By Lemma \ref{charsub-1}, $\{\|S^n
e_{k}\|^2\}_{n=0}^\infty$ is a Stieltjes moment sequence
for every $k\in \zbb_+$. This, together with Theorem
\ref{charsub}, completes the proof of the ``if'' part.
   \end{proof}
The ensuing proposition which concerns subnormal
extendibility of weights will be illustrated in
Example \ref{extend} in the context of directed
trees $\tcal_{\eta,\kappa}$.
   \begin{pro} \label{maxsub}
Let $\tcal = (V,E)$ be a subtree of a directed tree
$\hat\tcal=(\hat V,\hat E)$ such that $\dzit{\tcal}{w}
\neq \dzit{\hat\tcal}{w}$ for some $w \in V \setminus
\Ko{\tcal},$ and $\dest{\tcal} v = \dest{\hat\tcal} v$
for all $v \in \dzit{\tcal}w$. Suppose that $\slam \in
\ogr{\ell^2(V)}$ is a subnormal weighted shift on the
directed tree $\tcal$ with nonzero weights $\lambdab =
\{\lambda_u\}_{u \in V^\circ}$. Then the directed tree
$\tcal$ is leafless and there exists no subnormal
weighted shift $\slamh \in \ogr {\ell^2(\hat V)}$ on
$\hat \tcal$ with nonzero weights $\hat \lambdab =
\{\hat \lambda_u\}_{u \in \hat V^\circ}$ such that
\idx{$\lambdab \subseteq \hat \lambdab$}{71}, i.e.,
$\lambda_u = \hat \lambda_u$ for all $u \in V^\circ$.
   \end{pro}
Note that the directed tree $\hat\tcal$ in Proposition
\ref{maxsub} may not be leafless.
   \begin{proof}[Proof of Proposition  \ref{maxsub}]
Suppose that, contrary to our claim, such an $\slamh$
exists. It follows from Proposition \ref{hypcor} that
$\tcal$ and $\hat\tcal$ are leafless. Hence
$\dzit{\tcal}{w} \neq \varnothing$. Applying Lemma
\ref{charsub2} to $\slam$ and $u=\pa{w}$, we see that
$\mu_w^\tcal(\{0\})=0$. Next, applying Lemma \ref{charsub2}
to $\slam$ and $u=w$, we get
   \begin{align} \label{jedyneczka}
1=\sum_{v \in \dzit{\tcal}{w}} |\lambda_v|^2 \int_0^\infty
\frac 1 s\, \D \mu_v^{\tcal}(s) \quad \text{(see Notation
\ref{ozn2}).}
   \end{align}
The same is true for $\slamh$. Hence, we have
   \allowdisplaybreaks
   \begin{multline*}
1 = \sum_{v \in \dzit{\hat\tcal}{w}} |\hat\lambda_v|^2
\int_0^\infty \frac 1 s\, \D \mu_v^{\hat\tcal}(s)
   \\
\overset{(\star)}= \sum_{v \in \dzit{\tcal}{w}}
|\lambda_v|^2 \int_0^\infty \frac 1 s\, \D
\mu_v^{\tcal}(s) + \sum_{v \in \dzit{\hat\tcal}{w}
\setminus \dzit{\tcal}{w}} |\hat\lambda_v|^2
\int_0^\infty \frac 1 s\, \D \mu_v^{\hat\tcal}(s)
   \\
\overset{\eqref{jedyneczka}}= 1 + \sum_{v \in
\dzit{\hat\tcal}{w} \setminus
\dzit{\tcal}{w}}|\hat\lambda_v|^2 \int_0^\infty \frac 1 s\,
\D \mu_v^{\hat\tcal}(s),
   \end{multline*}
which implies that $\hat\lambda_v=0$ for all $v \in
\dzit{\hat\tcal}{w} \setminus \dzit{\tcal}{w} \neq
\varnothing$, a contradiction. The equality ($\star$)
follows from $\lambdab \subseteq \hat \lambdab$ and
the fact that $\mu_v^{\tcal} = \mu_v^{\hat\tcal}$ for
all $v \in \dzit{\tcal}w$, the latter being a direct
consequence of the equality
$\slam|_{\ell^2(\dest{\tcal}v)} =
\slamh|_{\ell^2(\dest{\hat\tcal}v)}$ which holds for
all $v \in \dzit{\tcal}w$. This completes the proof.
   \end{proof}
Note that Proposition \ref{maxsub} is no longer true when
$w=\ko{\tcal}$ (cf.\ Example \ref{extend}).
   \subsection{\label{subnkappa}Subnormality on assorted
directed trees}
   Classical weighted shifts are built on very special
directed trees which are characterized by the property
that each vertex has exactly one child (cf.\ Remark
\ref{re1-2}). In this section, we consider one step
more complicated directed trees, namely those with the
property that each vertex except one has exactly one
child; the exceptional vertex is assumed to be a
branching vertex (cf. \eqref{prec}).

Below we adhere to Notation \ref{poddrz}. Set
\idxx{$J_\iota$}{72} $J_\iota = \{k \in \nbb\colon
k\Le \iota\}$ for $\iota \in \zbb_+ \sqcup
\{\infty\}$. Note that $J_0=\varnothing$ and $J_\infty
= \nbb$.
   \begin{thm}\label{omega}
Suppose that $\tcal$ is a directed tree for which
there exists $\omega\in V$ such that
$\card{\dzi{\omega}}\Ge 2$ and $\card{\dzi{v}}=1$ for
every $v \in V \setminus \{\omega\}$. Let $\slam
\in\ogr{\ell^2(V)}$ be a weighted shift on the
directed tree $\tcal$ with nonzero weights $\lambdab =
\{\lambda_v\}_{v \in V^\circ}$. Then the following
assertions are valid.
   \begin{enumerate}
   \item[(i)] If $\omega\in \Ko{\tcal}$, then
$\slam$ is subnormal if and only if \eqref{consist} holds
for $u=\omega$ and $\{\|\slam^n e_v\|^2\}_{n=0}^\infty$ is
a Stieltjes moment sequence for every $v \in \dzi
{\omega}$.
   \item[(ii)] If $\tcal$ has a root and $\omega \neq \koo$,
then $\slam$ is subnormal if and only if any one of the
following two equivalent conditions holds\/{\em :}
   \begin{enumerate}
   \item[(\mbox{ii-a})] $\slamr
\omega$ is subnormal, \eqref{consist'} is valid for
$u=\omega$,
   \begin{align}  \label{znumerkiem}
\frac 1 {|\prod_{j=0}^{k-1}
\lambda_{\paa^j(\omega)}|^2} = \sum_{v \in \dzi
\omega} |\lambda_v|^2 \int_0^\infty \frac 1 {s^{k+1}}
\D \mu_v(s)
   \end{align}
for all $k\in J_{\kappa-1}$, and
   \begin{align}  \label{znumerkiem'}
\frac 1 {|\prod_{j=0}^{\kappa-1}
\lambda_{\paa^j(\omega)}|^2} \Ge \sum_{v \in \dzi \omega}
|\lambda_v|^2 \int_0^\infty \frac 1 {s^{\kappa+1}} \D
\mu_v(s),
   \end{align}
where $\kappa$ is a unique positive integer such that
$\paa^\kappa(\omega) = \koo${\em ;}
   \item[(\mbox{ii-b})] $\{\|\slam^n
e_{\koo}\|^2\}_{n=0}^\infty$ and $\{\|\slam^n
e_v\|^2\}_{n=0}^\infty$ are Stieltjes moment sequences
for all $v \in \dzi {\omega}$.
   \end{enumerate}
   \item[(iii)] If $\tcal$ is rootless,
then $\slam$ is subnormal if and only if any one of the
following two equivalent conditions holds\/{\em :}
   \begin{enumerate}
   \item[(\mbox{iii-a})] $\slamr
\omega$ is subnormal, \eqref{consist'} is valid for
$u=\omega$, and \eqref{znumerkiem} is valid for all $k\in
\nbb$,
   \item[(\mbox{iii-b})] $\{\|\slam^n e_{\paa^k(\omega)}
\|^2\}_{n=0}^\infty$ and $\{\|\slam^n
e_v\|^2\}_{n=0}^\infty$ are Stieltjes moment
sequenc\-es for infinitely many integers $k \Ge 1$ and
for all $v \in \dzi {\omega}$.
   \end{enumerate}
   \end{enumerate}
   \end{thm}
   \begin{proof}
(i) To prove the ``if'' part of the conclusion of (i),
denote by $\mathscr S$ the set of all $v \in V$ such
that $\big\{\|\slam^n e_v\|^2\big\}_{n=0}^\infty$ is a
Stieltjes moment sequence. Then, by Lemma
\ref{charsub2}, $\{\omega\} \cup \dzi{\omega}
\subseteq \mathscr S$. In turn, by Corollary
\ref{przem} and Lemma \ref{charsub-1}, we see that $V
\setminus (\{\omega\} \cup \dzi{\omega}) \subseteq
\mathscr S$, which together with Theorem \ref{charsub}
completes the proof of the ``if'' part. The ``only
if'' part is a direct consequence of Theorem
\ref{charsub} and Lemma \ref{charsub2}.

(ii) It follows from our assumptions on $\tcal$ and
Proposition \ref{xdescor2} that there exists a unique
$\kappa \in \nbb$ satisfying the equality
$\paa^\kappa(u)=\koo$, and that
   \begin{align} \label{vplus}
V = \{\paa^{j}(\omega) \colon j=1, \ldots, \kappa\} \sqcup
\des \omega.
   \end{align}
We first prove the ``only if'' part of the conclusion of
(ii). For this, suppose that $\slam$ is subnormal. Then
$\slamr \omega$ is subnormal as a restriction of $\slam$ to
its invariant subspace $\ell^2(\des \omega)$. Applying
Lemma \ref{charsub2} to $u=\paa^k(\omega)$, $k\in
J_\kappa$, we deduce that $\mu_{\paa^k(\omega)}(\{0\})=0$
for all $k = 0, \ldots, \kappa-1$. This, when combined with
Lemma \ref{charsub2}, applied to $u=\paa^k(\omega)$ with
$k=0, \ldots, \kappa$, leads to
   \begin{align} \label{ABC-A}
& \sum_{v \in \dzi{\omega}} |\lambda_v|^2 \int_0^\infty
\frac 1 s\, \D \mu_v(s) = 1,
   \\ \label{ABC-B}
& |\lambda_{\paa^k(\omega)}|^2 \int_0^\infty \frac 1 s\, \D
\mu_{\paa^k(\omega)}(s) = 1, \quad k \in \zbb,\, 0\Le k \Le
\kappa -2,
   \\ \label{ABC-C}
& |\lambda_{\paa^{\kappa-1}(\omega)}|^2 \int_0^\infty \frac
1 s\, \D \mu_{\paa^{\kappa-1}(\omega)}(s) \Le 1,
   \\ \label{ABC-A'}
& \mu_\omega(\sigma) = \sum_{v \in \dzi \omega}
|\lambda_v|^2 \int_\sigma \frac 1 s \D \mu_v(s), \quad
\sigma \in \borel \rbb,
   \\ \label{ABC-B'}
& \mu_{\paa^k(\omega)}(\sigma) =
|\lambda_{\paa^{k-1}(\omega)}|^2 \int_\sigma \frac 1 s
\D \mu_{\paa^{k-1}(\omega)}(s), \quad \sigma \in
\borel \rbb, \, k \in J_{\kappa-1}.
   \end{align}
Using an induction argument, we deduce from
\eqref{ABC-B}, \eqref{ABC-A'} and \eqref{ABC-B'} that
\eqref{znumerkiem} holds for every $k \in
J_{\kappa-1}$, and that the measures
$\mu_{\paa^k(\omega)}$, $k \in J_{\kappa-1}$, are
given by
   \begin{align}\label{literki}
\frac {\mu_{\paa^k(\omega)} (\sigma)} {|\prod_{j=0}^{k-1}
\lambda_{\paa^j(\omega)}|^2} = \sum_{v \in \dzi \omega}
|\lambda_v|^2 \int_{\sigma} \frac 1 {s^{k+1}}\, \D
\mu_v(s), \quad \sigma \in \borel{\rbb}.
   \end{align}
(To show that \eqref{znumerkiem} holds for $k+1$ in
place of $k$, we have to employ the formula
\eqref{literki}.) Next, we infer \eqref{znumerkiem'}
from \eqref{ABC-C}, \eqref{ABC-A'} (if $\kappa=1$) and
\eqref{literki} (if $\kappa \Ge 2$, with
$k=\kappa-1$). This means that (\mbox{ii-a}) holds.
Clearly, the condition (\mbox{ii-b}) is a direct
consequence of Theorem \ref{charsub}.

Let us turn to the proof of the ``if'' part of the
conclusion of (ii). Assume first that (\mbox{ii-a})
holds. By subnormality of $\slamr \omega =
\slam|_{\ell^2(\des \omega)}$, we have $\des \omega
\subseteq \mathscr S$ (cf.\ Theorem \ref{charsub}).
This and \eqref{ABC-A}, when combined with Lemma
\ref{charsub2}, yields \eqref{ABC-A'}. If $\kappa =1$,
then the formula \eqref{ABC-A'} for $\mu_\omega$
enables us to rewrite the inequality
\eqref{znumerkiem'} as $|\lambda_\omega|^2
\int_0^\infty \frac 1 s\, \D \mu_\omega(s) \Le 1$. By
Lemma \ref{charsub2}\,(ii), $\koo = \pa \omega \in
\mathscr S$, which together with \eqref{vplus} and
Theorem \ref{charsub} implies subnormality of $\slam$.
If $\kappa \Ge 2$, then the equality
\eqref{znumerkiem} with $k=1$ takes the form
$|\lambda_\omega|^2 \int_0^\infty \frac 1 s\, \D
\mu_\omega(s) = 1$. Applying Lemma \ref{charsub2}
again, we see that $\pa \omega \in \mathscr S$ and the
measure $\mu_{\pa \omega}$ is given by \eqref{literki}
with $k=1$. An induction argument shows that for every
$k \in J_{\kappa-1}$, $\paa^k(\omega) \in \mathscr S$
and the measure $\mu_{\paa^k(\omega)}$ is given by
\eqref{literki}. Finally, the formula \eqref{literki}
with $k=\kappa-1$ for $\mu_{\paa^{\kappa-1}(\omega)}$
enables us to rewrite the inequality
\eqref{znumerkiem'} as
$|\lambda_{\paa^{\kappa-1}(\omega)}|^2 \int_0^\infty
\frac 1 s\, \D \mu_{\paa^{\kappa-1}(\omega)}(s) \Le
1$. By Lemma \ref{charsub2}\,(ii), $\koo =
\paa^\kappa(\omega) \in \mathscr S$. This combined
with \eqref{vplus} and Theorem \ref{charsub} shows
that $\slam$ is subnormal.

Suppose now that \mbox{(ii-b)} holds. Employing Lemma
\ref{charsub-1} repeatedly first to
$u_0=\koo=\paa^{\kappa}(\omega)$, then to $u_0=\paa^{\kappa
- 1}(\omega)$ and so on up to $u_0=\paa^1(\omega)$, we see
that $\{\paa^{j}(\omega) \colon j=0, \ldots, \kappa\}
\subseteq \mathscr S$. The same procedure applied to
members of $\dzi \omega$ shows that $\des \omega \setminus
\{\omega\} \subseteq \mathscr S$, which by \eqref{vplus}
and Theorem \ref{charsub} completes the proof of (ii).

(iii) Let $J$ be an infinite subset of $\nbb$. In view
of Proposition \ref{xdescor}\,(iii), we can apply
Corollary \ref{subcyc} to $X= \{\paa^k(\omega)\colon k
\in J\}$. What we get is that $\slam$ is subnormal if
and only if $\slamr x$ is subnormal for every $x\in
X$. Since for every $x \in X$, $\slamr
x=\slam|_{\ell^2(\des{x})}$ and the directed tree
$\tcal_{\des{x}}$ has the property required in (ii),
we see that (iii) can be deduced from (ii) by applying
the aforesaid characterization of subnormality of
$\slam$. This completes the proof.
   \end{proof}
Our next aim is to rewrite Theorem \ref{omega} in
terms of weights of weighted shifts being studied. In
view of Proposition \ref{hypcor}, there is no loss of
generality in assuming that $\card{\dzi \omega} \Le
\aleph_0$. A careful look at the proof of Theorem
\ref{omega} reveals that the directed trees considered
therein can be modelled as follows (see Figure 6).
Given $\eta,\kappa \in \zbb_+ \sqcup \{\infty\}$ with
$\eta \Ge 2$, we define the directed tree
\idxx{$\tcal_{\eta,\kappa}$}{73} $\tcal_{\eta,\kappa}
= (V_{\eta,\kappa}, E_{\eta,\kappa})$ by (recall that
$J_\iota = \{k \in \nbb\colon k\Le \iota\}$ for $\iota
\in \zbb_+ \sqcup \{\infty\}$)
\idxx{$V_{\eta,\kappa}$}{74}
\idxx{$E_{\eta,\kappa}$}{75} \idxx{$E_\kappa$}{76}
   \allowdisplaybreaks
   \begin{align}  \label{varkappa}
   \begin{aligned}
V_{\eta,\kappa} & = \big\{-k\colon k\in J_\kappa\big\}
\sqcup \{0\} \sqcup \big\{(i,j)\colon i\in J_\eta,\,
j\in \nbb\big\},
   \\
E_{\eta,\kappa} & = E_\kappa \sqcup
\big\{(0,(i,1))\colon i \in J_\eta\big\} \sqcup
\big\{((i,j),(i,j+1))\colon i\in J_\eta,\, j\in
\nbb\big\},
   \\
E_\kappa & = \big\{(-k,-k+1) \colon k\in J_\kappa\big\}.
   \end{aligned}
   \end{align}
   \vspace{1.5ex}
   \begin{center}
   \includegraphics[width=7cm]
   %{Graf6.png}
   {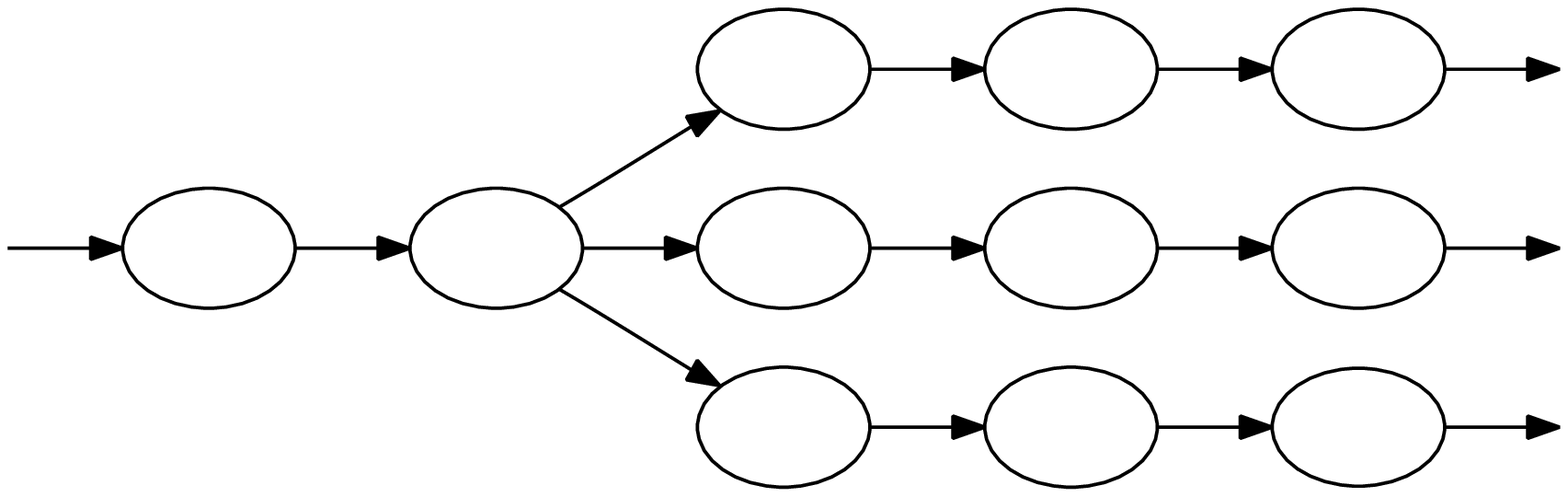}
   \\[1.5ex]
   {\small {\sf Figure 6}}
   \end{center}
   \vspace{1ex}
   If $\kappa < \infty$, then the directed tree
$\tcal_{\eta,\kappa}$ has a root and
$\ko{\tcal_{\eta,\kappa}}=-\kappa$. In turn, if
$\kappa=\infty$, then the directed tree
$\tcal_{\eta,\infty}$ is rootless. In all cases, the
branching vertex $\omega$ is equal to $0$. Note that
the simplest\footnote{\;This means that there is no
proper leafless subtree of the underlying directed
tree which is not isomorphic to $\zbb_+$.} leafless
directed tree which is not isomorphic to $\zbb_+$ and
$\zbb$ (cf.\ Remark \ref{re1-2}) coincides with
$\tcal_{2,0}$.

We are now ready to reformulate Theorem \ref{omega} in
terms of weights. Writing the counterpart of
(\mbox{iii-b}), being a little bit too long, is left to the
reader. Below we adhere to notation $\lambda_{i,j}$ instead
of a more formal expression $\lambda_{(i,j)}$.
   \begin{cor}\label{omega2}
Let $\slam \in\ogr{\ell^2(V_{\eta,\kappa})}$ be a
weighted shift on the directed tree
$\tcal_{\eta,\kappa}$ with nonzero weights $\lambdab =
\{\lambda_v\}_{v \in V_{\eta,\kappa}^\circ}$. Then the
following assertions hold.
   \begin{enumerate}
   \item[(i)] If $\kappa=0$, then
$\slam$ is subnormal if and only if there exist Borel
probability measures $\{\mu_i\}_{i=1}^\eta$ on $[0,\infty)$
such that
   \allowdisplaybreaks
   \begin{gather}     \label{zgod0}
\int_0^\infty s^n \D \mu_i(s) =
\Big|\prod_{j=2}^{n+1}\lambda_{i,j}\Big|^2, \quad n
\in \nbb, \; i \in J_\eta,
   \\
\sum_{i=1}^\eta |\lambda_{i,1}|^2 \int_0^\infty \frac 1 s\,
\D \mu_i(s) \Le 1. \label{zgod}
   \end{gather}
   \item[(ii)] If $0 < \kappa < \infty$,
then $\slam$ is subnormal if and only if one of the
following two equivalent conditions holds\/{\em :}
   \begin{enumerate}
   \item[(\mbox{ii-a})]  there exist Borel probability
measures $\{\mu_i\}_{i=1}^\eta$ on $[0,\infty)$ which
satisfy \eqref{zgod0} and the following requirements{\em :}
   \begin{align} \label{zgod'}
&\sum_{i=1}^\eta |\lambda_{i,1}|^2 \int_0^\infty \frac 1
s\, \D \mu_i(s) = 1,
   \\
& \frac 1 {|\prod_{j=0}^{k-1} \lambda_{-j}|^2} =
\sum_{i=1}^\eta|\lambda_{i,1}|^2 \int_0^\infty \frac 1
{s^{k+1}} \D \mu_i(s), \quad k \in J_{\kappa-1},
\label{widly1}
   \\
& \frac 1 {|\prod_{j=0}^{\kappa-1} \lambda_{-j}|^2} \Ge
\sum_{i=1}^\eta|\lambda_{i,1}|^2 \int_0^\infty \frac 1
{s^{\kappa+1}} \D \mu_i(s); \label{widly1'}
   \end{align}
   \item[(\mbox{ii-b})] there exist
Borel probability measures $\{\mu_i\}_{i=1}^\eta$ and
$\nu$ on $[0,\infty)$ which satisfy \eqref{zgod0} and
the equations below
   \begin{align*}
\int_0^\infty s^n \D \nu(s) =
   \begin{cases}
|\prod_{j=\kappa-n}^{\kappa-1}\lambda_{-j}|^2 &
\text{if } n \in J_\kappa,
   \\[1ex]
|\prod_{j=0}^{\kappa-1}\lambda_{-j}|^2
\big(\sum_{i=1}^\eta |\prod_{j=1}^{n-\kappa}
\lambda_{i,j}|^2\big) & \text{if } n \in \nbb
\setminus J_\kappa.
   \end{cases}
   \end{align*}
   \end{enumerate}
   \end{enumerate}
   \begin{enumerate}
   \item[(iii)] If $\kappa=\infty$,
then $\slam$ is subnormal if and only if there exist Borel
probability measures $\{\mu_i\}_{i=1}^\eta$ on $[0,\infty)$
satisfying \eqref{zgod0}, \eqref{zgod'} and \eqref{widly1}.
   \end{enumerate}
Moreover, if $\slam$ is subnormal and
$\{\mu_i\}_{i=1}^\eta$ are Borel probability
measures on $[0,\infty)$ satisfying
\eqref{zgod0}, then $\mu_i = \mu_{i,1}$ for all
$i \in J_\eta$.
   \end{cor}
   \begin{proof}  Apply Theorem \ref{omega}
(consult also \eqref{determ}).
   \end{proof}
Corollary \ref{omega2} suggests the possibility of singling
out a class of subnormal weighted shifts on
$\tcal_{\eta,\kappa}$ (with $\kappa \in \nbb$) whose
behaviour on $e_{\koo}$ is, in a sense, extreme.
   \begin{rem} \label{forwhile}
Suppose that $\kappa \in \nbb$. We say that a subnormal
weighted shift $\slam \in \ogr{\ell^2(V_{\eta,\kappa})}$ on
$\tcal_{\eta,\kappa}$ with nonzero weights $\lambdab =
\{\lambda_v\}_{v \in V_{\eta,\kappa}^\circ}$ is {\em
extremal} if
   \begin{align*}
\|\slam e_{\koo}\| = \max \|S_{\tilde \lambdab} e_{\koo}\|,
   \end{align*}
where the maximum is taken over all subnormal weighted
shifts $S_{\tilde \lambdab} \in
\ogr{\ell^2(V_{\eta,\kappa})}$ on $\tcal_{\eta,\kappa}$
with nonzero weights $\tilde \lambdab = \{\tilde
\lambda_v\}_{v \in V_{\eta,\kappa}^\circ}$ such that
$\slamr{-\kappa+1} =
S_{\tilde\lambdab_\rightarrow\!(-\kappa+1)}$, or
equivalently that $\lambda_v = \tilde \lambda_v$ for all $v
\neq -\kappa + 1$. It follows from Corollary \ref{omega2}
that a subnormal weighted shift $\slam \in
\ogr{\ell^2(V_{\eta,\kappa})}$ on $\tcal_{\eta,\kappa}$
with nonzero weights $\lambdab = \{\lambda_v\}_{v \in
V_{\eta,\kappa}^\circ}$ is extremal if and only if $\slam$
satisfies the condition (\mbox{ii-a}) with the inequality
in \eqref{widly1'} replaced by equality; in other words,
$\slam$ is extremal if and only if $\slam$ satisfies the
strong consistency condition at each vertex $u \in
V_{\eta,\kappa}$ (cf.\ \eqref{consist'}). Hence, if $\slam
\in \ogr{\ell^2(V_{\eta,\kappa})}$ is a subnormal weighted
shift on $\tcal_{\eta,\kappa}$ with nonzero weights
$\lambdab = \{\lambda_v\}_{v \in V_{\eta,\kappa}^\circ}$,
then any weighted shift $S_{\tilde \lambdab}$ on
$\tcal_{\eta,\kappa}$ with nonzero weights $\tilde \lambdab
= \{\tilde \lambda_v\}_{v \in V_{\eta,\kappa}^\circ}$
satisfying the following equalities
   \begin{align*}
|\tilde \lambda_v| =
   \begin{cases}
|\lambda_v| & \text{for } v \neq -\kappa+1,
   \\
\Big(\sum_{i=1}^\eta|\lambda_{i,1}|^2 \int_0^\infty \frac 1
{s^{\kappa+1}} \D \mu_i(s)\Big)^{-1/2} & \text{for } v =
-\kappa+1 \text{ if } \kappa =1,
   \\
\prod_{j=0}^{\kappa-2} |\lambda_{-j}|
\Big(\sum_{i=1}^\eta|\lambda_{i,1}|^2 \int_0^\infty \frac 1
{s^{\kappa+1}} \D \mu_i(s)\Big)^{-1/2} & \text{for } v =
-\kappa+1 \text{ if } \kappa > 1,
   \end{cases}
   \end{align*}
is bounded (cf.\ Proposition \ref{ogrs}), subnormal and
extremal.
   \end{rem}
If a weighted shift $\slam$ on the directed tree
$\tcal_{\eta,0}$ is an isometry, then $\slam$ is subnormal
and $\|\slam^n e_0\|^2=1$ for all $n \in \zbb_+$. It turns
out that the reverse implication holds as well even if
subnormality is relaxed into hyponormality. According to
Example \ref{2nitki} and Proposition \ref{izometria} below,
the assumption $\|\slam^n e_0\|^2=1$, $n\in \zbb_+$, by
itself does not imply the isometricity of $\slam$. This
phenomenon is quite different comparing with the case of
unilateral classical weighted shifts in which the aforesaid
assumption always implies isometricity.
   \begin{pro}\label{izometria}
If $\slam \in \ogr{\ell^2(V_{\eta,0})}$ is a weighted
shift on $\tcal_{\eta,0}$ with positive weights
$\{\lambda_v\}_{v \in V_{\eta,0}^\circ}$, then the
following conditions are equivalent\/{\em :}
   \begin{enumerate}
   \item[(i)] $\slam$ is an isometry,
   \item[(ii)]  $\slam$ is hyponormal and
$\|\slam^n e_0\|^2=1$ for all $n \in \zbb_+$,
   \item[(iii)]
$\sum_{i=1}^\eta \lambda_{i,1}^2=1$ and
$\lambda_{i,j}=1$ for all $i \in J_\eta$ and $j=2, 3,
\ldots$
   \end{enumerate}
   \end{pro}
   \begin{proof}
(i)$\Rightarrow$(ii) Obvious.

(ii)$\Rightarrow$(iii) Fix an integer $j\Ge 2$.
Consider first the case when $\lambda_{i,j} > 1$
for some $i\in J_\eta$. Since $\slam$ is
hyponormal, we see that $\lambda_{i,2} \Le
\lambda_{i,3} \Le \lambda_{i,4} \Le \ldots$ (cf.\
Theorem \ref{hyp}). Hence, we have
   \begin{align*}
(\lambda_{i,j}^2)^{n+1} \Big(\prod_{k=1}^{j-1}
\lambda_{i,k}^2\Big) \Le \prod_{k=1}^{j+n} \lambda_{i,k}^2
\Le \sum_{l=1}^\eta \prod_{k=1}^{j+n} \lambda_{l,k}^2 =
\|\slam^{j+n} e_0\|^2=1, \quad n \Ge 0,
   \end{align*}
which contradicts $\lambda_{i,j} > 1$. Thus, we must
have $\lambda_{i,j} \Le 1$ for all $i\in J_\eta$. This
in turn implies that $\lambda_{i,j} = 1$ for $i \in
J_\eta$, because if $\lambda_{i,j} < 1$ for some $i
\in J_\eta$, then
   \begin{align*}
\|\slam^{j-1}e_0\|^2 = \sum_{l=1}^\eta \prod_{k=1}^{j-1}
\lambda_{l,k}^2 < \sum_{l=1}^\eta \prod_{k=1}^{j}
\lambda_{l,k}^2 = \|\slam^{j}e_0\|^2,
   \end{align*}
which is a contradiction.

(iii)$\Rightarrow$(i) Apply Corollary \ref{chariso}.
   \end{proof}
   \begin{pro} \label{isokappa}
If $\slam \in \ogr{\ell^2(V_{\eta,\kappa})}$ is a
weighted shift on $\tcal_{\eta,\kappa}$ with positive
weights $\{\lambda_v\}_{v \in V_{\eta,\kappa}^\circ}$
and $\kappa \Ge 1$, then the following conditions are
equivalent\/{\em :}
   \begin{enumerate}
   \item[(i)] $\slam$ is an isometry,
   \item[(ii)] $\|\slam^n e_0\|^2=1$ for all $n \in
\zbb_+$, and either $\slam$ is subnormal and extremal if
$\kappa < \infty$, or $\slam$ is subnormal if $\kappa =
\infty$,
   \item[(iii)]
$\sum_{i=1}^\eta \lambda_{i,1}^2=1$, $\lambda_{i,j}=1$ for
all $i \in J_\eta$ and   $j=2, 3, \ldots$, and $\lambda_{-k}
= 1$ for all integers $k$ such that $0 \Le k < \kappa$.
   \end{enumerate}
   \end{pro}
   \begin{proof}
As in the previous proof, we verify that the only
implication which requires explanation is
(ii)$\Rightarrow$(iii). Applying Proposition
\ref{izometria} to the subnormal operator $\slamr{0}$, we
see that condition (iii) of this proposition holds. Since
evidently $\mu_0 = \delta_{1}$, we deduce from the strong
consistency condition at $u=-1$ (which is guaranteed either
by Lemma \ref{charsub2}, or by the extremality of $\slam$,
cf.\ Remark \ref{forwhile}) that $\lambda_0=1$. This in
turn implies that $\mu_{-1}=\delta_1$. Using this procedure
repeatedly completes the proof.
   \end{proof}
   \begin{exa}
In general, if $\kappa \in \nbb$, it is impossible to relax
extremal subnormality into subnormality in the condition
(ii) of Proposition \ref{isokappa} without affecting the
implication (ii)$\Rightarrow$(i). Indeed, using Proposition
\ref{isokappa}, we construct an isometric weighted shift
$\slam \in \ogr{\ell^2(V_{\eta,\kappa})}$ on
$\tcal_{\eta,\kappa}$ with positive weights
$\lambdab=\{\lambda_v\}_{v \in V_{\eta,\kappa}^\circ}$. In
particular, $\lambda_{-k} = 1$ for all integers $k$ such
that $0 \Le k < \kappa$. For a fixed $t \in (0,1)$, we
define the weighted shift $S_{\tilde\lambdab}\in
\ogr{\ell^2(V_{\eta,\kappa})}$ on $\tcal_{\eta,\kappa}$
with positive weights
$\tilde\lambdab=\{\tilde\lambda_v\}_{v \in
V_{\eta,\kappa}^\circ}$ by $\tilde \lambda_v = \lambda_v$
for every $v \neq -\kappa +1$, and $\tilde
\lambda_{-\kappa+1} =t$. Since $\delta_1$ is the
representing measure of the Stieltjes moment sequence
$\{\|\slam^n e_v\|^2\}_{n=0}^\infty$ for all $v \in
V_{\eta,\kappa}$, we infer from Corollary
\ref{omega2}\,(\mbox{ii-a}) that $S_{\tilde\lambdab}$ is
subnormal but not isometric.

It turns out that if $\kappa=\infty$, then the implication
(ii)$\Rightarrow$(i) of Proposition \ref{isokappa} is no
longer true if subnormality is replaced by hyponormality
(compare with Proposition \ref{izometria}). Indeed, using
Theorem \ref{hyp}, one can construct a non-isometric
hyponormal $\slam$ with isometric $\slamr{0}$.
   \end{exa}
   \subsection{\label{mod-sub}Modelling subnormality on
$\tcal_{\eta,\kappa}$}
   Corollary \ref{omega2} enables us to give a method
of constructing all possible bounded subnormal
weighted shifts on the directed tree
$\tcal_{\eta,\kappa}$ with nonzero weights (see
\eqref{varkappa} for the definition of
$\tcal_{\eta,\kappa}$). This is done in Procedure
\ref{twolin} below. By virtue of Theorem \ref{uni},
there is no loss of generality in assuming that
weighted shifts under consideration have positive
weights. Note also that, in view of Notation
\ref{ozn2}, the Borel probability measures
$\{\mu_i\}_{i=1}^\eta$ appearing in Corollary
\ref{omega2} are unique, concentrated on a common
finite subinterval of $[0,\infty)$ and
   \begin{align} \label{0<infty}
\int_0^\infty \frac 1 {s^k} \D \mu_i(s) < \infty, \quad
k\in J_{\kappa+1},\, i\in J_\eta.
   \end{align}
   \begin{opis} \label{twolin}
Let $\{\mu_i\}_{i=1}^\eta$ be a sequence of Borel
probability measures on $[0,\infty)$ satisfying
\eqref{0<infty} and concentrated on a common finite
subinterval of $[0,\infty)$. It is easily seen that
then
   \begin{align*} %\label{0<infty'}
0 < \int_0^\infty s^n \D \mu_i(s) < \infty, \quad n \in
\zbb, \, n \Ge -(\kappa+1),\, i\in J_\eta.
   \end{align*}
Set $M= \sup_{i\in J_\eta} \sup \, \supp(\mu_i)$, where
$\supp(\mu_i)$ stands for the closed support of the measure
$\mu_i$. Define
   \begin{align}    \label{lamij}
\lambda_{i,j} = \sqrt{\frac{\int_0^\infty s^{j-1} \D
\mu_i(s)}{\int_0^\infty s^{j-2} \D \mu_i(s)}}, \quad j
= 2,3,4, \ldots,\; i \in J_\eta.
   \end{align}
If $\kappa=0$, we take a sequence $\{\lambda_{i,1}\}_{i\in
J_\eta}$ of positive real numbers satisfying \eqref{zgod}.
Then clearly the weights $\lambdab = \{\lambda_v\}_{v \in
V_{\eta,0}^\circ}$ just defined satisfy \eqref{zgod0} and
\eqref{zgod}. If $\kappa \Ge 1$, we take a sequence
$\{\lambda_{i,1}\}_{i\in J_\eta}$ of positive real numbers
solving \eqref{zgod'} and satisfying the following
condition
   \begin{align}   \label{fine}
\sum_{i=1}^\eta \lambda_{i,1}^2 \int_0^\infty \frac 1
{s^{k}} \D \mu_i(s) < \infty, \quad k \in J_{\kappa+1}.
   \end{align}
(The question of the existence of such a sequence is
discussed in Lemma \ref{discus} below.) The remaining
weights $\{\lambda_{-k}\}_{k=0}^{\kappa-1}$ are defined by
   \begin{align}  \label{lambda-k}
\lambda_{-k} & = \sqrt{\frac{\sum_{i=1}^{\eta}
\lambda_{i,1}^2 \int_0^\infty \frac 1 {s^{k+1}} \D
\mu_i(s)}{\sum_{i=1}^\eta \lambda_{i,1}^2 \int_0^\infty
\frac 1 {s^{k+2}} \D \mu_i(s)}}, \quad k \in \zbb_+,\, 0
\Le k < \kappa -1,
   \\
\lambda_{-\kappa + 1} & = \vartheta, \text{ where } 0 <
\vartheta \Le \sqrt{\frac{\sum_{i=1}^{\eta} \lambda_{i,1}^2
\int_0^\infty \frac 1 {s^{\kappa}} \D
\mu_i(s)}{\sum_{i=1}^\eta \lambda_{i,1}^2 \int_0^\infty
\frac 1 {s^{\kappa+1}} \D \mu_i(s)}} \label{lambda-k'}.
   \end{align}
One can easily verify that the full system of weights
$\lambdab = \{\lambda_v\}_{v \in
V_{\eta,\kappa}^\circ}$ satisfies \eqref{zgod0},
\eqref{zgod'}, \eqref{widly1} and \eqref{widly1'} (of
course, if $\kappa=1$, then \eqref{widly1} and
\eqref{lambda-k} do not appear; similarly, if
$\kappa=\infty$, then \eqref{widly1'} and
\eqref{lambda-k'} do not appear). Thus we are left
with proving the boundedness of the weighted shift
$\slam$. We will also give an explicit formula for the
norm of $\slam$.

Fix $i\in J_\eta$ and set $\alpha_{i,n} :=
\int_0^\infty s^{n} \D \mu_i(s)$ for $n\in \zbb_+$.
Applying the Cauchy-Schwarz inequality in
$L^2(\mu_i)$, we get
   \begin{align*}
\alpha_{i,n+1}^2 \Le \alpha_{i,n} \cdot
\alpha_{i,n+2}, \quad n \in \zbb_+.
   \end{align*}
Hence, the sequence
$\{\frac{\alpha_{i,n+1}}{\alpha_{i,n}}\}_{n=1}^{\infty}$
is monotonically increasing. This combined with
\cite[Exercise 23, page 74]{Rud} gives
   \begin{align} \label{eq1}
\sup_{j\Ge 1} \|\slam e_{i,j}\|^2=\sup_{j\Ge 2}
\lambda_{i,j}^2 \overset{\eqref{lamij}}= \sup_{n\Ge 0}
\frac{\alpha_{i,n+1}}{\alpha_{i,n}} =
\lim_{n\to\infty} \frac{\alpha_{i,n+1}}{\alpha_{i,n}}
= \sup \, \supp(\mu_i).
   \end{align}
Since $\mu_i$, $i \in J_\eta$, are probability measures, we
deduce from \eqref{zgod} that
   \begin{align} \label{eq2}
\|\slam e_0\|^2 = \sum_{i\in J_\eta} \lambda_{i,1}^2 \Le M
\sum_{i\in J_\eta} \lambda_{i,1}^2 \int_0^M \frac 1 s \D
\mu_i(s) \Le M.
   \end{align}
Together with Proposition \ref{ogrs}, this implies
that if $\kappa = 0$, then $\slam$ is bounded and
$\|\slam\|^2 = M$. The case of $\kappa \Ge 1$ needs a
little more effort. By the Cauchy-Schwarz inequality
for integrals and series, we get
   \allowdisplaybreaks
   \begin{multline*}
\left(\sum_{i=1}^\eta \lambda_{i,1}^2 \int_0^\infty
\frac 1 {s^{k+2}} \D \mu_i(s)\right)^2
   \\
\Le \left(\sum_{i=1}^\eta \left(\lambda_{i,1}
\sqrt{\int_0^\infty \frac 1 {s^{k+1}} \D
\mu_i(s)}\,\right) \left(\lambda_{i,1}
\sqrt{\int_0^\infty \frac 1 {s^{k+3}} \D
\mu_i(s)}\,\right)\right)^2
   \\
\Le \left(\sum_{i=1}^\eta \lambda_{i,1}^2
\int_0^\infty \frac 1 {s^{k+1}} \D \mu_i(s)\right)
\left(\sum_{i=1}^\eta \lambda_{i,1}^2 \int_0^\infty
\frac 1 {s^{k+3}} \D \mu_i(s)\right), \quad k \in
\zbb_+,
   \end{multline*}
which, in view of \eqref{lambda-k} and \eqref{lambda-k'},
means that
   \begin{align} \label{eq3}
   \lambda_{-(k+1)} \Le \lambda_{-k} = \|\slam
e_{-(k+1)}\|, \quad k \in \zbb_+, \, 0 \Le k < \kappa -1.
   \end{align}
Combining \eqref{eq1}, \eqref{eq2} and \eqref{eq3}
with Proposition \ref{ogrs} and Corollary
\ref{omega2}, we see that $\slam$ is a bounded
subnormal operator and $\|\slam\|^2 = \max\{M,
\lambda_0^2\}$. Applying Theorem \ref{hyp} to $u=-1$
and then \eqref{eq2}, we deduce that $\lambda_0^2 \Le
\|\slam e_0\|^2 \Le M$. Summarizing, we have proved
that ($\kappa$ is arbitrary)
   \begin{align}   \label{normkappa}
\|\slam\|^2 = \sup_{i\in J_\eta} \sup \, \supp(\mu_i).
   \end{align}
   \end{opis}
We now discuss the question of the existence of a sequence
$\{\lambda_{i,1}\}_{i\in J_\eta}$ of positive real numbers
satisfying \eqref{zgod'} and \eqref{fine}.
  \begin{lem} \label{discus}
Let $\{\mu_i\}_{i=1}^\eta$ be a sequence of Borel
probability measures concentrated on a common finite
subinterval $[0,M]$ of $\rbb$. Then a sequence
$\{\lambda_{i,1}\}_{i\in J_\eta} \subseteq (0,\infty)$
satisfying \eqref{zgod'} and \eqref{fine} exists if and
only if one of the following two disjunctive conditions
holds\/{\em :}
   \begin{enumerate}
   \item[(i)] $\int_0^\infty \frac 1 {s^{\kappa+1}} \D \mu_i(s)
< \infty$ for all $i\in J_\eta$, provided $\kappa <
\infty$,
   \item[(ii)] $\int_0^\infty \frac 1 {s^{k}} \D \mu_i(s)
< \infty$ for all $k\in \nbb$ and $i\in J_\eta$,
provided $\kappa = \infty$.
   \end{enumerate}
   \end{lem}
   \begin{proof} Since the necessity is obvious,
we only need to consider the sufficiency.

Assume that $\kappa < \infty$. Since $\supp(\mu_i)
\subseteq [0,M]$, we see that $0 < \int_0^\infty \frac 1
{s^k} \D \mu_i(s) < \infty$ for all $k\in J_{\kappa +1}$
($i \in J_\eta$). This enables us to find a sequence
$\{\lambda_{i,1}\}_{i=1}^\eta$ of positive real numbers
such that $t_k:=\sum_{i=1}^\eta \lambda_{i,1}^2
\int_0^\infty \frac 1 {s^k} \D \mu_i(s)<\infty$ for all
$k\in J_{\kappa + 1}$. Replacing $\lambda_{i,1}$ by
$\lambda_{i,1}/\sqrt{t_1}$, we get the required sequence.

The same reasoning applies in the case of $\kappa =
\infty$; now the sequence $\{\lambda_{i,1}\}_{i=1}^\eta$
may be chosen as follows: $\lambda_{i,1}:=2^{-i}
(\max_{k\in J_{i}} \int_0^\infty \frac 1 {s^k} \D
\mu_i(s))^{-1/2}$ for all $i \in J_\eta$.
   \end{proof}
   If $\slam$ is a bounded subnormal weighted shift on
$\tcal_{\eta,\kappa}$ with positive weights, then by
Proposition \ref{desc} weights corresponding to
$\dzi{0}$ are always square summable. We now show that
for any square summable sequence $\{x_{i}\}_{i\in
J_\eta}$ of positive real numbers, there exists a
bounded subnormal weighted shift $\slam$ on
$\tcal_{\eta,\kappa}$ with positive weights $\lambdab
= \{\lambda_v\}_{v \in V_{\eta,\kappa}^\circ}$ such
that $\lambda_{i,1}=x_i$ for all $i\in J_\eta$. In
fact, $\slam$ can always be chosen to be a scalar
multiple of an isometry.
   \begin{exa} \label{sub2}
Let $\{x_i\}_{i \in J_\eta}$ be a square summable
sequence of positive real numbers. Take any sequence
$\{t_i\}_{i\in J_\eta}$ of positive real numbers such
that $0 < \inf_{i\in J_\eta} t_i$, $\sup_{i\in J_\eta}
t_i < \infty$ and $\sum_{i\in J_\eta} \frac {x_i^2}
{t_i} = 1$ (the simplest possible solution is the
constant one $t_i = \sum_{j\in J_\eta} x_j^2$).
Consider any family $\{\mu_i\}_{i\in J_\eta}$ of Borel
probability measures on $[0,\infty)$ such that $\frac
1 {t_i} = \int_0^\infty \frac 1 s \D \mu_i (s)$ for
all $i \in J_\eta$, $0 < \inf_{i \in J_\eta} \inf
\supp(\mu_i)$ and $M := \sup_{i\in J_\eta} \sup \,
\supp(\mu_i) < \infty$ (again this is always possible,
e.g., $\mu_i = \delta_{t_i}$ does the job). Applying
Procedure \ref{twolin} to $\lambda_{i,1}=x_i$, $i\in
J_\eta$, we obtain the required weighted shift $\slam$
on $\tcal_{\eta,\kappa}$; without loss of generality
we can also assume that $\slam$ is extremal in the
case when $\kappa \in \nbb$. Moreover, by
\eqref{normkappa}, we have $\|\slam\|^2= M$. In
particular, if $\mu_i = \delta_{t_i}$ for all $i\in
J_\eta$, then $\|\slam\|^2 = \sup_{i\in J_\eta} t_i$.
If additionally $t_i = \sum_{j\in J_\eta} x_j^2$ for
all $i\in J_\eta$, then $\|\slam\|^{-1}\slam$ is an
isometry (use Propositions \ref{izometria} and
\ref{isokappa}).
   \end{exa}
Let $S$ be a bounded subnormal unilateral classical
weighted shift with positive weights
$\{\alpha_n\}_{n=1}^\infty$. Then $\{\|S^n
e_0\|^2\}_{n=0}^\infty$ (cf.\ Remark \ref{re1-2}) is a
determinate Stieltjes moment sequence, whose unique
representing measure is called a {\em Berger measure} of
$S$. Given an integer $k\Ge 1$, we say that $S$ has a
\textit{subnormal $k$-step backward extension} if for some
positive scalars $x_{1}, \ldots, x_{k}$ the unilateral
classical weighted shift with weights $\{x_{k}, \ldots,
x_{1},\alpha_1, \alpha_2, \dots\}$ is subnormal. If $S$ has
a subnormal $k$-step backward extension for all $k\in
\nbb$, then $S$ is said to have a {\em subnormal
$\infty$-step backward extension}. The following elegant
characterization of subnormal $k$-step backward
extendibility is to be found in \cite[Corollary 6.2]{CL}
(see also \cite[Proposition 8]{cur} for the case $k=1$).
   \begin{align} \label{Lee-Cu}
   \begin{minipage}{30em}
If $k \in \nbb \cup \{\infty\}$, then $S$ has a subnormal
$k$-step backward extension if and only if $\int_0^\infty
\frac{1}{s^n} \D \mu(s) < \infty$ for all $n\in J_k$, where
$\mu$ is the Berger measure of $S$.
   \end{minipage}
   \end{align}
Note that according to Notation \ref{ozn2}, we have
$\mu=\mu_0$. Let us mention that there are subnormal
unilateral classical weighted shifts which have no
subnormal backward extensions. The famous Bergman
shift with weights
$\big\{\sqrt{\frac{n}{n+1}}\big\}_{n=1}^\infty$ is
among them. On the other hand, the Hardy shift with
weights $\{1,1,1, \ldots\}$ does have a subnormal
$\infty$-step backward extension (see
\cite{cur,HJL,CL} for more examples; see also
\cite{JLS} for subnormal backward extensions of
general subnormal operators and \cite{C-L-Y,C-Y,C-Y0}
for subnormal backward extensions of two-variable
weighted shifts). We now relate the subnormality of
weighted shifts on the directed tree
$\tcal_{\eta,\kappa}$ to that of unilateral classical
weighted shifts which have subnormal $(\kappa+1)$-step
backward extensions.
   \begin{pro}\label{2curto}
Let $\kappa, \eta \in \zbb_+ \sqcup \{\infty\}$ and let
$\eta \Ge 2$. If for every $i\in J_\eta$, $S_i$ is a
bounded unilateral classical weighted shift with positive
weights $\{\alpha_{i,n}\}_{n=1}^\infty$, then the following
two conditions are equivalent\/{\em :}
   \begin{enumerate}
   \item[(i)] there exists a system $\lambdab =
\{\lambda_{v}\}_{v \in V_{\eta,\kappa}^\circ}$ of
positive scalars such that the weighted shift $\slam$
on the directed tree $\tcal_{\eta,\kappa}$ is bounded
and subnormal, and
   \begin{align} \label{alla}
\alpha_{i,n}=\lambda_{i,n+1}, \quad n\in \nbb, \, i\in
J_\eta,
   \end{align}
   \item[(ii)]
$S_i$ has a subnormal $(\kappa+1)$-step backward
extension for every $i\in J_\eta$, and
$\sup_{i\in J_\eta} \|S_i\| < \infty$.
   \end{enumerate}
Moreover, if $\slam$ is as in {\em (i)}, then
$\|\slam\|=\sup_{i\in J_\eta} \|S_i\|$.
   \end{pro}
   \begin{proof}
(i)$\Rightarrow$(ii) It follows from \eqref{alla} that
$\slamr{i,1} = \slam|_{\ell^2(\des{i,1})}$ is
unitarily equivalent to $S_i$ for all $i\in J_\eta$.
Hence $\sup_{i\in J_\eta} \|S_i\| \Le \|\slam\|$ and
$\|\slam^n e_{i,1}\|^2 = \|S_i^n e_0\|^2$ for all $n
\in \zbb_+$, which implies that $\mu_{i,1}$ is the
Berger measure of $S_i$. Owing to Corollary
\ref{omega2}, $\int_0^\infty \frac 1 {s^k} \D
\mu_{i,1}(s) < \infty$ for all $k \in J_{\kappa+1}$.
By \eqref{Lee-Cu}, (ii) holds.

(ii)$\Rightarrow$(i) Let $\mu_i$ be the Berger measure
of $S_i$. Since $\sup \supp(\mu_i) = \|S_i\|^2$, we
deduce that the Borel probability measures $\mu_i$,
$i\in J_\eta$, are concentrated on the common finite
interval $[0,M]$ with $M:=\sup_{j\in J_\eta}
\|S_j\|^2$. As $S_i$ has a subnormal $(\kappa+1)$-step
backward extension, we infer from \eqref{Lee-Cu} that
$\int_0^\infty \frac 1 {s^k} \D \mu_{i}(s) < \infty$
for all $k \in J_{\kappa+1}$. Applying Lemma
\ref{discus} and Procedure \ref{twolin}, we find a
bounded subnormal weighted shift $\slam$ on
$\tcal_{\eta,\kappa}$ with positive weights $\lambdab
= \{\lambda_{v}\}_{v \in V_{\eta,\kappa}^\circ}$ which
satisfy \eqref{lamij}. Hence
$\alpha_{i,n}=\lambda_{i,n+1}$ for all $n\in \nbb$ and
$i\in J_\eta$. Employing \eqref{normkappa} completes
the proof.
   \end{proof}
We now illustrate the question of extendibility
discussed in Proposition \ref{maxsub}.
   \begin{exa} \label{extend}
Fix parameters $\eta,\eta^\prime,\iota,\kappa,\kappa^\prime
\in \zbb_+ \sqcup \{\infty\}$ such that $2 \Le \eta <
\eta^\prime$ and $1 \Le \iota \Le \kappa \Le
\kappa^\prime$. It is evident that $\tcal_{\eta,\kappa}$
can be regarded as a proper subtree of the directed tree
$\tcal_{\eta^\prime,\kappa^\prime}$. Denote by
$\tcal_{1,\iota}$ the proper subtree of the directed tree
$\tcal_{\eta,\kappa}$ with $V_{1,\iota}:= \big\{-k\colon
k\in J_\iota\big\} \sqcup \{0\} \sqcup \big\{(1,j)\colon
j\in \nbb\big\}$. Clearly, the directed tree
$\tcal_{1,\iota}$ can be identified with $\zbb_+$ when
$\iota < \infty$, or with $\zbb$ if $\iota = \infty$. It is
easily seen that each pair $(\tcal,\hat \tcal) \in
\{(\tcal_{1,\iota}, \tcal_{\eta,\kappa}),
(\tcal_{\eta,\kappa}, \tcal_{\eta^\prime,\kappa^\prime})\}$
satisfies the assumptions of Proposition \ref{maxsub} with
$w=0$. As a consequence, there are no bounded subnormal
weighted shifts $\slam$ and $S_{\hat \lambdab}$ on $\tcal$
and $\hat\tcal$ with nonzero weights
$\lambdab=\{\lambda_u\}_{u \in V^\circ}$ and $\hat
\lambdab=\{\hat \lambda_u\}_{u \in \hat V^\circ}$,
respectively, such that $\lambdab \subseteq \hat \lambdab$.
Of course, in the case of the pairs $(\tcal,\hat \tcal) \in
\{(\tcal_{1,\kappa},\tcal_{1,\kappa^\prime}),
(\tcal_{\eta,\kappa},\tcal_{\eta,\kappa^\prime})\}$ there
always exist bounded subnormal weighted shifts $\slam$ and
$\slamh$ with nonzero weights such that $\lambdab \subseteq
\hat \lambdab$. Isometric weighted shifts are the simplest
examples of such operators (cf.\ Corollary \ref{chariso}).

Consider now the pair $(\tcal,\hat \tcal) =
(\tcal_{\eta,0},\tcal_{\eta^\prime,0})$. We construct
subnormal weight\-ed shifts $\slam \in
\ogr{\ell^2(V)}$ and $S_{\hat\lambdab} \in
\ogr{\ell^2(\hat V)}$ on $\tcal$ and $\hat \tcal$ with
positive weights $\lambdab=\{\lambda_u\}_{u \in
V^\circ}$ and $\hat \lambdab=\{\hat \lambda_u\}_{u \in
\hat V^\circ}$, respectively, such that $\lambdab
\subseteq \hat \lambdab$. Set $\lambda_{i,j}=\hat
\lambda_{l,j}=1$ for $i \in J_\eta$, $l \in
J_{\eta^\prime}$ and $j=2,3,\ldots$ Let $\{\hat
\lambda_{i,1}\}_{i \in J_{\eta^\prime}}$ be a system
of positive real numbers such that $\sum_{i\in
J_{\eta^\prime}} \hat\lambda_{i,1}^2 \Le 1$. Set
$\lambda_{i,1} = \hat \lambda_{i,1}$ for $i\in
J_\eta$. It follows from Proposition \ref{ogrs} that
$\slam \in \ogr{\ell^2(V)}$ and $S_{\hat\lambdab} \in
\ogr{\ell^2(\hat V)}$. Since $\mu_{i,1}^{\tcal} =
\mu_{l,1}^{\hat\tcal} = \delta_1$ for all $i\in
J_\eta$ and $l\in J_{\eta^\prime}$, we see that
$\slam$ and $S_{\hat\lambdab}$ satisfy the consistency
condition \eqref{zgod} at $u=0$. As a consequence of
Corollary \ref{omega2}, $\slam$ and $S_{\hat\lambdab}$
are bounded subnormal weighted shifts with positive
weights such that $\lambdab \subseteq \hat \lambdab$.
This shows that Proposition \ref{maxsub} is no longer
true when $w=\ko{\tcal}$.
   \end{exa}
   \newpage
   \section{\label{ch7}Complete hyperexpansivity}
   \subsection{\label{s18}A general approach}
A sequence $\{a_n\}_{n=0}^\infty \subseteq \rbb$ is
said to be {\em completely alternating} if
   \begin{align*}
\sum_{j=0}^n (-1)^j \binom n j a_{m+j} \Le 0, \quad m \in
\zbb_+, \, n \in \nbb.
   \end{align*}
As an immediate consequence of the definition, we have
   \begin{align}  \label{ca+1}
\text{if a sequence $\{a_n\}_{n=0}^\infty$ is completely
alternating, then so is $\{a_{n+1}\}_{n=0}^\infty$.}
   \end{align}
By \cite[Proposition 4.6.12]{ber} (see also \cite[Remark
1]{ath}), a sequence $\{a_n\}_{n=0}^\infty \subseteq \rbb$
is completely alternating if and only if there exists a
positive Borel measure $\tau$ on the closed interval
$[0,1]$ such that
   \begin{align} \label{repch}
a_n = a_0 + \int_0^1 (1+ \ldots + s^{n-1}) \D \tau(s),
\quad n = 1,2, \ldots
   \end{align}
(from now on, we abbreviate $\int_{[0,1]}$ to $\int_0^1$).
The measure $\tau$ is unique (cf.\ \cite[Lemma 4.1]{jab})
and finite. Call it the {\em representing measure} of
$\{a_n\}_{n=0}^\infty$.
   \begin{lem} \label{monot-ca}
If a sequence $\{a_n\}_{n=0}^\infty \subseteq \rbb$ is
completely alternating and $a_0=1$, then $a_n \Ge 1$ for
all $n\in\zbb_+$ and the corresponding sequence of
quotients $\big\{\frac{a_{n+1}}{a_{n}}\big\}_{n=0}^\infty$
is monotonically decreasing.
   \end{lem}
   \begin{proof}
Argue as in the proof of \cite[Proposition 4]{ath}.
   \end{proof}
The question of backward extendibility of completely
alternating sequences has the following solution
(compare with Lemma \ref{bext}).
   \begin{lem} \label{bext-ca}
Let $\{a_n\}_{n=0}^\infty$ be a completely alternating
sequence with the representing measure $\tau$. Set
$a_{-1}=1$. Then the following conditions are
equivalent{\em :}
   \begin{enumerate}
   \item[(i)] the sequence $\{a_{n-1}\}_{n=0}^\infty$
is completely alternating,
   \item[(ii)] $1 + \int_0^1 \frac 1 s \D \tau(s) \Le a_0$.
   \end{enumerate}
If {\em (i)} holds, then $\tau(\nul)=0$, and the positive
Borel measure $\varrho$ on $[0,1]$ defined by
   \begin{align}   \label{nu-ch}
\varrho(\sigma) = \int_\sigma \frac 1 s \D \tau(s) +
\Big(a_0 - 1 - \int_0^1 \frac 1 s \D \tau(s)\Big)
\delta_0(\sigma), \quad \sigma \in \borel{[0,1]},
   \end{align}
is the representing measure of $\{a_{n-1}\}_{n=0}^\infty$.
Moreover, $\rho(\nul)=0$ if and only if $1 + \int_0^1 \frac
1 s \D \tau(s) = a_0$.
   \end{lem}
   \begin{proof}
(i)$\Rightarrow$(ii) Let $\varrho$ be a representing
measure of $\{a_{n-1}\}_{n=0}^\infty$. Then
   \begin{align} \label{a0}
a_0 = a_{1-1} = a_{-1} + \int_0^1 1 \D \varrho(s) = 1 +
\varrho([0,1]).
   \end{align}
Define the positive Borel measure $\tilde\tau$ on $[0,1]$
by $\D \tilde\tau(s) = s \D \varrho (s)$. Then, we have
   \begin{align*}
a_n = a_{(n+1)-1} & \hspace{1.5ex} = a_{-1} + \int_0^1 (1+
\ldots + s^{n}) \D \varrho(s)
   \\
& \hspace{1.5ex} = 1 + \varrho([0,1]) + \int_0^1 (1+ \ldots
+ s^{n-1}) s \D \varrho(s)
   \\
& \overset{\eqref{a0}} = a_0 + \int_0^1 (1+ \ldots +
s^{n-1}) \D \tilde\tau(s), \quad n = 1, 2, \ldots
   \end{align*}
By the uniqueness of the representing measure, we have
$\tilde \tau=\tau$. This implies that $\tau(\{0\})=0$,
and consequently
   \begin{align*}
a_0 \overset{\eqref{a0}} = 1 + \rho(\nul) + \rho((0,1]) = 1
+ \rho(\nul) + \int_0^1 \frac 1 s \D \tau(s).
   \end{align*}
Hence (ii) holds.

(ii)$\Rightarrow$(i) Define the positive Borel measure
$\varrho$ on $[0,1]$ by \eqref{nu-ch}. Then
   \begin{align*}
a_{-1} & + \int_0^1 (1+ \ldots + s^{n-1}) \D \varrho(s)
   \\
& = 1 + \int_0^1 \frac {1+ \ldots + s^{n-1}} s \D \tau(s) +
\Big(a_0 - 1 - \int_0^1 \frac 1 s \D \tau(s)\Big)
   \\
& = 1 + \int_0^1 \frac 1 s \D \tau(s) + \int_0^1 (1+ \ldots
+ s^{n-2}) \D \tau(s) + \Big(a_0 - 1 - \int_0^1 \frac 1 s
\D \tau(s)\Big)
   \\
&= a_0 + \int_0^1 (1+ \ldots + s^{n-2}) \D \tau(s) =
a_{n-1}, \quad n=2,3, \ldots
   \end{align*}
Since
   \begin{align*}
a_{-1} + \int_0^1 1 \D\rho(s) \overset{\eqref{nu-ch}} = 1 +
\int_0^1 \frac 1 s \D \tau(s) + \Big(a_0 - 1 - \int_0^1
\frac 1 s \D \tau(s)\Big) = a_0 = a_{1-1},
   \end{align*}
we deduce that the sequence
$\{a_{n-1}\}_{n=0}^\infty$ is completely
alternating with the representing measure
$\varrho$. This completes the proof.
   \end{proof}
We are now ready to recall the definition of our present
object of study from \cite{ath}. Let $\hh$ be a complex
Hilbert space. An operator $A\in \ogr{\hh}$ is said to be
{\em completely hyperexpansive} if the sequence $\{\|A^n
h\|^2\}_{n=0}^\infty$ is completely alternating for every
$h \in \hh$. In view of the above discussion, $A$ is
completely hyperexpansive if and only if (substitute $A^m
h$ in place of $h$)
   \begin{align}  \label{ch1}
\sum_{j=0}^n (-1)^j \binom n j\|A^jh\|^2 \Le 0, \quad n \in
\nbb, \, h\in\hh,
   \end{align}
Completely hyperexpansive operators are antithetical
to contractive subnormal operators in the sense that
their defining properties and behavior are related to
the theory of completely alternating functions on
abelian semigroups (subnormality is connected with
positive definiteness). This is their great advantage
and one of the reasons why they attract attention of
researchers (see e.g.,
\cite{Alem,AgSt1,AgSt2,AgSt3,ath,a-s,s-a,a-r1,a-r2,jab,ath2,jab2,e-j-l}).

Note that if $A$ is a completely hyperexpansive
operator, then $\|Ah\|\Ge \|h\|$ for all $h \in \hh$
(apply \eqref{ch1} to $n=1$), which means that $A$ is
injective. In view of this, we can deduce from
Proposition \ref{dzisdesz} that a directed tree which
admits a completely hyperexpansive weighted shift must
be leafless.
   \begin{pro} \label{chinj}
If $\slam \in \ogr {\ell^2(V)}$ is a completely
hyperexpansive weighted shift on a directed tree $\tcal$
with weights $\lambdab = \{\lambda_v\}_{v \in V^\circ}$,
then $\tcal$ is leafless and $\sum_{v\in \dzi u}
|\lambda_v|^2 > 0$ for all $u \in V$.
   \end{pro}
Complete hyperexpansivity of weighted shifts on directed
trees can be characterized as follows.
   \begin{thm} \label{charch}
If $\slam \in \ogr {\ell^2(V)}$ is a weighted shift on a
directed tree $\tcal$ with weights $\lambdab =
\{\lambda_v\}_{v \in V^\circ}$, then the following
conditions are equivalent{\em :}
   \begin{enumerate}
   \item[(i)] $\slam$ is completely hyperexpansive,
   \item[(ii)] $\big\{\sum\limits_{v \in \dzin{n}{u}}
|\lambda_{u\mid v}|^2\big\}_{n=0}^\infty$ is a
completely alternating sequence for all $u \in V,$
   \item[(iii)] $\{\|\slam^n e_u\|^2\}_{n=0}^\infty$
is a completely alternating sequence for all $u \in V,$
   \item[(iv)]
$\sum_{j=0}^n (-1)^j \binom n j\|\slam^{j} e_u\|^2 \Le 0$
for all $n \in \nbb$ and $u \in V.$
   \end{enumerate}
   \end{thm}
   \begin{proof}
The implications (i)$\Rightarrow$(iii) and
(iii)$\Rightarrow$(iv) are evident. By Lemma \ref{pot}, the
conditions (ii) and (iii) are equivalent.

(iv)$\Rightarrow$(i) Take $f \in \ell^2(V)$. It follows
from \eqref{sumo} that
   \begin{multline*}
\sum_{j=0}^n (-1)^j \binom n j\|\slam^{j} f\|^2 =
\sum_{j=0}^n (-1)^j \binom n j \sum_{u \in V} |f(u)
|^2\|\slam^{j} e_u \|^2
   \\
= \sum_{u \in V} |f(u)|^2 \Big(\sum_{j=0}^n (-1)^j
\binom n j \|\slam^j e_u\|^2 \Big)
\overset{\mathrm{(iv)}} \Le 0
   \end{multline*}
for all $n \in \nbb$. This, together with \eqref{ch1},
completes the proof.
   \end{proof}
As shown in Corollary \ref{chcyc} below, the study of
complete hyperexpansivity of weighted shifts on
directed trees reduces to the case of trees with root.
This is very similar to what happens in the case of
subnormality. The proof of Corollary \ref{chcyc} is
essentially the same as that of Corollary \ref{subcyc}
(use Theorem \ref{charch} instead of Theorem
\ref{charsub}).
   \begin{cor} \label{chcyc}
Let $\slam \in \ogr {\ell^2(V)}$ be a weighted shift on a
directed tree $\tcal$ with weights $\lambdab =
\{\lambda_v\}_{v \in V^\circ}$. Suppose that $X$ is a
subset of $V$ such that $V=\bigcup_{x\in X} \des x$. Then
$\slam$ is completely hyperexpansive if and only if the
operator $\slamr x$ is completely hyperexpansive for every
$x \in X$ $($cf.\ Notation {\em \ref{poddrz}}$)$.
   \end{cor}
We now formulate the counterpart of the small lemma (cf.\
Lemma \ref{charsub-1}) for completely hyperexpansive
weighted shifts on directed trees.
   \begin{lem} \label{charch-1}
Let $\slam \in \ogr {\ell^2(V)}$ be a weighted shift on a
directed tree $\tcal$ with weights $\lambdab =
\{\lambda_v\}_{v \in V^\circ}$ and let $u_0, u_1 \in V$ be
such that $\dzi{u_0} = \{u_1\}$. If the sequence
$\{\|\slam^n e_{u_0}\|^2\}_{n=0}^\infty$ is completely
alternating and $\lambda_{u_1}\neq 0$, then the sequence
$\{\|\slam^n e_{u_1}\|^2\}_{n=0}^\infty$ is completely
alternating.
   \end{lem}
   \begin{proof} Noticing that
   \begin{align*}
\|\slam^n e_{u_1}\|^2 \overset{\eqref{eu}}= \frac 1
{|\lambda_{u_1}|^2}\|\slam^{n+1} e_{u_0}\|^2, \quad n \in
\zbb_+,
   \end{align*}
and employing \eqref{ca+1}, we complete the proof.
   \end{proof}
It turns out that Lemma \ref{charch-1} is no longer true if
$\card{\dzi{u_0}} \Ge 2$. Indeed, the weighted shift
$\slam$ on $\tcal_{2,0}$ defined in Example \ref{2nitki}
has the property that the sequence
$\{\|\slam^ne_0\|^2\}_{n=0}^\infty = \{1,1, \ldots\}$ is
completely alternating (with the representing measure
$\tau=0$, cf.\ \eqref{repch}), and neither of the sequences
$\{\|\slam^ne_{1,1}\|^2\}_{n=0}^\infty = \{1, (\frac ba)^2,
1,1, \ldots\}$ and $\{\|\slam^ne_{2,1}\|^2\}_{n=0}^\infty =
\{1, (\frac ab)^2, 1,1, \ldots\}$ is completely alternating
(as neither of them is monotonically increasing).

Our next goal is to find relationships between
representing measures of completely alternating
sequences $\{\|\slam^n e_u\|^2\}_{n=0}^\infty$, $u \in
V$. Let us first fix the notation that is used
throughout (compare with Notation \ref{ozn2}).
   \begin{ozn} \label{ozn2-ca}
Let $\slam \in \ogr {\ell^2(V)}$ be a weighted shift
on a directed tree $\tcal$. If for some $u \in V$, the
sequence $\{\|\slam^n e_u\|^2\}_{n=0}^\infty$ is
completely alternating, then its unique representing
measure which is concentrated on $[0,1]$ will be
denoted by \idxx{$\tau_u$, $\tau_u^{\tcal}$}{77}
$\tau_u$ (or by $\tau_u^{\tcal}$ if we wish to make
clear the dependence of $\tau_u$ on $\tcal$).
   \end{ozn}
The result which follows is a counterpart of the big lemma
(cf.\ Lemma \ref{charsub2}) for completely hyperexpansive
weighted shifts on directed trees.
   \begin{lem} \label{charch2}
Let $\slam \in \ogr {\ell^2(V)}$ be a weighted shift on a
directed tree $\tcal$ with weights $\lambdab =
\{\lambda_v\}_{v \in V^\circ}$, and let $u \in V^\prime$ be
such that the sequence $\{\|\slam^n e_v\|^2\}_{n=0}^\infty$
is completely alternating for every $v \in \dzi u$. Then
the following conditions are equivalent\/{\em :}
   \begin{enumerate}
   \item[(i)] the sequence $\{\|\slam^n e_u\|^2\}_{n=0}^\infty$
is completely alternating,
   \item[(ii)]  $\slam$ satisfies
the consistency condition at $u$, i.e.,
   \begin{align} \label{consist-ch}
\sum_{v \in \dzi{u}} |\lambda_v|^2 \Ge 1 + \sum_{v \in
\dzi{u}} |\lambda_v|^2 \int_0^1 \frac 1 s\, \D \tau_v(s).
   \end{align}
   \end{enumerate}
   If {\em (i)} holds, then $\tau_v(\nul)=0$ for every $v
\in \dzi u$ such that $\lambda_v \neq 0$, and
   \begin{align} \label{muu-ch}
   \begin{aligned} \tau_u(\sigma)  & =
\sum_{v \in \dzi u} |\lambda_v|^2
\int_\sigma \frac 1 s \D \tau_v(s)
   \\
& \hspace{8ex}+ \Big(\sum_{v \in \dzi{u}} |\lambda_v|^2 - 1
- \sum_{v \in \dzi{u}} |\lambda_v|^2 \int_0^1 \frac 1 s\,
\D \tau_v(s)\Big) \delta_0(\sigma)
   \end{aligned}
   \end{align}
for every $\sigma \in \borel{[0,1]}$. Moreover,
$\tau_u(\nul)=0$ if and only if $\slam$ satisfies the
strong consistency condition at $u$, i.e.,
   \begin{align} \label{consist-ch'}
\sum_{v \in \dzi{u}} |\lambda_v|^2 = 1 + \sum_{v \in
\dzi{u}} |\lambda_v|^2 \int_0^1 \frac 1 s\, \D \tau_v(s).
   \end{align}
   \end{lem}
   \begin{proof}
Define the positive Borel measure $\tau$ on $[0,1]$ by
   \begin{align*}
\tau(\sigma) = \sum_{v \in \dzi u} |\lambda_v|^2
\tau_v(\sigma), \quad \sigma \in \borel{[0,1]}.
   \end{align*}
Applying a version of \eqref{leb2} (with $\mu=\tau$ and
$\mu_v=\tau_v$), we see that
   \allowdisplaybreaks
   \begin{align*}
\|\slam^{n+1} e_u\|^2 & \overset{\eqref{sln+1}}= \sum_{v
\in \dzi u} |\lambda_v|^2 \|\slam^n e_v\|^2
   \\
& \hspace{2.2ex}= \sum_{v \in \dzi u} |\lambda_v|^2
(1+\int_0^1 (1+ \ldots + s^{n-1}) \D \tau_v(s))
   \\
& \hspace{.55ex} \overset{\eqref{eu}} = \|\slam
e_u\|^2 + \int_0^1 (1+ \ldots + s^{n-1}) \D \tau (s),
\quad n =1,2, \ldots,
   \end{align*}
which means that the sequence $\{\|\slam^{n+1}
e_u\|^2\}_{n=0}^\infty$ is completely alternating with the
representing measure $\tau$. Employing now a version of
\eqref{1/t} and Lemma \ref{bext-ca} with $a_n =
\|\slam^{n+1} e_u\|^2$, we see that the conditions (i) and
(ii) are equivalent. The formula \eqref{muu-ch} can be
inferred from \eqref{nu-ch} by applying a version of
\eqref{leb2}. The rest of the conclusion is clearly true.
   \end{proof}
Because of uniqueness of representing measures $\tau_v$,
there is strong hope that Lemma \ref{charch2} holds for
unbounded completely hyperexpansive weighted shifts on
directed trees (see \cite{j-s,jab,jab2} for an invitation
to unbounded completely hyperexpansive operators).

It is interesting to note that the direct counterpart of
Proposition \ref{maxsub} for completely hyperexpansive
weighted shifts on directed trees is no longer true (cf.\
Example \ref{ch-n0restr}). However, its weaker version
remains valid (cf.\ Proposition \ref{maxsub-ch}).
   \subsection{\label{k-step}Complete hyperexpansivity on
$\tcal_{\eta,\kappa}$} Now we confine our attention to
discussing the question of complete hyperexpansivity
of weighted shifts on the directed trees
$\tcal_{\eta,\kappa}$ (cf.\ \eqref{varkappa}). Before
formulating the counterpart of Theorem \ref{omega} for
completely hyperexpansive weighted shifts, we first
recall the definition of completely hyperexpansive
$k$-step backward extendibility of unilateral
classical weighted shifts (cf.\ \cite[Definition
4.1]{j-j-s}). Given an integer $k\Ge 1$, we say that a
unilateral classical weighted shift with positive
weights $\{\lambda_n\}_{n=1}^\infty$ (cf.\ Remark
\ref{re1-2}) has a {\it completely hyperexpansive
$k$-step backward extension} if for some positive
scalars $\lambda_{-k+1}, \ldots, \lambda_{0}$, the
unilateral classical weighted shift with weights
$\{\lambda_{-k+n}\}_{n=1}^\infty$ is completely
hyperexpansive (note that completely hyperexpansive
unilateral classical weighted shifts are automatically
bounded, cf.\ \cite[Proposition 6.2\,(i)]{j-s}). Each
unilateral classical weighted shift which has a
completely hyperexpansive $k$-step backward extension
is automatically completely hyperexpansive (cf.\
\cite[Section 4]{j-j-s})).

Recall that a unilateral classical weighted shift $S$
with positive weights $\{\lambda_n\}_{n=1}^\infty$ is
completely hyperexpansive if and only if the sequence
$\{\|S^ne_0\|^2\}_{n=0}^\infty$ is completely
alternating (cf.\ \cite[Proposition 3]{ath}); the
representing measure of
$\{\|S^ne_0\|^2\}_{n=0}^\infty$ will be called the
{\em representing measure} of $S$. The following
characterization of completely hyperexpansive $k$-step
backward extendibility was given in \cite[Theorem
4.2]{j-j-s}.
   \begin{align} \label{jab-ju-st}
   \begin{minipage}{28em}
If $k \in \nbb$, then $S$ has a completely hyperexpansive
$k$-step backward extension if and only if $S$ is
completely hyperexpansive and $\int_0^1 \sum_{l=1}^k
\frac{1}{s^l} \D \tau(s) < 1$, where $\tau$ is the
representing measure of $S$.
   \end{minipage}
   \end{align}
Below we adhere to Notation \ref{ozn2-ca}.
   \begin{thm}\label{omega-ch}
Suppose that $\tcal$ is a directed tree for which
there exists $\omega\in V$ such that
$\card{\dzi{\omega}}\Ge 2$ and $\card{\dzi{v}}=1$ for
every $v \in V \setminus \{\omega\}$. Let $\slam
\in\ogr{\ell^2(V)}$ be a weighted shift on the
directed tree $\tcal$ with nonzero weights $\lambdab =
\{\lambda_v\}_{v \in V^\circ}$. Then the following
assertions hold.
   \begin{enumerate}
   \item[(i)] If $\omega = \koo$, then
$\slam$ is completely hyperexpansive if and only if the
sequence $\{\|\slam^n e_v\|^2\}_{n=0}^\infty$ is completely
alternating for every $v \in \dzi {\omega}$, and $\slam$
satisfies the consistency condition at $\omega$, i.e.,
\eqref{consist-ch} is valid for $u=\omega$.
   \item[(ii)] If $\tcal$ has a root and $\omega \neq \koo$,
then $\slam$ is completely hyperexpansive if and only if
one of the following two equivalent conditions holds\/{\em
:}
   \begin{enumerate}
   \item[(\mbox{ii-a})] $\slamr
\omega$ is completely hyperexpansive,
\eqref{consist-ch'} is valid for $u=\omega$,
   \begin{align*}  %\label{znumerkiem-ch}
|\lambda_{\paa^{k-1}(\omega)}|^2 = 1 +
\Big|\prod_{j=0}^{k-1} \lambda_{\paa^j(\omega)}\Big|^2
\sum_{v \in \dzi \omega} |\lambda_v|^2 \int_0^1 \frac 1
{s^{k+1}} \D \tau_v(s)
   \end{align*}
for all $k\in J_{\kappa-1}$, and
   \begin{align*}  %\label{znumerkiem-ch'}
|\lambda_{\paa^{\kappa-1}(\omega)}|^2 \Ge 1 +
\Big|\prod_{j=0}^{\kappa-1} \lambda_{\paa^j(\omega)}\Big|^2
\sum_{v \in \dzi \omega} |\lambda_v|^2 \int_0^1 \frac 1
{s^{\kappa+1}} \D \tau_v(s),
   \end{align*}
where $\kappa$ is a unique positive integer such that
$\paa^\kappa(\omega)=\koo$,
   \item[(\mbox{ii-b})] the sequences $\{\|\slam^n
e_{\koo}\|^2\}_{n=0}^\infty$ and $\{\|\slam^n
e_v\|^2\}_{n=0}^\infty$ are completely alternating for all
$v \in \dzi {\omega}$.
   \end{enumerate}
   \item[(iii)] If $\tcal$ is rootless,
then $\slam$ is completely hyperexpansive if and only if
$\slam$ is an isometry.
   \end{enumerate}
   \end{thm}
   \begin{proof}
The proofs of (i) and (ii) are essentially the same as the
corresponding parts of the proof of Theorem \ref{omega}. We
only have to use Lemmas \ref{charch-1} and \ref{charch2}
and Theorem \ref{charch} in place of Lemmas \ref{charsub-1}
and \ref{charsub2} and Theorem \ref{charsub}, respectively.
Moreover, in proving (ii) we need to exploit the explicit
formulas for representing measures $\tau_{\paa^k(\omega)}$,
$k\in J_{\kappa-1}$, which are given by
   \begin{align*}
\frac {\tau_{\paa^k(\omega)} (\sigma)} {|\prod_{j=0}^{k-1}
\lambda_{\paa^j(\omega)}|^2} = \sum_{v \in \dzi \omega}
|\lambda_v|^2 \int_{\sigma} \frac 1 {s^{k+1}}\, \D
\tau_v(s), \quad \sigma \in \borel{\rbb}
   \end{align*}
(the measure $\tau_\omega$ is given by $\tau_\omega(\sigma)
= \sum_{v \in \dzi \omega} |\lambda_v|^2 \int_\sigma \frac
1 s \D \tau_v(s)$ for $\sigma \in \borel{[0,1]}$).

(iii) Suppose that $\tcal$ is rootless and $\slam$ is
completely hyperexpansive. Then the sequences $\{\|\slam^n
e_{\paa^k(\omega)} \|^2\}_{n=0}^\infty$ and $\{\|\slam^n
e_v\|^2\}_{n=0}^\infty$ are completely alternating for all
$k \in \zbb_+$ and $v \in \dzi {\omega}$. Fix $k \in
\zbb_+$. Consider the unilateral classical weighted shift
$W_k$ with weights $\big\{\|\slam^{n}e_{\paa^k(\omega)}\|
\cdot
\|\slam^{n-1}e_{\paa^k(\omega)}\|^{-1}\big\}_{n=1}^\infty$.
By \cite[Proposition 3]{ath}, $W_k$ is completely
hyperexpansive. Since, by \eqref{eu},
   \begin{align*}
\slam^l e_{\paa^{k+l}(\omega)} =
\lambda_{\paa^{k+l-1}(\omega)} \cdots
\lambda_{\paa^{k}(\omega)} e_{\paa^{k}(\omega)}, \quad
l\in \nbb,
   \end{align*}
we deduce that
   \begin{align*}
\frac{\|\slam^{n+l}e_{\paa^{k+l}(\omega)}\|}
{\|\slam^{n-1+l}e_{\paa^{k+l}(\omega)}\|} =
\frac{\|\slam^{n}e_{\paa^k(\omega)}\|}
{\|\slam^{n-1}e_{\paa^k(\omega)}\|}, \quad n,l \Ge 1.
   \end{align*}
As $W_{k+l}$ is completely hyperexpansive, we see that
the unilateral classical weighted shift $W_k$ has a
completely hyperexpansive $l$-step backward extension
for all integers $l\Ge 1$. Hence, by \cite[Corollary
4.6\,(i)]{j-j-s}, the weights of $W_k$ are equal to
$1$, which, together with \eqref{eu}, implies that
   \begin{align} \label{con=1}
1 = \|\slam e_{\paa^k(\omega)}\|^2 =
   \begin{cases}
   |\lambda_{\paa^{k-1}(\omega)}|^2 & \text{ for } k \in
\nbb,
\\
   \sum_{v \in \dzi \omega} |\lambda_v|^2 & \text{ for }
k=0.
   \end{cases}
   \end{align}
Applying the consistency condition \eqref{consist-ch} at
$u=\omega$ and \eqref{con=1}, we deduce that $\tau_v=0$ for
all $v \in \dzi{\omega}$. Employing the integral
representation \eqref{repch}, we see that $\|\slam^n
e_v\|=1$ for all $n\in \zbb_+$ and $v\in \dzi{\omega}$.
Next, using an induction argument and \eqref{decom} (with
$u=\omega$), we infer that $\|\slam e_u\| = 1$ for all $u
\in \des {\omega}$. Finally, applying Proposition
\ref{xdescor}\,(iv) to $u=\omega$ and Corollary
\ref{chariso}, we conclude that $\slam$ is an isometry. The
reverse implication is obvious. This completes the proof.
   \end{proof}
   \begin{rem}
A careful look at the proof of Theorem \ref{omega-ch}
reveals that the characterization (\mbox{ii-a}) of complete
hyperexpansivity of bounded weighted shifts on $\tcal$ with
nonzero weights remains valid even if $\omega$ has only one
child.
   \end{rem}
As an immediate application of Theorem \ref{omega-ch} we
have the following result, being a counterpart of Corollary
\ref{omega2} for completely hyperexpansive weighted shifts.
As before, we adhere to notation $\lambda_{i,j}$ instead of
a more formal expression $\lambda_{(i,j)}$. Recall also
that $\eta,\kappa \in \zbb_+ \sqcup \{\infty\}$ and $\eta
\Ge 2$.
   \begin{cor}\label{omega2-ch}
Let $\slam \in\ogr{\ell^2(V_{\eta,\kappa})}$ be a
weighted shift on the directed tree
$\tcal_{\eta,\kappa}$ with nonzero weights
$\lambdab = \{\lambda_v\}_{v \in
V_{\eta,\kappa}^\circ}$. Then the following
assertions hold.
   \begin{enumerate}
   \item[(i)] If $\kappa=0$, then
$\slam$ is completely hyperexpansive if and only if there
exist positive Borel measures $\{\tau_i\}_{i=1}^\eta$ on
$[0,1]$ such that
   \allowdisplaybreaks
   \begin{gather}  \label{zgod0-ch}
1+\int_0^1 (1 + \ldots + s^{n-1}) \, \D \tau_i(s)
= \Big|\prod_{j=2}^{n+1}\lambda_{i,j}\Big|^2,
\quad n \in \nbb, \; i \in J_\eta,
   \\
\sum_{i=1}^\eta |\lambda_{i,1}|^2 \Ge 1 + \sum_{i=1}^\eta
|\lambda_{i,1}|^2 \int_0^1 \frac 1 s\, \D \tau_i(s).
\label{zgod-ch}
   \end{gather}
   \item[(ii)] If $0 < \kappa < \infty$,
then $\slam$ is completely hyperexpansive if and
only if one of the following two equivalent
conditions holds\/{\em :}
   \begin{enumerate}
   \item[(\mbox{ii-a})]  there exist positive Borel
measures $\{\tau_i\}_{i=1}^\eta$ on $[0,1]$ which satisfy
\eqref{zgod0-ch} and the following requirements{\em :}
   \begin{align}     \label{zgod-ch'}
&\sum_{i=1}^\eta |\lambda_{i,1}|^2 = 1 + \sum_{i=1}^\eta
|\lambda_{i,1}|^2 \int_0^1 \frac 1 s\, \D \tau_i(s).
   \\
\label{lam=1+} &|\lambda_{-(k-1)}|^2
 = 1 + \Big|\prod_{j=0}^{k-1} \lambda_{-j}\Big|^2
\sum_{i=1}^\eta|\lambda_{i,1}|^2 \int_0^1 \frac 1 {s^{k+1}}
\D \tau_i(s), \quad k \in J_{\kappa-1},
   \\
\label{lam=1+'} &|\lambda_{-(\kappa-1)}|^2 \Ge 1 +
\Big|\prod_{j=0}^{\kappa-1} \lambda_{-j}\Big|^2
\sum_{i=1}^\eta|\lambda_{i,1}|^2 \int_0^1 \frac 1
{s^{\kappa+1}} \D \tau_i(s);
   \end{align}
   \item[(\mbox{ii-b})] there exist
positive Borel measures $\{\tau_i\}_{i=1}^\eta$ and $\nu$
on $[0,1]$ which satisfy \eqref{zgod0-ch} and the equations
below
   \begin{align*}
1+\int_0^1 (1+ \ldots + s^{n-1}) \D \nu(s) =
   \begin{cases}
|\prod_{j=\kappa-n}^{\kappa-1}\lambda_{-j}|^2 &
\text{if } n \in J_\kappa,
   \\[1ex]
|\prod_{j=0}^{\kappa-1}\lambda_{-j}|^2 \big(\sum_{i=1}^\eta
|\prod_{j=1}^{n-\kappa} \lambda_{i,j}|^2\big) & \text{if }
n \in \nbb \setminus J_\kappa.
   \end{cases}
   \end{align*}
   \end{enumerate}
   \end{enumerate}
   \begin{enumerate}
   \item[(iii)] If $\kappa=\infty$,
then $\slam$ is completely hyperexpansive if and
only if $\slam$ is an isometry.
   \end{enumerate}
Moreover, if $\slam$ is completely hyperexpansive
and $\{\tau_i\}_{i=1}^\eta$ are positive Borel
measures on $[0,1]$ satisfying \eqref{zgod0-ch},
then $\tau_i = \tau_{i,1}$ for all $i \in
J_\eta$.
   \end{cor}
   \subsection{\label{mod-ch}Modelling complete
hyperexpansivity on $\tcal_{\eta,\kappa}$}
   Our aim in this section is to find a model for
completely hyperexpansive weighted shifts (with
nonzero weights) on $\tcal_{\eta,\kappa}$ (cf.\
\eqref{varkappa}). In view of Theorem \ref{uni} and
Corollary \ref{omega2-ch}, we can confine ourselves to
discussing the case when $\kappa$ is finite and the
weights of weighted shifts under consideration are
positive. We begin by formulating a simple necessary
condition which has to be satisfied by representing
measures $\tau_{i,1}$ (see Notation \ref{ozn2-ca}).
   \begin{lem}
If $\kappa \in \zbb_+$ and $\slam \in
\ogr{\ell^2(V_{\eta,\kappa})}$ is a completely
hyperexpansive weighted shift on $\tcal_{\eta,\kappa}$ with
nonzero weights, then $\sup_{i \in J_\eta}
\tau_{i,1}([0,1]) < \infty$.
   \end{lem}
   \begin{proof}  By  \eqref{repch}, we have
$\tau_{i,1} ([0,1]) = \|\slam e_{i,1}\|^2 -1 \Le
\|\slam\|^2 - 1$ for all $i \in J_\eta$, which completes
the proof.
   \end{proof}
We now show that a completely hyperexpansive weighted shift
on the directed tree $\tcal_{\eta,\kappa}$ is determined,
in a sense, by its weights which correspond to $\dzi{0}$.
   \begin{lem} \label{omega3-ch}
Let $\eta \in \{2,3, \ldots\} \sqcup \{\infty\}$ and
$\kappa \in \zbb_+$. Suppose that $\boldsymbol \tau
=\{\tau_i\}_{i=1}^\eta$ is a sequence of positive Borel
measures on $[0,1]$ such that $\sup_{i \in J_\eta}
\tau_{i}([0,1]) < \infty$. Then the following assertions
hold.
   \begin{enumerate}
   \item[(i)] If $\slam \in \ogr{\ell^2(V_{\eta,\kappa})}$
is a completely hyperexpansive weighted shift on
$\tcal_{\eta,\kappa}$ with positive weights $\lambdab=
\{\lambda_v\}_{v \in V_{\eta,\kappa}^\circ}$ such that
$\tau_{i,1} = \tau_i$ for all $i \in J_\eta$, then the
system $\teb:=\{t_{i}\}_{i=1}^{\eta}$ with $t_{i}:=
\lambda_{i,1}$ satisfies the following conditions{\em: }
   \allowdisplaybreaks
   \begin{align}
&\sum_{i=1}^\eta t_{i}^2 < \infty, \label{ka+1}
      \\
&\label{ka+2}
   \begin{cases}
   \begin{aligned}
&\sum_{i=1}^\eta t_{i}^2 \Ge 1 + \sum_{i=1}^\eta t_{i}^2
\int_0^1 \frac 1 s \D \tau_i(s) \quad \text{ if } \kappa=0,
   \\
& \sum_{i=1}^\eta t_{i}^2 = 1 + \sum_{i=1}^\eta t_{i}^2
\int_0^1 \frac 1 s \D \tau_i(s) \quad \text{ if } \kappa >
0,
   \end{aligned}
   \end{cases}
   \\
& \sum_{i=1}^\eta t_{i}^2 > \sum_{i=1}^\eta t_{i}^2
\int_0^1 \Big(\frac 1 s + \ldots + \frac 1
{s^{\kappa+1}}\Big) \D \tau_i(s). \label{ka+3}
   \end{align}
   \item[(ii)] Let $\teb =
\{t_{i}\}_{i=1}^{\eta} \subseteq (0,\infty)$ satisfy
\eqref{ka+1}, \eqref{ka+2} and \eqref{ka+3}. Then
there exists a completely hyperexpansive weighted
shift $\slam \in \ogr{\ell^2(V_{\eta,\kappa})}$ on
$\tcal_{\eta,\kappa}$ with positive weights $\lambdab=
\{\lambda_v\}_{v \in V_{\eta,\kappa}^\circ}$ such that
$\lambda_{i,1}=t_i$ and $\tau_{i,1} = \tau_i$ for all
$i \in J_\eta$. If $\kappa=0$, $\slam$ is unique. If
$\kappa \Ge 1$, all the weights of $\slam$, except for
$\lambda_{-\kappa+1}$, are uniquely determined by
$\teb$ and $\boldsymbol \tau$; the weight
$\lambda_{-\kappa+1}$ can be chosen arbitrarily within
the interval
$\Big[\sqrt{\frac{\zeta_\kappa}{\zeta_{\kappa+1}}},\infty\Big)$,
where
   \begin{align} \label{5USD}
\zeta_k =\sum_{i=1}^\eta t_i^2 \Big(1 - \int_0^1
\sum_{j=1}^{k} \frac 1 {s^j} \, \D \tau_i(s)\Big), \quad k
\in J_{\kappa + 1}.
   \end{align}
Moreover, the norm of $\slam$ is given by
   \begin{align} \label{norm-ch}
\|\slam\|^2 =
   \begin{cases}
   \max\big\{\sum_{i=1}^\eta t_i^2, \sup_{i\in J_\eta}
(1+\tau_i([0,1]))\big\} & \text{for } \kappa=0,
   \\[1ex]
   \max\big\{\lambda_{-(\kappa-1)}^2, \sup_{i\in J_\eta}
(1+\tau_i([0,1]))\big\} & \text{for } \kappa \Ge 1.
   \end{cases}
   \end{align}
   \end{enumerate}
   \end{lem}
It is worth noting that if $\kappa=0$, then \eqref{ka+2}
implies \eqref{ka+3}. Observe also that the quantities
$\zeta_k$ are defined only in the case when $\kappa \Ge 1$,
and that $\zeta_1=1$ (use \eqref{ka+2}).
   \begin{proof}[Proof of Lemma  \ref{omega3-ch}]
(i) In view of Corollary \ref{omega2-ch}, $\int_0^1 \frac 1
{s^{k+1}} \D \tau_i(s) < \infty$ for all $k\in J_\kappa$
and $i\in J_{\eta}$. By \eqref{eu}, we have
$\sum_{i=1}^\eta t_{i}^2 = \|\slam e_0\|^2 < \infty$, which
gives \eqref{ka+1}. The condition \eqref{ka+2} follows from
\eqref{zgod-ch} and \eqref{zgod-ch'}. Thus, it remains to
prove \eqref{ka+3}.

If $\kappa=0$, then, as noted just above, \eqref{ka+2}
implies \eqref{ka+3}.

If $\kappa=1$, then by applying the inequality
\eqref{lam=1+'} we get
   \allowdisplaybreaks
   \begin{align} \label{1=l02}
1 \hspace{-1.5ex} & \hspace{1.5ex} \Le \lambda_0^2 \Big(1 -
\sum_{i=1}^\eta t_i^2 \int_0^1 \frac 1 {s^2} \D
\tau_i(s)\Big)
   \\
&\hspace{1.5ex}= \lambda_0^2 \Big(\big(1 - \sum_{i=1}^\eta
t_i^2\big) + \sum_{i=1}^\eta t_i^2 \Big(1 - \int_0^1 \frac
1 {s^2} \D \tau_i(s)\Big)\Big) \notag
   \\
&\overset{\eqref{zgod-ch'}}= \lambda_0^2 \sum_{i=1}^\eta
t_i^2 \Big(1 - \int_0^1 \Big(\frac 1 s + \frac 1 {s^2}\Big)
\D \tau_i(s)\Big) = \lambda_0^2 \zeta_2, \notag
   \end{align}
where $\zeta_2$ is as in \eqref{5USD}. Hence, \eqref{ka+3}
follows from \eqref{1=l02}.

Assume now that $\kappa = 2$. Arguing as in \eqref{1=l02}
and using \eqref{lam=1+} in place of \eqref{lam=1+'}, we
obtain
   \begin{align} \label{1=l02'}
1 = \lambda_0^2 \zeta_2.
   \end{align}
It follows from \eqref{lam=1+'} that
   \allowdisplaybreaks
   \begin{align}  \label{1=104}
   1 \hspace{-1.5ex} & \hspace{1.5ex}\Le \lambda_{-1}^2
   \Big(1 - \lambda_0^2 \sum_{i=1}^\eta t_i^2 \int_0^1
   \frac 1 {s^3} \D \tau_i(s)\Big)
   \\
&\overset{\eqref{1=l02'}}= \lambda_{-1}^2 \left(1 -
\frac{\sum_{i=1}^\eta t_i^2 \int_0^1 \frac 1 {s^3} \D
\tau_i(s)}{\sum_{i=1}^\eta t_i^2 \Big(1 - \int_0^1
\sum_{j=1}^{2} \frac 1 {s^j} \D \tau_i(s)\Big)} \right)
\notag
   \\
& \hspace{1.5ex} = \lambda_{-1}^2 \frac{\zeta_3}{\zeta_2},
\notag
   \end{align}
which together with \eqref{1=l02'} implies \eqref{ka+3}.

Suppose now that $\kappa \Ge 3$. We claim that the
following two conditions hold for all $k\in \{2, \ldots,
\kappa -1\}$:
   \allowdisplaybreaks
   \begin{align} \label{1>}
& \sum_{i=1}^\eta t_{i}^2 > \sum_{i=1}^\eta t_{i}^2
\int_0^1 \sum_{j=1}^{k+1}\frac{1}{s^j} \, \D \tau_i(s),
   \\
& 1 = \lambda_{-(k-1)}^2 \frac{\zeta_{k+1}} {\zeta_k}.
\label{2>}
   \end{align}
Arguing as in \eqref{1=l02} and \eqref{1=104}, and
using \eqref{lam=1+} with $k=1,2$ in place of
\eqref{lam=1+'}, we get \eqref{1=l02'} and the
equality
   \begin{align*}
1 = \lambda_{-1}^2 \frac{\zeta_3}{\zeta_2},
   \end{align*}
which implies that \eqref{1>} and \eqref{2>} hold for
$k=2$. This proves our claim for $\kappa=3$. If $\kappa \Ge
4$, we proceed by induction. Fix an integer $n$ such that
$2 \Le n < \kappa-1$ and assume that \eqref{1>} and
\eqref{2>} hold for all $k=2, \ldots,n$. By \eqref{lam=1+},
applied to $k=n+1$, we obtain (note that \eqref{1=l02'} is
valid for $\kappa \Ge 2$)
   \allowdisplaybreaks
   \begin{align} \label{1=106}
   1 \hspace{-4.5ex} & \hspace{4.5ex}= \lambda_{-n}^2
\Big(1 - \lambda_0^2 \lambda_{-1}^2 \cdots
\lambda_{-(n-1)}^2 \sum_{i=1}^\eta t_i^2 \int_0^1 \frac 1
{s^{n+2}} \D \tau_i(s)\Big)
   \\
& \overset{\eqref{1=l02'} \& \eqref{2>}}= \lambda_{-n}^2
\Big(1 - \frac 1{\zeta_2} \frac {\zeta_{2}}{\zeta_{3}}
\cdots \frac {\zeta_{n}}{\zeta_{n+1}} \sum_{i=1}^\eta t_i^2
\int_0^1 \frac 1 {s^{n+2}} \D \tau_i(s)\Big) \notag
   \\
& \hspace{4.5ex} = \lambda_{-n}^2 \Big(1 - \frac
{1}{\zeta_{n+1}} \sum_{i=1}^\eta t_i^2 \int_0^1 \frac 1
{s^{n+2}} \D \tau_i(s)\Big) = \lambda_{-n}^2 \frac
{\zeta_{n+2}}{\zeta_{n+1}}, \notag
   \end{align}
which shows that \eqref{1>} and \eqref{2>} hold for
$k=n+1$. This proves our claim. Arguing as in the proof of
\eqref{1=106} with $n=\kappa-1$ and using \eqref{lam=1+'}
in place of \eqref{lam=1+}, we get
   \begin{align} \label{2>'}
   1 \Le \lambda_{-(\kappa-1)}^2 \frac{\zeta_{\kappa+1}}
{\zeta_\kappa},
   \end{align}
which, when combined with \eqref{1>} applied to
$k=\kappa-1$, implies \eqref{ka+3}. Hence (i) is proved.

(ii) Assume that $\teb := \{t_{i}\}_{i=1}^{\eta} \subseteq
(0,\infty)$ satisfies \eqref{ka+1}, \eqref{ka+2} and
\eqref{ka+3}. Our aim now is to define the system
$\lambdab= \{\lambda_v\}_{v \in V_{\eta,\kappa}^\circ}
\subseteq (0,\infty)$. For $i \in J_\eta$, we set
   \begin{align}    \label{lamij-ch}
   \begin{aligned}
   \lambda_{i,j} =
   \begin{cases}
   t_i & \text{for } j=1,
   \\[1ex]
   \sqrt{1+\tau_i([0,1])} & \text{for } j=2,
   \\[1ex]
   \displaystyle{\sqrt{\frac{1 + \int_0^1 (1+ \ldots +
s^{j-2}) \D \tau_i(s)}{1 + \int_0^1 (1+ \ldots + s^{j-3})
\D \tau_i(s)}}} & \text{for } j \Ge 3.
   \end{cases}
   \end{aligned}
   \end{align}

If $\kappa=0$, then the weights $\lambdab =
\{\lambda_v\}_{v \in V_{\eta,0}^\circ}$ just defined
satisfy \eqref{zgod0-ch} and \eqref{zgod-ch} (the
latter because of \eqref{ka+2}).

If $\kappa=1$, then $\lambda_0$ can be considered as
any number from the interval
$[1/\sqrt{\zeta_2},\infty)$. Clearly, \eqref{zgod0-ch}
is valid. It follows from \eqref{ka+2} with $\kappa=1$
that \eqref{zgod-ch'} holds. Hence, we can reverse the
reasoning in \eqref{1=l02} and verify that
\eqref{lam=1+'} is valid for $\kappa=1$.

If $\kappa \Ge 2$ and $\vartheta \in \Big[
\sqrt{\frac{\zeta_\kappa}{\zeta_{\kappa+1}}},
\infty\Big)$, then we define the weights
$\{\lambda_{-k}\}_{k=0}^{\kappa-1}$ by
   \begin{align}  \label{zeta-ch}
\lambda_{-k} =
   \begin{cases}
   \frac 1 {\sqrt{\zeta_2}} & \text{for } k=0,
   \\[1.5ex]
   \sqrt{\frac{\zeta_{k+1}}{\zeta_{k+2}}} & \text{for } k
\in J_{\kappa-2},
   \\[1.5ex] \vartheta & \text{for } k=\kappa-1.
   \end{cases}
   \end{align}
(Of course, if $\kappa=2$, then the middle expression
in \eqref{zeta-ch} does not appear.) According to
\eqref{ka+1} and \eqref{ka+3}, the above definition is
correct. Reversing the reasonings in \eqref{1=l02'}
and \eqref{1=106}, we deduce that \eqref{lam=1+}
holds. Arguing as in \eqref{1=106} with $n=\kappa -1$,
we see that \eqref{lam=1+'} holds. As above, we
conclude that \eqref{zgod0-ch} and \eqref{zgod-ch'}
are valid as well.

Thus, it remains to show that the weighted shift
$\slam$ is bounded. It follows from \eqref{lamij-ch}
and Lemma \ref{monot-ca} that for every $i\in J_\eta$
the sequence $\{\lambda_{i,j}\}_{j=2}^\infty$ is
monotonically decreasing. As a consequence, we have
   \begin{align} \label{copo}
\sup_{i\in J_\eta} \sup_{j\Ge 2} \lambda_{i,j}^2
= \sup_{i\in J_\eta} (1+\tau_i([0,1])) < \infty.
   \end{align}
Combining this with \eqref{ka+1} and $\kappa <
\infty$, we see that $\sup_{u \in V_{\eta,\kappa}}
\sum_{v\in \dzi{u}} \lambda_v^2 < \infty$. Hence, by
Proposition \ref{ogrs}, $\slam \in
\ogr{\ell^2(V_{\eta,\kappa})}$. Applying Corollary
\ref{omega2-ch}, we conclude that the weighted shift
$\slam$ is completely hyperexpansive and $\tau_{i,1} =
\tau_i$ for all $i \in J_\eta$. The uniqueness
assertion in (ii) can be deduced from
\eqref{lamij-ch}, \eqref{1=l02'} and \eqref{2>}.

We now prove the ``moreover'' part of (ii). If $\kappa=0$,
then the top equality in \eqref{norm-ch} follows from
\eqref{pnor} and \eqref{copo}. Assume that $\kappa \Ge 1$.
Since the sequence $\{\|\slam^n
e_{-\kappa}\|^2\}_{n=0}^\infty$ is completely alternating,
we infer from Lemma \ref{monot-ca} that the corresponding
sequence of quotients
   \begin{align} \label{ch-dec}
\Big\{\lambda_{-(\kappa-1)}^2, \ldots, \lambda_0^2,
\sum_{i=1}^\eta t_i^2, \frac{\sum_{i=1}^\eta t_i^2
(1+\tau_i([0,1]))}{\sum_{i=1}^\eta t_i^2}, \ldots\Big\}
   \end{align}
is monotonically decreasing. In particular, we have
$\lambda_{-(\kappa-1)}^2 \Ge \ldots \Ge \lambda_0^2 \Ge
\sum_{i=1}^\eta t_i^2$. This, combined with \eqref{pnor}
and \eqref{copo}, yields the bottom equality in
\eqref{norm-ch}.
   \end{proof}
   \begin{rem} \label{extrch}
As in the subnormal case, we can single out a class of
extremal completely hyperexpansive weighted shifts on
$\tcal_{\eta,\kappa}$ (cf.\ Remark \ref{forwhile}). Suppose
that $\kappa \in \nbb$. We say that a completely
hyperexpansive weighted shift $\slam \in
\ogr{\ell^2(V_{\eta,\kappa})}$ on $\tcal_{\eta,\kappa}$
with nonzero weights $\lambdab = \{\lambda_v\}_{v \in
V_{\eta,\kappa}^\circ}$ is {\em extremal} if
   \begin{align*}
\|\slam e_{\koo}\| = \min \|S_{\tilde \lambdab} e_{\koo}\|,
   \end{align*}
where the minimum is taken over all completely
hyperexpansive weighted shifts $S_{\tilde \lambdab} \in
\ogr{\ell^2(V_{\eta,\kappa})}$ on $\tcal_{\eta,\kappa}$
with nonzero weights $\tilde \lambdab = \{\tilde
\lambda_v\}_{v \in V_{\eta,\kappa}^\circ}$ such that
\linebreak $\slamr{-\kappa+1} =
S_{\tilde\lambdab_\rightarrow\!(-\kappa+1)}$, or
equivalently that $\lambda_v = \tilde \lambda_v$ for all $v
\neq -\kappa + 1$. It follows from Corollary
\ref{omega2-ch} that a completely hyperexpansive weighted
shift $\slam \in \ogr{\ell^2(V_{\eta,\kappa})}$ on
$\tcal_{\eta,\kappa}$ with nonzero weights $\lambdab =
\{\lambda_v\}_{v \in V_{\eta,\kappa}^\circ}$ is extremal if
and only if $\slam$ satisfies the condition (\mbox{ii-a})
with the inequality in \eqref{lam=1+'} replaced by
equality; in other words, $\slam$ is extremal if and only
if $\slam$ satisfies the strong consistency condition at
each vertex $u \in V_{\eta,\kappa}$ (cf.\
\eqref{consist-ch'}).
   \end{rem}
As stated in Theorem \ref{par-ch} below, extremal
completely hyperexpansive weight\-ed shifts on
$\tcal_{\eta,\kappa}$ with a fixed system of representing
measures $\{\tau_{i,1}\}_{i=1}^\eta$ are in one-to-one and
onto correspondence with sequences $\{t_{i}\}_{i=1}^\eta
\subseteq (0,\infty)$ satisfying the conditions
\eqref{ka+1}, \eqref{ka+2} and \eqref{ka+3}.
   \begin{thm} \label{par-ch}
Let $\eta \in \{2,3, \ldots\} \sqcup \{\infty\}$ and
$\kappa \in \zbb_+$. Assume $\boldsymbol
\tau=\{\tau_i\}_{i=1}^\eta$ is a sequence of positive
Borel measures on $[0,1]$ such that $\sup_{i \in
J_\eta} \tau_{i}([0,1]) < \infty$. Let \idx{$\mathscr
U_{\eta,\kappa}^{\boldsymbol \tau}$}{78} be the set of
all sequences $\teb = \{t_{i}\}_{i=1}^\eta \subseteq
(0,\infty)$ satisfying \eqref{ka+1}, \eqref{ka+2} and
\eqref{ka+3}, and let \idx{$\mathscr
V_{\eta,\kappa}^{\boldsymbol \tau}$}{79} be the set of
all completely hyperexpansive weighted shifts $\slam
\in \ogr{\ell^2(V_{\eta,\kappa})}$ on
$\tcal_{\eta,\kappa}$ with positive weights $\lambdab=
\{\lambda_v\}_{v \in V_{\eta,\kappa}^\circ}$ such that
$\tau_{i,1} = \tau_i$ for all $i \in J_\eta$. Denote
by \idx{$\mathscr W_{\eta,\kappa}^{\boldsymbol
\tau}$}{80} the set of all completely hyperexpansive
weighted shifts $\slam \in \mathscr
V_{\eta,\kappa}^{\boldsymbol \tau}$ which are
extremal. If $\kappa=0$, then the mapping
   \begin{align*}
\varPhi_{\eta,0}\colon \mathscr
U_{\eta,0}^{\boldsymbol \tau} \ni \teb \mapsto \slam
\in \mathscr V_{\eta,0}^{\boldsymbol \tau}
   \end{align*}
defined by \eqref{lamij-ch} is a bijection. If
$\kappa\Ge 1$, then the mapping
   \begin{align*}
   \varPhi_{\eta,\kappa}\colon \mathscr
U_{\eta,\kappa}^{\boldsymbol \tau} \ni \teb \mapsto
\slam \in \mathscr W_{\eta,\kappa}^{\boldsymbol \tau}
   \end{align*}
defined by \eqref{lamij-ch} and
   \begin{align}  \label{zeta-ch'}
\lambda_{-k} =
   \begin{cases}
\frac 1 {\sqrt{\zeta_2}} & \text{for } k=0,
   \\[1.5ex]
\sqrt{\frac{\zeta_{k+1}}{\zeta_{k+2}}} & \text{for } k \in
J_{\kappa-1},
   \end{cases}
   \end{align}
is a bijection $($see \eqref{5USD} for the definition of
$\zeta_k$$)$. Moreover, if $\slam \in \mathscr
V_{\eta,\kappa}^{\boldsymbol \tau}$, then $\teb
=\{t_{i}\}_{i=1}^{\eta} \in \mathscr
U_{\eta,\kappa}^{\boldsymbol \tau}$ with $t_{i}:=
\lambda_{i,1}$, $\lambda_{-\kappa+1} \in
\Big[\sqrt{\frac{\zeta_\kappa}
{\zeta_{\kappa+1}}},\infty\Big)$ and $\lambda_v=\tilde
\lambda_v$ for all $v\neq -\kappa +1$, where
$S_{\tilde\lambdab}=\varPhi_{\eta,\kappa}(\teb)$.
Conversely, if $\teb \in \mathscr
U_{\eta,\kappa}^{\boldsymbol \tau}$ and
$S_{\lambdab}=\varPhi_{\eta,\kappa}(\teb)$, then for every
$\vartheta \in
\Big[\sqrt{\frac{\zeta_\kappa}{\zeta_{\kappa+1}}},\infty\Big)$,
the weighted shift $S_{\tilde\lambdab}$ with weights
$\tilde \lambdab = \{\tilde \lambda_v\}_{v \in
V_{\eta,\kappa}^\circ}$ given by $\tilde
\lambda_v=\lambda_v$ for all $v\neq -\kappa +1$, and
$\tilde \lambda_{-\kappa +1} = \vartheta$ is a member of
$\mathscr V_{\eta,\kappa}^{\boldsymbol \tau}$.
   \end{thm}
   \begin{proof}
Apply Lemma \ref{omega3-ch} as well as its proof.
   \end{proof}
Our next aim is to find necessary and sufficient
conditions for the parameterizing set $\mathscr
U_{\eta,\kappa}^{\boldsymbol \tau}$ to be nonempty.
   \begin{pro} \label{nascfch}
If $\eta$, $\kappa$, $\boldsymbol
\tau=\{\tau_i\}_{i=1}^\eta$ and $\mathscr
U_{\eta,\kappa}^{\boldsymbol \tau}$ are as in Theorem {\em
\ref{par-ch}}, then the following two conditions are
equivalent\/{\em :}
   \begin{enumerate}
   \item[(i)] $\mathscr
U_{\eta,\kappa}^{\boldsymbol \tau} \neq
\varnothing$,
   \item[(ii)] $\int_0^1 \frac 1
{s^{\kappa+1}} \D \tau_i(s) < \infty$ for all $i
\in J_\eta$, and $\int_0^1 \sum_{l=1}^{\kappa+1}
\frac 1{s^l} \D \tau_{i_0}(s) < 1$ for some
$i_0\in J_\eta$.
   \end{enumerate}
   \end{pro}
   \begin{proof}
Note that if (i) or (ii) holds, then $\int_0^1 \frac 1
{s^j} \D \tau_i(s) < \infty$ for all $i \in J_\eta$ and
$j \in J_{\kappa+1}$.

First, we consider the case when $\kappa = 0$.

(i)$\Rightarrow$(ii) Suppose that, contrary to our claim,
$\int_0^1 \frac 1 {s} \D \tau_i(s) \Ge 1$ for all $i \in
J_\eta$. Take $\{t_{i}\}_{i=1}^{\eta} \in \mathscr
U_{\eta,\kappa}^{\boldsymbol \tau}$. Then $\sum_{i=1}^\eta
t_i^2 \int_0^1 \frac 1 {s} \D \tau_i(s) \Ge \sum_{i=1}^\eta
t_i^2$, which contradicts \eqref{ka+3}.

(ii)$\Rightarrow$(i) Set
   \begin{align*}
J_\eta^{+} = \Big\{i \in J_\eta \colon 1 -
\int_0^1 \frac 1 s \D \tau_i(s)>0 \Big\}, \quad
J_\eta^{-}=\Big\{i \in J_\eta \colon 1 - \int_0^1
\frac 1 s \D \tau_i(s) \Le 0\Big\}.
   \end{align*}
Then $J_\eta = J_\eta^{+} \sqcup J_\eta^{-}$ and $i_0
\in J_\eta^{+}$. Let $\{\tilde t_i\}_{i \in
J_\eta^{+}}$ and $\{\tilde t_i\}_{i \in J_\eta^{-}}$
be systems of positive real numbers such that (see the
convention preceding Proposition \ref{desc})
   \begin{align*}
\text{$\sum_{i \in J_\eta^{+}} \tilde t_i^2 < \infty$,
$\sum_{i \in J_\eta^{-}} \tilde t_i^2 < \infty$ and
$\beta:=\sum_{i \in J_\eta^{-}} \tilde t_i^2
\Big(\int_0^1 \frac 1 s \D \tau_i(s) - 1\Big) <
\infty$.}
   \end{align*}
Then clearly $\alpha:=\sum_{i \in J_\eta^{+}} \tilde
t_i^2 (1 - \int_0^1 \frac 1 s \D \tau_i(s)) < \infty$.
If we consider the sequence $\teb =
\{t_i\}_{i=1}^{\eta}$ given by
   \begin{align*}
t_i =
   \begin{cases}
r_1 \tilde t_i & \text{for } i \in J_\eta^{+},
   \\
r_2 \tilde t_i & \text{for } i \in J_\eta^{-},
   \end{cases}
   \end{align*}
where $r_1,r_2 \in \rbb$, then the inequality in
\eqref{ka+2} takes the form $1 \Le r_1^2 \alpha -
r_2^2\beta$. Since $\alpha \in (0,\infty)$, we deduce that
this inequality has a solution in positive reals $r_1$ and
$r_2$. Hence the sequence $\{t_i\}_{i=1}^{\eta}$ satisfies
\eqref{ka+1} and \eqref{ka+2}, and consequently
\eqref{ka+3}. In fact, the sequence $\teb$ can be chosen so
as to satisfy \eqref{ka+1} and the equality
$\sum_{i=1}^\eta t_{i}^2 = 1 + \sum_{i=1}^\eta t_{i}^2
\int_0^1 \frac 1 s \D \tau_i(s)$.

Consider now the case when $\kappa \Ge 1$.

(i)$\Rightarrow$(ii) Repeat the argument used in the proof
of the case $\kappa=0$.

(ii)$\Rightarrow$(i) Set
   \allowdisplaybreaks
   \begin{align*}
J_\eta^{++} & = \Big\{i \in J_\eta \colon 1 - \int_0^1
\frac 1 s \D \tau_i(s)>0 \;\; \& \;\; 1-\int_0^1
\sum_{l=1}^{\kappa+1} \frac 1{s^l} \D \tau_i(s) > 0
\Big\},
   \\
J_\eta^{+-} & = \Big\{i \in J_\eta \colon 1 - \int_0^1
\frac 1 s \D \tau_i(s)>0 \;\; \& \;\; 1-\int_0^1
\sum_{l=1}^{\kappa+1} \frac 1{s^l} \D \tau_i(s) \Le 0
\Big\},
   \\
J_\eta^{-+} & = \Big\{i \in J_\eta \colon 1 - \int_0^1
\frac 1 s \D \tau_i(s) \Le 0 \;\; \& \;\; 1-\int_0^1
\sum_{l=1}^{\kappa+1} \frac 1{s^l} \D \tau_i(s) > 0
\Big\},
   \\
J_\eta^{--} & = \Big\{i \in J_\eta \colon 1 - \int_0^1
\frac 1 s \D \tau_i(s) \Le 0 \;\; \& \;\; 1-\int_0^1
\sum_{l=1}^{\kappa+1} \frac 1{s^l} \D \tau_i(s) \Le 0
\Big\}.
   \end{align*}
It is clear that $J_\eta^{++} = \big\{i \in J_\eta \colon 1
- \int_0^1 \sum_{l=1}^{\kappa+1} \frac 1{s^l} \D
\tau_i(s)>0\big\} \neq \varnothing$, $J_\eta^{-+} =
\varnothing$ and $J_\eta^{--} = \big\{i \in J_\eta \colon 1
- \int_0^1 \frac 1 s \D \tau_i(s) \Le 0\big\}$. Moreover,
the sets $J_\eta^{++}$, $J_\eta^{+-}$ and $J_\eta^{--}$ are
pairwise disjoint and $J_\eta = J_\eta^{++} \cup
J_\eta^{+-} \cup J_\eta^{--}$. Arguing as in Section
\ref{mod-sub}, we find square summable systems $\{\tilde
t_i\}_{i \in J_\eta^{++}}, \{\tilde t_i\}_{i \in
J_\eta^{+-}}, \{\tilde t_i\}_{i \in J_\eta^{--}} \subseteq
(0,\infty)$ such that
   \allowdisplaybreaks
   \begin{align*}
1 & \phantom{:}= \sum_{i \in J_\eta^{++}} \tilde t_i^2
\Big(1 - \int_0^1 \frac 1 s \D \tau_i(s)\Big),
   \\
\vartheta_1 &:=\sum_{i \in J_\eta^{+-}} \tilde t_i^2 \Big(1
- \int_0^1 \frac 1 s \D \tau_i(s)\Big) < \infty,
   \\
\vartheta_2 &:= \sum_{i \in J_\eta^{--}} \tilde t_i^2
\Big(\int_0^1 \frac 1 s \D \tau_i(s) -1\Big) < \infty,
   \\
\alpha_1 &:=\sum_{i \in J_\eta^{+-}} \tilde t_i^2
\Big(\int_0^1 \sum_{l=1}^{\kappa+1} \frac 1 {s^l}\D
\tau_i(s)-1\Big) < \infty,
   \\
\alpha_2 & :=\sum_{i \in J_\eta^{--}} \tilde t_i^2
\Big(\int_0^1 \sum_{l=1}^{\kappa+1} \frac 1 {s^l} \D
\tau_i(s)-1\Big) < \infty.
   \end{align*}
Since the system $\{\tilde t_i\}_{i \in J_\eta^{++}}$ is
square summable, we deduce that
   \begin{align*}
0 < \alpha_0 :=\sum_{i \in J_\eta^{++}} \tilde t_i^2 \Big(1
- \int_0^1 \sum_{l=1}^{\kappa+1} \frac 1 {s^l} \D
\tau_i(s)\Big) < \infty.
   \end{align*}
If we consider the sequence $\teb = \{ t_i\}_{i=1}^\eta$
given by
   \begin{align*}
t_i =
   \begin{cases}
r_0\tilde t_i & \text{for } i \in J_\eta^{++},
   \\
r_1\tilde t_i & \text{for } i \in J_\eta^{+-},
   \\
r_2\tilde t_i & \text{for } i \in J_\eta^{--},
   \end{cases}
   \end{align*}
where $r_0,r_1,r_2 \in \rbb$, then the conditions
\eqref{ka+2} and \eqref{ka+3} take the following form:
   \begin{align} \label{syst2}
1 = r_0^2 + r_1^2 \vartheta_1 - r_2^2 \vartheta_2, \quad
r_0^2 \alpha_0 - r_1^2 \alpha_1 - r_2^2 \alpha_2 > 0.
   \end{align}
Since $\alpha_0 \in (0,\infty)$, we easily verify that
there exist positive real numbers $r_0$, $r_1$ and $r_2$,
which satisfy \eqref{syst2}. This completes the proof.
   \end{proof}
It may be worth noting that the proof of Proposition
\ref{nascfch} also works without assuming that $\sup_{i \in
J_\eta} \tau_{i}([0,1]) < \infty$.

We are now ready to give a method of constructing all
possible completely hyperexpansive bounded weighted
shifts with nonzero weights on the directed tree
$\tcal_{\eta,\kappa}$ with $\kappa < \infty$. The
reader is asked to compare this method with that for
subnormal weighted shifts described in Procedure
\ref{twolin}. In particular, Procedure \ref{twolin}
enables us to construct bounded subnormal weighted
shifts on $\tcal_{\eta,\infty}$ with nonzero weights
which are not isometric. In view of Corollary
\ref{omega2-ch}\,(iii), this never happens in the case
of completely hyperexpansive weighted shifts, because
such operators are isometric. Isometric weighted
shifts are discussed in Propositions \ref{izometria}
and \ref{isokappa} (the case of the directed tree
$\tcal_{\eta,\kappa}$) and in Corollary \ref{chariso}
(the general situation).
   \begin{opis} \label{proc-ch}
Fix $\eta \in \{2,3, \ldots\} \sqcup \{\infty\}$ and
$\kappa \in \zbb_+$. Let $\{\tau_i\}_{i=1}^\eta$ be a
sequence of positive Borel measures on $[0,1]$ such that
$\sup_{i \in J_\eta} \tau_{i}([0,1]) < \infty$, $\int_0^1
\frac 1 {s^{\kappa+1}} \D \tau_i(s) < \infty$ for all $i
\in J_\eta$ and $\int_0^1 \sum_{l=1}^{\kappa+1} \frac
1{s^l} \D \tau_{i_0}(s) < 1$ for some $i_0\in J_\eta$.
Using Proposition \ref{nascfch}, we get a sequence $\teb =
\{t_{i}\}_{i=1}^\eta$ of positive real numbers satisfying
the conditions \eqref{ka+1}, \eqref{ka+2} and \eqref{ka+3}.
Next, applying Theorem \ref{par-ch}, we get a completely
hyperexpansive weighted shift $\slam \in
\ogr{\ell^2(V_{\eta,\kappa})}$ on $\tcal_{\eta,\kappa}$
with positive weights $\lambdab= \{\lambda_v\}_{v \in
V_{\eta,\kappa}^\circ}$ such that $\lambda_{i,1}=t_i$ and
$\tau_{i,1} = \tau_i$ for all $i \in J_\eta$.
   \end{opis}
   \subsection{\label{s7.4}Completion of weights on
$\tcal_{\eta,\kappa}$}
   Using the modelling procedure described in Section
\ref{mod-ch}, we give a deeper insight into complete
hyperexpansivity of weighted shifts on
$\tcal_{\eta,\kappa}$. We begin with writing some
estimates (from above and from below) for
$\sum_{i=1}^\eta \lambda_{i,1}^2$ and
$\tau_{j,1}([0,1])$, $j\in J_\eta$. Under some
circumstances, this enables us to simplify the formula
\eqref{norm-ch} for the norm of a completely
hyperexpansive weighted shift on the directed tree
$\tcal_{\eta,\kappa}$.
   \begin{pro}\label{norm-sim}
Let $\eta \in \{2,3, \ldots\} \sqcup \{\infty\}$ and
$\kappa \in \zbb_+$. Assume that $\slam \in
\ogr{\ell^2(V_{\eta,\kappa})}$ is a completely
hyperexpansive weighted shift on $\tcal_{\eta,\kappa}$ with
positive weights $\lambdab= \{\lambda_v\}_{v \in
V_{\eta,\kappa}^\circ}$. Set $\tau_i=\tau_{i,1}$ and
$t_i=\lambda_{i,1}$ for $i \in J_\eta$. Then the following
assertions hold.
   \begin{enumerate}
   \item[(i)] There exists $i_0 \in J_\eta$ such
that $\tau_{i_0}([0,1]) < \frac 1 {\kappa+1}$.
   \vspace{.5ex}
   \item[(ii)]  $\sum_{i=1}^\eta
t_i^2 \Ge 1;$ moreover, $\sum_{i=1}^\eta t_i^2 = 1$ if and
only if either $\kappa=0$ and $\slam$ is an isometry or
$\kappa \Ge 1$ and $\slamr{-\kappa +1}$ is an isometry
$($if $\slam$ is extremal, then $\sum_{i=1}^\eta t_i^2 = 1$
if and only if $\slam$ is an isometry$)$.
   \vspace{.5ex}
   \item[(iii)] If $\kappa \Ge 1$, then
$\sum_{i=1}^\eta t_i^2 < \frac {\kappa + 1}{\kappa}$.
   \vspace{.5ex}
   \item[(iv)] $\sum_{i=1}^\eta t_i^2 \Ge
1+\inf\{\tau_i([0,1])\colon i \in J_\eta\}$.
   \vspace{.5ex}
   \item[(v)] If $\tau_i([0,1]) =
\tau_1([0,1])$ for all $i\in J_\eta$, then
   \begin{align} \label{formulaF1}
\|\slam\|^2 =
   \begin{cases}
\sum_{i=1}^\eta t_i^2 & \text{for } \kappa=0,
   \\[1ex]
\lambda_{-(\kappa-1)}^2 & \text{for } \kappa \Ge
1.
   \end{cases}
   \end{align}
   \end{enumerate}
   \end{pro}
   \begin{proof}
(i) It follows from Theorem \ref{par-ch} and Proposition
\ref{nascfch} that the exists $i_0\in J_\eta$ such that
$\int_0^1 \sum_{l=1}^{\kappa+1} \frac 1{s^l} \D
\tau_{i_0}(s) < 1$. As a consequence, we have
   \begin{align*}
(\kappa+1) \tau_{i_0}([0,1]) \Le \int_0^1
\sum_{l=1}^{\kappa+1} \frac 1{s^l} \D \tau_{i_0}(s) <1.
   \end{align*}

(ii) The inequality $\sum_{i=1}^\eta t_i^2 \Ge 1$ follows
from \eqref{ka+2}. Suppose that $\sum_{i=1}^\eta t_i^2 =
1$. Using \eqref{ka+2} again, we see that $\tau_i=0$ for
all $i \in J_\eta$, which in view of Theorem \ref{par-ch}
implies that $\sum_{v\in\dzi u} \lambda_v^2=1$ for all
$u\in V_{\eta,\kappa}$ when $\kappa = 0$, and for all $u
\in V_{\eta,\kappa} \setminus \{-\kappa\}$ when $\kappa \Ge
1$. Thus, by Corollary \ref{chariso}, $\slam$ is an
isometry when $\kappa=0$, and $\slamr{-\kappa +1}$ is an
isometry when $\kappa \Ge 1$. The reverse implication is
obvious. A similar reasoning applies to the case when
$\slam$ is extremal (cf.\ Remark \ref{extrch}).

(iii) It follows from Lemma \ref{omega3-ch}\,(i)
that
   \begin{align*}
\sum_{i=1}^\eta t_{i}^2 \overset{\eqref{ka+3}}
> \sum_{i=1}^\eta  t_{i}^2 \int_0^1
\Big(\frac 1 s + \ldots + \frac 1
{s^{\kappa+1}}\Big) \D \tau_i(s)
   & \hspace{1.5ex} \Ge (\kappa + 1) \sum_{i=1}^\eta
t_{i}^2 \int_0^1 \frac 1 s \D \tau_i(s)
   \\
& \overset{\eqref{ka+2}} = (\kappa + 1)
\Big(\sum_{i=1}^\eta t_{i}^2 - 1\Big),
   \end{align*}
which implies that $\sum_{i=1}^\eta t_i^2 < \frac {\kappa +
1}{\kappa}$.

(iv) Since the sequence $\{\|\slam^n
e_{0}\|^2\}_{n=0}^\infty$ is completely alternating,
we infer from \eqref{lamij-ch} and Lemma
\ref{monot-ca} that the corresponding sequence of
quotients
   \begin{align*}
\Big\{\sum_{i=1}^\eta t_i^2, \frac{\sum_{i=1}^\eta
 t_i^2 (1+\tau_i([0,1]))}{\sum_{i=1}^\eta t_i^2},
\ldots\Big\}
   \end{align*}
is monotonically decreasing. This implies that
$\sum_{i=1}^\eta t_i^2 \Ge 1+\inf_{i\in
J_\eta}\tau_i([0,1])$.

(v) Apply (iv), \eqref{norm-ch} and \eqref{ch-dec}.
   \end{proof}
Regarding parts (ii) and (iii) of Proposition
\ref{norm-sim}, it is worth noting that if $\kappa=0$, then
there is no upper bound for $\sum_{i=1}^\eta t_i^2$.
Moreover, for each $\varTheta \in [1,\frac
{\kappa+1}\kappa)$ (with the usual convention that $\frac 1
0 = \infty$) there exists a completely hyperexpansive
weighted shift $\slam \in \ogr{\ell^2(V_{\eta,\kappa})}$ on
$\tcal_{\eta,\kappa}$ with positive weights $\lambdab =
\{\lambda_v\}_{v \in V_{\eta,\kappa}^\circ}$ such that
$\varTheta = \sum_{i=1}^\eta \lambda_{i,1}^2$. In fact, we
can prove a more general result (see Example \ref{sub2} for
the discussion of the case of subnormality).
   \begin{pro}\label{kap-ch}
Let $\eta \in \{2,3, \ldots\} \sqcup \{\infty\}$ and
$\kappa \in \zbb_+$. If $\{t_i\}_{i=1}^\eta$ is a sequence
of positive real numbers such that $1 \Le \sum_{i=1}^\eta
t_i^2 < \frac{\kappa +1}\kappa$, then there exists a
completely hyperexpansive weighted shift $\slam \in
\ogr{\ell^2(V_{\eta,\kappa})}$ on $\tcal_{\eta,\kappa}$
with positive weights $\lambdab = \{\lambda_v\}_{v \in
V_{\eta,\kappa}^\circ}$ such that $\lambda_{i,1} = t_i$ for
all $i \in J_\eta$.
   \end{pro}
   \begin{proof}
Set $\varTheta = \sum_{i=1}^\eta t_i^2$. According to our
assumptions, we have $\varTheta \in [1,\frac{\kappa
+1}\kappa)$. Define the sequence $\{\tau_i\}_{i=1}^\eta$ of
positive Borel measures on $[0,1]$ by $\tau_i =
\frac{\varTheta-1}{\varTheta} \delta_1$ for $i \in J_\eta$.
It is easily seen that
   \begin{align*}
\sum_{i=1}^\eta t_i^2 = 1 + \sum_{i=1}^\eta t_i^2 \int_0^1
\frac 1 s \D \tau_i(s).
   \end{align*}
Hence, if $\kappa=0$, then by applying Lemma
\ref{omega3-ch}\,(ii) we get the required weighted shift
$\slam$. If $\kappa \Ge 1$, then $1 \Le \varTheta <
\frac{\kappa +1}\kappa$ implies that $1-(\kappa + 1)
\frac{\varTheta -1}\varTheta > 0$. Thus
   \begin{align*}
\sum_{i=1}^\eta t_{i}^2 \Big(1 - \int_0^1 \Big(\frac 1 s +
\ldots + \frac 1 {s^{\kappa+1}}\Big) \D \tau_i(s)\Big) =
\sum_{i=1}^\eta t_{i}^2
\Big(1-(\kappa+1)\frac{\varTheta-1}\varTheta\Big)>0,
   \end{align*}
which enables us once more to employ Lemma
\ref{omega3-ch}\,(ii).
   \end{proof}
   If total masses of representing measures $\tau_i$ are
not identical, then the formula \eqref{formulaF1} for the
norm of $\slam$ is no longer true.
   \begin{exa}
Consider the case when $\kappa=1$, $\eta=2$, $\tau_1=0$ and
$\tau_2=\delta_1$. Set $ t_1=1$ and take any $ t_2 \in (0,
1)$. We easily verify that the conditions
\eqref{ka+1}-\eqref{ka+3} are satisfied. By Theorem
\ref{par-ch}, there exists a unique extremal completely
hyperexpansive weighted shift $\slam \in
\ogr{\ell^2(V_{2,1})}$ on $\tcal_{2,1}$ with positive
weights $\lambdab= \{\lambda_v\}_{v \in V_{2,1}^\circ}$
such that $\lambda_{i,1}= t_i$ and $\tau_{i,1} = \tau_i$
for $i = 1,2$. Applying \eqref{zeta-ch'}, we see that
$\lambda_0=\frac1{\sqrt{1- t_2^2}}$.

If $ t_2 \in (0,\frac 1{\sqrt 2})$, then $\lambda_0\in
(1,\sqrt 2)$ and so $\lambda_0^2 < 2 =
\max\{1+\tau_i([0,1])\colon i=1,2\}$. Hence, by
\eqref{norm-ch}, $\|\slam\|^2= \max\{1+\tau_i([0,1])\colon
i=1,2\}$, which means that \eqref{formulaF1} does not hold.

If $ t_2 = \frac 1{\sqrt 2}$, then $\lambda_0 = \sqrt 2$
and $\|\slam\|^2= \lambda_0^2=\max\{1+\tau_i([0,1])\colon
i=1,2\}$. Thus \eqref{formulaF1} holds.

Finally, if $ t_2 \in (\frac 1{\sqrt 2},1)$, then
$\lambda_0\in (\sqrt 2,\infty)$, which implies that
$\|\slam\|^2= \lambda_0^2 > \max\{1+\tau_i([0,1])\colon
i=1,2\}$. Therefore \eqref{formulaF1} also holds.
   \end{exa}
The following result is a counterpart of Proposition
\ref{2curto} for completely hyperexpansive weighted shifts
(see Section \ref{k-step} for the definition of $k$-step
backward extendibility). The reader should be aware of the
difference between the condition (ii) of Proposition
\ref{2curto} and its counterpart in Proposition
\ref{2curto-ch} below.
   \begin{pro}\label{2curto-ch}
Let $\eta \in \{2,3, \ldots\} \sqcup \{\infty\}$ and
$\kappa \in \zbb_+$. If for every $i\in J_\eta$, $S_i$
is a bounded unilateral classical weighted shift with
positive weights $\{\alpha_{i,n}\}_{n=1}^\infty$, then
the following two conditions are equivalent\/{\em :}
   \begin{enumerate}
   \item[(i)] there exists a system $\lambdab =
\{\lambda_{v}\}_{v \in V_{\eta,\kappa}^\circ}$ of positive
scalars such that the weighted shift $\slam$ on the
directed tree $\tcal_{\eta,\kappa}$ is bounded and
completely hyperexpansive, and
   \begin{align} \label{alla-ch}
\alpha_{i,n}=\lambda_{i,n+1}, \quad n\in \nbb, \, i\in
J_\eta,
   \end{align}
   \item[(ii)] the operator $S_i$ is
completely hyperexpansive and $\int_0^1 \frac 1
{s^{\kappa+1}} \D \tau_i(s) < \infty$ for every $i \in
J_\eta$ $($$\tau_i$ is the representing measure of
$S_i$$)$, $S_{i_0}$ has a completely hyperexpansive
$(\kappa+1)$-step backward extension for some $i_0\in
J_\eta$, and $\sup_{i\in J_\eta} \|S_i\| < \infty$.
   \end{enumerate}
   \end{pro}
   \begin{proof}  We argue essentially as in the
proof of Proposition \ref{2curto}.

   (i)$\Rightarrow$(ii) By \eqref{alla-ch}, the
operator $\slam|_{\ell^2(\des{i,1})}$ is unitarily
equivalent to $S_i$, and hence $\sup_{i\in J_\eta}
\|S_i\| \Le \|\slam\|$. Since $\|\slam^n e_{i,1}\|^2 =
\|S_i^n e_0\|^2$ for all $n \in \zbb_+$, we see that
$S_i$ is completely hyperexpansive and $\tau_{i,1}$ is
the representing measure of $S_i$. Owing to Theorem
\ref{par-ch} and Proposition \ref{nascfch}, $\int_0^1
\frac 1 {s^{\kappa+1}} \D \tau_{i,1}(s) < \infty$ for
all $i \in J_\eta$, and $\int_0^1
\sum_{l=1}^{\kappa+1} \frac 1{s^l} \D \tau_{i_0,1}(s)
< 1$ for some $i_0\in J_\eta$. Applying
\eqref{jab-ju-st} to $S_{i_0}$, we get (ii).

(ii)$\Rightarrow$(i) Since the weights of a completely
hyperexpansive unilateral classical weighted shift are
monotonically decreasing (use \eqref{repch} and Lemma
\ref{monot-ca}), we deduce that $\|S_i\|^2 = 1 +
\tau_i([0,1])$. This yields $\sup_{i \in J_\eta}
\tau_{i}([0,1]) < \infty$. It follows from
\eqref{jab-ju-st}, applied to $S_{i_0}$, that $\int_0^1
\sum_{l=1}^{\kappa+1} \frac 1{s^l} \D \tau_{i_0}(s) < 1$.
Employing Procedure \ref{proc-ch}, we get a bounded
completely hyperexpansive weighted shift $\slam$ on
$\tcal_{\eta,\kappa}$ with positive weights $\lambdab =
\{\lambda_{v}\}_{v \in V_{\eta,\kappa}^\circ}$ such that
$\tau_{i,1} = \tau_i$ for all $i \in J_\eta$. Hence
$\|\slam^n e_{i,1}\|^2 = \|S_i^n e_0\|^2$ for all $n \in
\zbb_+$ and $i \in J_\eta$, which implies \eqref{alla-ch}.
   \end{proof}
   \begin{cor} \label{cur-wn}
Let $\eta$, $\kappa$, $\{\alpha_{i,n}\}_{n=1}^\infty$ and
$S_i$ be as in Proposition {\em \ref{2curto-ch}}. If $S_i$
has a completely hyperexpansive $(\kappa+1)$-step backward
extension for every $i\in J_\eta$ and $\sup_{i\in J_\eta}
\|S_i\| < \infty$, then there exists a completely
hyperexpansive weighted shift $\slam \in
\ogr{\ell^2(V_{\eta,\kappa})}$ on $\tcal_{\eta,\kappa}$
with positive weights $\lambdab = \{\lambda_{v}\}_{v \in
V_{\eta,\kappa}^\circ}$ such that
$\alpha_{i,n}=\lambda_{i,n+1}$ for all $n\in \nbb$ and
$i\in J_\eta$.
   \end{cor}
   \begin{proof}
Apply \eqref{jab-ju-st} and Proposition \ref{2curto-ch}.
   \end{proof}
The converse of Corollary \ref{cur-wn} does not hold. In
fact, one can construct a completely hyperexpansive
weighted shift $\slam\in \ogr{\ell^2(V_{\eta,\kappa})}$ on
$\tcal_{\eta,\kappa}$ with positive weights such that the
set of all $i\in J_\eta$ for which $S_i$ has a completely
hyperexpansive $(\kappa+1)$-step backward extension
consists of one point (e.g., the required weighted shift
can be obtained by applying Procedure \ref{proc-ch} to the
measures $\{\tau_i\}_{i\in J_\eta}$ given by $\tau_1 = 0$
and $\tau_i=\delta_1$ for $i \neq 1$).
   \subsection{\label{s751}Graph extensions}
   It turns out that the direct counterpart of
Proposition \ref{maxsub} for complete hyperexpansivity
is no longer true (cf.\ Example \ref{ch-n0restr}). In
fact, the situation is now more complicated, and so we
have to make stronger assumptions.
   \begin{pro} \label{maxsub-ch}
Let $\tcal = (V,E)$ be a subtree of a directed tree
$\hat \tcal = (\hat V,\hat E)$ such that, for some $w
\in V \setminus \Ko{\tcal}$, $\dzit{\tcal}{w} \neq
\dzit{\hat \tcal}{w}$, $\dzit{\tcal}{\pa{w}} =
\dzit{\hat \tcal}{\pa{w}}$, and $\dest{\tcal}{v} =
\dest{\hat \tcal}{v}$ for all $v \in \dzit{\tcal}{w}
\cup \big(\dzit{\tcal}{\pa{w}} \setminus \{w\}\big)$.
Assume that $\slam \in \ogr{\ell^2(V)}$ is a
completely hyperexpansive weighted shift on $\tcal$
with nonzero weights $\lambdab = \{\lambda_u\}_{u \in
V^\circ}$. If $\slam$ satisfies one of the following
conditions{\em :}
   \begin{enumerate}
   \item[(i)] $w \in V \setminus
\big(\Ko{\tcal} \cup \dzi{\Ko{\tcal}}\big)$,
   \item[(ii)] $\slam$ satisfies the strong consistency
condition at $u=\pa{w}$, i.e., \eqref{consist-ch'} is
valid for $u = \pa{w}$,
   \end{enumerate}
then there exists no completely hyperexpansive weighted
shift $\slamh \in \ogr {\ell^2(\hat V)}$ on $\hat \tcal$
with nonzero weights $\hat \lambdab = \{\hat \lambda_u\}_{u
\in \hat V^\circ}$ such that $\lambdab \subseteq \hat
\lambdab$, i.e., $\lambda_u = \hat \lambda_u$ for all $u
\in V^\circ$.
   \end{pro}
   \begin{proof}
Applying Lemma \ref{charch2} to $u=\paa^2(w)$, we see that
(i) implies (ii). Assume that (ii) holds. Suppose that,
contrary to our claim, such an $\slamh$ exists. It follows
from Proposition \ref{chinj} that $\tcal$ and $\hat\tcal$
are leafless. Hence $\varnothing \neq \dzit{\tcal}{w}
\subsetneq \dzit{\hat\tcal}{w}$. Applying Lemma
\ref{charch2} to $u=\pa{w}$, we deduce that
$\tau_w^{\tcal}{(\nul)}=0$, which, again by Lemma
\ref{charch2} applied now to $u=w$, yields
   \begin{align} \label{jedyneczka-ch}
1=\sum_{v \in \dzit{\tcal}{w}} |\lambda_v|^2 \Big(1 -
\int_0^1 \frac 1 s\, \D \tau_v^\tcal(s)\Big).
   \end{align}
The same is true for $\slamh$. Since $\lambdab \subseteq
\hat \lambdab$ and $\tau_v^{\tcal} = \tau_v^{\hat\tcal}$
for all $v \in \dzit{\tcal}{w}$ (see the proof of
Proposition \ref{maxsub}), we have
   \allowdisplaybreaks
   \begin{align*}
1&\hspace{2ex}=\sum_{v \in \dzit{\hat\tcal}{w}}
|\hat\lambda_v|^2 \Big(1 - \int_0^1 \frac 1 s\, \D
\tau_v^{\hat\tcal}(s)\Big)
   \\
&\hspace{2ex} = \sum_{v \in \dzit{\tcal}{w}} |\lambda_v|^2
\Big(1 - \int_0^1 \frac 1 s\, \D \tau_v^\tcal(s)\Big)
   \\
& \hspace{5ex}+ \sum_{v \in \dzit{\hat\tcal}{w} \setminus
\dzit{\tcal}{w}} |\hat\lambda_v|^2 \Big(1 - \int_0^1 \frac
1 s\, \D \tau_v^{\hat\tcal}(s)\Big)
   \\
& \overset{\eqref{jedyneczka-ch}}= 1 + \sum_{v \in
\dzit{\hat\tcal}{w} \setminus \dzit{\tcal}{w}}
|\hat\lambda_v|^2 \Big(1 - \int_0^1 \frac 1 s\, \D
\tau_v^{\hat\tcal}(s)\Big),
   \end{align*}
which implies that
   \begin{align} \label{=0ch}
\sum_{v \in \dzit{\hat\tcal}{w} \setminus \dzit{\tcal}{w}}
|\hat\lambda_v|^2 \Big(1 - \int_0^1 \frac 1 s\, \D
\tau_v^{\hat\tcal}(s)\Big) = 0.
   \end{align}

We now turn to the second part of the proof. Applying
Lemma \ref{charch2} to $u=\pa{w}$ (as well as to both
operators $\slam$ and $\slamh$), and using the
assumption $\dzit{\tcal}{\pa{w}} = \dzit{\hat
\tcal}{\pa{w}}$, we deduce that
   \allowdisplaybreaks
   \begin{gather}       \label{=0ch'}
1 + \sum_{v \in \dzit{\tcal}{\pa w}} |\lambda_v|^2 \int_0^1
\frac 1 {s}\, \D \tau_v^{\tcal}(s) \overset{\mathrm{(ii)}}=
\sum_{v \in \dzit{\tcal}{\pa{w}}} |\lambda_v|^2
   \\
\overset{\substack{\lambdab \subseteq \hat \lambdab
\\[-.6ex]\phantom{.}}} = \sum_{v \in \dzit{\hat\tcal}{\pa{w}}} |\hat
\lambda_v|^2 \Ge 1+ \sum_{v \in \dzit{\hat\tcal}{\pa w}}
|\hat\lambda_v|^2 \int_0^1 \frac 1 {s}\, \D
\tau_v^{\hat\tcal}(s). \notag
   \end{gather}
It follows from our assumptions that
$\tau_v^{\tcal}=\tau_v^{\hat\tcal}$ for all $v \in
\dzit{\tcal}{\pa{w}} \setminus \{w\}$. Hence, by
\eqref{=0ch'} and $\lambdab \subseteq \hat \lambdab$,
we have
   \begin{align} \label{=0ch''}
\int_0^1 \frac 1 s \D \tau_w^{\tcal}(s) \Ge \int_0^1 \frac
1 s \D \tau_w^{\hat \tcal}(s).
   \end{align}
Applying Lemma \ref{charch2} to $u=w$ (recall that
$\tau_w^{\tcal}(\nul) = \tau_w^{\hat \tcal}(\nul)
=0$), we deduce from \eqref{muu-ch} that
   \begin{align*} %\label{=0ch'''}
&\sum_{v \in \dzit{\tcal}{w}} |\lambda_v|^2 \int_0^1
\frac 1 {s^{2}}\, \D \tau_v^{\tcal}(s) = \int_0^1
\frac 1 s \D \tau_w^{\tcal}(s)
   \\
   &\hspace{5ex} \overset{\eqref{=0ch''}} \Ge \int_0^1
\frac 1 s \D \tau_w^{\hat \tcal}(s) = \sum_{v \in
\dzit{\hat\tcal}{w}} |\hat\lambda_v|^2\int_0^1 \frac 1
{s^{2}}\, \D \tau_v^{\hat\tcal}(s)
   \\
& \hspace{5.5ex} \overset{\substack{\lambdab \subseteq
\hat \lambdab
\\[-.6ex]\phantom{.}}} = \sum_{v \in \dzit{\tcal}{w}}
|\lambda_v|^2 \int_0^1 \frac 1 {s^{2}}\, \D
\tau_v^{\tcal}(s) + \sum_{v \in \dzit{\hat \tcal}{w}
\setminus \dzit{\tcal}{w}} |\hat \lambda_v|^2 \int_0^1
\frac 1 {s^{2}}\, \D \tau_v^{\hat\tcal}(s),
   \end{align*}
which implies that
   \begin{align*} %\label{=0ch4'}
\sum_{v \in \dzit{\hat \tcal}{w} \setminus \dzit{\tcal}{w}}
|\hat \lambda_v|^2 \int_0^1 \frac 1 {s^{2}}\, \D
\tau_v^{\hat\tcal}(s)=0.
   \end{align*}
Since $\dzit{\hat \tcal}{w} \setminus \dzit{\tcal}{w} \neq
\varnothing$ and all the weights $\hat \lambda_v$ are
nonzero, we conclude that $\tau_v^{\hat\tcal}=0$ for all $v
\in \dzit{\hat \tcal}{w} \setminus \dzit{\tcal}{w}$, which
contradicts \eqref{=0ch}. This completes the proof.
   \end{proof}
Regarding Proposition \ref{maxsub-ch}, note that if
$\tcal=\tcal_{\eta,1}$ and $w=0$, then (ii) is equivalent
to assuming that $\slam$ is extremal.

We now show by example that the conclusion of
Proposition \ref{maxsub-ch} can fail if one of the
assumptions (i) or (ii) is not satisfied. In fact, we
give a method of constructing such examples. The
reader who is interested in a simple example may
consider the measures $\{\tau_i\}_{i \in J_{\hat
\eta}}$ given by $\tau_i=0$ for $i\in J_\eta$ and
$\tau_i=\delta_1$ for $i \in J_{\hat \eta} \setminus
J_\eta$.
   \begin{exa} \label{ch-n0restr}
Let $\eta, \hat \eta \in \{2,3, \ldots\} \sqcup
\{\infty\}$ be such that $\eta < \hat \eta$ and let
$\hat \kappa \in \zbb_+$. Set $\tcal=\tcal_{\eta,0}$,
$\tilde \tcal=\tcal_{\eta,\hat\kappa}$ and $\hat
\tcal=\tcal_{\hat \eta,\hat\kappa}$. Take a (finite)
sequence $\{\tau_i\}_{i=1}^\eta$ of positive Borel
measures on $[0,1]$ such that $\sup_{i \in J_\eta}
\tau_{i}([0,1]) < \infty$, $\int_0^1 \frac 1 {s^{\hat
\kappa + 1}} \D\tau_i(s) < \infty$ for all $i\in
J_\eta$ and $\int_0^1 \sum_{l=1}^{\hat\kappa+1} \frac
1{s^l} \D \tau_{i_0}(s) < 1$ for some $i_0 \in
J_\eta$. Applying Procedure \ref{proc-ch} to
$\tilde\tcal$, we deduce that there exists a
completely hyperexpansive weighted shift
$S_{\tilde{\lambdab}} \in
\ogr{\ell^2(V_{\eta,\hat\kappa})}$ on $\tilde\tcal$
with positive weights $\tilde{\lambdab} =
\{\tilde\lambda_v\}_{v \in V_{\eta,\hat\kappa}^\circ}$
such that $\tau_{i,1}^{\tilde\tcal}=\tau_i$ for all $i
\in J_\eta$. It follows from \eqref{ka+3}, applied to
$S_{\tilde{\lambdab}}$, that
   \begin{align} \label{dodsum}
\sum_{i=1}^\eta \tilde\lambda_{i,1}^2\Big(1-
\int_0^1 \sum_{l=1}^{\hat\kappa+1} \frac 1 {s^l}
\D \tau_i(s)\Big) > 0.
   \end{align}
Consider now a supplementary sequence $\{\tau_i\}_{i \in
J_{\hat\eta} \setminus J_\eta}$ of positive Borel measures
on $[0,1]$ such that
   \begin{align} \label{int0'}
& \int_0^1 \frac 1 {s^{\hat \kappa + 1}} \D\tau_i(s) <
\infty, \quad i\in J_{\hat\eta} \setminus J_\eta,
   \\
& \label{int0} \int_0^1 \frac 1 s \D\tau_i(s)=1, \quad
i\in J_{\hat\eta} \setminus J_\eta.
   \end{align}
Since $\int_0^1 \frac 1 s \D \tau_i(s) \Ge \tau_i([0,1])$,
we deduce from \eqref{int0} that $\sup_{i\in J_{\hat \eta}}
\tau_i([0,1]) < \infty$. By \eqref{dodsum} and
\eqref{int0'}, there exists a square summable sequence
$\{t_i\}_{i\in J_{\hat \eta} \setminus J_{\eta}}$ of
positive real numbers such that $\sum_{i=\eta+1}^{\hat\eta}
t_i^2\big|1- \int_0^1 \sum_{l=1}^{\hat\kappa+1} \frac 1
{s^l} \D \tau_i(s)\big| < \infty$ and
   \begin{align}  \label{ka+3-ch}
\sum_{i=1}^\eta \tilde\lambda_{i,1}^2\Big(1-
\int_0^1 \sum_{l=1}^{\hat\kappa+1} \frac 1 {s^l}
\D \tau_i(s)\Big) + \sum_{i=\eta+1}^{\hat\eta}
t_i^2\Big(1- \int_0^1 \sum_{l=1}^{\hat\kappa+1}
\frac 1 {s^l} \D \tau_i(s)\Big) > 0.
   \end{align}
Note that the sequence $\hat\teb:=\{\hat
t_i\}_{i=1}^{\hat\eta}$ given by $\hat t_i =
\tilde\lambda_{i,1}$ for $i \in J_\eta$ and $\hat t_i =
t_i$ for $i\in J_{\hat\eta} \setminus J_{\eta}$ is square
summable and (apply \eqref{ka+2} to $S_{\tilde\lambdab}$)
   \begin{align} \label{ka+2-ch}
\sum_{i=1}^{\hat\eta} \hat t_i^2 \Big(1 - \int_0^1 \frac 1
s \D \tau_i(s)\Big) \overset{\eqref{int0}}=
\sum_{i=1}^{\eta} \tilde\lambda_{i,1}^2 \Big(1 - \int_0^1
\frac 1 s \D \tau_i(s)\Big)
   \begin{cases}
\Ge 1 & \text{if } \hat\kappa = 0,
   \\[1ex]
=1 & \text{if } \hat \kappa \Ge 1.
   \end{cases}
   \end{align}
In view of\,\footnote{\;The reader should be aware of
the fact that $\sum_{i=1}^{\hat\eta} \hat t_i^2
\int_0^1 \sum_{l=1}^{\hat\kappa+1} \frac 1 {s^l} \D
\tau_i(s) < \infty$.} \eqref{ka+3-ch} and
\eqref{ka+2-ch}, we can apply Theorem \ref{par-ch} to
the directed tree $\hat\tcal$ and to the sequences
$\{\tau_i\}_{i\in J_{\hat \eta}}$ and $\hat\teb$. In
this way we obtain a completely hyperexpansive
weighted shift $\slamh \in
\ogr{\ell^2(V_{\hat\eta,\hat\kappa})}$ on $\hat\tcal$
(which is extremal when $\hat \kappa \Ge 1$) with
positive weights $\hat\lambdab= \{\hat\lambda_v\}_{v
\in V_{\hat\eta,\hat\kappa}^\circ}$ such that
   \begin{align} \label{tilde-ch}
\text{$\hat\lambda_{i,1}=\hat t_i$ and
$\tau_{i,1}^{\hat\tcal} = \tau_i$ for all $i \in
J_{\hat\eta}$.}
   \end{align}

We are now ready to show that the conclusion of
Proposition \ref{maxsub-ch} may not hold if $w \in
\Ko\tcal$. Indeed, applying Theorem \ref{par-ch} to
$\tcal_{\eta,0}$, $\{\tau_i\}_{i=1}^\eta$ and $\teb
=\{\tilde\lambda_{i,1}\}_{i=1}^\eta$, we get a
completely hyperexpansive weighted shift $\slam \in
\ogr{\ell^2(V_{\eta,0})}$ on $\tcal$ with positive
weights $\lambdab=\{\lambda_v\}_{v\in
V_{\eta,0}^\circ}$ such that $\tau_{i,1}^\tcal=\tau_i$
and $\lambda_{i,1}=\tilde\lambda_{i,1}$ for all $i \in
J_\eta$. Since, by \eqref{tilde-ch}, $\lambda_{i,1} =
\hat\lambda_{i,1}$ and
$\tau_{i,1}^\tcal=\tau_{i,1}^{\hat\tcal}$ for all
$i\in J_\eta$, we deduce from \eqref{lamij-ch} that
$\lambdab \subseteq \hat\lambdab$.

Next example shows that the direct counterpart of
Proposition \ref{maxsub} for complete hyperexpansivity
breaks down when $w \in \dzi{\Ko{\tcal}}$. Moreover,
it exhibits that Proposition \ref{maxsub-ch} is no
longer true if $\slam$ does not satisfy the strong
consistency condition at $u = \pa w$ (even though
$S_{\hat\lambdab}$ satisfies the strong consistency
condition at $u = \pa w$). For this purpose, we
consider the pair $(\tilde\tcal,\hat\tcal)$ with
$\hat\kappa=1$, i.e., $\tilde\tcal = \tcal_{\eta,1}$
and $\hat\tcal=\tcal_{\hat\eta,1}$. Let $S_{\tilde
\lambdab}$ and $S_{\hat\lambdab}$ be as in the
penultimate paragraph. Define the new system
$\lambdab^\flat = \{\lambda_v^\flat\}_{v \in
V_{\eta,1}^\circ}$ of positive weights by modifying
the old one $\tilde{\lambdab}$ as follows:
$\lambda_v^\flat = \tilde\lambda_v$ for all $v\neq 0$
and $\lambda_0^\flat = \hat\lambda_0$. Since
$S_{\hat\lambdab}$ is extremal, we infer from
\eqref{zeta-ch'} and \eqref{tilde-ch} that
   \begin{align*}
\lambda_0^\flat & = \frac 1{\sqrt{\sum_{i=1}^{\hat\eta}
\hat\lambda_{i,1}^2 \Big(1- \int_0^1 \sum_{l=1}^{2} \frac 1
{s^l} \D \tau_i(s)\Big)}}
   \\
   & = \frac 1{\sqrt{\sum_{i=1}^{\eta}
\tilde\lambda_{i,1}^2 \Big(1- \int_0^1 \sum_{l=1}^{2} \frac
1 {s^l} \D \tau_i(s)\Big) - \sum_{i=\eta+1}^{\hat\eta}
t_i^2 \Big(\int_0^1 \sum_{l=1}^{2} \frac 1 {s^l} \D
\tau_i(s) - 1\Big)}}
   \\
& > \frac 1{\sqrt{\sum_{i=1}^{\eta} \tilde\lambda_{i,1}^2
\Big(1- \int_0^1 \sum_{l=1}^{2} \frac 1 {s^l} \D
\tau_{i,1}^{\tilde\tcal}(s)\Big)}},
   \end{align*}
where the last inequality is a consequence of the following
estimate
   \begin{align*}
1 \overset{\eqref{int0}}= \int_0^1 \frac 1s \D \tau_i (s) <
\int_0^1 \sum_{l=1}^2 \frac 1{s^l} \D \tau_i (s), \quad
i\in J_{\hat\eta} \setminus J_\eta.
   \end{align*}
Hence, by Theorem \ref{par-ch}, $S_{\lambdab^\flat} \in
\ogr{\ell^2(V_{\eta,1})}$ is a completely hyperexpansive
weight\-ed shift on $\tilde\tcal$ with positive weights
$\lambdab^\flat$ such that $\lambdab^\flat \subseteq
\hat\lambdab$. Certainly, $S_{\lambdab^\flat}$ is not
extremal, or equivalently $S_{\lambdab^\flat}$ does not
satisfy the strong consistency condition at $u=\pa{w}=-1$.
However, $S_{\hat\lambdab}$ does satisfy the strong
consistency condition at $u = \pa w$.
   \end{exa}
   \newpage
   \section{\label{ch8}Miscellanea}
   \subsection{Admissibility of assorted
weighted shifts} In this section we characterize
directed trees admitting weighted shifts with assorted
properties. To be more precise, a directed tree
$\tcal$ is said to {\em admit} a weighted shift with a
property $\mathcal P$ if there exists a weighted shift
on $\tcal$ with this property. First, we describe
directed trees admitting weighted shifts with dense
range. For this, we prove the following lemma (see
Remarks \ref{re1-2} and \ref{surp} for the definitions
of directed trees $\zbb$ and $\zbb_-$).
   \begin{lem} \label{denran0}
If $\slam$ is a densely defined weighted shift on
a directed tree $\tcal$ with weights $\lambdab =
\{\lambda_v\}_{v \in V^\circ}$, then the
following two conditions are equivalent\/{\em :}
   \begin{enumerate}
   \item[(i)] $\ob{\slam}$ is dense in $\ell^2(V)$,
   \item[(ii)] the directed tree $\tcal$ is
isomorphic either to $\zbb_-$ or to $\zbb$, and
$\lambda_u \neq 0$ for all $u \in V$.
   \end{enumerate}
   \end{lem}
   \begin{proof}
(i)$\Rightarrow$(ii) It follows from Proposition
\ref{polar}\,(ii) that the directed tree $\tcal$
is rootless and $\dim\big(\ell^2(\dzi u) \ominus
\langle \lambdab^u \rangle\big)=0$ for all $u \in
V^\prime$. It is a matter of routine to verify
that the latter implies that
   \begin{align} \label{ca1}
\card{\dzi{u}}=1, \quad u \in V^\prime,
   \end{align}
and $\lambdab^u \neq 0$ for all $u \in V^\prime$.
Since $\tcal$ is rootless, we deduce that $\lambda_v
\neq 0$ for all $v \in V$. Fix $w \in V$. By
\eqref{ca1} and Proposition \ref{xdescor}\,(i) and
(iv), the entries of the sequence
$\{\paa^k(w)\}_{k=1}^\infty$ are distinct and $V =
\{\paa^k(w)\}_{k=1}^\infty \sqcup \des{w}$. This,
together with \eqref{decom} and \eqref{ca1}, implies
that either $\dzin{n}w \neq \varnothing$ for all $n\in
\nbb$ and consequently $\tcal$ is isomorphic to
$\zbb$, or there exists a smallest positive integer
$n$ such that $\dzin{n}w = \varnothing$ and
consequently $\tcal$ is isomorphic to $\zbb_-$.

(ii)$\Rightarrow$(i) Evident due to \eqref{eu}.
   \end{proof}
By Lemma \ref{denran0}, the only densely defined
weighted shifts on directed trees with dense
range are either bilateral classical weighted
shifts with nonzero weights or the adjoints of
unilateral classical weighted shifts with nonzero
weights.
   \begin{pro}
If $\tcal$ is a directed tree, then the following
conditions are equivalent\/{\em :}
   \begin{enumerate}
   \item[(i)] $\tcal$ admits a densely defined
weighted shift with dense range,
   \item[(ii)] the directed tree $\tcal$ is isomorphic
either to $\zbb_-$ or to $\zbb$.
   \end{enumerate}
Moreover, if $($i$)$ holds, then each densely defined
weighted shift on $\tcal$ with nonzero weights has dense
range.
   \end{pro}
   \begin{proof}
(i)$\Rightarrow$(ii) Apply Lemma \ref{denran0}.

(ii)$\Rightarrow$(i) Consider the bounded
weighted shift on $\zbb_-$ (or on $\zbb$) with
$\lambda_u\equiv 1$.
   \end{proof}
The question of when a directed tree admits a weighted
shift which is respectively hyponormal, subnormal and
completely hyperexpansive, has a simple answer.
   \begin{pro}\label{adhyp}
If $\tcal$ is a directed tree with $V^\circ \neq
\varnothing$, then the following conditions are
equivalent{\em :}
   \begin{enumerate}
   \item[(i)] $\tcal$ admits a bounded hyponormal weighted
shift with nonzero weights,
   \item[(ii)] $\tcal$ admits a bounded subnormal weighted
shift with nonzero weights,
   \item[(iii)] $\tcal$ admits an isometric weighted
shift with nonzero weights,
   \item[(iv)] $\tcal$ admits a bounded completely hyperexpansive
weighted shift with nonzero weights,
   \item[(v)] $\tcal$ is leafless and $\card{V} =
\aleph_0$,
   \item[(vi)] $\tcal$ admits a bounded injective weighted
shift with nonzero weights.
   \end{enumerate}
   \end{pro}
   \begin{proof}
(i)$\Rightarrow$(v) Apply Proposition \ref{hypcor}.

(v)$\Rightarrow$(iii)\&(v)$\Rightarrow$(vi) Argue as in the
proof of Proposition \ref{przeldz} and use Corollary
\ref{chariso}.

(iv)$\Rightarrow$(v) Apply Propositions \ref{przeldz} and
\ref{chinj}.

Since the implications (iii)$\Rightarrow$(iv),
(iii)$\Rightarrow$(ii) and (ii)$\Rightarrow$(i) are
obvious, and the implication (vi)$\Rightarrow$(v) is a
consequence of Propositions \ref{dzisdesz} and
\ref{przeldz}, the proof is complete.
   \end{proof}
Proposition \ref{adhyp} fails to hold if the requirement of
nonzero weights is dropped. Indeed, if $\tcal$ is a
directed tree which comes from $\zbb_+$ by gluing a leaf to
the directed tree $\zbb_+$ at its root, and $\slam$ is the
weighted shifts on $\tcal$ with $0$ weight attached to the
glued leaf, the remaining weights being equal to $1$, then
$\slam$ is subnormal, but $\tcal$ is not leafless.

As stated below, admissibility of adjoints of isometric (in
short:\ {\em coisometric}) weight\-ed shifts is much more
restrictive.
   \begin{pro}\label{adcoiso}
If $\tcal$ is a directed tree, then the following
assertions hold.
   \begin{enumerate}
   \item[(i)] $\tcal$ admits a coisometric weighted
shift if and only if the directed tree $\tcal$ is
isomorphic either to $\zbb_-$ or to $\zbb$; moreover,
if $\slam$ is a coisometric weighted shift on $\tcal$,
then all its weights are nonzero.
   \item[(ii)] $\tcal$ admits a unitary weighted shift
if and only if $\tcal$ is isomorphic to $\zbb$.
   \end{enumerate}
   \end{pro}
   \begin{proof}
The assertion (i) is a direct consequence of Lemma
\ref{denran0} because each coisometry is surjective
(consult also Remark \ref{surp}). In turn, the assertion
(ii) follows from (i) because any bounded weighted shift on
$\zbb_-$ is not injective.
   \end{proof}
Using Theorem \ref{cohyp-opis}, one can construct
bounded cohyponormal weighted shifts on directed trees
with nonzero weights which are non-injective and
non-coisometric. Most of the model trees appearing in
Theorem \ref{cohyp-opis}\,(ii) are far from being
isomorphic to the directed trees $\zbb_-$ and $\zbb$.
Hence, by Lemma \ref{denran0}, weighted shifts on
these model trees (except for $\zbb_-$ and $\zbb$) do
not have dense range.

We now discuss the question of when a given directed
tree admits a bounded normal weighted shift with
nonzero weights.
   \begin{lem}\label{normal}
If $\slam \in \ogr {\ell^2(V)}$ is a nonzero
weighted shift on a directed tree $\tcal$ with
weights $\lambdab = \{\lambda_v\}_{v \in
V^\circ}$, then the following two conditions are
equivalent\/{\em :}
   \begin{enumerate}
   \item[(i)] $\slam$ is normal,
   \item[(ii)] there exists a sequence
$\{u_{n}\}_{n=-\infty}^{\infty} \subseteq V$ such
that
   \begin{align*}
\text{$u_{n-1} = \pa {u_n}$ and
$|\lambda_{u_{n-1}}|= |\lambda_{u_{n}}|$}
   \end{align*}
for all $n \in \zbb$, and $\lambda_v = 0$ for all
$v \in V \setminus {\{u_{n}\colon n\in\zbb\}}$.
   \end{enumerate}
   \end{lem}
   \begin{proof}
(i)$\Rightarrow$(ii) Note first that if $u\in
V^\circ$, then by \eqref{eu} and \eqref{sl*} we
have
   \begin{align}    \label{normal1}
\|\slam e_u\|^2 e_u = \slam^* \slam e_u = \slam
\slam^* e_u = |\lambda_u|^2 e_u + \sum_{v \in
\dzi{\pa{u}} \setminus \{u\}} \lambda_v
\overline{\lambda_u} e_v.
   \end{align}
Hence, if $\|\slam e_u\|=0$ for some $u \in
V^\circ$, then $\lambda_u=0$. This, combined with
Theorem \ref{cohyp-opis} and \eqref{normal1},
establishes the implication (i)$\Rightarrow$(ii)
(observe that the situation described in Theorem
\ref{cohyp-opis}\,(ii) is excluded because it
forces $\slam$ to be the zero operator).

(ii)$\Rightarrow$(i) Argue as in \eqref{normal1}.
   \end{proof}
   \begin{pro}
If $\tcal$ is a directed tree with $V^\circ \neq
\varnothing$, then the following two conditions are
equivalent\/{\em :}
   \begin{enumerate}
   \item[(i)] $\tcal$ admits a bounded normal
weighted shift with nonzero weights,
   \item[(ii)] the directed tree $\tcal$ is
isomorphic to $\zbb$.
   \end{enumerate}
   \end{pro}
   \begin{proof}
(i)$\Rightarrow$(ii) Apply Lemma \ref{normal}.

(ii)$\Rightarrow$(i) Obvious.
   \end{proof}
It turns out that quasinormal weighted shifts on directed
trees with nonzero weights are scalar multiplies of
isometric operators. Recall that an operator $A \in \ogr
\hh$ acting on a complex Hilbert space $\hh$ is said to be
{\em quasinormal} if $A|A|=|A|A$, or equivalently if
$AA^*A=A^*AA$. It is well known that normal operators are
quasinormal and quasinormal operators are subnormal, but
neither of these implications is reversible in general
(cf.\ \cite{con2}).
   \begin{pro}\label{qn-ch}
Let $\slam \in \ogr{\ell^2(V)}$ be a weighted
shift on a directed tree $\tcal$ with weights
$\lambdab = \{\lambda_v\}_{v \in V^\circ}$. Then
the following conditions are equivalent\/{\em :}
   \begin{enumerate}
   \item[(i)] $\slam$ is quasinormal,
   \item[(ii)] $\|\slam
e_u\|=\|\slam e_v\|$ for all $u \in V$ and $v \in
\dzi u$ such that $\lambda_v \neq 0$.
   \end{enumerate}
Moreover, if $V^\circ \neq \varnothing$ and $\lambda_v \neq
0$ for all $v\in V^\circ$, then $\slam$ is quasinormal if
and only if $\|\slam\|^{-1}\slam$ is an isometry.
   \end{pro}
   \begin{proof}
It follows from Proposition \ref{3} that
   \begin{align*}
\slam(\slam^*\slam) e_u = \|\slam e_u\|^2 \slam
e_u \overset{\eqref{eu}}= \sum_{v \in \dzi u}
\|\slam e_u\|^2 \lambda_v e_v, \quad u \in V,
   \end{align*}
and
   \begin{align*}
(\slam^*\slam)\slam e_u = \sum_{v \in \dzi u}
\lambda_v (\slam^*\slam) e_v = \sum_{v \in \dzi
u} \|\slam e_v\|^2 \lambda_v e_v, \quad u \in V.
   \end{align*}
Putting this all together completes the proof of
the equivalence (i)$\Leftrightarrow$(ii).

Suppose now that $\slam$ is quasinormal, $V^\circ \neq
\varnothing$ and $\lambda_v \neq 0$ for all $v\in V^\circ$.
First, we claim that $\|\slam e_u\| = \text{const}$. For
this, take $u \in V$. Using an induction argument and the
implication (i)$\Rightarrow$(ii), we see that $\|\slam
e_v\| = \|\slam e_u\|$ for all $v \in \dzin{n}u$ and $n\in
\zbb_+$. In view of \eqref{decom}, this implies that
$\|\slam e_v\| = \|\slam e_u\|$ for all $v \in \des{u}$. An
application of Proposition \ref{witr} proves our claim.
Hence, by \eqref{eu} and Corollary \ref{chariso}, the
operator $\|\slam\|^{-1}\slam$ is an isometry. The reverse
implication is obvious.
   \end{proof}
Note that if $\tcal$ is a directed tree (with or without
root) such that
   \begin{align*}
   1 \Le \card{\dzi{u}}=\card{\dzi{v}} < \infty, \quad u,v
\in V,
   \end{align*}
then the weighted shift $\slam$ on $\tcal$ with weights
$\lambda_v=\text{const}$ is bounded and quasinormal.
   \subsection{$p$-hyponormality}
Recall that an operator $A \in \ogr{\hh}$ acting on a
complex Hilbert space $\hh$ is said to be {\em
$p$-hyponormal}, where $p$ is a positive real number,
if $|A^*|^{2p} \Le |A|^{2p}$. By the L\"{o}wner-Heinz
inequality, for all positive real numbers $p,q$ such
that $p<q$, if $A\in \ogr{\hh}$ is $q$-hyponormal,
then $A$ is $p$-hyponormal (see \cite{xia} and
\cite{furr} for more information on the subject).
Clearly, the notions of $1$-hyponormality and
hyponormality coincide. This means that the following
characterization of $p$-hyponormality can be thought
of as a generalization of Theorem \ref{hyp}.
    \begin{thm} \label{p-hyp}
Let $\slam \in \ogr{\ell^2(V)}$ be a weighted shift on a
directed tree $\tcal$ with weights $\lambdab =
\{\lambda_v\}_{v \in V^\circ}$, and let $p$ be a positive
real number. Then the following assertions are
equivalent{\em :}
    \begin{enumerate}
    \item[(i)] $\slam$ is $p$-hyponormal,
    \item[(ii)] the following two conditions
hold\/{\em :}
    \begin{gather}
\text{for every $u \in V$, if } v \in \dzi u \text{
and } \|\slam e_v\|=0, \text{ then } \lambda_v = 0,
\label{p-wkwhyp0}
    \\
\|\slam e_u\|^{2(p-1)}\sum_{v \in \dziplus u} \frac
{|\lambda_v|^2}{\|\slam e_v\|^{2p}} \Le 1, \quad u \in
\vplus. \label{p-wkwhyp}
    \end{gather}
    \end{enumerate}
    \end{thm}
Note that if $\slam \in \ogr{\ell^2(V)}$ is a
$p$-hyponormal weighted shift on a directed tree
$\tcal$, then for every $u \in \vplus$, the left-hand
side of the inequality \eqref{p-wkwhyp} never
vanishes.
   \begin{proof}[Proof of Theorem  \ref{p-hyp}]
We make use of some ideas from the proof of Theorem
\ref{hyp}. Let $\slam = U |\slam|$ be the polar
decomposition of $\slam$. It follows from Propositions
\ref{sprz}\,(iii) and \ref{polar} that for every $f \in
\ell^2(V)$,
   \begin{align} \label{u*}
(U^*f)(u) =
   \begin{cases}
   \frac{1}{\|\slam e_u\|}\sum_{v \in \dzi{u}}
\overline{\lambda_v}f(v) & \text{for } u \in
\vplus,
   \\[1ex]
   0 & \text{for } u \in V \setminus \vplus.
   \end{cases}
   \end{align}
Since $|\slam^*|^{2p} = U |\slam|^{2p} U^*$ (cf.\
\cite[Theorem 4 in \S2.2.2]{furr}), we deduce
from Proposition \ref{3}\,(iv) that
   \begin{align} \label{msl*2}
   \begin{aligned}
   \is{|\slam^*|^{2p}f}{f} &
\hspace{1.5ex}= \is{|\slam|^{2p}U^*f}{U^*f}
   \\
& \hspace{1.5ex} = \sum_{u\in V}
(|\slam|^{2p}U^*f)(u) \overline{(U^*f)(u)}
   \\
& \hspace{1.5ex} = \sum_{u\in V} \|\slam
e_u\|^{2p} |(U^*f)(u)|^2
   \\
& \overset{\eqref{u*}}= \sum_{u\in \vplus}
\|\slam e_u\|^{2(p-1)} \Big|\sum_{v \in \dzi{u}}
\overline{\lambda_v}f(v)\Big|^2, \quad f \in
\ell^2(V).
   \end{aligned}
   \end{align}
Similar reasoning leads to
   \begin{align}  \label{msl2}
   \begin{aligned}
   \is{|\slam|^{2p}f}{f} & = \sum_{u\in V}
\|\slam e_u\|^{2p} |f(u)|^{2}
   \\
&\overset{\eqref{sumchi}} = \|\slam
e_{\koo}\|^{2p} |f(\koo)|^{2}
   \\
& \hspace{15ex} + \sum_{u\in V} \sum_{v \in
\dziplus u} \|\slam e_v\|^{2p} |f(v)|^{2}, \quad
f \in \ell^2(V),
   \end{aligned}
   \end{align}
where the term $\|\slam e_{\koo}\|^{2p}
|f(\koo)|^{2}$ appears in \eqref{msl2} only if
$\tcal$ has a root. Combining \eqref{msl*2} with
\eqref{msl2} (see also the proof of Theorem
\ref{slssl*}), we deduce that $\slam$ is
$p$-hyponormal if and only if
   \begin{align} \label{p-hyp-nier}
   \begin{aligned}
\sum_{u\in \vplus} \|\slam e_u\|^{2(p-1)} &
\Big|\sum_{v \in \dzi{u}}
\overline{\lambda_v}f(v)\Big|^2
   \\
& \Le \sum_{u\in V} \sum_{v \in \dziplus u}
\|\slam e_v\|^{2p} |f(v)|^{2}, \quad f \in
\ell^2(V).
   \end{aligned}
   \end{align}

Suppose that $\slam$ is $p$-hyponormal. If $v \in
\dzi u$ is such that $\|\slam e_v\|=0$, then by
substituting $f = e_v$ into \eqref{p-hyp-nier} we
obtain $\lambda_v=0$, which proves
\eqref{p-wkwhyp0} (note that if $u \in V
\setminus \vplus$, then automatically
$\lambda_v=0$). In view of \eqref{p-wkwhyp0} and
\eqref{p-hyp-nier}, we see that for every $u \in
\vplus$,
   \begin{align*}
\|\slam e_u\|^{2(p-1)} \Big|\sum_{v \in \dziplus
u} \overline{\lambda_v}f(v)\Big|^2 \Le \sum_{v
\in \dziplus u} \|\slam e_v\|^{2p} |f(v)|^{2},
\quad f \in \ell^2(\dziplus u).
   \end{align*}
This implies \eqref{p-wkwhyp} (consult the part of the
proof of Theorem \ref{slssl*} which comes after the
inequality \eqref{in2}). It is a simple matter to verify
that the above reasoning can be reversed. This completes
the proof.
   \end{proof}
The following well known fact is a direct consequence
of Theorem \ref{p-hyp}.
   \begin{cor} \label{p-hyp-clas}
Let $p\in (0,\infty)$. A bounded unilateral or
bilateral classical weighted shift $S$ with nonzero
weights is $p$-hyponormal if and only if it is
hyponormal.
   \end{cor}
   \begin{proof}
The inequalities \eqref{p-wkwhyp} are easily seen to
be equivalent to the fact that the moduli of weights
of $S$ form a monotonically increasing sequence, which
in turn is equivalent to the hyponormality of $S$.
   \end{proof}
Theorem \ref{p-hyp} provides us with a handy
characterization of the $p$-hyponormality of weighted
shifts on the directed tree $\tcal_{\eta,\kappa}$
defined in \eqref{varkappa}.
   \begin{cor} \label{another-ch}
Let $\eta \in \{2,3, \ldots\} \sqcup \{\infty\}$, $\kappa
\in \zbb_+ \sqcup \{\infty\}$ and $p \in (0,\infty)$. A
weighted shift $\slam \in \ogr{\ell^2(V_{\eta,\kappa})}$ on
$\tcal_{\eta,\kappa}$ with nonzero weights $\lambdab =
\{\lambda_v\}_{v\in V_{\eta,\kappa}^\circ}$ is
$p$-hyponormal if and only if $\slam$ satisfies the
following conditions\/{\em:}
   \begin{align*}
& |\lambda_{i,j}| \Le |\lambda_{i,j+1}| \text{ for all } i
\in J_\eta \text{ and } j \Ge 2,
   \\
& \big(\sum_{i=1}^\eta |\lambda_{i,1}|^2\big)^{p-1}
\Big(\sum_{i=1}^\eta
\frac{|\lambda_{i,1}|^2}{|\lambda_{i,2}|^{2p}}\Big) \Le 1,
   \\
& |\lambda_0|^2 \Le \sum_{i=1}^\eta |\lambda_{i,1}|^2,
\text{ provided $\kappa \Ge 1$,}
   \\
& |\lambda_{-(k+1)}| \Le |\lambda_{-k}| \text{ for } k=0,
\ldots, \kappa -2, \text{ provided $\kappa \Ge 2$.}
   \end{align*}
   \end{cor}
We now show how to separate $p$-hyponormality classes
with weighted shifts on the directed tree
$\tcal_{2,1}$ (see \cite{JLP,JLL} and \cite{b-j-l} for
analogous results for weighted shifts with special
matrix weights and composition operators,
respectively).
   \begin{exa} \label{p-hyp-sep}
Let $a,b$ be positive real numbers. Consider a weighted
shift $\slam$ on $\tcal_{2,1}$ with weights $\lambdab =
\{\lambda_v\}_{v\in V_{2,1}^\circ}$ such that $\lambda_0
\in (0,\infty)$, $\lambda_{1,1}=\lambda_{2,1}=1/\sqrt 2$
and $\lambda_{1,j}=1/a$, $\lambda_{2,j}=1/b$ for $j=2,3,
\dots$ By Corollary \ref{parc}, $\slam \in
\ogr{\ell^2(V_{2,1})}$. It follows from Corollary
\ref{another-ch} that
   \begin{align} \label{ap+bp}
\text{$\slam$ is $p$-hyponormal if and only if $\lambda_0
\Le 1$ and $(a,b) \in \varDelta_p$,}
   \end{align}
where $\varDelta_p = \{(x,y)\in \rbb^2 \colon x,y
> 0, \, x^{2p} + y^{2p} \Le 2\}$. Observe that the set
$\varDelta_p$ consists of all points of the first open
quarter of the plane which lie on or below the graph
of the function $x \mapsto
\sqrt[\uproot{2}2p]{2-x^{2p}}$ (see Figure $7$). By
more or less elementary calculations, one can verify
that $\varDelta_q \varsubsetneq \varDelta_p$ for all
$p,q\in (0,\infty)$ such that $p<q$. What is more, if
$0 < p < q$, then $(1,1)$ is the only point of
$\varDelta_q$ which is in the topological boundary of
$\varDelta_p$. One can also check that
   \begin{align} \label{infty-ph}
\varDelta_\infty := \bigcap_{p>0} \varDelta_p & =
\{(x,y)\in \rbb^2\colon 0 < x \Le 1, \, 0< y \Le 1\},
   \\
\varDelta_0 := \bigcup_{p>0} \varDelta_p & = \{(x,y)\in
\rbb^2\colon x,y >0, \, xy < 1\} \cup \{(1,1)\}. \notag
   \end{align}
The sets $\varDelta_p$ are plotted in Figure 7 for
some choices of $p$\hspace{.1ex}; the most external
one corresponds to $p=0$, while the most internal to
$p=\infty$.
   \vspace{1.5ex}
   \begin{center}
   \includegraphics[width=9cm,height=8.7cm]
   %{p-hyp.png}
   {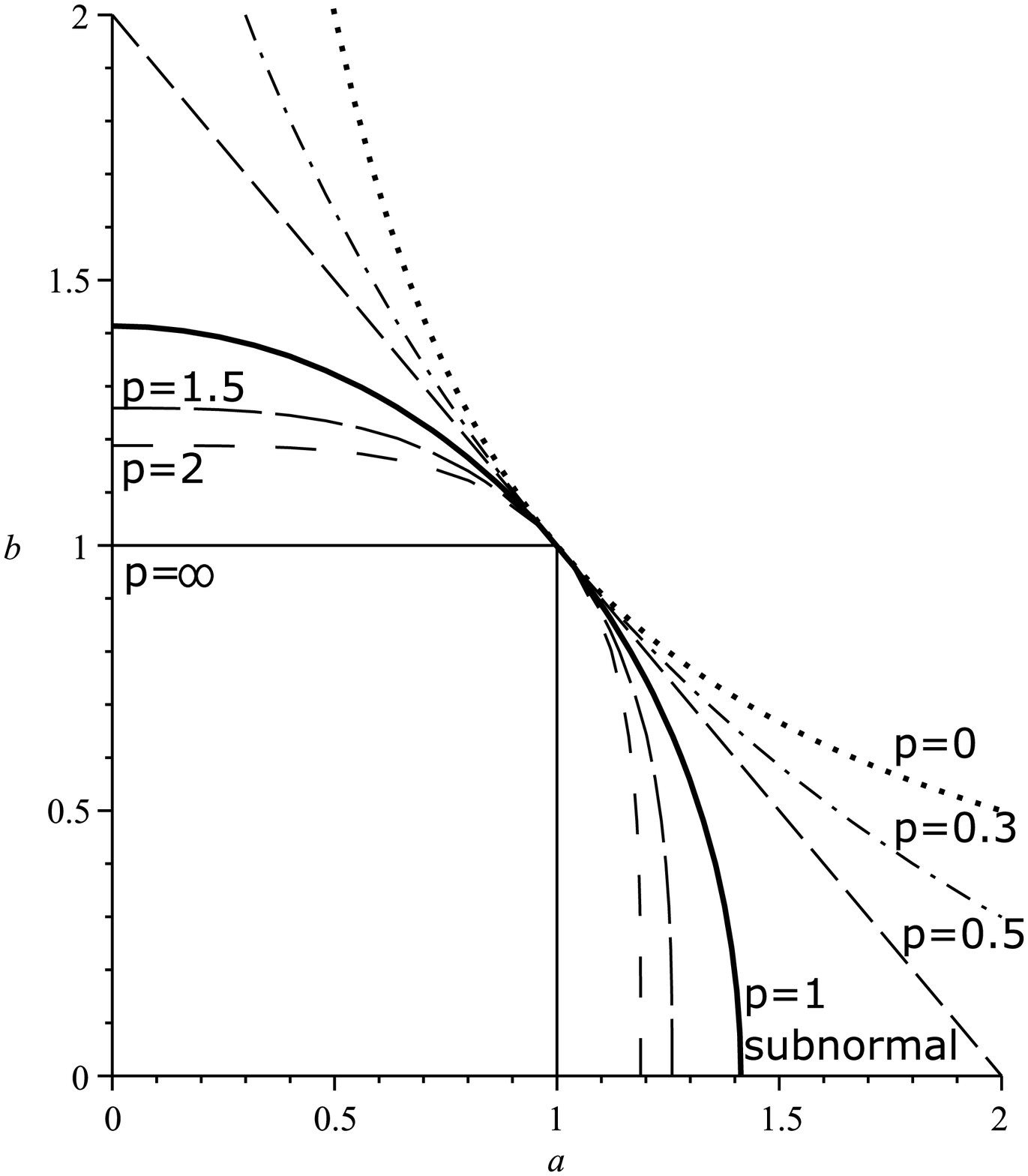}
   \\
   {\small {\sf Figure $7$}}
   \end{center}
   \vspace{1ex} By \eqref{ap+bp} and \eqref{infty-ph},
the operator $\slam$ is $\infty$-hyponormal (i.e.,
$p$-hyponormal for all $p\in (0,\infty)$) if and only
if $\lambda_0,a,b\Le 1$. Owing to Proposition
\ref{isokappa}, $\slam$ is an isometry if and only if
$\lambda_0=a=b=1$. Subnormality of $\slam$ can also be
described in terms of the parameters $a$, $b$ and
$\lambda_0$. Namely, applying Corollary
\ref{omega2}\,(ii) to $\mu_1=\delta_{1/a^2}$ and
$\mu_2=\delta_{1/b^2}$, we deduce that
   \begin{align} \label{ap+bp-sub}
\slam \text{ is subnormal if and only if } \frac {a^4 +
b^4}2 \Le \frac 1 {\lambda_0^2} \text{ and } a^2 + b^2 = 2.
   \end{align}
Fix now any real $\lambda_0$ such that $0 < \lambda_0 \Le
\frac 1{\sqrt{2}}$. Since $x^2 + y^2 = 2$ implies $\frac
{x^4 + y^4}2 < \frac 1 {\lambda_0^2}$ whenever $x,y>0$, we
deduce from \eqref{ap+bp} and \eqref{ap+bp-sub} that
$\slam$ is $p$-hyponormal if and only if $(a,b) \in
\varDelta_p$, and $\slam$ is subnormal if and only if $a^2+
b^2=2$. In view of the above discussion, if $(a,b) \neq
(1,1)$, then the operator $\slam$ is simultaneously
subnormal and $p$-hyponormal if and only if $0 < p \Le 1$
and $a^2 + b^2 = 2$. What is more, if $(a,b) \in
\varDelta_\infty \setminus \{(1,1)\}$, then $\slam$ is
$\infty$-hyponormal but not subnormal. On the other hand,
if $a^2 + b^2 = 2$ and $(a,b) \neq (1,1)$, then $\slam$ is
subnormal but not $\infty$-hyponormal (see \cite[Examples
3.2 and 3.3]{b-j-l} for the case of composition operators).
   \end{exa}

\vspace{3ex}

   \textbf{Acknowledgement}. A substantial part of
this paper was written while the first and the third
authors visited Kyungpook National University during
the springs and the autumns of 2008 and 2009; they
wish to thank the faculty and administration of this
unit for their warm hospitality.
   \newpage
   \bibliographystyle{amsalpha}
   
   \newpage
   \renewcommand\baselinestretch{1.3}
   \printindex
   \end{document}